\documentclass[11pt,a4paper]{article}
 \usepackage[top=0.9in, bottom=0.9in, left=0.9in, right=0.9in]{geometry}
 \usepackage{hyperref}
\hypersetup{hidelinks}
\usepackage[english]{babel}
\usepackage{bbold}
\usepackage{amsmath,amssymb,enumerate,latexsym,theorem}
\usepackage{graphicx} 
\usepackage{subcaption}
\numberwithin{equation}{section}
\allowdisplaybreaks
\title{Nonuniqueness of Leray--Hopf solutions to the unforced incompressible 3D Navier--Stokes
Equation}
\author{Thomas Y. Hou, Yixuan Wang, Changhe Yang}
\date{October 30, 2026}
\newcommand{\Beta}{\mathrm{B}}
\newcommand{\Epsilon}{\mathrm{E}}

\newcommand{\assign}{:=}
\newcommand{\backassign}{=:}
\newcommand{\cdummy}{\cdot}

\newcommand{\infixor}{\text{ or }}
\newcommand{\mathd}{\mathrm{d}}
\newcommand{\nospace}{}
\newcommand{\nosymbol}{}
\usepackage[ruled,vlined]{algorithm2e}

\newcommand{\tmop}[1]{\ensuremath{\operatorname{#1}}}

\newcommand{\tmstrong}[1]{\textbf{#1}}

\newenvironment{enumeratenumeric}{\begin{enumerate}[1.] }{\end{enumerate}}
\newenvironment{itemizedot}{\begin{itemize} }{\end{itemize}}
\newenvironment{proof}[1][Proof]{\noindent\textbf{#1\ }}{\hspace*{\fill}$\Box$\medskip}

\newenvironment{tmparmod}[3]{\begin{list}{}{\setlength{\topsep}{0pt}\setlength{\leftmargin}{#1}\setlength{\rightmargin}{#2}\setlength{\parindent}{#3}\setlength{\listparindent}{\parindent}\setlength{\itemindent}{\parindent}\setlength{\parsep}{\parskip}} \item[]}{\end{list}}

\newcounter{tmcounter}

\newtheorem{definition}{Definition}
\newtheorem{lemma}{Lemma}
\newtheorem{proposition}{Proposition}
{\theorembodyfont{\rmfamily}\newtheorem{remark}{Remark}}
\newtheorem{theorem}{Theorem}

\newtheorem{assumption}{Assumption}

\begin{document}

\maketitle
\begin{abstract}
  The nonuniqueness of Leray--Hopf solutions to the unforced incompressible 3D Navier--Stokes equations is one of the central open problems in mathematical fluid dynamics. In this paper, we provide, to our knowledge, the first rigorous computer-assisted proof demonstrating such nonuniqueness. Inspired by earlier works \cite{jia2015incompressible, albritton2022non,guillod2023numerical}, we construct a self-similar profile and then establish the existence of a second solution by analyzing the stability of the linearized operator around this profile, showing that it corresponds to an unstable perturbation. To achieve this, we develop an innovative numerical method that computes candidate solutions with high precision and propose a framework for rigorously establishing exact solutions in a neighborhood of these candidates, {where the invertibility of the linearized operator is the key}. Inspired by \cite{chen2022stable}, we decompose the linearized operator into a coercive part plus a compact perturbation, which is further approximated by a finite-rank operator up to a small error. The invertibility of the linearized operator restricted to the image of this finite-rank approximation is then rigorously verified {via construction of the inversion on its finite set of basis functions with residues incorporated properly}. This certifies the existence of {a profile and similarly an unstable eigenpair}. With regularity bootstrapped, it yields a second solution--indeed, infinitely many Leray--Hopf solutions. The code is publicly available at: https://github.com/HouGroup2026/3d-navier-stokes-nonuniqueness
\end{abstract}
\section{Introduction}
\label{Sec:introd}

In this paper, we study the three-dimensional incompressible Navier--Stokes equation on the whole space without
forcing:
\begin{equation}
  \left\{\begin{array}{l}
    \partial_t u + u \cdot \nabla u - \Delta u + \nabla p = 0,\\
    \mathrm{div} u = 0,\\
    u (t = 0, x) = u_{\tmop{in}} (x) .
  \end{array}\right. \label{eq:NS eq}
\end{equation}
Leray \cite{leray1934}, and later Hopf \cite{hopf1951} for bounded domains, constructed weak solutions that exist for all time and satisfy energy inequalities. These solutions are termed Leray--Hopf solutions, and Leray's construction was moreover suitable, in the sense that it satisfies local energy inequality.

The main contribution of this paper is to develop a rigorous computer-assisted proof of the nonuniqueness of the
Leray--Hopf solutions to \eqref{eq:NS eq}.
\begin{theorem}[Nonuniqueness of
Leray--Hopf solutions]
  \label{thm:main}There exist infinitely many distinct suitable Leray--Hopf
  solutions to the Navier--Stokes equation \eqref{eq:NS eq} on $\mathbb{R}^3
  \times [0, 1]$ with the same divergence-free initial condition $u_{\tmop{in}} =u_{\tmop{loc}}$ of compact support, with $u_{\tmop{loc}}\in C^\infty(\mathbb{R}^3\setminus\{0\})\cap L^q$ for any $q<3$. The solutions are $ L^s ([0, 1] ; L^q (\mathbb{R}^3)),$ for any $q\geqslant2$, $\frac{3}{q} + \frac{2}{s}
> 1$, {and smooth for positive times}. They just miss the Prodi--Serrin condition that guarantees a unique solution, see \eqref{eq:Serrin}.
\end{theorem}
Our work is inspired by the works of Jia and Sverak \cite{jia2014local,jia2015incompressible}, which studied the forward self-similar Navier--Stokes equation and reduced nonuniqueness to spectral assumptions of the self-similar profiles. In \cite{guillod2023numerical}, Guillod and Sverak provided numerical evidence that such spectral assumptions are satisfied for the self-similar solution profiles that they constructed numerically. In \cite{albritton2022non}, Albritton--Brue--Colombo built on this idea of unstable profiles in self-similar variables, but added a singular forcing term so that their background solution does not need to solve the Navier--Stokes equation exactly. In fact, their profile can be any divergence-free function. They instead leveraged the unstable vortex solutions constructed by Vishik \cite{vishik2018instability,vishik2018instability1,albritton2021instability}
that are steady states of the 2D Euler equations, and lifted them to 3D axisymmetric vortex rings. The case of the unforced Navier--Stokes equations is much more challenging since it is almost impossible to find an explicit self-similar profile with an unstable eigenvalue. 

In this paper, we adopt a different strategy by developing a novel computational method and an innovative computer-assisted strategy to rigorously verify the existence of a self-similar profile that solves the unforced Navier--Stokes equation exactly with an unstable eigenvalue. We then use the tools developed in \cite{jia2015incompressible} to establish the nonuniqueness of Leray--Hopf solutions, {while reducing the spectral assumption simply to an unstable eigenvalue inspired by \cite{albritton2022non}}.   

\subsection{Leray--Hopf solution and strong-weak uniqueness}\label{subsec:ss}
Two important types of solutions to \eqref{eq:NS eq} are the {\tmstrong{Leray--Hopf solutions}} and
the {\tmstrong{strong solutions}}. We begin with the definition of Leray--Hopf
solutions.  We denote divergence-free function spaces using a subscript $\sigma$, and locally integrable function spaces using a subscript $\tmop{loc}$.

\begin{definition}[Leray--Hopf
solution]{\cite{sohr2012navier}}
  Let $T > 0$, $u_{\tmop{in}} \in L_{\sigma}^2$. A Leray--Hopf solution on
  $\mathbb{R}^3 \times [0, T)$ with initial data $u_{\tmop{in}}$  is a divergence-free vector field
  \[ u \in L^{\infty}_{\tmop{loc}} ([0, T) ; L^2_{\sigma} (\mathbb{R}^3)) \cap
     L^2_{\tmop{loc}} ([0, T) ; H^{1}_{\sigma} (\mathbb{R}^3)) , \] with the following properties:
  \begin{enumeratenumeric}
    \item It holds that for all test functions $v \in C^{\infty}_0 ([0, T) ;
    C^{\infty}_{0, \sigma} (\mathbb{R}^3))$:
    \[ - \int_0^T \langle u, \partial_t v \rangle + \int_0^T \langle \nabla u,
       \nabla v \rangle + \int_0^T \langle u \cdot \nabla u, v \rangle =
       \langle u_{\tmop{in}}, v (0, \cdummy) \rangle . \]
    \item $u$ satisfies the energy inequality:
    \[ \frac{1}{2} \| u (t, \cdummy) \|^2_{L^2} + \int_0^{t} \| \nabla u \|^2_{L^2} \leqslant
       \frac{1}{2} \| u_{\tmop{in}} \|_{L^2}^2 . \]
  \end{enumeratenumeric}
\end{definition}
\begin{remark}[Suitability]
    Solutions built via Leray's original construction are, moreover, suitable, which means they satisfy the stronger local energy inequality:
    \begin{equation}
    (\partial_t-\Delta)\frac{1}{2}|u|^2+|\nabla u|^2+\mathrm{div}((\frac{1}{2}|u|^2+p)u)\leqslant 0
        \label{suitability}
    \end{equation}
    in a distributional sense. 
    Local energy inequality plays an important role in partial regularity theory on the dimension of the potential singular sets, as in \cite{caffarelli1982partial,lin1998new}.
\end{remark}
A Leray--Hopf solution is weakly continuous in time {\cite{sohr2012navier}} and exists globally. Its uniqueness is not guaranteed.
A strong solution is a Leray--Hopf solution that satisfies the Prodi--Serrin \cite{prodi1959teorema, serrin1963initial}
condition in $L^sL^q$ spaces:
\begin{equation}
  u \in L^s_{\tmop{loc}} ([0, T) ; L^q (\mathbb{R}^3)),\quad 3 < q < \infty,\quad 2 < s < \infty,\quad \frac{3}{q} + \frac{2}{s}
\leqslant 1. \label{eq:Serrin}
\end{equation}
The critical case of $q=3$ was established by Escauriaza--Seregin--Sverak  \cite{sverak2003}.
 Under this condition, the solution exhibits higher regularity \cite{prodi1959teorema, serrin1963initial, sverak2003},
  \[ u \in L^{\infty}_{\tmop{loc}} ((0, T) ; H^{1}_{\sigma} (\mathbb{R}^3))
     \cap L^2_{\tmop{loc}} ((0, T) ; H^{2}_{\sigma} (\mathbb{R}^3)). \]
 A global strong solution ensures the uniqueness of weak solutions. They are locally well-posed. However, the
global-in-time existence of such solutions remains an open question.
 Therefore, we shift focus to studying the nonuniqueness of the
Leray--Hopf solutions arising from blow-up times of strong solutions. In other
words, we consider singular initial data $u_{\tmop{in}}$.

\subsection{Self-similar setting} \label{sec:ss}
Self-similar solutions are thus a special class of solutions of interest that lie precisely at
the boundary of the known well-posedness theory.  A scale-invariant initial condition takes the form:
\begin{equation}
  u_{\tmop{in}} (x) = \frac{1}{| x |} A ( \frac{x}{| x |} ) ,
  \label{eq:self-similar-id-ansatz0}
\end{equation}
with a singularity of $O (| x |^{- 1})$ at the origin.
The authors of {\cite{jia2014local,jia2015incompressible}} seek
solutions of the form:
\begin{equation}
  u (t, x) = \frac{1}{\sqrt{t}} \widetilde{U} ( \frac{x}{\sqrt{t}} ) ,
  \label{eq:self-similar-ansatz0}
\end{equation}
which preserves the scaling structure.  The readers should not confuse this self-similar setting, which starts from singular initial data, with the backward self-similar setting related to finite-time singularities from smooth initial data. Plugging \eqref{eq:self-similar-ansatz0} into the Navier--Stokes equation, we obtain:
\begin{equation}
  \left\{\begin{array}{l}
    - \frac{1}{2}  \widetilde{U} - \frac{1}{2} \xi \cdummy \nabla \widetilde{U} + \Pi
    (\widetilde{U} \cdot \nabla \widetilde{U}) - \Delta \widetilde{U} = 0,\\
    \mathrm{div}  \widetilde{U} = 0,
  \end{array}\right. \label{eq:NS-eq-self-similar0}
\end{equation}
where $\xi = \frac{x}{\sqrt{t}}$ is the similarity variable and $\Pi$ is the Leray projection. We stress that the pressure and the Leray projection formulations are interchangeable. The initial
data, as per  \eqref{eq:self-similar-id-ansatz0}, translates into
the following boundary condition for \eqref{eq:NS-eq-self-similar0}, 
\begin{equation}
    \label{boundary u0}\widetilde{U} (\xi) = \frac{1}{| \xi |} A ( \frac{\xi}{| \xi |} ) +
   o ( \frac{1}{| \xi |} ), \quad| \xi | \to + \infty,
\end{equation} 
For the self-similar solutions constructed in \cite{jia2014local}, the remainder decays at the sharper rate $O( \frac{1}{| \xi |^3} )$. The decay of our exact profile is established in Proposition~\ref{prop decay}.

The divergence-free condition induces the following constraint on $A$:
\begin{equation}
  \tmop{div} ( \frac{1}{| x |} A ( \frac{x}{| x |} ) ) =
  0. \label{eq:A constraint}
\end{equation}

In {\cite{jia2014local}}, the authors proved that for any admissible
$A$, there exists a self-similar solution for \eqref{eq:NS eq}. However, whether this solution is unique remains an open question.
\subsection{Eigenvalue of the linearized operator and nonuniqueness}
In {\cite{jia2015incompressible}}, a sufficient condition for nonuniqueness was
proposed. It involves analyzing the eigenvalues of the linearized operator
$\mathcal{L}_{\widetilde{U}}$, defined by:
\begin{equation}
  \mathcal{L}_{\widetilde{U}} v \assign  \frac{1}{2} v + \frac{1}{2} \xi \cdummy
  \nabla v - \Pi (\widetilde{U} \cdot \nabla v + v \cdot \nabla \widetilde{U}) +
  \Delta v. \label{eq:LU def}
\end{equation}

To establish nonuniqueness, a second solution is sought in the form:
\[ \widetilde{u} (t, x) = \frac{1}{\sqrt{t}} \widetilde{U} ( \frac{x}{\sqrt{t}}
   ) + \frac{1}{\sqrt{t}} \phi ( t, \frac{x}{\sqrt{t}} ) . \]
Plugging this ansatz into \eqref{eq:NS-eq-self-similar0} and introducing the rescaled
time $\tau = \log t$,  $\phi$ satisfies:
\begin{equation}
  \left\{\begin{array}{l}
    \partial_{\tau} \phi -\mathcal{L}_{\widetilde{U}} \phi + \Pi (\phi \cdot
    \nabla \phi) = 0,\\
    \mathrm{div} \phi = 0.
  \end{array}\right. \label{eq:NS eq self-similar pert rescale}
\end{equation}
We aim to find a nontrivial solution to \eqref{eq:NS eq self-similar pert
rescale} with initial data
\[ \phi (\tau = - \infty, x) = 0. \]

In {\cite{jia2015incompressible}}, the author considered a series of solutions
parameterized by $\alpha$. For a given parameterized initial data
$A_{\alpha}$, let $\widetilde{U}_{\alpha}$ denote the corresponding solution to
\eqref{eq:NS eq} with the boundary condition $\widetilde{U}_{\alpha} (\xi) \sim
\frac{A_{\alpha} (\xi)}{\xi}$ as $| \xi | \to + \infty$. Suppose that
$L_{\widetilde{U}_{\alpha}}$ has a discrete eigenvalue $\lambda_{\alpha}$
satisfying $\Re (\lambda_{\alpha}) > -\frac{1}{4}$. They showed that if the
trajectory of $\lambda_{\alpha}$ could cross the imaginary axis from negative
to positive as $\alpha$ increases, then, under certain auxiliary assumptions, there exists a nontrivial solution to \eqref{eq:NS eq self-similar pert
rescale} and thus the solution to \eqref{eq:NS eq} is nonunique.

\begin{remark}
  Note that due to the slow decay of $u_{\tmop{in}} = O ( \frac{1}{| x |}
  )$, the solution is not in $L^2$ and thus not a
  Leray--Hopf solution. However, in {\cite{jia2015incompressible}}, the authors
  showed that such solutions can be localized to yield valid Leray--Hopf
  solutions. 
\end{remark}

In {\cite{albritton2022non}}, the authors employed a similar approach to prove the nonuniqueness of Leray--Hopf solutions for the forced Navier--Stokes
equation. The force corresponds to the right-hand side of the first equation of \eqref{eq:NS eq} of any divergence-free $U$, while we have to find a divergence-free profile $U$ that solves the first equation of \eqref{eq:NS-eq-self-similar0} exactly without any external forcing. They assumed that an unstable eigenvalue,
with real part $a > 0$, exists for the linearized operator
$\mathcal{L}_{\widetilde{U}}$ and found the unstable manifold associated with the
most unstable eigenvalue. The perturbation of the solution follows the asymptotics
$\phi = \widetilde{U}^{\tmop{lin}} + O (e^{ 2 a \tau})$. Here,
$\widetilde{U}^{\tmop{lin}}$ solves the linearized equation:
\begin{equation}
    \partial_{\tau} \widetilde{U}^{\tmop{lin}} =\mathcal{L}_{\widetilde{U}} 
   \widetilde{U}^{\tmop{lin}} , \label{main lin}
\end{equation} 
and decays at the rate $O (e^{ a \tau})$, converging to zero as $\tau
\to - \infty$, indicating nonuniqueness.

\paragraph{Verification of eigenvalues.}
Having related the nonuniqueness of the Leray--Hopf solutions to the eigenvalue of the linearized operator in self-similar variables, we now focus on confirming the existence of an unstable eigenvalue. We aim to find a solution to the following system:
\begin{equation}
  \left\{\begin{array}{l}
    - \frac{1}{2}  \widetilde{U} - \frac{1}{2} \xi \cdummy \nabla \widetilde{U} + \Pi
    (\widetilde{U} \cdot \nabla \widetilde{U}) - \Delta \widetilde{U} = 0,\\
    - \frac{1}{2}  \widetilde{v} - \frac{1}{2} \xi \cdummy \nabla \widetilde{v} + \Pi
    (\widetilde{U} \cdot \nabla \widetilde{v} + \widetilde{v} \cdot \nabla \widetilde{U}) -
    \Delta \widetilde{v} = \widetilde\lambda \widetilde{v}, \quad \widetilde\lambda < 0,\\
    \mathrm{div}  \widetilde{U} = 0,\\
    \tmop{div} \widetilde{v} = 0.
  \end{array}\right. \label{eq:NS eq eig}
\end{equation}
Here, $\Pi$ denotes the Leray projection. 
We remark that we flip a sign of the linearized operator in the second equation of \eqref{eq:NS eq eig} due to the convenience of elliptic estimates.
If we can prove that $\mathcal{L}_{\widetilde{U}}$ has an eigenvalue
with a positive real part for some $\widetilde{U}$ satisfying \eqref{eq:NS-eq-self-similar0}, then one can hope to show that the solution to \eqref{eq:NS eq} is nonunique. Through
localization, this leads to the nonuniqueness of the Leray--Hopf solution. We will elaborate on this point in Section \ref{sec:settt}.
Numerical evidence supporting the existence of such profiles was provided in
{\cite{guillod2023numerical}}.
\begin{theorem}
    \label{thm:stationary solu main}
    There
  exists a weak solution $(\widetilde{U},\widetilde{v},\widetilde{\lambda})$ to the system \eqref{eq:NS eq eig} with $\widetilde{U}\in H^1_{\tmop{loc}},\widetilde{v}\in H^1$ such that:\begin{enumerate}
  \item $\widetilde{U}$ is smooth.
      \item In the far field, $\widetilde{U}$ satisfies the boundary condition \begin{equation}
    \widetilde{U} (\xi)=\frac{1}{| \xi |} A ( \frac{\xi}{| \xi |} ) +
   O( {| \xi |^{-2}} ), \quad| \xi | \to + \infty,
\end{equation} for a suitable $A$ with the divergence-free constraint \eqref{eq:A constraint}.
  \end{enumerate}
\end{theorem}

We adopt a computer-assisted framework to verify existence for \eqref{eq:NS eq eig}. The strategy is to
compute a high-accuracy numerical profile $\overline U$ and then construct a perturbation $U$ so that
$\overline U+U$ solves the equation exactly. This requires (i) a certified small residual for $\overline U$
and (ii) an invertibility estimate for the linearized operator around $\overline U$. To obtain (ii), we split
the linearized operator into a coercive part and a compact remainder; the latter is rigorously approximated
by a finite-rank operator with a quantitatively controlled tail. We then verify the spectral
properties needed on this small matrix for the nonlinear estimates, combine them with coercivity in the linear estimates, and close a
fixed-point argument; see Sections~\ref{Sec:anal} and~\ref{sec:analysis}. A similar strategy also applied to $(\widetilde{v},\widetilde{\lambda})$.

A subtle point is the regularity and decay of the exact profile
$\widetilde U:=\overline U+U$. The weak solution obtained from the computer-assisted proof \eqref{eq:NS eq eig} alone will not have sufficient regularity to imply Theorem \ref{thm:main}. Instead, we need to look at a more fine-grained version obtained from our computer-assisted proof in Propositions \ref{prop:U} and \ref{prop:u lambda} along with the asymptotics of the approximate profiles in Assumption \ref{asm:num}. They will imply the two properties in Theorem \ref{thm:stationary solu main} via a bootstrap of the regularity from the more precise prescription of the far field; summarized also in Propositions \ref{prop smooth} and \ref{prop decay}. To be precise, from the computer-assisted proof in   Proposition~\ref{prop:U} we only know $U\in H^1$ with
$\|U\|_{H^1}\lesssim \varepsilon^U$, which does not by itself identify $\widetilde U$ as a
Leray--Hopf profile. To bridge this, we introduce a smooth radial cutoff  and write
$\widetilde U=U_{\mathrm{near}}+U_{\mathrm{far}}$, where $U_{\mathrm{far}}$ is divergence free, smooth, and
constructed to match the exterior asymptotics of $\overline U$. In these variables, $U_{\mathrm{near}}$ solves the stationary self-similar
Navier--Stokes system with variable coefficients determined by $U_{\mathrm{far}}$; these coefficients are
smooth, $C^k$--bounded, and decay polynomially, and the corresponding forcing is smooth and decays like
$O(|\xi|^{-3})$ by our far-field analysis. 
We therefore combine local elliptic bootstrapping
 using the coercivity of the linearized operator to obtain
$U_{\mathrm{near}}\in H^m$ for all $m$, together with weighted energy bounds
\[
|\xi|\,U_{\mathrm{near}},\quad |\xi|\,\nabla U_{\mathrm{near}},\quad
|\xi|^{2}\,\nabla U_{\mathrm{near}},\quad |\xi|^{2}\,\nabla^{2} U_{\mathrm{near}} \ \in L^{2}.
\]
A dyadic Agmon estimate then yields the quantitative decay
$|U_{\mathrm{near}}(\xi)|=O(|\xi|^{-2})$ as $|\xi|\to\infty$; see
Propositions~\ref{prop smooth} and~\ref{prop decay}. Consequently, $\widetilde U$ is smooth and satisfies the required far-field asymptotics. Finite-energy Leray--Hopf solutions are obtained after localization.

\paragraph{Outline of the nonuniqueness proof (Section~\ref{Sec2:nonuniqueness}).}
We combine the localization idea of Jia--Sverak \cite{jia2015incompressible} with the fixed-point framework of Jia--Sverak \cite{jia2015incompressible} and Albritton--Brue--Colombo \cite{albritton2022non}. {The key innovation compared to \cite{jia2015incompressible} is that we perform localization and nonuniqueness arguments at the same time, bypassing the sophisticated spectral assumptions therein.} Writing the solution as
\(u=\widetilde u+u^{\mathrm{cut}}+u^{\mathrm{cor}}\), we first \emph{localize the scale-invariant initial data} by a Bogovskii
decomposition \(u_{\mathrm{in}}=u_{\tmop{loc}}+w\), where \(u_{\tmop{loc}}\) is compactly supported and divergence free, and then set
\(u^{\mathrm{cut}}=-e^{t\Delta}w\) to cut off the far-field tail. Unlike \cite{albritton2022non}, our
self-similar profile \(\widetilde U\) is not compactly supported; the new ingredient is to control the
resulting \emph{far-field interaction} via heat smoothing, which makes the associated linear term small when the localization radius \(R\) is large. Passing to similarity variables, we solve for the
correction \(U^{\mathrm{cor}}\) as a fixed point of a contraction Duhamel map generated by the semigroup
\(e^{\tau\mathcal L_{\widetilde U}}\). A key idea is to use a Riesz projection to split the dynamics into a stable part and an unstable part as in \cite{jia2015incompressible}; see also \cite{chen2024vorticity, chen2025vorticity}.  The finitely many unstable modes will serve as free parameters at
\(\tau=0\). Different choices of these unstable coordinates yield distinct solutions on
\((-\infty,0]\), producing infinitely many Leray--Hopf solutions with the same \(L^2\) initial data \(u_{\tmop{loc}}\).

In our construction of the approximate self-similar profile $\overline U$ and the eigenfunction $\overline v$, 
we impose even symmetry in $z$ for $\overline U$ and odd symmetry in $z$ for $\overline v$. The odd parity of $\overline v$ breaks the base even symmetry of $\overline U$; thus the Leray--Hopf branch bifurcating in the $\overline v$-direction is symmetry-breaking, which was also observed by the previous work  \cite{guillod2023numerical,albritton2022non}.
\subsection{Computer-assisted proofs}\label{subsec:cap}
Our approach
motivates the following assumption for approximate profiles:

\begin{assumption}[Approximate profiles]
  \label{asm:num} The tuple $(\overline{U}, \overline{P}, \overline{v},
  \overline{q}, \overline{\lambda})$ is an approximate solution of \eqref{eq:NS
  eq eig} with small residuals:
\begin{equation*}
    \left\{\begin{array}{l}
      - \frac{1}{2}  \overline{U} - \frac{1}{2} \xi \cdummy \nabla \overline{U}
      + \overline{U} \cdot \nabla \overline{U} - \Delta \overline{U} + \nabla
      \overline{P} = \Epsilon^U,\\
      - \frac{1}{2}  \overline{v} - \frac{1}{2} \xi \cdummy \nabla \overline{v}
      + \overline{U} \cdot \nabla \overline{v} + \overline{v} \cdot \nabla
      \overline{U} + \nabla \overline{q} - \Delta \overline{v} =
      \overline{\lambda}  \overline{v} + \Epsilon^v, \quad \overline{\lambda}
      < 0,\\
      \mathrm{div}  \overline{U} = 0,\\
      \tmop{div} \overline{v} = 0,
    \end{array}\right. \label{eq:NS rescale numerical}
  \end{equation*}
  where $\overline{v}$ is normalized (i.e., $\| \overline{v} \|_{L^2} = 1$), and the
  orthogonality condition $\langle \Epsilon^v, \overline{v} \rangle = 0$ is
  satisfied upon a possible modification of $\overline{\lambda}$. The small residuals $\Epsilon^U$ and $\Epsilon^v$ are bounded by
  \[ \| \Epsilon^U \|_{L^2} \leqslant \varepsilon^U, \| \Epsilon^v \|_{L^2} \leqslant
     \varepsilon^v. \]
     The residues have to be small to satisfy the boundedness needed in later Assumptions \ref{eq:cond ref U} and \ref{asm:cond ref v}.
  Here $\overline{U},\overline{v}\in
  C^{\infty}_\sigma (\mathbb{R}^3 \backslash \{ 0 \})\cap H^2_{\tmop{loc}}$. In the far field, $\overline{U}$ satisfies the boundary condition \begin{equation}
    \label{boundary u}\overline{U} (\xi)=\frac{1}{| \xi |} A ( \frac{\xi}{| \xi |} ) +
   O( {| \xi |^{-3}} ), \quad| \xi | \to + \infty,
\end{equation} which is consistent with the boundary condition in Section \ref{sec:ss}, for a suitable $A$ with the divergence-free constraint \eqref{eq:A constraint}. Moreover, for any index $j\geqslant 0$, we have \begin{equation}
    \label{higher boundary u}
\nabla^j\overline{U}(\xi)=\nabla^j(\frac{1}{| \xi |} A ( \frac{\xi}{| \xi |} ))+O(|\xi|^{-j-3}).\end{equation} Our choice of the numerical spectral basis in the compactified domain guarantees the integrability and decay; see Sections \ref{sec:ext} and \ref{sec:div-free}.
\end{assumption}

Suppose we have constructed a sufficiently accurate approximate solution. The next step is to perform rigorous error analysis to confirm the existence of a
nearby exact solution. This leads to our main theoretical results, summarized in the two propositions below:

\begin{proposition}[Exact self-similar profile]
  \label{prop:U}If Assumptions \ref{asm:num} and \ref{eq:cond ref U} hold, then there exists a $U$
  satisfying the following equation in the distributional sense:
  \begin{equation}
    \left\{\begin{array}{l}
      - \frac{1}{2} U - \frac{1}{2} \xi \cdot \nabla U - \Delta U + \Pi
      (\overline{U} \cdot \nabla U + U \cdot \nabla \overline{U} + U \cdot
      \nabla U) = - \Pi \Epsilon^U,\\
      \tmop{div} U = 0,
    \end{array}\right. \label{eq:U pert}
  \end{equation}so that 
  \begin{equation}
  \label{Decomp-Utilde}
\widetilde{U} = \overline{U} + U
\end{equation}
solves the first and third equations of
  \eqref{eq:NS eq eig}.
  Moreover, the perturbation $U\in H^1$ satisfies:
  \[ \| U \|_{L^2} \lesssim \varepsilon^U . \]

\end{proposition}

\begin{proposition}[Exact unstable eigenpair]
  \label{prop:u lambda}If Assumptions \ref{asm:num}, \ref{eq:cond ref U}, and \ref{asm:cond ref v} hold, then there exists
  $(\lambda, v)$ satisfying the following equation in the distributional sense:
  \begin{equation}
    \left\{\begin{array}{l}
      - \frac{1}{2} v - \frac{1}{2} \xi \cdummy \nabla v - \Delta v + \Pi
      (\widetilde{U} \cdot \nabla v + U \cdot \nabla \overline{v} + v \cdot \nabla
      \widetilde{U} + \overline{v} \cdot \nabla U) = \overline{\lambda} v +
      \lambda \overline{v} + \lambda v - \Pi\Epsilon^v,\\
      \tmop{div} v = 0.
    \end{array}\right. \label{eq:u pert}
  \end{equation}
  Here, $\widetilde{U} = \overline{U} + U$ is the solution stated in Proposition \ref{prop:U}. Hence, $(\widetilde{\lambda}, \widetilde{v}) = (\overline{\lambda} + \lambda,
  \overline{v} + v)$ is an eigenpair of the linearized operator
  and solves the second and fourth equations of
  \eqref{eq:NS eq eig}.  
  Additionally,
  $v\in H^1$ satisfies the orthogonality condition $ \langle v, \overline{v} \rangle = 0$, and the following estimate holds for $(v, \lambda)$:
  \[ \| v \|_{L^2} + | \lambda | \lesssim \varepsilon^v+\varepsilon^U . \]
  Therefore, $\widetilde{v}$ is nontrivial, and the eigenvalue $\widetilde{\lambda}$ is strictly
  negative.
\end{proposition}
We remark that the precise bound on the $L^2$ (and also $H^1$ norms), which will be obtained later via the bounded mapping in Assumptions \ref{eq:cond ref U} and \ref{asm:cond ref v}, is not important for the nonuniqueness argument. It is only the nontriviality of the profile and the unstable eigenpairs that matters.

 One of the main contributions of this paper is that we
develop an innovative numerical framework to construct a smooth, axisymmetric, divergence-free candidate profile with high precision. A distinguished advantage of this framework is that we can start with any initial guess to search for a candidate unstable eigenmode. We refine our search iteratively by using gradient descent to locate a maximally unstable configuration. Specifically, we first solve the Navier--Stokes equations in a large finite domain using a finite element method, then extend it to the whole space by an analytic change of variables $r =\tan(\beta)$ and interpolate it on an exactly divergence-free spectral basis, which is essential for rigorous computer-assisted verification. We also use an iterative optimization method to further improve the accuracy of the approximate profile  until a sufficiently small residual is reached. An important property of our numerical framework is that our constructed numerical solution $\overline{U}$ or $\overline{v}$ is smooth except at the origin. Moreover, $\overline{U}$ satisfies the far field boundary condition and the decay property to all orders.
For more details,
see Section \ref{Sec:num}.

We would like to point out that our
Navier--Stokes nonuniqueness verification is far lighter than that for the 3D Euler singularity program  \cite{chen2022stable,ChenHou2023b,ChenHou2025}, because viscosity yields genuine coercivity.
In particular, the linearized operators enjoy relatively large coercivity coefficients in $L^2$:
about $0.0541$ for the $U_1$--equation and about $0.059$ for the $v_1$--equation in the fixed-point estimates
(see Section~\ref{Sec:num prev}). This allows us to work in standard $L^2/H^1$ norms without
singular weights and to close the a~posteriori bounds using the $L^2$ residuals of
$\overline U$ and $\overline v$ together with a few pointwise envelopes of $L^\infty$ estimates. The discrete operators we assemble are modest (roughly $1300\times1300$), and retaining about $20$--$25$ modes suffices to close the estimates via the Schauder fixed point theorem.

Moreover, we use a trigonometric representation with $r=\tan\beta$. For $\overline U$, we use a larger basis with maximum wavenumbers $1400$ in $\beta$ and $700$ in $\theta$; for all other quantities, we use a basis with maximum wavenumbers approximately $600$ in $\beta$ and $300$ in $\theta$. This permits rigorous $L^2$ residual evaluation by quadrature and sharp $L^\infty$ bounds on the profiles and their gradients. The only
delicate numerical step is to evaluate the $Q$-inner products for the generalized eigenfunctions, which involve Bessel functions along the $\beta$-direction. We need to approximate the inner products using a 8th order quadrature and supply a rigorous $W^{8,\infty}$ error bound on the integrand, see Section \ref{sec7} for more discussion. Altogether, viscous coercivity, low-dimensional spectral checks, and
analytic trigonometric calculus make the present verification substantially simpler than that in the
Euler setting.

We remark that a candidate solution was found in \cite{guillod2023numerical}, but this approximate solution seems to be too
large to apply our computer-assisted proof,  in the sense that the approximation of the associated compact part of the linearized operator by a finite rank operator would require a large number of modes to reduce the residual error small enough for us to close the bootstrap argument, making the computer verification too expensive; see Section \ref{subsec:finiterank} for details.  

\subsection{Organization of the paper}
In summary, our strategy consists of four main steps:
\begin{enumerate}
  \item Construct an accurate approximate solution to the eigenvalue system \eqref{eq:NS eq eig} (Assumption \ref{asm:num}). 
  
  \item Establish the existence of an exact solution via a computer-assisted proof (Propositions \ref{prop:U} and \ref{prop:u lambda}).

  \item Bootstrap the regularity of the profile to conclude Theorem \ref{thm:stationary solu main} (Propositions \ref{prop smooth} and \ref{prop decay}).
  \item Prove the nonuniqueness of Leray--Hopf solutions (Theorem
  \ref{thm:main}). 
\end{enumerate}
The rest of the paper is structured as follows.  We first conclude nonuniqueness based on Assumption \ref{asm:num} and Propositions \ref{prop:U} and \ref{prop:u lambda} in Section \ref{sec:settt}. With the asymptotics of the approximate profile in Assumption \ref{asm:num}, we establish the smoothness of our profile by bootstrapping its regularity using elliptic estimates, even if the numerical profile is nonsmooth by itself only at the origin. We also show that the profile satisfies the boundary conditions necessary for the self-similar setting by establishing the fast decay of the perturbation, which concludes Theorem \ref{thm:stationary solu main}. Then we adopt a fixed-point approach similar to \cite{albritton2022non} and \cite{jia2015incompressible} to establish nonuniqueness based on the unstable eigenpair, with a cutoff to obtain Leray--Hopf solutions.

In Section \ref{Sec:anal}, we outline the
framework for proving Propositions \ref{prop:U} and \ref{prop:u lambda}, which serves as a
foundation for the numerical analysis of a broad class of problems.  We believe that our framework is interesting in itself to establish eigenvalue properties rigorously for other problems.
In Section \ref{sec:analysis}, we apply this framework to our specific
system and establish Propositions \ref{prop:U} and \ref{prop:u lambda}. We demonstrate the construction of a finite-rank operator to
approximate the compact part of the linear operator, verify that this
approximation has a sufficiently small error, and derive the necessary
nonlinear estimates to complete the argument. For both the nonlinear terms and the
approximation error, sharp estimates are essential to tighten the final bound.

Section \ref{Sec:num} discusses the numerical implementation in detail. We
introduce the setting of the finite element method used to obtain an approximate solution, and interpolate it on a spectral basis
to ensure exact divergence-freeness. Section \ref{Sec:numfr} details the numerical construction of the finite-rank approximation and verifies the invertibility of the
linear operator on a finite-dimensional subspace. Finally, we describe the numerical methods that we will use to rigorously verify our numerical computations using interval arithmetic in Section \ref{sec7} and collect some proofs in the Appendix.
\subsection{Literature review} Now we review related literature on nonuniqueness besides the programs of Jia--Sverak and Guillod, and Albritton--Brue--Colombo. We also discuss relevant literature on computer-assisted proofs.
\paragraph{Nonuniqueness.}
One of the earliest examples of nonuniqueness to \eqref{eq:NS eq} dates back to  Ladyzhenskaya \cite{ladyzhenskaya1969example}, in a time-varying domain with force. Similar to the program of Jia, Sverak, and Guillod, Bressan, Murray, and Shen \cite{bressan2021posteriori, bressan2020self} proposed some suggestive numerical evidence for nonuniqueness to the 2D Euler equations without forcing.  In addition to the work using self-similar solutions, there is another important line of work leveraging 
 convex integration, a powerful tool introduced to fluid dynamics that resulted in the proof of the Onsager conjecture by Isett \cite{isett2018proof} and  Buckmaster--De-Lellis--Szekelyhidi--Vicol \cite{ buckmaster2019onsager}. In particular, Buckmaster--Vicol \cite{buckmaster2019nonuniqueness} established that for any $L^2$ initial data, distributional solutions to \eqref{eq:NS eq} are nonunique, and Cheskidov--Luo  \cite{cheskidov2022sharp, cheskidov20232} established nonuniqueness in $ L_t^{2-}L_x^\infty$ and $ L_t^\infty L_x^{2-}$ for dimension $d\geqslant 2$. However, solutions obtained from convex integration schemes are very far from Leray regularity $H^1_x$. Very recently, Coiculescu--Palasek established nonuniqueness in a scaling critical space $BMO^{-1}$ \cite{coiculescu2025non} where global well-posedness for small initial data holds \cite{koch2001well}, using techniques from a dyadic model introduced by Palasek \cite{palasek2024non}. Later on, Cheskidov--Dai--Palasek \cite{cheskidov2025instantaneous} established nonuniqueness in the same space along with instantaneous blowup. 

\paragraph{Computer-assisted proofs.} 
The idea of combining rigorous analysis with numerical candidates has been successfully applied to problems in fluid dynamics, where the inherent complex nature of solutions prohibited analytical candidates: particularly in the search for backward self-similar singularities, culminating in the computer-assisted proof for the 3D incompressible Euler equation with cylindrical boundary by Chen--Hou \cite{chen2022stable, ChenHou2023b,ChenHou2025} based on numerical investigation by Luo--Hou \cite{luo2014potentially, luo2014toward}. We draw inspiration from the works of Chen--Hou \cite{chen2022stable}, Elgindi--Pasqualotto \cite{elgindi2023instability} on the singularity formulation of the Boussinesq equation with only radial nonsmoothness, and Castro--Cordoba--Gomez-Serrano \cite{castro2020global} on the construction of smooth rotating solutions of the inviscid surface quasi-geostrophic equation. We also remark the linear framework of Liu--Nakao--Oishi for eigenvalue verifications \cite{liu2013verified,liu2015framework,liu2022computer}.

Stability analysis, especially linear stability, around the approximate solution is essential to conclude the computer-assisted proof, and we apply a coercive-compact decomposition of the linear operator. The compact part is approximated by a finite-rank operator, whose error, along with the residual of the approximate solution, will need to be verified rigorously by interval arithmetic. Our decomposition is obvious in a regular $H^1$ norm, unlike in \cite{chen2022stable}, where building on results of 1D singularities in fluid models \cite{chen2020finite,chen2022asymptotically}, one needs to use singularly weighted estimates; see also \cite{hou20242,chen2024stability,liu2025finite} for the use of singular weights for singularity beyond self-similarity or full stability.

For numerical computation, we use traditional finite element methods for profile computation and spectral
methods for rigorous verification.
We remark that there has been encouraging recent progress in using Physics-Informed Neural Networks (PINNs) to construct self-similar blowup profiles, including the 2D Boussinesq system with boundary \cite{wang2023asymptotic, wang2025high, wang2025discovery}. However,  a rigorous computer-assisted proof is not yet available to rigorously justify these numerical profiles. 
\subsection{Notations}

We denote 
by $\langle \cdummy, \cdummy
\rangle_{\Omega}$ the $L^2$-inner product over a domain $\Omega \subset \mathbb{R}^3$. The Sobolev norm $\parallel \cdot
\parallel_{W^{k, p}
(\Omega)}$ for functions in the Sobolev space $W^{k, p}
(\Omega)$ is defined by
\[ \| \nosymbol f \|_{W^{k, p}
(\Omega)} \assign ( \sum_{| \alpha | \leqslant
   k} \| \partial^{\alpha} f \|_{L^p (\Omega)}^p )^{\frac{1}{p}}, \]
where $\alpha=(\alpha_1 , \alpha_2 ,
\alpha_3)$ denotes a multi-index and $| \alpha | = \alpha_1 + \alpha_2 +
\alpha_3$ is its order. 
In particular, we write $H^k (\Omega) \assign W^{k, 2} (\Omega)$. 
These notations extend naturally to vector- and matrix-valued functions induced by the Frobenius norm. For example, for an
$m \times n$ matrix function $A (x)$ with entries $A_{i j}$, we define its $L^p$ norm as: 
  \[ \| A \|_{L^p (\Omega)} \assign \| ( \sum_{
       1 \leqslant i \leqslant m, 1 \leqslant j \leqslant n
     } | A_{i j} |^2 )^{\frac{1}{2}} \|_{L^p (\Omega)}. \]
When $\Omega =\mathbb{R}^3$, we can omit $\Omega$.

For $\eta>0$, define the weighted $H^1$ norm and inner product by
\begin{equation}\label{eq:weighted-H1-notation}
\|f\|_{H^1(\Omega),\eta}^2
 :=\eta\|f\|_{L^2(\Omega)}^2+\|\nabla f\|_{L^2(\Omega)}^2,\qquad \langle f,g\rangle_{H^1(\Omega),\eta}
 :=\eta\langle f,g\rangle_\Omega+\langle\nabla f,\nabla g\rangle_\Omega.
\end{equation}
We likewise omit the domain when $\Omega=\mathbb R^3$.

We use $C$ to denote a generic absolute constant, which may change from line
to line. We write $C_R$ for constants depending on $R$.
The notation $A \lesssim B$ indicates that $A \leqslant C B$ for some
constant $C > 0$. Constants such as $K_0, K_1, \ldots$ are fixed throughout
the paper. Constants $\eta_1, \eta_2, \ldots$ are tunable parameters that arise in both the analytical estimates and the computational components of the proof. We use $\varepsilon_1, \varepsilon_2, \ldots$ to denote error
bounds that can be made arbitrarily small. Constants like $M_0, M_1, \ldots$
are abbreviations for long expressions.
For operators and finite matrices, $\|\cdot\|$ denotes the induced $L^2$ and Euclidean norms, respectively.
Constants $\Theta$ denote bounds for operator norms.
To indicate dependence, we use subscripts and superscripts such as $K^U$, though not necessarily functional
dependence.

\section{Nonuniqueness Based on Unstable Eigenpairs}\label{sec:settt}
In this section, we detail the proof of Theorem \ref{thm:main} based on Assumption \ref{asm:num} and Propositions \ref{prop:U} and \ref{prop:u lambda}, addressing the
nonuniqueness of Leray--Hopf solutions. First, we demonstrate the smoothness
of the profile $\widetilde{U}$, along with its far-field decay. In particular, it satisfies the boundary conditions of a self-similar solution.  Then, via an approach inspired by prior works \cite{albritton2022non, jia2015incompressible}, we establish the nonuniqueness of Leray--Hopf weak solutions
starting from homogeneous initial data after cutoff, as in {\cite{jia2015incompressible}}. An important idea is to perform a stable-unstable decomposition of the semigroup and demonstrate that each appropriate unstable mode at $t=1$ corresponds to a distinct solution with the same initial data.

\subsection{Regularity of the profile}
\label{subsec2.1}
We begin by proving the smoothness of the profile and that it satisfies the appropriate boundary condition, thus concluding Theorem \ref{thm:stationary solu main}. Notice that by Proposition \ref{prop:U} and  Assumption \ref{asm:num}, we know that $\widetilde{U}\in H^1_{\tmop{loc}}$. We will perform a decomposition into far and near fields of the profile, and bootstrap the regularity based on the steady state equation \eqref{eq:NS eq eig}.  

To be precise,  by Assumption \ref{asm:num}, $\overline{U} \in
  C^{\infty} (\mathbb{R}^3 \backslash \{ 0 \})\cap H^1_{\tmop{loc}}$ is divergence free, and we 
  construct $U_{\tmop{far}} \in C^\infty$, also divergence
  free, such that $\overline{U} - U_{\tmop{far}} \in H^1$.
    Let
  $\widetilde\chi$ be a smooth truncation function:
  \[ \widetilde\chi (\xi) = 0, \quad |\xi| \leqslant 1, \quad \widetilde\chi (\xi) = 1, \quad |\xi| \geqslant
     2, \]and define:
  \begin{equation}
  \label{Definition-Ufar}
   U_{\tmop{far}} = \Pi (\widetilde\chi \overline{U}) . 
   \end{equation}
  Then, $U_{\tmop{far}}$  satisfies:
  \begin{equation}
  \label{Definition-Ufar2} \left\{\begin{array}{l}
       U_{\tmop{far}} = \widetilde\chi \overline{U} - \nabla P_c,
      \\
       \tmop{div} U_{\tmop{far}} = 0,
     \end{array}\right. 
     \end{equation}
  where $P_c$ solves:
  \begin{equation}
  \label{Definition-Pc}
   \Delta P_c = \tmop{div} (\widetilde\chi \overline{U}) = \overline{U} \cdummy \nabla
     \widetilde\chi . 
     \end{equation}
     Here, the source term is of compact support and mean-zero, and we can invert the Laplacian in 3D to conclude that
  the function $P_c$ is smooth, and $\nabla^j P_c$ decays sufficiently fast as
  $O ( {|\xi|^{-j-2}} )$, ensuring $\overline{U} - U_{\tmop{far}} =
  (1 - \widetilde\chi) \overline{U} + \nabla P_c \in H^1$. 
  Then we decompose $\widetilde{U}$ into two components:
  \begin{equation}
  \label{Decomp-Ufar-Unear} \widetilde{U} = U_{\tmop{near}} + U_{\tmop{far}} . 
  \end{equation}
  
  By Proposition \ref{prop:U} and \eqref{Decomp-Utilde}, we know $U\in H^1$, and thus $U_{\tmop{near}}=\widetilde{U}-U_{\tmop{far}} = U+\overline{U} - U_{\tmop{far}}\in H^1$. Moreover, $U_{\tmop{far}}$ has the same far field asymptotics as $\overline{U}$: \begin{equation}
    \label{higher boundary ufar}
\nabla^jU_{\tmop{far}}(\xi)=\nabla^j(\frac{1}{| \xi |} A ( \frac{\xi}{| \xi |} ))+O(|\xi|^{-j-3}).\end{equation}With this decomposition, we are ready to establish the smoothness and decay properties of $\widetilde{U}$.

\begin{proposition}[Smoothness]
  The solution $\widetilde{U}$ in Proposition \ref{prop:U} is smooth. \label{prop smooth}
\end{proposition}

\begin{proof}[Proof of Proposition \ref{prop smooth}]
Recall that we will use induction to prove $U_{\mathrm{near}} \in H^\infty$. Plugging the decomposition $\widetilde{U} = U_{\mathrm{near}} + U_{\mathrm{far}}$ into the steady profile equation, we have the equation for $U_{\mathrm{near}}$
\begin{equation}\label{eq:unear_eq}
-\tfrac{1}{2} U_{\mathrm{near}} - \tfrac{1}{2}\xi \cdot \nabla U_{\mathrm{near}} + \Pi(\widetilde{U} \cdot \nabla \widetilde{U}) - \Delta U_{\mathrm{near}} = F, \qquad \nabla \cdot U_{\mathrm{near}} = 0,
\end{equation}
with the right-hand side
\begin{equation}\label{eq:F_def}
F := \tfrac{1}{2} U_{\mathrm{far}} + \tfrac{1}{2}\xi \cdot \nabla U_{\mathrm{far}} + \Delta U_{\mathrm{far}}, \qquad \nabla^k F(\xi) = O(|\xi|^{-3-k}) \quad (k \geqslant 0).
\end{equation}
Since $U_{\mathrm{far}} \in C^\infty(\mathbb{R}^3)$ is smooth near the origin, the operator $\frac{1}{2} + \frac{1}{2}\xi \cdot \nabla$ annihilates the $(-1)$-homogeneous leading far-field term of $U_{\mathrm{far}}$, and $\Delta U_{\mathrm{far}}(\xi) = O(|\xi|^{-3})$, $F$ is globally smooth on $\mathbb{R}^3$ with $|\nabla^k F(\xi)| \lesssim (1 + |\xi|)^{-3-k}$ for all $k \geqslant 0$, whence $F \in H^\infty(\mathbb{R}^3)$. Likewise, the far-field asymptotics of $U_{\mathrm{far}}$ give $U_{\mathrm{far}} \in L^4(\mathbb{R}^3) \cap W^{\infty,\infty}(\mathbb{R}^3)$ and $\nabla U_{\mathrm{far}} \in H^\infty(\mathbb{R}^3)$.

Fix a standard mollifier $\rho_\epsilon$ with $\epsilon \in (0, 1)$ and denote
\begin{equation}\label{eq:u_eps_def}
u^\epsilon = \rho_\epsilon * U_{\mathrm{near}}, \qquad F^\epsilon = \rho_\epsilon * F.
\end{equation}
We know that $u^\epsilon \in H^\infty$ is divergence-free, and satisfies the pointwise equation
\begin{equation}\label{eq:u_eps_eq}
-\frac{1}{2} u^\epsilon - \frac{1}{2}\xi \cdot \nabla u^\epsilon - \frac{1}{2}[\rho_\epsilon *, \xi \cdot \nabla] U_{\mathrm{near}} + \rho_\epsilon * \Pi(\widetilde{U} \cdot \nabla \widetilde{U}) - \Delta u^\epsilon = F^\epsilon,
\end{equation}
where the commutator is computed as
\[
[\rho_\epsilon *, \xi \cdot \nabla] U_{\mathrm{near}} = \rho_\epsilon * (\xi \cdot \nabla U_{\mathrm{near}}) - \xi \cdot \nabla(\rho_\epsilon * U_{\mathrm{near}}) = -\int_{\mathbb{R}^3} \rho_\epsilon(y) y \cdot \nabla U_{\mathrm{near}}(\xi - y)\,dy.
\]
The commutator vanishes as $\epsilon \to 0$. More precisely, we have the estimate of its $H^{k-1}$-seminorm by Young's inequality as
\begin{equation}\label{eq:comm_est}
\|\nabla^{k-1}[\rho_\epsilon *, \xi \cdot \nabla] U_{\mathrm{near}}\|_{L^2} \leqslant \|\rho_\epsilon(y) y\|_{L^1} \|\nabla^k U_{\mathrm{near}}\|_{L^2} \lesssim \epsilon \|U_{\mathrm{near}}\|_{H^k} < \|U_{\mathrm{near}}\|_{H^k}, \qquad \forall k > 0.
\end{equation}
We recall two basic facts that will be used repeatedly: the boundedness of the Leray projection in $L^2$, and that mollification is bounded in $H^k$ via Young's inequality: for any integer $k \geqslant 0$,
\begin{equation}\label{eq:moll_hk}
\|\nabla^k u^\epsilon\|_{L^2} = \|\rho_\epsilon * \nabla^k U_{\mathrm{near}}\|_{L^2} \leqslant \|\rho_\epsilon\|_{L^1} \|\nabla^k U_{\mathrm{near}}\|_{L^2} = \|\nabla^k U_{\mathrm{near}}\|_{L^2}.
\end{equation}

Next, we will obtain an estimate of the $H^{k-1}$-seminorm of the nonlinear term for $k \geqslant 2$. Recall $U_{\mathrm{far}} \in W^{\infty,\infty}$ and $\nabla U_{\mathrm{far}} \in H^\infty$. A canonical term in the expansion of $\nabla^{k-1}(\rho_\epsilon * \Pi(\widetilde{U} \cdot \nabla \widetilde{U}))$ has the form $\rho_\epsilon * \Pi(\nabla^i U_{\mathrm{s}} \otimes \nabla^{k-i} U_{\mathrm{t}})$, where $\mathrm{s}, \mathrm{t} \in \{\mathrm{far}, \mathrm{near}\}$, $0 \leqslant i < k$. By Young's inequality, all of the terms where at least one of $\mathrm{s}, \mathrm{t}$ is $\mathrm{far}$ have $L^2$-norm bounded by $1 + \|U_{\mathrm{near}}\|_{H^k}$. When $\mathrm{s}, \mathrm{t} = \mathrm{near}$ and $i > 0$, we have by Young's inequality, H\"older's inequality, and the Sobolev embedding that
\begin{equation}\label{eq:nl_mid}
\|\rho_\epsilon * \Pi(\nabla^i U_{\mathrm{s}} \otimes \nabla^{k-i} U_{\mathrm{t}})\|_{L^2} \leqslant \|\nabla^i U_{\mathrm{near}}\|_{L^4} \|\nabla^{k-i} U_{\mathrm{near}}\|_{L^4} \lesssim \|U_{\mathrm{near}}\|_{H^k}^2.
\end{equation}
When $\mathrm{s}, \mathrm{t} = \mathrm{near}$, $i = 0$, and $k \geqslant 2$, the Sobolev embedding $H^k(\mathbb{R}^3) \hookrightarrow L^\infty(\mathbb{R}^3)$ and Young's inequality directly yield $\|\rho_\epsilon * \Pi(U_{\mathrm{near}} \otimes \nabla^k U_{\mathrm{near}})\|_{L^2} \leqslant \|U_{\mathrm{near}}\|_{L^\infty} \|\nabla^k U_{\mathrm{near}}\|_{L^2} \lesssim \|U_{\mathrm{near}}\|_{H^k}^2$. Combined with \eqref{eq:nl_mid}, we collect a good nonlinear estimate for $k \geqslant 2$:
\begin{equation}\label{eq:nl_ge2}
\|\nabla^{k-1} \rho_\epsilon * \Pi(\widetilde{U} \cdot \nabla \widetilde{U})\|_{L^2} \lesssim 1 + \|U_{\mathrm{near}}\|_{H^k}^2 \qquad (k \geqslant 2).
\end{equation}
When $k=1$, the preceding nonlinear estimate cannot be used, since
$H^1(\mathbb R^3)$ does not embed into $L^\infty(\mathbb R^3)$.
We therefore first establish $U_{\mathrm{near}}\in H^2(\mathbb R^3)$
by a difference-quotient argument. Set
\[
u:=U_{\mathrm{near}},\qquad b:=U_{\mathrm{far}},\qquad
 W:=\widetilde U=u+b.
\]
By construction, $u\in H^1(\mathbb R^3)$ and $u,b,W$ are divergence-free.
Moreover, \eqref{higher boundary ufar}--\eqref{eq:F_def} imply
$b\in L^4\cap W^{2,\infty}$, $\nabla b\in H^2$, and $F\in H^1$.
In particular, $u\in L^4$ by interpolation and hence $W\in L^4$.
Define the pressure by
$P=\sum_{i,j}\mathcal R_i\mathcal R_j(W_iW_j)$, where
$\mathcal R_j=\partial_j(-\Delta)^{-1/2}$.
The $L^2$ boundedness of the Riesz transforms yields
$\|P\|_{L^2}\lesssim\|W\|_{L^4}^2$.
Thus \eqref{eq:unear_eq} can be written in velocity--pressure form as
\begin{equation*}
 -\frac12 u-\frac12\xi\cdot\nabla u-\Delta u
      +W\cdot\nabla W+\nabla P=F
 \quad\text{in }\mathcal D'(\mathbb R^3).
\end{equation*}
The nonlinear term belongs to $L^{3/2}+L^2$:
$u\cdot\nabla u\in L^{3/2}$, while the three terms involving $b$
belong to $L^2$. Consequently, the weak equation may be tested
against compactly supported $H^1$ fields, without imposing
zero divergence on the test function.

Fix $\ell\in\{1,2,3\}$ and $0<h<1$. Write
$\tau_h f(\xi)=f(\xi+he_\ell)$,
$D_h f=h^{-1}(\tau_h f-f)$, and set
$V=D_hu$, $B_h=D_hb$.
The fundamental theorem of calculus and translation invariance give
\begin{equation*}\|V\|_{L^2}\leq\|\partial_\ell u\|_{L^2}\leq\|\nabla u\|_{L^2},
 \qquad
\|B_h\|_{L^\infty}+\|\nabla B_h\|_{L^2\cap L^3}\leq C_b,
\end{equation*}
where $C_b$ is independent of $h$.
The discrete product rule gives the identities
\begin{align*}
 D_h(\xi\cdot\nabla u)&=\xi\cdot\nabla V+\tau_h\partial_\ell u,\\
 D_h(W\cdot\nabla W)
   &=(\tau_hW\cdot\nabla)(V+B_h)
      +((V+B_h)\cdot\nabla)W.
\end{align*}
Let $\chi_R(\xi)=\chi(\xi/R)$, where $\chi=1$ on $B_1$,
$\chi=0$ outside $B_2$, and $|\nabla\chi_R|\lesssim R^{-1}$.
We test the original weak equation against
$-D_{-h}(\chi_R V)$, equivalently testing its difference quotient
against $\chi_RV$.
For each \emph{fixed} $h$, integration by parts gives
\begin{align*}
 \left\langle-\tfrac12V-\tfrac12\xi\cdot\nabla
 V,\chi_RV\right\rangle &\longrightarrow\tfrac14\|V\|_{L^2}^2,\\
 \langle-\Delta V,\chi_RV\rangle &\longrightarrow\|\nabla V\|_{L^2}^2.
\end{align*}
Here the errors are bounded, respectively, by
$C\int_{R<|\xi|<2R}|V|^2$ and
$CR^{-1}\|\nabla V\|_{L^2}\|V\|_{L^2}$.
The pressure and the transport of $V$ vanish in the same limit,
because $\nabla\cdot V=\nabla\cdot W=0$ and
\begin{align*}
 |\langle\nabla D_hP,\chi_RV\rangle|
    &\lesssim R^{-1}\|D_hP\|_{L^2}\|V\|_{L^2},\\
|\langle(\tau_hW\cdot\nabla)V,\chi_RV\rangle|
    &\lesssim R^{-1}\|W\|_{L^4}\|V\|_{L^4}\|V\|_{L^2}.
\end{align*}
Both bounds tend to zero for fixed $h$; in particular,
$D_hP\in L^2$ for fixed $h$ follows from $P\in L^2$.
We can therefore send $R\to\infty$ and obtain
\begin{equation}\label{eq:dq_energy}
\begin{aligned}
 \|\nabla V\|_{L^2}^2+\frac14\|V\|_{L^2}^2
&=\frac12\langle\tau_h\partial_\ell u,V\rangle
     +\langle D_hF,V\rangle \\
 &\quad-\langle(\tau_hW\cdot\nabla)B_h,V\rangle
       -\langle((V+B_h)\cdot\nabla)W,V\rangle.
\end{aligned}
\end{equation}

We next estimate the right-hand side \emph{uniformly in $h$}.
The first two terms are bounded by
$\frac12\|\nabla u\|_{L^2}^2$
and $\|\partial_\ell F\|_{L^2}\|\nabla u\|_{L^2}$.
Using $H^1\hookrightarrow L^6$ and the above bounds on $B_h$,
\[
|\langle(\tau_hW\cdot\nabla)B_h,V\rangle|
\leq\bigl(\|u\|_{L^6}\|\nabla B_h\|_{L^3}  +\|b\|_{L^\infty}\|\nabla B_h\|_{L^2}\bigr)\|V\|_{L^2}.
\]
Similarly, the terms
$\langle(V\cdot\nabla)b,V\rangle$ and
$\langle(B_h\cdot\nabla)W,V\rangle$
are controlled by
\[
 \|\nabla b\|_{L^\infty}\|V\|_{L^2}^2
 +\|B_h\|_{L^\infty}
  (\|\nabla u\|_{L^2}+\|\nabla b\|_{L^2})\|V\|_{L^2}.
\]
Hence all these terms are bounded by
$C_{b,F}(1+\|u\|_{H^1}^2)$, independently of $h$.
The only delicate term is the near--near interaction. By H\"older,
the three-dimensional Gagliardo--Nirenberg inequality,
and the difference-quotient bound for $V$,
\begin{align*}
|\langle(V\cdot\nabla)u,V\rangle|
 &\leq\|\nabla u\|_{L^2}\|V\|_{L^4}^2\\
 &\lesssim\|\nabla u\|_{L^2}   \|V\|_{L^2}^{1/2}\|\nabla V\|_{L^2}^{3/2}\\
&\lesssim\|u\|_{H^1}^{3/2}\|\nabla V\|_{L^2}^{3/2}\\
 &\leq\tfrac12\|\nabla V\|_{L^2}^2+C\|u\|_{H^1}^6.
\end{align*}
Inserting these bounds into \eqref{eq:dq_energy} and absorbing the gradient term,
we obtain
\[
\sup_{0<h<1}\|D_h\nabla u\|_{L^2}^2
 \leq C_{b,F}\bigl(1+\|u\|_{H^1}^6\bigr),
 \qquad \ell=1,2,3.
\]
Since $D_h\nabla u\to\partial_\ell\nabla u$ in distributions
as $h\downarrow0$, weak compactness in $L^2$ implies
$\partial_\ell\nabla u\in L^2$ for each $\ell$.
Thus $u=U_{\mathrm{near}}\in H^2(\mathbb R^3)$, completing the
base case of the induction.

We are now ready to perform an induction on $k$ to show that $U_{\mathrm{near}} \in H^k$. The base cases $k = 0, 1, 2$ are established; supposing it holds for any index no greater than $k \geqslant 2$, we take derivatives of \eqref{eq:u_eps_eq}:
\begin{equation}\label{eq:u_eps_k}
-\frac{1+k}{2}\nabla^k u^\epsilon - \frac{1}{2}(\xi \cdot \nabla)\nabla^k u^\epsilon - \frac{1}{2}\nabla^k([\rho_\epsilon *, \xi \cdot \nabla] U_{\mathrm{near}}) + \nabla^k(\rho_\epsilon * \Pi(\widetilde{U} \cdot \nabla \widetilde{U})) - \Delta \nabla^k u^\epsilon = \nabla^k F^\epsilon.
\end{equation}
Take inner products with $\chi_R \nabla^k u^\epsilon$ and take $R \to \infty$.

The drift term has the estimate via an integration by parts as
\[
\langle (\xi \cdot \nabla)\nabla^k u^\epsilon, \chi_R \nabla^k u^\epsilon \rangle = -\frac{1}{2}\langle \nabla \cdot (\chi_R \xi) \nabla^k u^\epsilon, \nabla^k u^\epsilon \rangle = -\frac{3}{2}\langle \chi_R \nabla^k u^\epsilon, \nabla^k u^\epsilon \rangle + O\Big(\int_{R < |\xi| < 2R} |\nabla^k u^\epsilon|^2\Big),
\]
where we use the fact that $|\nabla \chi_R| \lesssim 1/R$. The diffusion term similarly yields
\[
\langle -\Delta \nabla^k u^\epsilon, \chi_R \nabla^k u^\epsilon \rangle = \langle \nabla^{k+1} u^\epsilon, \chi_R \nabla^{k+1} u^\epsilon \rangle + \langle \nabla^{k+1} u^\epsilon, \nabla \chi_R \nabla^k u^\epsilon \rangle = \langle \nabla^{k+1} u^\epsilon, \chi_R \nabla^{k+1} u^\epsilon \rangle + O\Big(\frac{1}{R}\Big).
\]
In the same fashion, the commutator term enjoys the estimate by \eqref{eq:comm_est}
\[
|\langle \nabla^k([\rho_\epsilon *, \xi \cdot \nabla] U_{\mathrm{near}}), \chi_R \nabla^k u^\epsilon \rangle| = |\langle \nabla^{k-1}([\rho_\epsilon *, \xi \cdot \nabla] U_{\mathrm{near}}), \chi_R \nabla^{k+1} u^\epsilon \rangle| + O\Big(\frac{1}{R}\Big) \lesssim \|\nabla^{k+1} u^\epsilon\|_{L^2} + O\Big(\frac{1}{R}\Big),
\]
and the nonlinear term enjoys a similar estimate by \eqref{eq:nl_ge2} (for $k \geqslant 2$)
\[
|\langle \nabla^k(\rho_\epsilon * \Pi(\widetilde{U} \cdot \nabla \widetilde{U})), \chi_R \nabla^k u^\epsilon \rangle| \lesssim \|\nabla^{k+1} u^\epsilon\|_{L^2} + O\Big(\frac{1}{R}\Big).
\]
Again, by Young's inequality, the right-hand side of \eqref{eq:u_eps_k} has the trivial estimate
\[
\langle \nabla^k F^\epsilon, \chi_R \nabla^k u^\epsilon \rangle \leqslant \|\nabla^k F^\epsilon\|_{L^2} \|\nabla^k u^\epsilon\|_{L^2} \leqslant \|\nabla^k F\|_{L^2} \|\nabla^k U_{\mathrm{near}}\|_{L^2} \lesssim 1.
\]

We can take the limit $R \to \infty$ and collect the estimates, treating $\|U_{\mathrm{near}}\|_{H^k}$ as a known constant:
\[
\|\nabla^{k+1} u^\epsilon\|_{L^2}^2 \lesssim 1 + \|\nabla^{k+1} u^\epsilon\|_{L^2},
\]
and we conclude a uniform bound on $\|u^\epsilon\|_{H^{k+1}}$ independent of $\epsilon$. Thus, we conclude that $U_{\mathrm{near}} \in H^{k+1}$ by the theory of mollifiers. We can take a weakly convergent subsequence of $u^\epsilon$ in $H^{k+1}$. Since they converge to $U_{\mathrm{near}}$ in distributions, we know the limits are identical and conclude $U_{\mathrm{near}} \in H^{k+1}$.

We conclude by induction that $U_{\mathrm{near}} \in H^\infty$.
\end{proof}

In particular, along with the decay of $U_{\mathrm{far}}$ \eqref{higher boundary ufar}, we get that \begin{equation}
    \label{l4u}\widetilde{U}\in L^4\cap L^\infty.
\end{equation}

  \begin{proposition}[Far field asymptotics]
  The solution $\widetilde{U}$ in Proposition \ref{prop:U} satisfies the boundary condition \eqref{boundary u0} with the same $A$ as that associated with $\overline{U}$ in Assumption \ref{asm:num}. More precisely,
\[
 \widetilde U(\xi)=\frac{1}{|\xi|}A\!\left(\frac{\xi}{|\xi|}\right)+O(|\xi|^{-2}),\qquad |\xi|\to\infty.
\]
\label{prop decay}
\end{proposition}
\begin{proof}[Proof of Proposition \ref{prop decay}]
  Since $U_{\tmop{far}}$ satisfies the desirable far field condition \eqref{higher boundary ufar}, we only need to establish $U_{\tmop{near}}$ has decay $O(|\xi|^{-2})$ in the far field.
 We will establish that $|\xi|U_{\tmop{near}},|\xi|^2\nabla^2 U_{\tmop{near}}\in L^2$, and use a local Agmon inequality to conclude the far field decay. 
  
 \smallskip
\noindent\textbf{Step 1: $L^2$-level estimate, that $|\xi|U_{\tmop{near}}, |\xi|\nabla U_{\tmop{near}}\in L^2$.} Recall that $U_{\tmop{near}}$ is smooth, and \eqref{eq:unear_eq} holds in a strong sense with $F\in H^\infty$, with decay at the infinity $\nabla^k F=O(|\xi|^{-3-k})$. We rewrite \eqref{eq:unear_eq} in the velocity-pressure formation \begin{equation}
- \frac{1}{2} U_{\tmop{near}} - \frac{1}{2} \xi \cdummy \nabla
       U_{\tmop{near}} + \widetilde{U} \cdot \nabla \widetilde{U}+\nabla P - \Delta
       U_{\tmop{near}} = F,
    \label{pre unear}
\end{equation}
where we can represent the pressure term by taking divergence and using Riesz transform $R_i$ as \begin{equation}
\Delta P = -\tmop{tr}((\nabla \widetilde{U})^2),\quad P=\sum_{i,j}R_iR_j(\widetilde{U}_i\widetilde{U}_j).
    \label{pressure unear}
\end{equation}

We introduce the following cutoff version of the radial weight $|\xi|^2$ as $\phi_{L}$:
\[\phi_{L}=|\xi|^2\chi_L+(2L)^2(1-\chi_L).\]
It holds that $\phi_{L}=|\xi|^2$ if $|\xi|<L$, $\phi_{L}=|2L|^2$ if $|\xi|>2L$, $\phi_{L}$ is radially increasing since $\chi_L$ is radially decreasing, and that $|\Delta\phi_{L}|\lesssim 1$ is uniformly bounded in $L$ due to the estimates on derivatives of $\chi_L$. Now we take inner products of \eqref{pre unear} with $\chi_R\phi_{L} U_{\tmop{near}}$, where the cutoff radius $R$ is much larger than $L$, and we get the following estimates as we take $R\to\infty$.

The drift term has the estimate via an integration by parts as \[\begin{aligned}
&\langle - \frac{1}{2} \xi \cdummy \nabla
       U_{\tmop{near}},\chi_R\phi_{L}U_{\tmop{near}}\rangle=\frac34\langle 
       U_{\tmop{near}},\chi_R\phi_{L}U_{\tmop{near}}\rangle+\frac{1}{4}\langle 
       U_{\tmop{near}},\xi\cdot\nabla(\chi_R\phi_{L})U_{\tmop{near}}\rangle \\\geqslant&  \frac34\langle 
       U_{\tmop{near}},\chi_R\phi_{L}U_{\tmop{near}}\rangle+\frac{L^2}{4}O(\int_{R<|\xi|<2R}|U_{\tmop{near}}|^2), 
\end{aligned}\]where we use again the fact that $|\nabla\chi_R|\lesssim1/R$, and the radial increasing property of $\phi_{L}$. To see this, 
since $\phi_L=\phi_L(|\xi|)$ is radially increasing, we have
$\xi\!\cdot\!\nabla\phi_L=|\xi|\,\phi_L'(|\xi|)\geqslant 0$, hence
\[
\xi\!\cdot\!\nabla(\chi_R\phi_L)=\phi_L\,\xi\!\cdot\!\nabla\chi_R
+\chi_R\,\xi\!\cdot\!\nabla\phi_L \geqslant  \phi_L\,\xi\!\cdot\!\nabla\chi_R.
\]
Therefore, for a lower bound, we may discard the nonnegative contribution
$\chi_R\,\xi\!\cdot\!\nabla\phi_L$ and bound the remaining term on
$\mathrm{supp}\,\nabla\chi_R=\{R<|\xi|<2R\}$ using
$\phi_L=(2L)^2$ and $|\xi\!\cdot\!\nabla\chi_R|\lesssim 1$.

The diffusion term has a similar estimate via an integration by parts twice as \[\begin{aligned}
&\langle - \Delta
       U_{\tmop{near}},\chi_R\phi_{L}U_{\tmop{near}}\rangle=\langle \nabla
       U_{\tmop{near}},\chi_R\phi_{L}\nabla U_{\tmop{near}}\rangle-\frac12\langle 
       U_{\tmop{near}},\Delta(\chi_R \phi_{L})U_{\tmop{near}}\rangle\\\geqslant&\langle \nabla
       U_{\tmop{near}},\chi_R\phi_{L}\nabla U_{\tmop{near}}\rangle+O(\frac{L^2}{R^2}+\frac{L}{R})+O(\langle 
       U_{\tmop{near}}, U_{\tmop{near}}\rangle),\end{aligned}\]
where we use the uniform bound on the Laplacian of $\phi_L$. The nonlinear term splits into 
\[\begin{aligned}
&|\langle \widetilde{U} \cdot \nabla \widetilde{U},\chi_R\phi_{L} U_{\tmop{near}}\rangle|=|\langle \widetilde{U} \cdot \nabla U_{\tmop{far}},\chi_R\phi_{L} U_{\tmop{near}}\rangle-\frac12\langle \widetilde{U}  \cdot\nabla (\chi_R\phi_{L}), |U_{\tmop{near}}|^2\rangle|\\\lesssim&|\langle |\xi|\widetilde{U} \cdot \nabla U_{\tmop{far}},\sqrt{\chi_R\phi_{L}} U_{\tmop{near}}\rangle|+\frac {L^2}{R}+\langle 
       \sqrt{\chi_R\phi_{L}}U_{\tmop{near}}, U_{\tmop{near}}\rangle\lesssim \frac {L^2}{R}+\sqrt{\langle 
       U_{\tmop{near}},\chi_R\phi_{L}U_{\tmop{near}}\rangle},&\end{aligned}\]
where we use the divergence-free of $\widetilde{U}$ in the first equality, $\widetilde{U}\in L^\infty$ and the estimate $|\nabla\phi_L|\lesssim\sqrt{\phi_L}$ in the first inequality. For the second inequality, we use the fact by \eqref{higher boundary ufar} and \eqref{l4u} that $|\xi|\nabla U_{\tmop{far}}, \widetilde{U}\in L^4\cap L^\infty$, and then
apply the Cauchy--Schwarz inequality.
       
The pressure term is controlled by \eqref{pressure unear} using Calderón–Zygmund theory $\|P\|\lesssim \|\widetilde{U}\|^2_{L^4}\lesssim1$:
\[|\langle \nabla P,\chi_R\phi_{L} U_{\tmop{near}}\rangle|=|-\langle P\nabla (\chi_R\phi_{L}), U_{\tmop{near}}\rangle|\lesssim  \frac {L^2}{R}+\sqrt{\langle 
       U_{\tmop{near}},\chi_R\phi_{L}U_{\tmop{near}}\rangle}.\]
The forcing term has the trivial estimate by the Cauchy--Schwarz inequality \[\langle F,\chi_R\phi_{L} U_{\tmop{near}}\rangle\lesssim \langle |\xi|F, \sqrt{\chi_R\phi_{L}}U_{\tmop{near}}\rangle\lesssim\sqrt{\langle 
       U_{\tmop{near}},\chi_R\phi_{L}U_{\tmop{near}}\rangle}.\]

       We can take $R\to\infty$ and collect the estimates that \[\frac14\langle 
       U_{\tmop{near}},\phi_{L}U_{\tmop{near}}\rangle+\langle \nabla
       U_{\tmop{near}},\phi_{L}\nabla U_{\tmop{near}}\rangle\lesssim1+\sqrt{\langle 
       U_{\tmop{near}},\phi_{L}U_{\tmop{near}}\rangle}.\]
       As a consequence, $\langle 
       U_{\tmop{near}},\phi_{L}U_{\tmop{near}}\rangle+\langle \nabla
       U_{\tmop{near}},\phi_{L}\nabla U_{\tmop{near}}\rangle$ has a uniform bound in $L$ and we can take the limit $L \rightarrow \infty$ to conclude that  
       \begin{equation}
       \label{First-bound}
       |\xi|U_{\tmop{near}} \in L^2, \quad \quad|\xi|\nabla U_{\tmop{near}}\in L^2.
       \end{equation}
\smallskip
\noindent\textbf{Step 2: $H^1$-level estimate, that $|\xi|^2\nabla U_{\tmop{near}}, |\xi|^2\nabla^2 U_{\tmop{near}}\in L^2$.} Take derivatives of \eqref{pre unear} to get \begin{equation}
- \nabla U_{\tmop{near}} - \frac{1}{2} \xi \cdummy \nabla^2
       U_{\tmop{near}} + (\nabla\widetilde{U})^2+\widetilde{U} \cdot \nabla^2 \widetilde{U}+\nabla^2 P - \Delta
       \nabla U_{\tmop{near}} = \nabla F.
    \label{pre unear1}
\end{equation}
Now we take inner products of \eqref{pre unear1} with $\chi_R\phi_{L}|\xi|^\alpha\nabla U_{\tmop{near}}$, where the cutoff radius $R$ is much larger than $L$, and we get the following estimates as we take $R\to\infty$. We will first take $\alpha=0$ to derive an intermediate result $|\xi|\nabla^2 U_{\tmop{near}}\in L^2$, and then take $\alpha=2$ to conclude. 

For $\alpha=0,2$, the drift term has a similar estimate as that in Step 1:
\[\begin{aligned}
&\langle - \frac{1}{2} \xi \cdummy \nabla^2
       U_{\tmop{near}},\chi_R\phi_{L}|\xi|^\alpha\nabla U_{\tmop{near}}\rangle=\frac{\alpha+3}4\langle 
       \nabla U_{\tmop{near}},\chi_R\phi_{L}|\xi|^\alpha\nabla U_{\tmop{near}}\rangle+\frac{1}{4}\langle 
       \nabla U_{\tmop{near}},|\xi|^\alpha\xi\cdot\nabla(\chi_R\phi_{L})\nabla U_{\tmop{near}}\rangle \\&\geqslant  \frac{\alpha+3}4\langle 
       \nabla U_{\tmop{near}},\chi_R\phi_{L}|\xi|^\alpha\nabla U_{\tmop{near}}\rangle+\frac{L^2}{4}O(\int_{R<|\xi|<2R}|\xi|^\alpha|\nabla U_{\tmop{near}}|^2).
\end{aligned}\]The diffusion term has a similar estimate\[\begin{aligned}
&\langle - \Delta\nabla U_{\tmop{near}},\chi_R\phi_{L}|\xi|^\alpha\nabla U_{\tmop{near}}\rangle=\langle \nabla^2
       U_{\tmop{near}},\chi_R\phi_{L}|\xi|^\alpha\nabla^2 U_{\tmop{near}}\rangle-\frac12\langle 
       \nabla U_{\tmop{near}},\Delta(\chi_R \phi_{L}|\xi|^\alpha)\nabla U_{\tmop{near}}\rangle\\\geqslant&\langle \nabla^2
       U_{\tmop{near}},\chi_R\phi_{L}|\xi|^\alpha\nabla^2 U_{\tmop{near}}\rangle+O(\frac{L^2}{R^2}+\frac{L+L^2}{R})+O(\langle 
       \nabla U_{\tmop{near}}, |\xi|^\alpha\nabla U_{\tmop{near}}\rangle),\end{aligned}\]
       where we use the bound $\Delta(\phi_{L}|\xi|^\alpha)\lesssim|\xi|^\alpha$ uniformly in $L$. The first nonlinear term has the estimate\[|\langle ( \nabla \widetilde{U})^2,\chi_R\phi_{L} |\xi|^\alpha\nabla U_{\tmop{near}}\rangle|\lesssim\||\xi|^{\frac{2+\alpha}{2}}(\nabla \widetilde{U})^2\|_{L^2}\sqrt{\langle \nabla
       U_{\tmop{near}},\chi_R\phi_{L}|\xi|^\alpha\nabla U_{\tmop{near}}\rangle}.\]
       When $\alpha=0$, we have the estimate via \eqref{First-bound} and the decay of $U_{\tmop{far}}$:
       \[\||\xi|^{\frac{2+\alpha}{2}}(\nabla \widetilde{U})^2\|_{L^2}\lesssim\||\xi|^{\frac12}\nabla U_{\tmop{far}}\|_{L^4}^2+\||\xi|\nabla U_{\tmop{near}}\|_{L^2}\|\nabla U_{\tmop{near}}\|_{L^\infty}\lesssim1,\]
       and when $\alpha=2$, we have the estimate by the Gagliardo–Nirenberg inequality that \[\||\xi|^{\frac{2+\alpha}{2}}(\nabla \widetilde{U})^2\|_{L^2}\lesssim\||\xi|\nabla U_{\tmop{far}}\|_{L^4}^2+\||\xi|\nabla U_{\tmop{near}}\|_{L^2}^{1/2}\|\nabla(|\xi|\nabla U_{\tmop{near}})\|_{L^2}^{3/2}\lesssim1+\||\xi|\nabla^2 U_{\tmop{near}}\|_{L^2}^{3/2}.\]
       We collect the useful estimate for the first nonlinear term as\begin{equation}
           \label{non 111}\||\xi|^{\frac{2+\alpha}{2}}(\nabla \widetilde{U})^2\|_{L^2}\lesssim1+\||\xi|^{\alpha/2}\nabla^2 U_{\tmop{near}}\|_{L^2}^{3/2}.
       \end{equation}
       The second nonlinear term splits, similarly as in step 1, into \[\begin{aligned}
&|\langle \widetilde{U} \cdot \nabla^2 \widetilde{U},\chi_R\phi_{L} |\xi|^\alpha\nabla U_{\tmop{near}}\rangle|=|\langle \widetilde{U} \cdot \nabla^2 U_{\tmop{far}},\chi_R\phi_{L} |\xi|^\alpha\nabla U_{\tmop{near}}\rangle-\frac12\langle \widetilde{U}  \cdot\nabla (\chi_R\phi_{L}|\xi|^\alpha), |\nabla U_{\tmop{near}}|^2\rangle|\\&\lesssim|\langle |\xi|^2\widetilde{U} \cdot \nabla^2 U_{\tmop{far}},\sqrt{\chi_R\phi_{L}|\xi|^\alpha}\nabla U_{\tmop{near}}\rangle|+\frac {L^2}{R}+\langle 
       \sqrt{\chi_R\phi_{L}|\xi|^\alpha}\nabla U_{\tmop{near}}, \sqrt{|\xi|^\alpha}\nabla U_{\tmop{near}}\rangle\\&\lesssim \frac {L^2}{R}+\sqrt{\langle 
       \nabla U_{\tmop{near}},\chi_R\phi_{L}|\xi|^\alpha \nabla U_{\tmop{near}}\rangle}.&\end{aligned}\]
       The pressure term is controlled by \eqref{pressure unear} using the Calderón–Zygmund theory and the fact that $|\xi|^\alpha$ is an $A_2$-weight for $-3<\alpha<3$ (see  Thm.~5.24 in \cite{grafakos2014}), so that for $\alpha=0,2$, we have by \eqref{non 111} that $$\|\nabla P|\xi|^{\alpha/2}\|_{L^2}\lesssim \|\nabla (\widetilde{U}\otimes\widetilde{U})|\xi|^{\alpha/2}\|_{L^2}\lesssim1+\sqrt{\||\xi|^{\frac{2+\alpha}{2}}(\nabla \widetilde{U})^2\|_{L^2}\| \widetilde{U}^2\|_{L^2}}\lesssim1+\||\xi|^{\alpha/2}\nabla^2 U_{\tmop{near}}\|_{L^2}^{3/4},$$
Here the additional constant controls the region $|\xi|\le1$, where $\widetilde U$ is smooth; on $|\xi|\ge1$, the stated H\"older bound holds. This constant is independent of $L$ and $R$.
As a consequence we have \[\begin{aligned}
           &|\langle \nabla^2 P,\chi_R\phi_{L}|\xi|^\alpha\nabla U_{\tmop{near}}\rangle|=|-\langle \nabla  P, \nabla U_{\tmop{near}}^T\nabla (\chi_R\phi_{L}|\xi|^\alpha)\rangle|\\\lesssim & \frac {L^2}{R}+(1+\||\xi|^{\alpha/2}\nabla^2 U_{\tmop{near}}\|_{L^2}^{3/4})\sqrt{\langle 
       \nabla U_{\tmop{near}},\chi_R\phi_{L}|\xi|^\alpha\nabla U_{\tmop{near}}\rangle}.
       \end{aligned}\]
       Finally the forcing term has the trivial estimate by the Cauchy-Schwarz inequality\[\langle \nabla F,\chi_R\phi_{L}|\xi|^\alpha\nabla U_{\tmop{near}}\rangle\lesssim \langle |\xi|^2\nabla F, \sqrt{\chi_R\phi_{L}|\xi|^\alpha}\nabla U_{\tmop{near}}\rangle\lesssim\sqrt{\langle 
       \nabla U_{\tmop{near}},\chi_R\phi_{L}|\xi|^\alpha\nabla U_{\tmop{near}}\rangle}.\]

       We first choose $\alpha=0$ and take $R\to\infty$ to conclude a uniform bound on $\langle \nabla^2
       U_{\tmop{near}},\phi_{L}\nabla^2 U_{\tmop{near}}\rangle$. Taking the limit $L\rightarrow \infty$ yields $|\xi|\nabla^2 U_{\tmop{near}}\in L^2$.
       Now we choose $\alpha=2$ and take $R\to\infty$ to get \[\frac14\langle 
       \nabla U_{\tmop{near}},\phi_{L}|\xi|^2\nabla U_{\tmop{near}}\rangle+\langle \nabla^2
       U_{\tmop{near}},\phi_{L}|\xi|^2\nabla^2 U_{\tmop{near}}\rangle\lesssim1+\sqrt{\langle 
       \nabla U_{\tmop{near}},\phi_{L}|\xi|^2\nabla U_{\tmop{near}}\rangle}.\]
       As a consequence, $\langle 
       \nabla U_{\tmop{near}},\phi_{L}|\xi|^2\nabla U_{\tmop{near}}\rangle+\langle \nabla^2
       U_{\tmop{near}},\phi_{L}|\xi|^2\nabla^2 U_{\tmop{near}}\rangle$ has a uniform bound in $L$ and we can take the limit $L\rightarrow \infty$ to conclude that \begin{equation}
       \label{Second-bound}
       |\xi|^2\nabla U_{\tmop{near}} \in L^2, \quad \quad |\xi|^2\nabla^2 U_{\tmop{near}}\in L^2.
       \end{equation}

We have hence conclude the proof.
We remark that by the same argument, we can moreover show that $|\xi|^l\nabla^k U_{\tmop{near}}$ for any $l\leqslant k+1$ so that we can extract a decay of $O(|\xi|^{-5/2-k})$ for $\nabla^k U_{\tmop{near}}$ using the dyadic argument. This is not necessary for our purpose.
\smallskip

\noindent\textbf{Step 3: Dyadic rescaling and decay.} Fix $R>1$ and the annulus $A_R=\{\xi\in\mathbb{R}^3:\ R<|\xi|<2R\}$.
Define $v(y):=U_{\mathrm{near}}(R y)$ on $A_1=\{y:\ 1<|y|<2\}$. The local Agmon inequality on $A_1$ implies
\[
\begin{aligned}
\|U_{\mathrm{near}}\|_{L^\infty(A_R)}
&=\|v\|_{L^\infty(A_1)}\\
&\leqslant C\left(\|v\|_{L^2(A_1)}
+\|\nabla_y v\|_{L^2(A_1)}^{1/2}
\|\nabla_y^2 v\|_{L^2(A_1)}^{1/2}\right)\\
&= C\left(R^{-3/2}\|U_{\mathrm{near}}\|_{L^2(A_R)}
+\|\nabla U_{\mathrm{near}}\|_{L^2(A_R)}^{1/2}
\|\nabla^2 U_{\mathrm{near}}\|_{L^2(A_R)}^{1/2}\right).
\end{aligned}
\]
Since $|\xi|\sim R$ on $A_R$ and the global weighted bounds imply
$\|U_{\mathrm{near}}\|_{L^2(A_R)}\leqslant CR^{-1}$ and $\|\nabla^kU_{\mathrm{near}}\|_{L^2(A_R)}\leqslant CR^{-2}$ for $k=1,2$, we conclude
\[
\|U_{\mathrm{near}}\|_{L^\infty(A_R)}\leqslant C\,R^{-2},
\]
that is, $U_{\mathrm{near}}(\xi)=O(|\xi|^{-2})$ as $|\xi|\to\infty$, the desired decay.

\end{proof}

Thus $\widetilde U$ induces a self-similar local-Leray solution; localization below produces finite-energy Leray--Hopf solutions.

\subsection{Nonuniqueness of Leray--Hopf solutions}
\label{Sec2:nonuniqueness}
\begin{proof}[Proof of Theorem \ref{thm:main}] Let $\widetilde{u}$ be the self-similar solution corresponding to the profile $\widetilde{U}$ in Theorem \ref{thm:stationary solu main}. Recall that due to the slow decay of the scale-invariant initial data $u_{\text{in}}$, $\widetilde{u}$ is not a Leray-Hopf solution. Therefore, we introduce a suitable cutoff function $u^{\text{cut}}$ and
consider the following ansatz $u=\widetilde{u}+u^{\text{cut}}+u^{\text{cor}}$ to the Navier--Stokes equation \eqref{eq:NS eq}, where  $u^{\text{cor}}$ is the correction to be constructed that renders non-unique solutions. We adopt the change of variables between physical coordinates $(x,t)$ and self-similar coordinates $(\xi,\tau)$:  \begin{equation}
    \widetilde{u}(x,t)=\frac{1}{\sqrt{t}}\widetilde{U}(\frac{x}{\sqrt{t}}),\quad u^{\text{cut}}(x,t)=\frac{1}{\sqrt{t}}U^{\text{cut}}(\frac{x}{\sqrt{t}},\log t),\quad u^{\text{cor}}(x,t)=\frac{1}{\sqrt{t}}U^{\text{cor}}(\frac{x}{\sqrt{t}},\log t).\label{cov}
\end{equation}  Please note that $u=\widetilde{u}+u^{\text{cut}}+u^{\text{cor}}$ is the Leray-Hopf solution that we construct in this subsection. It does not correspond to the correction $U$ that we used in \eqref{Decomp-Utilde} to adjust $\overline{U}$; consequently, $u$ does not satisfy the relation $u(t,x) = \frac{1}{\sqrt{t}}U(\frac{x}{\sqrt{t}},\log t)$ with that $U$. We will first construct the cutoff function, as in \cite{jia2015incompressible}. Then we demonstrate that different choices of $u^{\text{cor}}(\cdot,1)$ correspond to different solutions in the rescaled time interval $\tau\in (-\infty,0]$, as in \cite{albritton2022non} and \cite{jia2015incompressible}. Here, it is crucial to treat the stable and unstable parts separately via the Riesz projection, as in \cite{jia2015incompressible}; see also \cite{chen2024vorticity, chen2025vorticity}.

\paragraph{Construction of cutoff.} As in \cite{jia2015incompressible}, via the Bogovskii operator, one can construct a decomposition of the initial condition $u_{\tmop{in}}=u_{\tmop{loc}}+w$ such that $u_{\tmop{loc}}$ is compactly supported, divergence free, $w(x)=0$ for $|x|<R$ and $\|w\|_{L^p}\lesssim R^{-(p-3)/p}$ based on the decay of $u_{\tmop{in}}$, for any $p>3$. We will choose $u_{\tmop{loc}}$ of compact support as the modified initial data for our Leray--Hopf solution $u$.

Now we define our cutoff function as\[u^{\text{cut}}=-e^{t
\Delta}w,\]
where $e^{t
\Delta}$ is the standard notation for the heat semigroup. Recall the standard smoothing estimate:
\begin{equation}\label{eq:heat-smoothing}
\|\nabla_x^k e^{t\Delta}f\|_{L^r}\lesssim t^{-\frac{k}{2}}\,t^{-\frac{3}{2}(\frac1p-\frac1r)}\|f\|_{L^p},
\quad 1\leqslant p\leqslant r\leqslant\infty.
\end{equation}
By the similarity transform, we have the identities
\begin{equation}\label{eq:scaling}
\nabla_\xi^k U^{\text{cut}}(\xi,\tau)
=t^{\frac{k}{2}+\frac12}(\nabla_x^k u^{\text{cut}})(x,t)\Big|_{x=\sqrt t\xi},
\quad 
\|g(\sqrt t\,\cdot)\|_{L^r}=t^{-\frac{3}{2r}}\|g\|_{L^r}.
\end{equation}
Combining \eqref{eq:scaling}, \eqref{eq:heat-smoothing}, and the bound on the $p$-norm of $w$, we conclude the estimate
\begin{equation}
    \|U^{\text{cut}}\|_{W^{k,r}}
\lesssim  t^{\frac12-\frac{3}{2p}}\|w\|_{L^p}
\lesssim (\frac{e^{\frac\tau2}}{R})^{(1-\frac{3}{p})}.
\label{ucut est}
\end{equation}
We fix $p=4$ from now on. 
The constants in the inequalities are independent of $R$ and $\tau$.
\paragraph{Duhamel setup.}
We will find solutions $u^{\text{cor}}$ to the following perturbative equation to \eqref{eq:NS eq} \[\partial_tu^{\text{cor}}-\Delta u^{\text{cor}}+\Pi(u\cdot\nabla u-\widetilde{u}\cdot\nabla\widetilde{u})=0,\quad \tmop{div}u^{\text{cor}}=0,\]
with $u^{\text{cor}}\to0$ in $L^2$ as $t\to 0$ so that $u=\widetilde{u}+u^{\text{cut}}+u^{\text{cor}}$ will be a solution to the Navier--Stokes equation in $L^2$ with initial data $u_{\tmop{in}}-w=u_{\tmop{loc}}$. In  similarity coordinates, the equation for $U^{\text{cor}}$ reads \begin{equation}
    \partial_\tau U^{\text{cor}}=\mathcal{L}_{\widetilde{U}}U^{\text{cor}}+\mathcal{L}_{\tmop{far}}U^{\text{cor}}+F+N(U^{\text{cor}},U^{\text{cor}}),\label{ucor eqn}
\end{equation}where we use the definition of $\mathcal{L}_{\widetilde{U}}$ in \eqref{eq:LU def} and define the linear, quadratic, and residue terms as \[\mathcal{L}_{\tmop{far}}f\assign-\Pi (U^{\text{cut}}\cdot\nabla f+f\cdot\nabla U^{\text{cut}}),\quad N(f,g)\assign-\Pi (f\cdot\nabla g),\quad F\assign-\Pi (U^{\text{cut}}\cdot\nabla U^{\text{cut}}+\widetilde{U}\cdot\nabla U^{\text{cut}}+U^{\text{cut}}\cdot\nabla \widetilde{U}).\]

By \eqref{ucut est},  the fact that $H^2$ is an algebra, namely,
\[
\| f g \|_{H^2} \lesssim \| f \|_{H^2} \| g \|_{H^2},
\]
together with the fact that $U^{\text{cut}}, \widetilde{U}$ are in $W^{\infty,4}$, we have \begin{equation}
\|\mathcal{L}_{\tmop{far}}f(\cdot,s)\|_{H^2}\lesssim(\frac{e^{\frac s2}}{R})^{(1-\frac{3}{p})}\|f\|_{H^3},\quad\|N(f,g)\|_{H^2}\lesssim\|f\|_{H^3}\|g\|_{H^3},\quad\|F(\cdot,s)\|_{H^2}\lesssim(\frac{e^{\frac s2}}{R})^{(1-\frac{3}{p})}.
    \label{nl est}
\end{equation}
\paragraph{Decomposition via Riesz projection.}
We use the complex $L^2_\sigma$ realization of $\mathcal L_{\widetilde U}$ with domain
$D=\{f\in H^2\cap L^2_\sigma:\ \xi\cdot\nabla f\in L^2\}$, as in \cite[(3.18)]{albritton2022non}.
For $\mathcal L_0=\Delta+\tfrac12(1+\xi\cdot\nabla)$, the weak eigenvalue equation and
$\widetilde U\in W^{1,\infty}$, $\widetilde v\in H^1_\sigma$ give $\mathcal L_0\widetilde v\in L^2$.
Local elliptic regularity and the $L^2$ maximal-domain characterization of $\mathcal L_0$
in \cite[(4.1) and Theorem~4.1]{metafune2002domain} therefore imply $\widetilde v\in D$.

Write $\mathcal L_{\widetilde U}=\mathcal L_0+\mathcal B$.
By \eqref{higher boundary ufar} and Proposition~\ref{prop smooth}, both
$\widetilde U$ and $\nabla\widetilde U$ are bounded and vanish at infinity.
For each $s>0$, $\mathcal B e^{s\mathcal L_0}$ is compact on $L^2_\sigma$,
by parabolic smoothing, local Rellich compactness, and decay of these coefficients.
Moreover,
\[
 \|\mathcal B e^{s\mathcal L_0}\|_{L^2\to L^2}\lesssim s^{-1/2},\qquad 0<s\le1.
\]
The Duhamel formula and the integrability of this bound at $s=0$ show that
$e^{t\mathcal L_{\widetilde U}}-e^{t\mathcal L_0}$ is compact for every $t>0$.
Since $\|e^{t\mathcal L_0}\|_{L^2\to L^2}\le e^{-t/4}$, the essential semigroup
growth bound satisfies $\omega_{\mathrm{ess}}(\mathcal L_{\widetilde U})\le-1/4$.
This follows the compact-semigroup argument of Lemma~2.7 in
\cite{jia2015incompressible}; see also Section~4.1 of \cite{albritton2022non}.
Lemma~4.2 of \cite{albritton2022non} then yields the finite-dimensional unstable decomposition
and the stable semigroup estimate \eqref{bounds}.

We proceed as in the proof of Theorem 4.1 in \cite{jia2015incompressible} to decompose into stable and unstable subspaces $L^2_\sigma=X_s\oplus_jX_j  $ via the Riesz projection $P_{\lambda_j}:L^2_\sigma\to X_j$. For every eigenvalue $\lambda_j$ with a positive real part, we define $P_{\lambda_j}=\frac{1}{2\pi i}\int_{\Gamma_j} (\lambda-\mathcal{L}_{\widetilde{U}})^{-1}d\lambda$, where $\Gamma_j$ is a suitable contour containing ${\lambda_j}$ but no other spectral values
. Here, each  $X_j$ is finite dimensional, and we define $P_s=I-\sum_jP_{\lambda_j}$
. Finally, we denote $A_j=\mathcal{L}_{\widetilde{U}}|_{X_j}$ and $A_s=\mathcal{L}_{\widetilde{U}}|_{X_s}$.  Using Jordan block computation, we have the following estimates for any $\delta>0$ by
\cite{engel2000one}:\begin{equation}
     \begin{aligned}
         &e^{\Re(\lambda_j)\tau-\delta|\tau|}\|f\|_{L^2}\lesssim \|e^{\tau A_j}f\|_{L^2} \lesssim e^{\Re(\lambda_j)\tau+\delta|\tau|}\|f\|_{L^2},\quad \forall f\in X_j, \forall\tau,\\&\|e^{ \tau A_s}f\|_{L^2} \lesssim e^{\delta\tau}\|f\|_{L^2},\quad \forall f\in X_s, \tau>0,
     \end{aligned}\label{bounds}
 \end{equation} 
 We define the fixed-point map $\mathcal T$ by\begin{equation}
     \begin{aligned}
         &\mathcal{T}(U)_s(\cdot,\tau)\assign\int_{-\infty}^\tau e^{(\tau-s)A_s}P_s(\mathcal{L}_{\tmop{far}}U+F+N(U,U))ds,\\&\mathcal{T}(U)_j(\cdot,\tau)\assign e^{\tau A_j}U_{j0}-\int_{\tau}^0 e^{(\tau-s)A_j}P_j(\mathcal{L}_{\tmop{far}}U+F+N(U,U))ds,
     \end{aligned}\label{duh}
 \end{equation} and we need to show the existence of fixed points $\mathcal{T}(U)_s=P_sU$, $\mathcal{T}(U)_j=P_jU$, for appropriately chosen $U_{j0}=P_jU(\cdot,
 0)\in X_j$. We remark that by Theorem \ref{thm:stationary solu main}, we have at least one unstable eigenmode with eigenvalue $\lambda_j$. The Riesz projections are bounded by their resolvent-contour definitions.
\paragraph{Smoothing estimates of $A_s$.} With the growth bound \eqref{bounds}, we obtain the smoothing estimate as in Lemma 4.4 of \cite{albritton2022non}: for any $\delta>0$, $\sigma_2\geqslant \sigma_1\geqslant0$, there exists a constant depending on $\delta,\sigma_2,\sigma_1$ only, such that for any $f_s\in X_s\cap H^{\sigma_1}$ and $\tau>0$, we have \begin{equation}
    \|e^{\tau A_s}f_s\|_{H^{\sigma_2}}\lesssim \tau^{-(\sigma_2-\sigma_1)/2} e^{\delta\tau}\|f_s\|_{H^{\sigma_1}},
    \label{smooth eqn}
\end{equation}
whose proof essentially combines the $L^2$ growth bound with a gain of regularity for the short time, as in Step 1 in the proof of Lemma 4.4 in \cite{albritton2022non}. Compact support of $\widetilde{U}$ is not essential for the proof therein, as only its boundedness in $C^k$ norms is used in the standard estimates of energy dissipation along with a gain of regularity by the heat kernel, relating back to the equation in physical time.

Every finite-dimensional unstable eigenspace $X_j$ is the kernel of $(A_j-\lambda_j)^{k_j}$ for some finite $k_j$. Thus it is spanned by basis functions $\{f_{j,k}\}$, where every $f_{j,k}$ satisfies
\[
(A_j-\lambda_j)^{k_j} f_{j,k} = 0.
\]
Similarly, we have a smoothing estimate for $A_j$, so the basis functions $f_{j,k}$ are smooth. 
\paragraph{Contraction mapping.} We recall $p=4$ and fix $\delta=\frac{1}{2}\min\{\frac{p-3}{2p},\min_j\Re(\lambda_j)\}>0$. Let $X$ consist of continuous $H^3_\sigma$-valued functions on $(-\infty,0]$ with finite norm
\[
 \|U\|_X\assign\sup _{\tau\le0} R^{\frac{p-3}{2p}}e^{-\delta\tau}\|U(\cdot,\tau)\|_{H^3}.
\] We will show that there exists a sufficiently large $R$ such that the Duhamel mapping $\mathcal{T}$ in \eqref{duh} induces a contraction on the unit ball of $X$. 

Thanks to the integral computation for $\delta_1>\delta$ that \[\int_{-\infty}^\tau (\tau-s)^{-1/2}e^{\delta(\tau-s)}e^{\delta_1 s}ds\leqslant e^{\delta\tau}(\int_{-\infty}^{\tau-1} e^{(\delta_1-\delta) s}ds+\int_{\tau-1}^\tau (\tau-s)^{-1/2}ds)\lesssim e^{\delta\tau},\]
by \eqref{duh}, \eqref{smooth eqn}, and \eqref{nl est},  we have for $\tau<0$, \begin{equation}
\begin{aligned}
    \|\mathcal{T}(U)_s(\tau)\|_{H^3}&\lesssim\int_{-\infty}^\tau (\tau-s)^{-1/2}e^{\delta(\tau-s)}((\frac{e^{\frac s2}}{R})^{(1-\frac{3}{p})}(1+\|U(s)\|_{H^3})+\|U(s)\|_{H^3}^2)ds\\&\lesssim e^{\delta\tau}R^{-1+\frac{3}{p}}(1+\|U\|_X+\|U\|_X^2),\\
\|\mathcal{T}(U)_s(\tau)-\mathcal{T}(V)_s(\tau)\|_{H^3}&\lesssim\int_{-\infty}^\tau (\tau-s)^{-1/2}e^{\delta(\tau-s)}((\frac{e^{\frac s2}}{R})^{(1-\frac{3}{p})}+\|U\|_{H^3}+\|V\|_{H^3})\|U-V\|_{H^3}ds\\&\lesssim e^{\delta\tau}R^{-1+\frac{3}{p}}(1+\|U\|_X+\|V\|_X)\|U-V\|_X.\end{aligned}   \label{H3norm-2} 
\end{equation}
Multiplying $R^{\frac{p-3}{2p}}e^{-\delta\tau}$ to \eqref{H3norm-2}, and taking $\sup_{\tau<0}$ give
\begin{equation}
\|\mathcal{T}(U)_s\|_{X}\lesssim R^{-\frac{p-3}{2p}}\big(1+\|U\|_{X}+\|U\|_{X}^{2}\big),\quad
\|\mathcal{T}(U)_s-\mathcal{T}(V)_s\|_{X}\lesssim R^{-\frac{p-3}{2p}} \big(1+\|U\|_{X}+\|V\|_{X}\big)\|U-V\|_{X}.
\label{H3norm-3} 
\end{equation}
Thus, the stable part of $\mathcal{T}$ is a strict contraction on the unit ball of $X$ for $R$ sufficiently large.

Since $U_{j0} \in X_j$ are smooth, using \eqref{bounds}, 
for $\tau<0$, we have similarly that\begin{equation*}
\begin{aligned}\|\mathcal{T}(U)_j(\cdot,\tau)\|_{H^3}&\lesssim e^{(\Re(\lambda_j)-\delta)\tau}\|U_{j0}\|_{H^3}+\int_{\tau}^0 e^{(\Re(\lambda_j)-\delta)(\tau-s)}((\frac{e^{\frac s2}}{R})^{(1-\frac{3}{p})}(1+\|U\|_{H^3})+\|U\|_{H^3}^2)ds\\&\lesssim e^{\delta\tau}\|U_{j0}\|_{H^3}+e^{\delta\tau} R^{-1+\frac{3}{p}}(1+\|U\|_X+\|U\|_X^2),\end{aligned}\label{H3norm-4} 
\end{equation*}\begin{equation*}
\begin{aligned}\|\mathcal{T}(U)_j(\cdot,\tau)-\mathcal{T}(V)_j(\cdot,\tau)\|_{H^3}&\lesssim\int_{\tau}^0 e^{(\Re(\lambda_j)-\delta)(\tau-s)}((\frac{e^{\frac s2}}{R})^{(1-\frac{3}{p})}+\|U\|_{H^3}+\|V\|_{H^3})\|U-V\|_{H^3}ds\\&\lesssim e^{\delta\tau}R^{-1+\frac{3}{p}}(1+\|U\|_X+\|V\|_X)\|U-V\|_X.\end{aligned}\label{H3norm-5} 
\end{equation*}
Similarly, by multiplying $R^{\frac{p-3}{2p}}e^{-\delta\tau}$ to the above inequalities and taking $\sup_{\tau<0}$ give\begin{equation}
\|\mathcal{T}(U)_j\|_{X}\lesssim R^{\frac{p-3}{2p}}\|U_{j0}\|_{H^3}+R^{-\frac{p-3}{2p}}(1+\|U\|_X+\|U\|_X^2),
\label{H3norm-6} 
\end{equation}
\begin{equation}
\|\mathcal{T}(U)_j-\mathcal{T}(V)_j\|_{X}\lesssim R^{-\frac{p-3}{2p}}(1+\|U\|_X+\|V\|_X)\|U-V\|_X.
\label{H3norm-7} 
\end{equation}

Denote \(\alpha:=1-\tfrac{3}{p}=\frac14\). 
For a small $\eta\in(0,1)$, consider the nonempty admissible set of the unit ball of $X$ with fixed unstable projections at $\tau=0$:
\[
\mathcal D_{R,\eta,U_{j0}}:=\Bigl\{\|U\|_X\leqslant1: \ P_jU(\cdot,0)=U_{j0}\quad \forall j\Bigr\},
\]
where $\sum_j R^{\alpha/2}\|U_{j0}\|_{H^3}\leqslant\eta$. Let $C$ denote the maximal absolute constant in the inequalities \eqref{H3norm-3}, \eqref{H3norm-6}, and \eqref{H3norm-7}. 
Choose \(R\) so large that \(C \sum_jR^{-\alpha/2}\leqslant\frac{1}{12}\), and then choose \(\eta>0\) small enough so that
\(C\eta\leqslant\tfrac14\). For any \(U\in\mathcal D_{R,\eta,U_{j0}}\), we get by \eqref{H3norm-3} and \eqref{H3norm-6} that
\[
\|\mathcal T(U)\|_{X}\ \leqslant\ C R^{\alpha/2}\sum_j\|U_{j0}\|_{H^3} + 2 C \sum_jR^{-\alpha/2}\bigl(1+\|U\|_X+\|U\|_X^2\bigr)
 \leqslant\tfrac14 + \tfrac12 <1,
\]
and hence we have \(\mathcal T(\mathcal D_{R,\eta,U_{j0}})\subset \mathcal D_{R,\eta,U_{j0}}\).
Moreover, we obtain by \eqref{H3norm-3} and \eqref{H3norm-7} that
\[
\|\mathcal T(U)-\mathcal T(V)\|_{X}
 \leqslant2C \sum_j R^{-\alpha/2}\bigl(1+\|U\|_X+\|V\|_X\bigr)\|U-V\|_X
 \leqslant\tfrac12\,\|U-V\|_X.
\]
Thus \(\mathcal T\) is a strict contraction on \(\mathcal D_{R,\eta,U_{j0}}\).
 Hence, by the Banach fixed-point theorem, $\mathcal T$ has a fixed point for each sufficiently small admissible choice of $U_{j0}$. Each fixed point solves \eqref{ucor eqn}, and varying $U_{j0}$ yields infinitely many distinct solutions, distinguished by their projections at $\tau=0$, i.e., $t=1$.

\paragraph{Conclusion of the proof.}
Set $V=\widetilde U-e^\Delta u_{\mathrm{in}}\in H^1$. Then
\[
 u(x,t)=e^{t\Delta}u_{\mathrm{loc}}(x)
 +t^{-1/2}V(x/\sqrt t)+u^{\mathrm{cor}}(x,t).
\]
The far-field decay established in Proposition~\ref{prop decay}, together with the heat-kernel asymptotics, gives $V\in L^q$ for every $q>2$.
Since $u_{\mathrm{loc}}$ is compactly supported and behaves like $|x|^{-1}$
near zero, it belongs to $L^2$ and to $L^{3-\epsilon}$ for $0<\epsilon<1$.
The heat semigroup estimates yield
\begin{equation}\label{norm uuuu}
 \|\widetilde u(t)+u^{\mathrm{cut}}(t)\|_{L^q}
 \lesssim 1+t^{\frac{3}{2q}-\frac{1}{2(1-\epsilon/3)}},
 \qquad
 \lim_{t\downarrow0}\|\widetilde u(t)+u^{\mathrm{cut}}(t)
 -u_{\mathrm{loc}}\|_2=0.
\end{equation}
Indeed, the latter limit follows from $e^{t\Delta}u_{\mathrm{loc}}\to u_{\mathrm{loc}}$ in $L^2$
and $\|t^{-1/2}V(\cdot/\sqrt t)\|_2=t^{1/4}\|V\|_2$.
The correction satisfies
$\|u^{\mathrm{cor}}(t)\|_2=t^{1/4}\|U^{\mathrm{cor}}(\log t)\|_2\to0$,
so $u(t)\to u_{\mathrm{loc}}$ strongly in $L^2$.
For the gradients, the profile remainder $t^{-1/2}V(x/\sqrt t)$ and the correction $u^{\mathrm{cor}}(x,t)$ have norms
$t^{-1/4}\|\nabla V\|_2$ and
$t^{-1/4}\|\nabla U^{\mathrm{cor}}(\log t)\|_2$, respectively, whose
squares are integrable at zero. The heat part satisfies
$\int_0^1\|\nabla e^{t\Delta}u_{\mathrm{loc}}\|_2^2\,dt
\le\tfrac12\|u_{\mathrm{loc}}\|_2^2$.
Thus $u\in L^\infty_tL^2_x\cap L^2_tH^1_x$.
The smoothing estimates and \eqref{duh} bootstrap the correction to all
spatial Sobolev orders on each positive-time interval, so $u$ is smooth
for $t>0$. Its global energy equality on $[s,t]\subset(0,1]$ is
\[
 \tfrac12\|u(t)\|_2^2+\int_s^t\|\nabla u(r)\|_2^2\,dr
 =\tfrac12\|u(s)\|_2^2.
\]
Strong initial $L^2$ convergence permits $s\downarrow0$, giving the
Leray--Hopf energy equality from the prescribed initial data.
Positive-time smoothness also implies that \eqref{suitability} holds with equality, and hence the solutions are suitable.
Thus, we obtain infinitely many distinct suitable Leray--Hopf solutions with the same initial data $u_{\mathrm{loc}}$.

By \eqref{norm uuuu}, the scaling \eqref{eq:scaling}, and  $U^{\text{cor}}\in H^3\subset L^q$, we compute for $q\geqslant2$,  $$\|u\|_{L^q}\leqslant\|u^{\text{cut}}+\widetilde u\|_{L^q}+\|u^{\text{cor}}\|_{L^q}\lesssim t^{\frac{3}{2q}-\frac{1}{2(1-\epsilon/3)}}+1+t^{\frac{3}{2q}-\frac12}\lesssim t^{\frac{3}{2q}-\frac{1}{2(1-\epsilon/3)}}+1.$$ Choosing $\epsilon>0$ sufficiently small, we conclude the $L^sL^q$ boundedness  and the proof of Theorem \ref{thm:main}.    
\end{proof}
\begin{remark}\label{rmk4}
    {Our choice of localization is propagated in time via the heat smoothing. It would be interesting to explore an alternative approach based on a static spatial cutoff and direct estimates of the boundary terms generated by the cutoff. In particular, one may hope to combine such an approach with inner–outer gluing ideas used by Albritton--Brue--Colombo \cite{albritton2023gluing}, to investigate whether an analogue of Theorem \ref{thm:main} could hold on smooth bounded domains (or the torus), e.g. under no-slip boundary conditions, at comparable levels of regularity. 
    A natural starting point, similar to \cite{albritton2023gluing}, would be to consider an ansatz of the form
\[
u = \eta(\tilde u + u^{\mathrm{cor}}) + \psi,
\]
where \(\eta\) is a cutoff, \(\psi\) corrects boundary/cutoff errors, and \(u^{\mathrm{cor}}\) is intended to capture the nonuniqueness mechanism associated with unstable directions. Again, it is essential to run backward for the unstable directions, so that no assumption on the smallness of positive eigenvalues is needed for the cutoff terms. Another difficulty is that the correction fields $u^{\mathrm{cor}}$ and $\psi$ are expected to satisfy a coupled system whose analysis appears more delicate in bounded geometries. We do not pursue these questions here and leave them for future work.
}
\end{remark}

\section{A Framework for Numerical Analysis}\label{Sec:anal}

In this section, we introduce a framework for \textit{a posteriori} numerical analysis.
Consider a general nonlinear partial differential equation of the form:
\begin{equation}
  J (\widetilde{w}) = 0. \label{eq:general PDE}
\end{equation}
Suppose we have an approximate solution $\overline{w}$, which satisfies
\[ J (\overline{w}) = \Epsilon, \]
where the residual $\Epsilon$ is small in a sense to be clarified later. Our
goal is to find an exact solution $\widetilde{w} = \overline{w} + w$ that is close to $\overline{w}$ and
satisfies
\[ J (\overline{w} + w) = 0, \]
where $w$ represents the error between the exact and approximate solutions.

We expand the equation in terms of the linearized operator $L$ and the nonlinear part $N$ as
\begin{equation}
  L w + N (w) + \Epsilon = 0. \label{eq:U eq}
\end{equation}
If we can solve \eqref{eq:U eq} for some $w$, then $\widetilde{w} = \overline{w} +
w$ solves  the original equation \eqref{eq:general PDE}. Since $w$ is small,
the nonlinear term $N (w)$ should be negligible and \eqref{eq:U eq} is almost
linear. Therefore, the invertibility of the linear operator $L$ essentially
determines the well-posedness of the equation \eqref{eq:U eq}, and we can use a fixed-point argument to conclude the existence of a solution. 

In this section, we will develop a framework of computer-assisted proof for
the existence of solutions to equation  \eqref{eq:U eq}, making the above heuristics rigorous. We will first
introduce a simplified example of a nonlinear Poisson equation to illustrate and motivate our method in Section
\ref{sec:toy}, followed by a general result in Section \ref{sec:general}, when $L$ is not exactly invertible but can be decomposed into a coercive part plus a compact part. We remark that instead of invoking the Banach fixed-point theorem via the proof of a contraction mapping, we will use the Schauder fixed-point theorem since our fixed-point map can be shown to have a gain of regularity. We spend more effort on compactness arguments for analysis, so that the computer-assisted proofs will be easier and to the minimal extent possible.  

\subsection{Toy model: nonlinear Poisson equation}\label{sec:toy}

We illustrate our framework with a simple toy model: the nonlinear Poisson
equation on $D=[0,1]^3$:
\begin{equation*}
  \left\{\begin{array}{l}
    - \Delta \widetilde{w} + \widetilde{w}^2 = f, \text{\quad in } D,\\
    \widetilde{w} = 0, \text{\quad on } \partial D.
  \end{array}\right. \label{eq:nonlin Poisson}
\end{equation*}
In general, we do not know the existence of the solution. Suppose we have a bounded
numerical solution $\overline{w}$ satisfying
\begin{equation*}
  \left\{\begin{array}{l}
    - \Delta \overline{w} + \overline{w}^2 = f + \Epsilon_w, \text{\quad in }
    D,\\
    \overline{w} = 0, \text{\quad on } \partial D,
  \end{array}\right. \label{eq:nonlin Poisson num}
\end{equation*}
where $\Epsilon_w$ stands for the residual. If an exact solution exists, the
error $w = \widetilde{w} - \overline{w}$ should satisfy
\begin{equation}
  \left\{\begin{array}{l}
    - \Delta w + 2 \overline{w} w + w^2 = - \Epsilon_w, \text{\quad in } D,\\
    w = 0, \text{\quad on } \partial D.
  \end{array}\right. \label{eq:nonlin Poisson err}
\end{equation}
In this formulation, we identify the linear and nonlinear components as:
\[ L w \assign - \Delta w + 2 \overline{w} w, \quad N (w) \assign w^2 . \]
We now analyze two cases. In the first case, we assume $L$ is invertible and use a fixed-point argument to construct the solution. In the second case,
where the invertibility of $L$ is unknown, we will use a finite-rank perturbation to approximate the compact part and verify numerically the invertibility of $L$ on the image of the finite-rank perturbation. In this subsection, since all of the norms and inner products are defined on $D$, we will omit the $D$ dependence.

\subsubsection{Simple case: \texorpdfstring{$\overline{w} \geqslant 0$}{overline{w} >= 0}}

First, we consider a simple case where $\overline{w} \geqslant 0$ and $L$ is
thus coercive. Specifically, we have
\[ \langle L w, w \rangle = \| \nabla w \|_{L^2}^2 + \langle 2 \overline{w}, w^2\rangle
   \geqslant \| \nabla w \|_{L^2}^2 . \]
Hence, by the Lax-Milgram theorem, $L$ is invertible in $H^1_0 (D)$. Now we define
a fixed-point map $w^{\ast} = T (w)$ by solving in $H^1_0 (D)$
\[ L w^{\ast} = - \Epsilon_w - N (w) . \]
The solution $w^{\ast}$ exists and satisfies
\[ \| \nabla w^{\ast} \|_{L^2}^2 \leqslant \langle L w^{\ast}, w^{\ast} \rangle
   \leqslant | \langle \Epsilon_w, w^{\ast} \rangle | + | \langle w^2,
   w^{\ast} \rangle | . \]
By Holder's inequality with multiple functions and Sobolev embedding, we have
\[ \| \nabla w^{\ast} \|_{L^2}^2 \leqslant \| \Epsilon_w \|_{L^2} \| w^{\ast} \|_{L^2} +
   \| w \|_{L^3}^2 \| w^{\ast} \|_{L^3}\lesssim (\| \Epsilon_w \|_{L^2} + \|
   \nabla w \|_{L^2}^2) \| \nabla w^{\ast} \|_{L^2} , \]where we have used the Poincar\'e inequality $\|w^{*}\|_{L^2}\lesssim\|\nabla w^{*}\|_{L^2}$ since $w^{*}$ vanishes on $\partial D$.

Assume that the residual is small, i.e.,
\[ \| \Epsilon_w \|_{L^2} \leqslant \varepsilon_w \ll 1. \]
Then the following estimate holds for $w^{\ast}$:
\[ \| \nabla w^{\ast} \|_{L^2} \leqslant C (\varepsilon_w + \| \nabla w \|_{L^2}^2) .
\]
If $\varepsilon_w$ is sufficiently small, we can choose a small radius $r
$ such that
\[ C (\varepsilon_w + r ^2) \leqslant r
    . \]
This implies the map $T$ sends the ball
\[ B_{r} \assign \{ w \in H^1_0 (D) : \| \nabla w \|_{L^2}
   \leqslant r  \} \]
to itself. Next, we prove that $T$ is continuous in a weaker topology. Let $w^{\ast} = T
(w)$, $\widehat{w}^{\ast} = T
(\widehat{w})$, for $w,\widehat{w}\in B_{r }$. Then we get
\[ L (w^{\ast} - \widehat{w}^{\ast}) = \widehat{w}^2 - w^2 = (\widehat{w} - w) (w + \widehat{w}), \]
and similarly, using coercivity and Sobolev embedding, we have 
\[ \|  w^{\ast} -  \widehat{w}^{\ast} \|_{L^3} \lesssim \| \nabla( w^{\ast} -  \widehat{w}^{\ast}) \|_{L^2} \lesssim (\| \nabla w \|_{L^2}
   + \| \nabla \widehat{w} \|_{L^2}) \|  w -  \widehat{w} \|_{L^3} \lesssim r
    \|  w-  \widehat{w} \|_{L^3} . \]
We know that $T$ is continuous in $L^3$-norm on $B_{r }$. Since the bounded ball in $H^1$ is compact in $L^3$, by the Schauder fixed-point theorem, a solution to the equation \eqref{eq:nonlin
Poisson err} exists.

\subsubsection{Difficult case: \texorpdfstring{$\overline{w} \leqslant 0$}{overline{w} <= 0}}

Next, we consider a more difficult case, where $\overline{w} \leqslant 0$. As
before, we aim to establish the existence of a solution to equation
\eqref{eq:nonlin Poisson err}. However, the operator $L$ is no longer
coercive, and directly proving its invertibility is challenging. To address
this, we observe that the operator can be decomposed as $L = L_1 + L_2$, where
$L_1 w \assign - \Delta w$ is coercive and $L_2 w \assign  2 \overline{w}
w$ is a compact perturbation.

Following the strategy in {\cite{chen2022stable}}, we leverage
computer-assisted methods to avoid directly verifying the invertibility of
$L$. Specifically, we decompose $w$ into $w = w_1 + w_2$ with:
\[\begin{aligned}
    &- \Delta w_1  =  - \Epsilon_w - (w_1 + w_2)^2,  \\
  &- \Delta w_2 + L_2 w_2  =  -L_2 w_1 .  
\end{aligned}\]

Summing up the two equations implies
 \eqref{eq:nonlin Poisson err}. The key observation is that $w_2 =
-L^{- 1} L_2 w_1$. If $L_2$ were of of finite-rank instead, 
we would only need to verify the invertibility of $L$ on a
finite-dimensional space, which can be done numerically.

In our case, however, $L_2 =  2 \overline{w}$ is compact but not of finite rank.
To circumvent this, we approximate $w_2$ using a finite-rank operator. That
is, we modify the system to:
\begin{align}
  &- \Delta w_1 + L_{\tmop{res}} w_1  =  - \Epsilon_w - (w_1 + w_2)^2, 
  \label{eq:nonlin Poisson w1}\\
  &- \Delta w_2 + L_2 w_2  =  - L_{\tmop{ap}} w_1 .  \label{eq:nonlin
  Poisson w2}
\end{align}
Here, $ L_{\tmop{ap}}$ is a finite rank approximation of $L_2$, and the
residual $ L_{\tmop{res}} : = L_2 - L_{\tmop{ap}}$ is small and can be
absorbed by the coercivity of $L_1$. Therefore, the modified linear part of
equation \eqref{eq:nonlin Poisson w1}, namely $- \Delta w_1 + L_{\tmop{res}} w_1$, remains coercive,  enabling a fixed-point argument for the existence
of $w_1$.

Now we show how to construct $ L_{\tmop{ap}}$. To begin with, we choose an
interpolation operator $I_n$ with a finite rank $n$ to approximate $w$. This
operator, typically derived from finite element or spectral methods, ensures a
small approximation error:
\[ \| w - I_n w \|_{L^2} \leqslant \delta_0 (n) \| \nabla w \|_{L^2},   \]where $\delta_0 (n)$ tends to zero as $n$ increases.
We use $L_2 I_n$ to approximate $L_2$. Consider  an associated basis $\Psi = (\psi_1, \ldots, \psi_n)$ of the image of $I_n$. We can represent
$I_n w_1$ as:
\[ I_n w_1 = \sum_{j = 1}^n s_j (w_1) \psi_j, \]
where $s_j (w_1)$ are the projection coefficients. Next, we verify the
invertibility of $L$ by solving
\[ - \Delta \phi_i + L_2 \phi_i = -L_2 \psi_i, i = 1, \ldots, n. \]
If this equation can be exactly solved, the solution of \eqref{eq:nonlin
Poisson w2} is thus given by
\begin{equation}
  w_2 = G w_1 \assign \sum_{j = 1}^n s_j (w_1) \phi_j . \label{eq:w2 def}
\end{equation}

In practice, small numerical residuals are unavoidable. We will get
approximate solutions $\phi_i$ with numerical residual $r_i$, that is:
\begin{equation*}
  L \phi_i = -L_2 \psi_i + r_i, i = 1, \ldots, n. \label{eq:gk method}
\end{equation*}
These residuals are expected to be
small, ensuring that the overall approximation remains accurate enough. We can
absorb this residual into $\psi_i$ by slightly modifying $L_{\tmop{ap}}$
as:
\begin{equation*}
  L_{\tmop{ap}} w_1 \assign L_2 I_n w_1 + R w_1,\quad R w_1 \assign -\sum_{j
  = 1}^n s_j (w_1) r_j.\label{eq:Lapn tilde def w1}
\end{equation*}
Then $w_2$ defined by \eqref{eq:w2 def} solves the
equation \eqref{eq:nonlin Poisson w2} precisely. Here, we need to numerically
verify that $G$ is bounded and $R$ is small in $L^2 (D)$, i.e.,
\[ \| R w_1 \|_{L^2} \leqslant \varepsilon_R \| w_1 \|_{L^2}, \quad\| G w_1 \|_{L^2}
   \lesssim \| w_1 \|_{L^2} . \]
The boundedness of $G$ ensures the invertibility of $L$ within the
finite-dimensional space. Moreover, since the residual $R$ is small,
$ L_{\tmop{res}} = L_2 (I - I_n) - R$ also remains small as we want.

Finally, by substituting \eqref{eq:w2 def} into \eqref{eq:nonlin Poisson w1}, we address the first equation in our modified system:
\[ - \Delta w_1 + L_{\tmop{res}} w_1 + (w_1 + G w_1)^2 = - \Epsilon_w .
\]
Using the Poincare's inequality, which gives $\| w_1 \|_{L^2} \lesssim \|\nabla w_1 \|_{L^2}$, we obtain
\[\begin{aligned}
  &\langle - \Delta w_1 + L_{\tmop{res}} w_1, w_1 \rangle  =  \langle -
  \Delta w_1, w_1 \rangle + \langle 2 \overline{w} (w_1 - I_n w_1), w_1
  \rangle - \langle R w_1, w_1 \rangle\\
  & \geqslant  \| \nabla w_1 \|_{L^2}^2 - 2 \| \overline{w} \|_{L^\infty} \| w_1
  - I_n w_1 \|_{L^2} \| w_1 \|_{L^2} - \| R w_1 \|_{L^2} \| w_1 \|_{L^2} \geqslant  (1 - C \delta_0 (n)  - C
  \varepsilon_R) \| \nabla w_1 \|_{L^2}^2.
\end{aligned}\]
Thus,
coercivity is preserved with large $n$ and small $\varepsilon_R$.
Consequently, a fixed-point argument similar to the simpler scenario ensures
the existence of a solution to equation \eqref{eq:nonlin Poisson err}.

\subsection{Generalization}\label{sec:general}

In this section, we formalize the methods and ideas introduced earlier into a
general theorem. For our toy model, the essential structure of the linear operator is characterized by its decomposition into a coercive part and a
compact perturbation. Accordingly, we impose the following assumption for the
linear operator in the equation \eqref{eq:U eq}:

\begin{assumption}[A priori bounds]
  \label{asm:linear op}Let $\mathfrak{H}^0$ be a Hilbert space equipped with
  inner product $\langle \cdummy, \cdummy \rangle_{\mathfrak{H}^0}$ and
  corresponding norm $\| \cdummy \|_{\mathfrak{H}^0}$. Suppose  $\mathfrak{H}^1\subset \mathcal{B}\subset \mathfrak{H}^0$ are  dense subsets of $\mathfrak{H}^0$, respectively endowed with stronger norms ${\mathfrak{H}^1}$, ${\mathcal{B}}$, and that the identity embedding $\mathfrak{H}^1\to \mathcal{B}$ is compact. In particular, there exists a constant $K_0 > 0$ such that \begin{equation}
      \| f \|_{\mathfrak{H}^0} \leqslant K_0 \| f \|_{\mathfrak{H}^1}, \quad  \| f \|_{\mathfrak{H}^0} \lesssim \| f \|_{\mathcal{B}}, \quad
      \forall f \in \mathfrak{H}^{1}.  \label{eq:embed}
    \end{equation}
    The operator
  $L$ admits the decomposition:
  \[ L = L_1 + L_2 , \]
  where $L_1$ is coercive and bounded, and $L_2$ is a compact perturbation. Specifically, we impose:
  \begin{enumeratenumeric}
    \item Coercivity of $L_1$ for a constant $K_1 > 0$:
    \begin{equation}
      \langle L_1 f, f \rangle_{\mathfrak{H}^0} \geqslant K_1 \| f
      \|_{\mathfrak{H}^1}^2, \quad \forall f \in \mathfrak{H}^1.
      \label{eq:coercive}
    \end{equation}

    \item Boundedness of $L_1$ as a bilinear form, where the exact constant is not important:
    \begin{equation}
      |\langle L_1 f, g \rangle_{\mathfrak{H}^0}| \lesssim \| f
      \|_{\mathfrak{H}^1}\| g
      \|_{\mathfrak{H}^1}, \quad \forall f \in \mathfrak{H}^1.
      \label{asm:lax}
    \end{equation}
    
    \item Compactness of $L_2$ via boundedness in $\mathfrak{H}^0$ for a constant $K_2 > 0$:
    \begin{equation}
      \| L_2 f \|_{\mathfrak{H}^0} \leqslant K_2 \| f \|_{\mathfrak{H}^0}, \quad
      \forall f \in \mathfrak{H}^0.\label{eq:L2 bound}
    \end{equation}
  \end{enumeratenumeric}
\end{assumption}
\subsubsection{Finite rank approximation}\label{subsub fr}

Following our earlier approach, we now construct a finite-rank approximation
of the operator $L_2$. First, we select an interpolation operator $I_n$
satisfying:
\[ \| f - I_n f \|_{\mathfrak{H}^0} \leqslant \delta_1 (n) \| f
   \|_{\mathfrak{H}^1}, \]
where $\delta_1 (n)$ tends to zero as $n$ increases.  Consider an associated basis $\Psi = (\psi_1,
\ldots, \psi_n)$ of the image of $I_n$. We express $I_n f$ in terms of its basis functions:
\[ I_n f = S (f) \cdummy \Psi \assign \sum_{j = 1}^n s_j (f) \psi_j . \]
 Next, we construct a set of
functions $\Phi = (\phi_1, \ldots, \phi_n)$ satisfying:
\begin{equation}
  L \phi_i = -L_2 \psi_i + r_i, i = 1, \ldots, n, \label{eq:gk method1}
\end{equation}
where $R = (r_1, \ldots, r_n)^T$ represents the small residual terms. The
finite-rank approximation operator $ L_{\tmop{ap}}$ is defined as:
\begin{equation*}
  L_{\tmop{ap}} w_1 = \sum_{j = 1}^n s_j (w_1) (L_2 \psi_j - r_j).
  \label{eq:Lapn tilde def}
\end{equation*}
The residual of this approximation is given by:
\begin{equation}
    \label{def lresn}
L_{\tmop{res}} w_1 \assign (L_2- L_{\tmop{ap}})w_1= L_2 (I - I_n) w_1 + \sum_{j = 1}^n s_j (w_1) r_j .
\end{equation}
We then decompose the equation \eqref{eq:U eq} into the following system:
\begin{equation}
  L_1 w_1 + L_{\tmop{res}} w_1 = -\Epsilon - N (w_1+w_2), \quad Lw_2 = 
  - L_{\tmop{ap}} w_1 . \label{eq:eigen_num}
\end{equation}
With the definitions of $\phi_j$ and $s_j$ established, the solution $w_2$ is
explicitly represented as:
\begin{equation*}
  w_2 = G w_1 \assign  \sum_{j = 1}^n s_j (w_1) \phi_j.\label{eq:u2 def}
\end{equation*}
The boundedness of the operator $G$ guarantees the invertibility of $L$
within this finite-dimensional space. In summary, to validate the numerical
approach, we will verify the following assumption:

\begin{assumption}[Finite-rank approximation]
  \label{asm:finite_cap}For any $n$, consider the finite-rank operator $I_n$.
  We assume that the interpolation error satisfies:
  \begin{equation}
    \| I_n f - f \|_{\mathfrak{H}^0} \leqslant \delta_1 (n) \| f
    \|_{\mathfrak{H}^1}, \quad \forall f \in \mathfrak{H}^1 . \label{eq:compact}
  \end{equation}
  Suppose $\Phi$ solves equation \eqref{eq:gk method1} with a small residual $R$, specifically:
  \[ \| R \|_{\mathfrak{H}^0} \leqslant \varepsilon_1 , \]
  where $\varepsilon_1$, though dependent on $n$, can be made sufficiently
  small in practice. The coefficients $S (f)$ and the solution $\Phi$ satisfy the bounds:
  \begin{equation}
    | S (f) | \lesssim \| f \|_{\mathfrak{H}^0},\quad| S (f) | \leqslant \gamma_1 (n) \| f \|_{\mathfrak{H}^1},\quad \| \Phi
    \|_{\mathfrak{H}^1} \leqslant \gamma_2 (n) . \label{eq:G bound}
  \end{equation}

\end{assumption}
  Here, in the first estimate of \eqref{eq:G bound}, we therefore have a gain of regularity that \[\|Gw_1\|_{\mathfrak{H}^1}\lesssim\|w_1\|_{\mathfrak{H}^0},\] which is crucial for the continuity estimate in a weaker topology $\mathcal{B}$ to follow, while the precise constant is not important.
\subsubsection{Residual and nonlinear terms}\label{Sec:nonl est}

In this subsection, we analyze the first equation of 
\eqref{eq:eigen_num}. Having established the coercivity of $L_1$ and the smallness of $ L_{\tmop{res}}$, we now shift our focus to the residual and nonlinear
terms. 
Regarding the residual term, we pose the following assumption:
\begin{assumption}[Residual estimate]
  \label{asm:res small}The residual $\Epsilon$ is sufficiently small:
  \[ \| \Epsilon \|_{\mathfrak{H}^0} \leqslant \varepsilon_2, \]
  where $\varepsilon_2$ quantifies the
  permissible magnitude of the residual.
\end{assumption}

For the nonlinear term, we pose the following assumption:

\begin{assumption}[Nonlinear estimate]
  \label{asm:nonlinear}The nonlinear term $N (w)$ satisfies the bounded estimate:
  \begin{equation*}
    | \langle N (u), w \rangle_{\mathfrak{H}^0} | \leqslant K_3 \| u
    \|_{\mathfrak{H}^1}^2 \| w \|_{\mathfrak{H}^1}, \label{eq:Nu asm}
  \end{equation*}
  and the continuous estimate
  \begin{equation*}
    | \langle N (u_1) - N (u_2), w_1 - w_2 \rangle_{\mathfrak{H}^0} | \lesssim
   (\| u_1 \|_{\mathfrak{H}^1} + \| u_2 \|_{\mathfrak{H}^1}) \| u_1 - u_2
    \|_{\mathcal{B}} \| w_1 - w_2 \|_{\mathcal{B}}. \label{eq:dtNu asm}
  \end{equation*}
\end{assumption}

\begin{remark}
We remark that the constant in the continuity estimate is not important, and that the nonlinear estimates ensure a continuous mapping in $\mathcal{B}$ and a bounded mapping in $\mathfrak{H}^1$, in preparation for the Schauder fixed-point argument.

Essentially, a quadratic nonlinearity combined with Sobolev embedding will satisfy the above assumptions.
  These assumptions can be relaxed to incorporate the nonlinearities of other orders. 
\end{remark}


\subsubsection{Main theorem}

With the assumptions above, we establish the following result:

\begin{theorem}[Existence of exact solutions in the abstract case]
  \label{thm:general}If Assumptions \ref{asm:linear op} and
  \ref{asm:nonlinear} hold,  Assumption
  \ref{asm:finite_cap} holds for sufficiently large $n$ and a
  sufficiently small $\varepsilon_1$, ensuring:
  \begin{equation}
    K_1 - K_0 K_2 \delta_1 (n) - K_0 \gamma_1 (n) \varepsilon_1 > 0, \label{eq:N constraint}
  \end{equation}
  and Assumption \ref{asm:res small} holds with a
  sufficiently small $\varepsilon_2$ such that the determinant \begin{equation}M_4\assign(K_1 - K_0 K_2 \delta_1  - K_0 \gamma_1  \varepsilon_1)^2 - 4\varepsilon_2K_0 K_3 (1 +
    \gamma_1 \gamma_2)^2\geqslant0,\label{eq:e2 constraint}
  \end{equation} then the equation \eqref{eq:eigen_num} has a solution $w_1$ with:
  \begin{equation}
    \| w_1 \|_{\mathfrak{H}^1} \leqslant  R_w\assign\frac{K_1 - K_0 K_2 \delta_1  - K_0 \gamma_1  \varepsilon_1-\sqrt{M_4}}{2K_3 (1 +
    \gamma_1 \gamma_2)^2}  . \label{solve}
  \end{equation}
  Consequently, there exists a solution $\widetilde{w}$ to equation
  \eqref{eq:general PDE} such that:
  \begin{equation*}
    \| \overline{w} - \widetilde{w} \|_{\mathfrak{H}^1} \leqslant (1+\gamma_1\gamma_2)R_w . \label{eq:error bd}
  \end{equation*}
\end{theorem}

\begin{remark}
  The logical sequence for addressing the assumptions in the preceding theorem
  is as follows: First, Assumptions \ref{asm:linear op} and
  \ref{asm:nonlinear} describe the qualitative properties of the operators and
  must initially be verified. Next, we select a sufficiently large $n$ to construct the
  approximation operator so that Assumption \ref{asm:finite_cap} and \eqref{eq:N constraint} hold. Finally, the residue $\varepsilon_2$ is chosen
  sufficiently small satisfying Assumption \ref{asm:res small} and \eqref{eq:e2 constraint}.
 
\end{remark}

\begin{proof}
  We employ a Schauder fixed-point argument to establish this theorem.
  
  \paragraph{Well-definedness of the fixed point map.} We will first establish that  $L_1 + L_{\tmop{res}}$ induces a coercive and bounded bilinear form.

  We begin by quantifying the smallness of the residual operator $L_{\tmop{res}}$. Recall \eqref{def lresn} that:
  \[ L_{\tmop{res}} f = L_2 (I - I_n) f + \sum_{j = 1}^n
     s_j (f) r_j. \]
  Using the boundedness of $L_2$ \eqref{eq:L2 bound} and the finite-rank approximation \eqref{eq:compact}, we have:
  \begin{equation*}
    \| L_2 (I - I_n) f\|_{\mathfrak{H}^0} 
    \leqslant K_2 \delta_1 \| f \|_{\mathfrak{H}^1} .
  \end{equation*}
  For the second term, by applying the Cauchy-Schwarz inequality, we have:
  \[
    \| \sum_{j = 1}^n s_j (f) r_j\|_{\mathfrak{H}^0} 
     \leqslant  \sqrt{\sum_{j = 1}^n | s_j (f) |^2} \sqrt{\sum_{j
    = 1}^n \| r_j \|_{\mathfrak{H}^0}^2 }= | S (f) | \| R \|_{\mathfrak{H}^0}\leqslant \gamma_1 \varepsilon_1 \| f
    \|_{\mathfrak{H}^1} ,
  \]
  where we use Assumption \ref{asm:finite_cap} in the last inequality.
  We conclude that\begin{equation}
      \label{eq:Lres est}\| L_{\tmop{res}}f\|_{\mathfrak{H}^0}\leqslant (K_2 \delta_1 +\gamma_1 \varepsilon_1) \| f
    \|_{\mathfrak{H}^1}.
  \end{equation}

We conclude the coercivity of $L_1 + L_{\tmop{res}}$ using estimates \eqref{eq:coercive}, \eqref{eq:Lres est}, and the embedding \eqref{eq:embed}:
  \begin{equation}
    \langle (L_1 + L_{\tmop{res}}) f, f \rangle_{\mathfrak{H}^0} \geqslant (K_1 - K_0 K_2 \delta_1  - K_0 \gamma_1  \varepsilon_1)\| f \|_{\mathfrak{H}^1}^2. \label{eq:LHS est}
  \end{equation}
  The boundedness is trivial by \eqref{eq:Lres est}, \eqref{eq:embed}, and Assumption \ref{asm:linear op}. By the Lax-Milgram theorem, we define the mapping $T$ by
  setting $T (w_1)=w_1^{\ast}$, where $w_1^{\ast}$ solves:
  \begin{equation}
    (L_1 + L_{\tmop{res}}) w_1^{\ast} = -\Epsilon - N (w_1 + G w_1) .
    \label{eq:Tmap eq}
  \end{equation}

  \paragraph{A priori estimate in $\mathfrak{H}^1$.}
  From Assumption \ref{asm:res small} and \ref{asm:nonlinear}, the residual term and the nonlinear term are respectively controlled by:
  \begin{equation}
    | \langle \Epsilon, w_1^{\ast} \rangle_{\mathfrak{H}^0} | \leqslant
    \varepsilon_2 \| w_1^{\ast} \|_{\mathfrak{H}^0}\leqslant
    \varepsilon_2K_0 \| w_1^{\ast} \|_{\mathfrak{H}^1}, \label{eq:res est}
  \end{equation}
  \begin{equation*}
    | \langle N (w_1 + G w_1), w_1^{\ast} \rangle_{\mathfrak{H}^0} |
    \leqslant K_3 \| w_1 + G w_1 \|_{\mathfrak{H}^1}^2 \|w_1^{\ast}
    \|_{\mathfrak{H}^1} \leqslant K_3 (\| w_1 \|_{\mathfrak{H}^1} + \| G w_1
    \|_{\mathfrak{H}^1})^2 \|w_1^{\ast} \|_{\mathfrak{H}^1} .
  \end{equation*}
  The estimate of $G w_1$ follows from the Cauchy-Schwarz inequality:
  \begin{equation} \| G w_1 \|_{\mathfrak{H}^1}  \leqslant \sum_{j = 1}^n |s_j (w_1)| \|
     \phi_j \|_{\mathfrak{H}^1} \leqslant | S (w_1) | \| \Phi
     \|_{\mathfrak{H}^1} \leqslant \gamma_1 \gamma_2 \| w_1 \|_{\mathfrak{H}^1} 
     , \label{G collected} \end{equation}
     where we use \eqref{eq:G bound} from Assumption \ref{asm:finite_cap}.
  We finally have
  \begin{equation}
    | \langle N (w_1 + G w_1), w_1^{\ast} \rangle_{\mathfrak{H}^0} |
    \leqslant K_3 (1 + \gamma_1 \gamma_2)^2 \| w_1 \|_{\mathfrak{H}^1}^2
    \|w_1^{\ast} \|_{\mathfrak{H}^1} . \label{eq:nonl est}
  \end{equation}
  
 By collecting the estimates \eqref{eq:LHS est}, \eqref{eq:res est}, and \eqref{eq:nonl est}, we obtain:
  \begin{equation*}
    (K_1 - K_0 K_2 \delta_1  - K_0 \gamma_1  \varepsilon_1)\| w_1^{\ast} \|_{\mathfrak{H}^1} \leqslant \varepsilon_2K_0 + K_3 (1 +
    \gamma_1 \gamma_2)^2 \| w_1 \|_{\mathfrak{H}^1}^2 . \label{eq:collect}
  \end{equation*}
  Solving the inequality as in \eqref{solve} shows that
  if $\| w_1 \|_{\mathfrak{H}^1}\leqslant R_w$, then   $\| w_1^{\ast} \|_{\mathfrak{H}^1}\leqslant R_w$.
  Let $B_{R_w}$ be the closed ball of radius $R_w$ centered at zero in $\mathfrak H^1$.
  $T$ maps the ball $B_{R_w}$ to itself.

  \paragraph{Continuity estimate in $\mathcal{B}$.}
  Next, we establish the continuity of $T$. Let $w_1^{\ast} = T
  (w_1), \widehat{w}_1^{\ast} = T (\widehat{w}_1)$. Subtracting the equation \eqref{eq:Tmap eq} that they satisfy, we get:
  \[ (L_1 + L_{\tmop{res}})  (w_1^{\ast} - \widehat{w}_1^{\ast}) = N
     (\widehat{w}_1 + G \widehat{w}_1) - N (w_1 + G w_1) . \]
     Instead of the estimate \eqref{G collected}, we need the gain of regularity in $G$ by \eqref{eq:G bound}:
  \begin{equation*} \| G f \|_{\mathfrak{H}^1}  \leqslant | S (f) | \| \Phi
     \|_{\mathfrak{H}^1} \lesssim \| f \|_{\mathfrak{H}^0} \lesssim \| f \|_{\mathcal{B}} \lesssim \| f \|_{\mathfrak{H}^1} 
     , \label{G gain} \end{equation*}
     where we use the comparison of the norms ${\mathfrak{H}^0},\mathcal{B},\mathfrak{H}^1$ in Assumption \ref{asm:linear op}.
  Combining with the nonlinear estimate in Assumption \ref{asm:nonlinear}, 
    the difference between the nonlinear term is
  bounded by\begin{equation}
  \begin{aligned}
    & | \langle N (\widehat{w}_1 + G \widehat{w}_1) - N (w_1 + G w_1),
    \widehat{w}_1^{\ast} - w_1^{\ast} \rangle_{\mathfrak{H}^0} | \\
    &\lesssim   (\| w_1 + G w_1 \|_{\mathfrak{H}^1} + \| \widehat{w}_1 +
    G \widehat{w}_1 \|_{\mathfrak{H}^1}) \| \widehat{w}_1 - w_1 + G \widehat{w}_1
    - G w_1 \|_{\mathcal{B}} \| \widehat{w}_1^{\ast} - w_1^{\ast}
    \|_{\mathcal{B}} \\
     &\lesssim   (\| w_1 \|_{\mathfrak{H}^1} + \|
    \widehat{w}_1 \|_{\mathfrak{H}^1}) \| \widehat{w}_1 - w_1 \|_{\mathcal{B}} \|
    \widehat{w}_1^{\ast} - w_1^{\ast} \|_{\mathcal{B}} 
    \lesssim    \|
    \widehat{w}_1 - w_1 \|_{\mathcal{B}} \| \widehat{w}_1^{\ast} - w_1^{\ast}
    \|_{\mathcal{B}} ,  \label{eq:nonlin diff est}
  \end{aligned}    
  \end{equation}
  for $\widehat{w}_1, {w}_1\in B_{R_w}$. 
  By the coercivity estimate \eqref{eq:LHS est} and the embedding \eqref{eq:embed}, it follows that
  \[ \| w_1^{\ast} - \widehat{w}_1^{\ast} \|^2_{\mathcal{B}} \lesssim\| w_1^{\ast} - \widehat{w}_1^{\ast} \|^2_{\mathfrak{H}^1} \lesssim  \| \widehat{w}_1 - w_1
     \|_{\mathcal{B}}\| w_1^{\ast} - \widehat{w}_1^{\ast} \|_{\mathcal{B}} . \]
If $w_1^{\ast}\neq \widehat{w}_1^{\ast}$, we can cancel the common factor to obtain
\[
\| w_1^{\ast} - \widehat{w}_1^{\ast} \|_{\mathcal B}
\;\lesssim\; \| \widehat{w}_1 - w_1 \|_{\mathcal B}.
\]
Hence $T$ is (Lipschitz) continuous in the weaker norm $\mathcal B$.
  
    The ball $B_{R_w}$ in $\mathfrak{H}^1$ is compact in $\mathcal{B}$ by Assumption \ref{asm:linear op}, and by the Schauder fixed-point theorem, there exists
  a  solution $w_1$ to equation \eqref{eq:eigen_num} satisfying
  $\| w_1 \|_{\mathfrak{H}^1} \leqslant R_w . $ By \eqref{G collected}, we know that $$\| \overline{w} - \widetilde{w} \|_{\mathfrak{H}^1}=\| w_1+Gw_1 \|_{\mathfrak{H}^1} \leqslant (1+\gamma_1\gamma_2)R_w . $$
  This concludes the proof.
\end{proof}

\section{Proof of Theorem \ref{thm:stationary solu main}: Analysis Part}\label{sec:analysis}

In this section, we apply the framework developed in the previous section to establish Theorem
\ref{thm:stationary solu main}, whose proof consists of two parts: demonstrating the
existence of the profile $\widetilde{U}$, and proving the existence of the
eigenpair $(\widetilde{\lambda}, \widetilde{v})$, via Propositions \ref{prop:U} and \ref{prop:u lambda}.
We will use superscript $U$ and $v$, such as $L^U$ and $L^v$, to indicate dependence on the operators associated with Propositions \ref{prop:U} and \ref{prop:u lambda}, respectively. The strategies are largely the same, with minor modifications highlighted. From this section onward, we will only work in the self-similar variables, and upon an abuse of notation, we still use $x$ rather than $\xi$ for the spatial variables.

\begin{remark} [Key idea of using truncated domain and taking weak limits]
One key observation is that on the whole space, the linear operator $x\cdot\nabla$ is no longer bounded, and that $H^1$ is not compactly embedded into $L^2$. To verify the compactness and boundedness in Assumptions \ref{asm:linear op}, it is crucial that we work in the truncated spaces $ \mathfrak{H}^1=\mathcal{H}^1_R$ defined as the closure in $H^1(\mathbb R^3)$ of functions in
$\mathcal{H}^1$ with compact support in the open ball $B_R(0)$. $\mathcal{H}^1$ consists of some subspaces of divergence-free functions measured in $H^1$ norm, whose definition will be made precise in \eqref{def h0h1} and \eqref{Def-H0H1-space}. $\mathfrak{H}^0, \mathcal{B}$-norms will be $L^2$ and $L^4$ norms respectively.

We truncate the domain so that the associated bilinear form will be bounded, as in  Assumption \ref{asm:linear op} to apply the Lax-Milgram theorem. Moreover, thanks to the gain of regularity, we show that the fixed point map maps a bounded set in $H^1$ to itself and is continuous in $L^4$. Therefore, in the truncated domain, we can apply the Schauder fixed-point theorem to obtain a fixed point,  and taking weak limits yields the desired solution. The continuity estimate and the boundedness of the bilinear form depend on $R$, while all the a priori estimates, in particular the coercivity estimate, remain uniform in the truncations, since we use an integration by parts to handle the unbounded operator $x\cdot\nabla$ in the weak formulation against compactly supported test functions.
\end{remark}
\begin{remark} [Key estimate] In order to obtain sharper constants, we choose to 
verify Assumptions \ref{asm:linear op} and \ref{asm:res small}, while
replacing Assumptions \ref{asm:finite_cap} and \ref{asm:nonlinear} by obtaining the
key estimates \eqref{eq:LHS est},  \eqref{eq:nonl est}, and \eqref{eq:nonlin
diff est}  from the proof of Theorem \ref{thm:general} directly. We remark that the key challenge lies in the linear estimate \eqref{eq:LHS est}, while the nonlinear and residual estimates are more straightforward. \end{remark}

We first discuss the concrete setting and the linear operators in Section \ref{Sec:summary}. The finite-rank approximation for a general type of compact perturbations applicable to $U$ and $v$ simultaneously is discussed in Section \ref{subsec:finiterank}. 
Proposition \ref{prop:U} is addressed in Section \ref{Sec:prop 3}, as
Proposition \ref{prop:u lambda} follows a similar verification process in
Section \ref{Sec:prop 4}.  Section
\ref{Sec:num prev} briefly presents our numerical results and records the verified numerical constants used in the proof. Implementation details of the numerical
algorithm are deferred to Sections \ref{Sec:num} and \ref{Sec:numfr}, and a roadmap for rigorous computer-assisted proof is in Section \ref{sec7}.

\subsection{Setting and decomposition of linear operators}\label{Sec:summary}
We discuss the setting and the definition of the linear operators.

\paragraph{Setting for $U$ for Proposition \ref{prop:U}.}
For $U$, we separate the linear and nonlinear components in equation \eqref{eq:U
pert}:
\begin{equation}\label{eq:profile-operators}
\begin{aligned}
  L^U U & \assign  - \frac{1}{2} U - \frac{1}{2} x \cdot \nabla U - \Delta U
  + \Pi (\overline{U} \cdot \nabla U + U \cdot \nabla \overline{U}),\\
  N^U (U) & \assign  \Pi (U \cdot \nabla U) .
\end{aligned}
\end{equation}
Our strategy involves decomposing the linear operator $L^U$ into two parts as in Section \ref{Sec:anal}: a
coercive operator $L_1^U$ and a compact operator $L_2^U$. All spaces below are restricted to the axisymmetric even (profile) or odd (eigenpair) symmetry class of Section~\ref{sec:sym}. The respective fixed-point maps preserve these classes and, for the eigenpair, orthogonality to $\overline v$. We first specify the Hilbert spaces under consideration: \begin{equation}\label{def h0h1}\mathcal{H}^0 =
L^2_{\sigma} (\mathbb{R}^3), \mathcal{H}^1 = H^1_{\sigma} (\mathbb{R}^3)
\end{equation}
 Noticing the cancellation via an integration by parts, an initial attempt at this decomposition could be:
\[ L_1^U U = -\frac{1}{2} U - \frac{1}{2} x \cdot \nabla U - \Delta U + \Pi
   (\overline{U} \cdot \nabla U), \quad L_2^U U = \Pi (U \cdot \nabla
   \overline{U}) . \]
While this decomposition is correct, the construction of a finite rank approximation prompts
us to ensure the matrix function in $L_2^U$ is compactly supported. To achieve
this, we truncate the gradient $\nabla \overline{U}$ as follows. We first
decompose $\nabla \overline{U}$ into antisymmetric and symmetric parts:
\[ \nabla \overline{U} (x) = Q_a (x) + Q_s (x), \quad Q_a (x) \assign
   \frac{\nabla \overline{U} (x) - \nabla^T \overline{U} (x)}{2}, \quad Q_s
   (x) \assign \frac{\nabla \overline{U} (x) + \nabla^T \overline{U} (x)}{2} .
\]
Note that $Q_a$ is antisymmetric and thus does not contribute to the scalar
product $\langle L_2^U U, U \rangle$. Therefore, we focus on truncating the
symmetric part $Q_s$.

To truncate $Q_s$, we choose a constant $0 < \eta_1 < \frac{1}{4}$.
We choose a compactly supported $Q\succeq0$ such that $Q_+:=Q_s+\eta_1I_3+Q\succeq0$.
The operator $L_U$ is finally decomposed as:
    \begin{align}
         L_1^U U & \assign  ( - \eta_1 - \frac{1}{2} ) U - \frac{1}{2} x
  \cdot \nabla U - \Delta U + \Pi (\overline{U} \cdot \nabla U + U \cdot (Q_a
  + Q_{+}) ), \label{eq:L1 def}\\ 
  L_2^U U & \assign  -\Pi (U \cdot Q), \label{eq:L2 def}
    \end{align}
This gives the following coercivity:
\begin{equation}
    \langle L_1^U U, U \rangle = ( \frac{1}{4} - \eta_1 ) \| U
   \|_{L^2}^2 + \| \nabla U \|_{L^2}^2 + \langle Q_{+} U, U \rangle \geqslant (
   \frac{1}{4} - \eta_1 ) \| U \|_{L^2}^2 + \| \nabla U \|_{L^2}^2 . \label{l1u c}
\end{equation} 
Since $Q$ is bounded of compact support, $L_2^U$ is a compact operator from
$\mathcal{H}^1$ to $\mathcal{H}^0$ and bounded from $\mathcal{H}^0$ to $\mathcal{H}^0$. The estimates above establish coercivity and compactness of the perturbation,
with constants independent of the exhaustion radius. The boundedness of the
drift-containing bilinear form and the compact embedding required for the
fixed-point construction are used only on each fixed truncated space
$\mathcal H^1_R$, as explained in Remark 6. The concrete semi-implicit maps
and their whole-space limits are established in Sections 4.3--4.4; no
whole-space boundedness assumption for the self-similar drift is being made.

\paragraph{Setting for $v$ for Proposition \ref{prop:u lambda}.}
For $v$, the only nonlinear term is $\lambda v$, and we need to decompose $\lambda$ appropriately. Starting from the equation
\eqref{eq:u pert}, we take the inner product with $\overline{v}$ to obtain:
\begin{align*}
  \lambda & =  \langle - \frac{1}{2} v - \frac{1}{2} x \cdummy \nabla v
  - \Delta v + \Pi (\widetilde{U} \cdot \nabla v + U \cdot \nabla \overline{v} + v
  \cdot \nabla \widetilde{U} + \overline{v} \cdot \nabla U) - \overline{\lambda} v
  - \lambda v + \Pi\Epsilon^v, \overline{v} \rangle\\
  & =  \langle - \frac{1}{2} x \cdummy \nabla v - \Delta v + \widetilde{U}
  \cdot \nabla v + v \cdot \nabla \widetilde{U} +
  \overline{v} \cdot \nabla U, \overline{v} \rangle ,
\end{align*}
where we use orthogonalities to $\overline{v}$ and divergence-free assumptions in Assumption \ref{asm:num} combined with an integration by parts.
It can be decomposed into:
\begin{equation}\label{eq:eigenvalue-correction-decomposition}
\begin{aligned}
  \lambda (v) &= \lambda^\text{lin} (v) + \lambda^\text{con},\\
  \lambda^\text{lin} (v) &\assign \langle - \frac{1}{2} x \cdummy \nabla v -
  \Delta v + \widetilde{U} \cdot \nabla v + v \cdot \nabla \widetilde{U}, \overline{v}
  \rangle,\\
  \lambda^\text{con} &\assign \langle \overline{v}
  \cdot \nabla U, \overline{v} \rangle .
\end{aligned}
\end{equation}
Here $\lambda^\text{lin}$ is linear in $v$, while $\lambda^\text{con}$ is independent of $v$. Consequently, equation
\eqref{eq:u pert} can be decomposed as:
\[ L^v v =  - \widehat{\Epsilon}^v-N^v (v), \]
where we define\begin{equation}
  \label{v-def-eqn}
\begin{aligned}
  L^v v & \assign - ( \frac{1}{2} + \overline{\lambda} + \lambda^\text{con}
  ) v - \frac{1}{2} x \cdummy \nabla v - \Delta v + \Pi (\widetilde{U} \cdot
  \nabla v + v \cdot \nabla \widetilde{U})  - \lambda^\text{lin} (v) \overline{v},\\
  N^v (v) & \assign  -\lambda^\text{lin} (v) v,\quad
  \widehat{\Epsilon}^v  \assign  \Pi\Epsilon^v + \Pi (U \cdot \nabla \overline{v} +
  \overline{v} \cdot \nabla U) - \lambda^\text{con} \overline{v} .
\end{aligned}\end{equation}

We observe that $L^v$ resembles $L^U$,
except for a few additional terms in the linear operator. The
$(\overline{\lambda} + \lambda^\text{con}) v$ term provides enhanced coercivity, while the
extra terms $U \cdot \nabla v$ and $v \cdot \nabla U$ are treated as
perturbations.
The term $- \lambda^\text{lin} (v) \overline{v}$ can be interpreted as a projection.
More explicitly, we define:
\begin{equation}\label{eq:projection-bar-v}
 P_{\overline{v}} f \assign f - \langle f, \overline{v} \rangle \overline{v} .
\end{equation}
Since $\overline{v}$ is divergence-free, it follows that $P_{\overline{v}}$
commutes with the Leray projection operator $\Pi$. With this projection, the
linear operator can be represented as:
\begin{equation*}\label{lin op v}
     L^v v = - ( \frac{1}{2} + \overline{\lambda} + \lambda^\text{con} ) v -
   \frac{1}{2} P_{\overline{v}}  (x \cdummy \nabla v) - P_{\overline{v}}\Delta v +
   P_{\overline{v}} \Pi (\widetilde{U} \cdot \nabla v + v \cdot \nabla \widetilde{U})
    .
\end{equation*}  We
decompose the linear operator $L$ into a coercive operator $L_1^v$
and a compact perturbation $L_2^v$.

Recall the definitions of $Q_a, Q_{+}, Q$ from the decomposition of $L^U$. Similarly, we split the linear operator $L^v$ into two components:
\begin{align}
  L_1^v v & \assign  ( - \eta_1 - \frac{1}{2} - \overline{\lambda} -
  \lambda^\text{con} ) v - \frac{1}{2} P_{\overline{v}} (x \cdot \nabla v) -
  P_{\overline{v}}\Delta v + P_{\overline{v}} \Pi (\widetilde{U} \cdot \nabla v + v \cdummy \nabla
  U + v \cdot (Q_a + Q_{+}) ), \nonumber\\
  L_2^v v & \assign  - P_{\overline{v}} \Pi (v \cdot Q) .  \label{eq:L2 def2}
\end{align}
We choose the relevant Hilbert spaces as follows:
\begin{equation}
\label{Def-H0H1-space}
 \mathcal{H}^0 = \{ u \in L^2_{\sigma} (\mathbb{R}^3) : \langle u,
   \overline{v} \rangle = 0 \},\quad \mathcal{H}^1 = \{ u \in H^1_{\sigma}
   (\mathbb{R}^3) : \langle u, \overline{v} \rangle = 0 \} . 
   \end{equation}
   We verify the coercivity of the operator $L_1^v$. Indeed, recalling that $\overline{\lambda}$ is negative and $| \lambda^\text{con} |$ is small, it holds that
\[ \langle L_1^v v, v \rangle \geqslant ( \frac{1}{4} - \eta_1 -
   \overline{\lambda} - | \lambda^\text{con} | ) \| v \|_{L^2}^2 + \| \nabla v \|_{L^2}^2
   - | \langle v \cdummy \nabla U, v \rangle | , \]
ensuring coercivity for sufficiently small $\| \nabla U \|_{L^2}$ and $| \lambda^\text{con}
|$. Finally, since $Q$ is bounded of compact support, $L_2^v$ is a compact operator from
$\mathcal{H}^1$ to $\mathcal{H}^0$ and bounded from $\mathcal{H}^0$ to $\mathcal{H}^0$. The estimates above establish coercivity and compactness of the perturbation,
with constants independent of the exhaustion radius. The boundedness of the
drift-containing bilinear form and the compact embedding required for the
fixed-point construction are used only on each fixed truncated space
$\mathcal H^1_R$, as explained in Remark 6. The concrete semi-implicit maps
and their whole-space limits are established in Sections 4.3--4.4; no
whole-space boundedness assumption for the self-similar drift is being made.
 
\subsection{Finite rank approximation}
\label{subsec:finiterank}
For a general compact operator in the form of \eqref{eq:L2 def} and \eqref{eq:L2 def2} as \begin{equation}
    L_2=-\mathcal{P}(U\cdot Q),\label{L22 def}
\end{equation} where $Q$ is compactly supported and positive semidefinite, and $\mathcal{P}$ denotes a projection on $L^2$ that equals the Leray projection $\Pi$ for the case of $U$ and $P_{\overline v}\Pi$ for the case of $v$. Here, $L_2^U$ and $L_2^v$ only differ by the projection operator. Notice that $L$ maps the image of the projection operator to itself: for any $f=\mathcal{P} f$, we have $\mathcal{P} L f=L f$. 

Recall that we want the finite-rank operator to have a number of basis functions as small as possible so that our numerical verification of invertibility on the space can be easier. 
Ideally, we should choose the eigenspace corresponding to the largest eigenvalues of $L_2$ so that the residue of the finite-rank approximation is small. This is not tractable in practice, however, since the eigenvectors are not explicit. We instead solve the eigenvalue problem on an explicitly constructed finite-dimensional subspace numerically, and require that $L_2$ acting outside of the space is small.
To be precise, the construction involves three main steps:
\begin{enumeratenumeric}
  \item Choose a domain $\Omega$ sufficiently large to encompass
  the support of $Q$, and construct a finite-dimensional collection of spectral basis functions on $\Omega$. 
  
  \item Compute the eigenvectors corresponding to the largest eigenvalues and
  construct their spanning eigenspace in the finite-dimensional space.   
  
  \item Construct
  approximations of $L_2$ and invert $L$ numerically restricted to the eigenspace. 
\end{enumeratenumeric}
\subsubsection{Construction of the finite rank approximation}
We first identify a suitable domain $\Omega$ that fully encompasses the
support of the matrix function $Q (x)$. 
Now, we introduce spectral approximations for functions
defined on $\Omega$. Let $V_N$ denote the finite-dimensional space with the first $N$ spectral basis functions. 
Use the weighted inner product $\langle\cdot,\cdot\rangle_{H^1(\Omega),\eta_2}$
from \eqref{eq:weighted-H1-notation}, where we choose
\begin{equation}\label{eq:finite-rank-eta2}
 \eta_2=0.005.
\end{equation}
Notice that by construction of the spectral basis, the $L^2$-projection $I_N$ defined by \eqref{eq:I def} in Section \ref{Sec:numfr} onto the space $V_N$ is also an $H^1$-projection.

To obtain an efficient finite-rank approximation, we consider the following
eigenvalue problem within the space of spectral  basis functions $V_N$:
\begin{equation*}
  \langle u, Q v \rangle_{\Omega} = \lambda \langle u, v \rangle_{H^1(\Omega),\eta_2} .
\end{equation*}

Specifically, we arrange the eigenvalues as:
\[ \lambda^{Q}_{1} \geqslant\cdots\geqslant \lambda^{Q}_k > 0 = \lambda^{Q}_{k+1} = \cdots
  = \lambda^{Q}_N, \]
and denote their corresponding eigenfunctions as $\{ \psi^{Q}_{1}, \ldots,
\psi^{Q}_N \}$. The
eigenfunctions are orthonormal in  the $H^1$-inner
product:
\[
  \langle \psi^{Q}_i, \psi^{Q}_i \rangle_{H^1(\Omega),\eta_2} = 1;\quad
  \langle \psi^{Q}_i, \psi^{Q}_j \rangle_{H^1(\Omega),\eta_2} = \langle \psi^{Q}_i, Q \psi^{Q}_j
  \rangle_{\Omega} = 0, i \neq j.
\]

We retain only eigenfunctions with index $j \leqslant m$, for a small $m$.
For each $j 
\leqslant m$, we numerically obtain the eigenpair $(\overline{\lambda}_j,\overline{\psi}_j)$, and solve for $\overline{\phi}_j$ as:
\begin{equation}
  L \overline{\phi}_j= \mathcal{P} Q \overline\psi_j + r_j, \quad \mathcal{P} \overline\phi_j =\overline\phi_j, \label{eq:phi
  eq}
\end{equation}
where $r_j$ denotes a small residual term that satisfies $\mathcal{P} r_j =r_j$ by definition. As in Section \ref{subsub fr}, to ensure an exact inversion of $L$, we define the final finite-rank approximation as:
\begin{equation}
  L_{\tmop{ap}} f \assign -\sum_{j = 1}^m\overline{\lambda}_j^{-1} \langle f, Q\overline\psi_j \rangle_{
  \Omega} \mathcal{P} (Q \overline\psi_j + r_j) , \label{eq:Lap def}
\end{equation}
with the finite-rank map $G$, satisfying $LG=-L_{\tmop{ap}}$, given by\[Gf= \sum_{j = 1}^m\overline{\lambda}_j^{-1} \langle f, Q\overline\psi_j \rangle_{
  \Omega} \overline\phi_j.\]

\subsubsection{Estimates of the finite rank approximation}\label{Estimate_finite_rank}
Since our approximate eigensystem incurs numerical errors, we first present a lemma concerning projection onto the eigensystem, whose proof is deferred to Appendix \ref{appen:lem1}. We denote the eigenvalues $\overline\Lambda=\tmop{diag}\{\overline\lambda_1,\cdots,\overline\lambda_m\}$, and the space $V_{\tmop{lg}}=\tmop{span}\{\overline\psi_j\}_{j=1}^m$.   The Gram matrices \begin{equation*}
{G_E}\assign(\langle \overline\psi_i, \overline\psi_j \rangle_{H^1(\Omega),\eta_2})_{1\leqslant i, j\leqslant m},\quad  G_Q\assign(\langle \overline\psi_i, Q\overline\psi_j \rangle_{ \Omega})_{1\leqslant i, j\leqslant m},
      \label{eq:qusi-or}
  \end{equation*} are close to the identity matrix $\mathbb{1}_{m}$ and $\overline\Lambda$ respectively. 

\begin{lemma}[Quasi-orthonormal projection]
  \label{lem:quasi ort}

  For any $f$ with the orthogonal projection $f^{Q,\perp}_{\tmop{lg}}\in V_{\tmop{lg}}$ such that $\langle f-f^{Q,\perp}_{\tmop{lg}},Q\overline\psi_j\rangle_{\Omega}=0$,   we have the estimates \begin{equation}
      |\|Q^{\frac{1}{2}}f^{Q,\perp}_{\tmop{lg}}\|^2_{L^2}-\sum_{j = 1}^m \overline\lambda_j^{-1} \langle f, Q\overline\psi_j \rangle_{\Omega} ^2|\leqslant \varepsilon_3\| \lambda_{\max} (Q) \|_{L^\infty}\|f\|^2_{{L^2}(\Omega)}.\label{fn 1}
  \end{equation}
  Here we use the notation that \begin{equation}
  \label{def-eps3}
  \varepsilon_3\assign \|\mathbb{1}_{m}-(G_Q)^{\frac{1}{2}}\overline\Lambda^{-1}(G_Q)^{\frac{1}{2}}\|,
  \end{equation}
measuring how close the approximate eigenpairs are to satisfying exact orthonormality in the $Q$-geometry.
  \end{lemma}
  We now present a lemma on the quasi-orthogonality in the $H^1$-inner product; see its proof in Appendix \ref{appen:lem1}.
\begin{lemma}[Quasi-orthogonality in $Q$- and $H^1$- inner products]
\label{lem:quasi ort1}
Define the residues $r_j^e$ of the approximate eigenvector problem by
\[
 \langle\overline\psi_j,Qv\rangle_\Omega
 =\overline\lambda_j\langle\overline\psi_j,v\rangle_{H^1(\Omega),\eta_2}
 +\langle r_j^e,v\rangle_{H^1(\Omega),\eta_2},\qquad v\in V_N.
\]
For any $f\in V_{\mathrm{lg}}$ and $u$ in $V_N$
orthogonal to $V_{\mathrm{lg}}$ in the $Q$-inner product, we have
\[
 |\langle f,u\rangle_{H^1(\Omega),\eta_2}|\le\varepsilon_4\|f\|_{H^1(\Omega),\eta_2}\|u\|_{H^1(\Omega),\eta_2}.
\]
Here we use the notation
\begin{equation*}
 \varepsilon_4\assign
 \frac{\big(\sum_j\overline\lambda_j^{-2}\|r_j^e\|_{H^1(\Omega),\eta_2}^2\big)^{1/2}}
      {\sqrt{1-\|I-{G_E}\|_{\mathrm F}}},
\end{equation*}
assuming $\|I-{G_E}\|_{\mathrm F}<1$, where $\|\cdot\|_{\mathrm F}$ is the Frobenius norm.
\end{lemma}

Now we work on bounded estimates.
The approximation rate associated with $V_N$ is defined as \begin{equation}\label{eq:K5-definition}
  K_5\assign{\| \lambda_{\max} (Q) \|_{L^\infty}}\sup_{u \neq 0, u \bot_{L^2} V_N} \frac{\|
    u \|^2_{L^2(\Omega)}}{\|  u \|^2_{H^1(\Omega),\eta_2}}.\end{equation}$K_5$ tends to $0$ as $N$ increases by our spectral basis construction of $V_N$, since the Rayleigh quotients are upper bounded by the inverse of eigenvalues of the Laplacian. 
    Defining the $L^2$-projection on $V_N$ as $I_N$, we have the estimate \begin{equation}
      \label{k5}\| Q^{\frac{1}{2}} (u - I_N u)
    \|_{L^2}^2\leqslant K_5 \|  (u - I_N u) \|_{H^1(\Omega),\eta_2}^2.\end{equation}
  
  We define the following constant
  \begin{equation}
  K_6\assign\sup_{u\in V_N, u \neq 0, u \bot_{Q}V_\text{lg}} 
      \frac{\|Q^{\frac{1}{2}}
    u \|^2_{L^2}}{\|  u \|^2_{H^1(\Omega),\eta_2}},\label{constant small}
  \end{equation}the maximal Rayleigh quotient of $Q$ measured in the $H^1$--energy on the $Q$--orthogonal complement of the captured subspace $V_{\mathrm{lg}}=\mathrm{span}\{\overline\psi_1,\dots,\overline\psi_m\}\subset V_N$,
i.e. the top generalized eigenvalue of $Q$ on $V_{\mathrm{lg}}^{\perp_Q}$. 
Moreover,
     $K_6$ can be computed rigorously in the matrix equivalent characterization of $$K_6\leqslant\sup_{u\in V_N, u \neq 0} \frac{\langle u, Q
    u \rangle_{\Omega}-\sum_{j = 1}^m \overline\lambda_j^{-1} \langle u, Q\overline\psi_j \rangle_{\Omega} ^2}{\|  u \|^2_{H^1(\Omega),\eta_2}},$$
which should approximately be the eigenvalue smaller than $\overline\lambda_m$. In practice, since we only need to provide an upper bound of $K_6$, we can increase $K_6$ a bit and only need to ensure and verify a positive-definiteness in the following sense: 
\begin{equation*}
     \forall u\in V_N,\qquad K_6\|  u \|^2_{H^1(\Omega),\eta_2}+\sum_{j = 1}^m \overline\lambda_j^{-1} \langle u, Q\overline\psi_j \rangle_{\Omega} ^2-\langle u, Q
    u \rangle_{\Omega}\geq0,
\end{equation*}
which we can verify in the matrix representation via a numerical $LDL^T$ decomposition and by showing the residue is small and diagonal entries are away from $0$.

With the bounded estimates in hand, we are in a position to present
the following lemma on a rigorous estimate of the approximation error
associated with $L_{\tmop{ap}}$.

\begin{lemma}[Finite-rank approximation]
  \label{lem:fe est}For $L_{\tmop{ap}}$ defined by \eqref{eq:Lap def}
  and $L_2$ defined by \eqref{L22 def},
    we have the estimate for any $f$ such that $\mathcal{P} f=f$:
  \[  |\langle (L_2 - L_{\tmop{ap}}) f, f \rangle |\leqslant M_1 \|
     \nabla f \|_{L^2(\mathbb R^3)}^2 + M_2 \| f \|_{L^2(\mathbb R^3)}^2, \]
  where
  \begin{equation}\label{eq:finite-rank-constants} M_1\assign K_5+\frac{K_6}{(1-\varepsilon_4)^2},\quad M_2 \assign \eta_2M_1+\varepsilon_3\| \lambda_{\max} (Q) \|_{L^\infty}+{\Theta_{r}}. \end{equation}

Here we set
\[
 \Theta_r\assign\sup_{\substack{\|f\|_2=1\\\mathcal P f=f}}
 \left|\sum_{j=1}^m\overline\lambda_j^{-1}
 \langle f,Q\overline\psi_j\rangle_\Omega
 \langle r_j,f\rangle_{\mathbb R^3}\right|.
\]

  Finally, we remark that by construction, $\overline\phi_j \in H^2$.
\end{lemma}
We defer the proof of the lemma to Appendix \ref{appen:lem1}.

Define
\begin{equation}\label{eq:phi gram U}
 \Theta_\phi\assign\|G\|,\qquad
 \Theta_{\nabla\phi}\assign\|\nabla G\|.
\end{equation}
Superscripts $U,v$ specify the two families.

\begin{remark}[Finite-rank norm estimates]\label{rem:finite-rank-norm-certification}
For $Tf=\sum_j\langle f,a_j\rangle b_j$, put
$C=(\langle a_i,a_j\rangle)_{i,j}$,
$D=(\langle b_i,b_j\rangle)_{i,j}$, and
$X=(\langle a_i,b_j\rangle)_{i,j}$ when the two families lie in the same space.
For operators on one space, form
\[
 W=\begin{pmatrix}C&X\\X^T&D\end{pmatrix},\quad
 K=\frac12\begin{pmatrix}0&I\\I&0\end{pmatrix}.
\]
Then, to estimate $\Theta_r$, we use
\[
 \sup_{\|f\|_2=1}|\langle f,Tf\rangle|
 =\left\|\frac{T+T^*}{2}\right\|
 =\rho\!\left(W^{1/2}KW^{1/2}\right),
\]
where $\rho$ denotes the spectral radius.
For $\|T\|$, both spectral and trace bounds are used:
\[
 \sup_{\|f\|_2=1}\|Tf\|_2^2
 =\|T\|^2=\lambda_{\max}(C^{1/2}DC^{1/2})
 \le\operatorname{tr}(CD).
\]
These bounds apply to $\Theta_\phi$ and $\Theta_{\nabla\phi}$ with
$a_j=\overline\lambda_j^{-1}Q\overline\psi_j\mathbf1_\Omega$ and
$b_j=\overline\phi_j$ or $\nabla\overline\phi_j$, respectively.
\end{remark}

\subsection{Proof of Proposition \ref{prop:U}}\label{Sec:prop 3}

We begin by addressing Proposition \ref{prop:U}.
We construct a finite-rank
approximation $L_{\tmop{ap}}^{U}$ of $L_2^U$ by the strategy in Section \ref{subsec:finiterank} as
\begin{equation}
  L_{\tmop{ap}}^{U} f \assign -\sum_{j = 1}^m (\overline{\lambda}_j^U)^{-1} \langle f, Q\overline\psi^U_j \rangle_{
  \Omega} \Pi (Q \overline\psi^U_j + r_j^U)  . \label{eq:Lapu def}
\end{equation}
By Lemma \ref{lem:fe est} and the coercivity of $L_1^U$ in \eqref{l1u c}, we have the  linear estimate:
\begin{equation}
  \langle L_1^U U_1 + (L_2^U - L_{\tmop{ap}}^{U}) U_1,
  U_1 \rangle \geqslant ( \frac{1}{4} - \eta_1 - M_2^U ) \| U_1
  \|_{L^2}^2 + (1 - M_1^U) \| \nabla U_1 \|_{L^2}^2 . \label{eq:TU lin est}
\end{equation}
This verifies \eqref{eq:LHS est}. We will subsequently choose the parameters
appropriately to ensure that both coefficients on the right-hand side are
strictly positive.

This approximation allows us to split the solution $U = U_1 + U_2$ and
rewrite equation \eqref{eq:U pert} as:
\begin{align}
  &L_1^U U_1 + (L_2^U - L_{\tmop{ap}}^{U}) U_1  =  -N^U (U_1 + U_2) -
  \Pi\Epsilon^U,\label{eq: fix u 1} \\
  &L^U U_2  =  -L_{\tmop{ap}}^{U} U_1 . \label{eq: fix u 2}
\end{align}

Finally, we define:
\begin{equation}
    U_2 = G^{U} U_1 \assign \sum_{j = 1}^m (\overline{\lambda}_j^U)^{-1} \langle U_1, Q\overline\psi_j^U \rangle_{
   \Omega} \overline\phi_j^U  , \label{nl def u}
\end{equation} 
which gives an exact solution to \eqref{eq: fix u 2}. This follows directly from the finite sum and the prescribed residual identities; no inverse of the full operator $L^U$ is assumed.
   By construction, $\overline\phi_j^U$ and thus $U_2$ are divergence-free.
We detail our fixed-point approach for \eqref{eq: fix u 1} as follows. 

\paragraph{Definition of the fixed point map.} Typically, the fixed-point map
$T^U$ is defined as:
\[ L_1^U U_1^{\ast} + (L_2^U - L_{\tmop{ap}}^{U}) U_1^{\ast} = -\Pi ((U_1 +
   G^{U} U_1) \cdot \nabla (U_1 + G^{U} U_1)) - \Pi\Epsilon^U , \]
where $U_1^{\ast} \assign T^U (U_1)$. To simplify nonlinear estimates, we replace $\nabla U_1$ by $\nabla U_1^{\ast}$ in the nonlinear term. In the weak form, we slightly
modify the map $T^U_R$ to:
\begin{equation}  \langle L_1^U U_1^{\ast} + (L_2^U - L_{\tmop{ap}}^{U}) U_1^{\ast}+(U_1 +
   G^{U} U_1) \cdot \nabla U_1^{\ast} , g \rangle =  \langle-(U_1 +
   G^{U} U_1) \cdot \nabla   G^{U} U_1 - \Epsilon^U,g \rangle, \label{eq:Ustar def}
\end{equation}
 for all $g\in \mathcal{H}^1_R$. Here we denote $U_1^{\ast} \assign T^U_R (U_1)$, and 
  work in the closed zero-trace space $\mathcal H_R^1$ specified above. Recall that $ \mathcal{H}^1$ consists of divergence-free functions, so we can remove the Leray projection. We require $R$ sufficiently large so that $\Omega\subset B_R(0)$ in the finite rank approximation.
The finite-rank correction $G^UU_1$ is kept as a whole-space field. 

     \paragraph{Well-definedness of the fixed point map.} 
     Integration by parts yields:
\begin{equation} \langle  (U_1 +G^{U} U_1) \cdot \nabla g, g \rangle = 0,
\label{eq:Ustar canq}
\end{equation}
and thus we know by \eqref{eq:TU lin est} that \begin{equation}
    \langle L_1^Uf + (L_2^U - L_{\tmop{ap}}^{U})f+(U_1 + G^{U} U_1) \cdot \nabla f,g\rangle\label{lax-mil fix}
\end{equation} induces a coercive bilinear form for $(f,g)$ on $\mathcal{H}^1_R$.

By definition of the finite rank approximation in \eqref{eq:Lapu def}, we compute using the boundedness of $Q$, $\overline\psi_j^U$ and $r_j^U$ that $\|L_{\tmop{ap}}^{U}f\|_{L^2}\lesssim\|f\|_{L^2(\Omega)}\lesssim\|f\|_{\mathcal{H}^1}.$
Using Sobolev embedding, along with the forms \eqref{eq:L1 def} and \eqref{eq:L2 def}, we deduce that \eqref{lax-mil fix} induces a bounded bilinear form for $(f,g)$ in $\mathcal{H}^1_R$. 

By definition \eqref{nl def u} and Lemma \ref{lem:fe est}, $G^{U} U_1\in H^2(\mathbb R^3)$, and by Sobolev embedding we have $$-(U_1 +
   G^{U} U_1) \cdot \nabla G^{U} U_1 - \Epsilon^U\in L^2\,.$$

As a consequence of the Lax-Milgram theorem applied to the bilinear form \eqref{lax-mil fix}, we know that there exists $U_1^{\ast}=T_R^U(U_1)\in\mathcal{H}^1_R$ such that \eqref{eq:Ustar def} holds.

\paragraph{A priori estimate in $H^1$.} 
By \eqref{eq:Ustar def}, \eqref{eq:TU lin est}, and \eqref{eq:Ustar canq}, we obtain:
\begin{equation}
  ( \frac{1}{4} - \eta_1 - M_2^U ) \| U_1^{\ast}
  \|_{L^2}^2 + (1 - M_1^U) \| \nabla U_1^{\ast} \|_{L^2}^2 \leqslant - \langle (U_1 +G^{U} U_1) \cdot \nabla G^{U} U_1 + \Epsilon^U,
  U_1^{\ast} \rangle . \label{eq:Ustar eq}
\end{equation}
The residual term is  bounded by:
\[ | \langle \Epsilon^U, U_1^{\ast} \rangle | \leqslant \| \Epsilon^U \|_{L^2} \|
   U^{\ast}_1 \|_{L^2} \leqslant \varepsilon^U \| U^{\ast}_1 \|_{L^2}, \]
with $\varepsilon^U$ sufficiently small defined in Assumption \ref{asm:num}.

To handle the nonlinear term, we establish the following bound:
\begin{equation}
  | \langle (U_1 + G^{U} U_1) \cdot \nabla G^{U} U_1, U_1^{\ast} \rangle |\leqslant\|(U_1 + G^{U} U_1) \cdot \nabla G^{U} U_1\|_{L^2}\|U_1^{\ast}\|_{L^2}
  \leqslant M_3^U \|  U_1 \|_{L^2}^2\| U_1^{\ast} \|_{L^2}  ,
  \label{eq:TU nonlin est}
\end{equation}

where
\begin{align}
 \Theta_Q^U&\assign\sup_{\|f\|_2=1}
 \sum_{j=1}^m\left|(\overline\lambda_j^U)^{-1}
 \langle f,Q\overline\psi_j^U\rangle_\Omega\right|^2,
 \label{eq:qusi-or1}\\
 M^U_3&\assign\sqrt{\Theta_Q^U\sum_{j = 1}^m \| \nabla \overline\phi_j ^U\|_{L^\infty}^2} + \Theta_Q^U\sqrt{\sum_{j = 1}^m\sum_{i = 1}^m \| \overline\phi_i^U \cdummy \nabla \overline\phi_j^U
  \|_{L^2}^2}.
 \label{eq:U-M3-definition}
\end{align}
We estimate $\Theta_Q^U$ using the bound in Remark~\ref{rem:finite-rank-norm-certification}.

\begin{proof}[Proof of \eqref{eq:TU nonlin est}]
  Now we estimate the LHS by the definition in \eqref{nl def u} as
  \begin{align*}
    &\|U_1 \cdummy \nabla G^{U} U_1 + G^{U} U_1 \cdummy \nabla G^{U} U_1 \|_{L^2}\leqslant
     \|\sum_{j = 1}^m (\overline{\lambda}_j^U)^{-1} \langle U_1, Q\overline\psi_j^U \rangle_{\Omega}  U_1
    \cdummy \nabla \overline\phi_j^U \|_{L^2}\\
    &+ \|\sum_{j = 1}^m \sum_{i = 1}^m (\overline{\lambda}_j^U)^{-1} \langle U_1, Q\overline\psi_j^U \rangle_{\Omega}(\overline{\lambda}_i^U)^{-1} \langle U_1, Q\overline\psi_i^U \rangle_{\Omega}  \overline\phi_i ^U\cdummy
    \nabla \overline\phi_j^U\|_{L^2}\backassign I_4+I_5
  \end{align*}
  The absolute value of the first term is bounded by Cauchy-Schwarz as
  \[|I_4|\leqslant
    \sqrt{\sum_{j = 1}^m  \langle U_1, (\overline{\lambda}_j^U)^{-1}Q\overline\psi_j^U \rangle_{\Omega}^2}
    \sqrt{\sum_{j = 1}^m \| \nabla \overline\phi_j ^U\|_{L^\infty}^2} \bigg(\| U_1 \|_{L^2}\bigg) .
\]
  The second term is similarly bounded by
  \[ |I_5|\leqslant\sqrt{\sum_{j = 1}^m   \langle U_1, (\overline{\lambda}_j^U)^{-1}Q\overline\psi_j^U \rangle_{\Omega}^2 \sum_{i = 1}^m \langle U_1, (\overline{\lambda}_i^U)^{-1}Q\overline\psi_i^U \rangle_{\Omega}^2}    \sqrt{\sum_{j = 1}^m
     \sum_{i = 1}^m \| \overline\phi_i^U \cdummy \nabla \overline\phi_j^U \|_{L^2}^2} 
     . \]
Applying the coefficient bound \eqref{eq:qusi-or1} to $I_4$ and $I_5$ gives \eqref{eq:TU nonlin est}.
     \end{proof}
     
     We collect the estimates in $H^1$ by \eqref{eq:Ustar eq}: 
  \begin{equation*}
  ( \frac{1}{4} - \eta_1 - M_2^U ) \| U_1^{\ast}
  \|_{L^2}^2 + (1 - M_1^U) \| \nabla U_1^{\ast} \|_{L^2}^2 \leqslant  (\varepsilon^U+M_3^U \|  U_1 \|_{L^2}^2)\| U_1^{\ast} \|_{L^2}. \label{eq:final est}
\end{equation*} We analyze the bounded mapping with the following assumption.
\begin{assumption}[Bounded mapping assumption for $U$]\label{eq:cond ref U}
    Choose $a^U,b^U>0$ with
    $a^U\le\frac14-\eta_1-M_2^U$ and $b^U\le1-M_1^U$.
    We assume\begin{equation}\label{eq:U-M4-definition}
    M_4^U\assign(a^U)^2-4M_3^U\varepsilon^U>0.
  \end{equation}
\end{assumption}
We will verify Assumption \ref{eq:cond ref U} in Subsection \ref{Sec:num prev}.

Define
\begin{equation}\label{eq:weighted-H1-norm}
 \eta_U:=\frac{a^U}{b^U},
 \qquad
 R_U:=\frac{2\sqrt{\eta_U}\,\varepsilon^U}
 {a^U+\sqrt{M_4^U}}.
\end{equation}
Under Assumption~\ref{eq:cond ref U}, we have
\begin{equation*}
 \varepsilon^U+\frac{M_3^U}{\eta_U}R_U^2
 =b^U\sqrt{\eta_U}\,R_U.
\end{equation*}
For the weighted energy ball
\begin{equation*}
 \mathcal B_R^U:=\{U_1\in\mathcal H_R^1:\|U_1\|_{H^1,\eta_U}\le R_U\},
\end{equation*}
the preceding estimates give, for $U_1\in\mathcal B_R^U$,
\[
 \|U_1^\ast\|_{H^1,\eta_U}
 \le\frac{\varepsilon^U+\frac{M_3^U}{\eta_U}R_U^2}
 {b^U\sqrt{\eta_U}}
 \le R_U,
\]
so $T_R^U(\mathcal B_R^U)\subset\mathcal B_R^U$.

\paragraph{Continuity estimate in $L^4$.}
We estimate the difference between $U_1^{\ast}\assign T_R^U(U_1)\in\mathcal{H}^1_R$ and $\widehat{U}_1^{\ast}\assign T_R^U(\widehat{U}_1)\in\mathcal{H}^1_R$. By \eqref{eq:Ustar def}, we take differences and plug in $g=U_1^{\ast}-\widehat{U}_1^{\ast}$, to get
\begin{equation}
   \langle L_1^U \delta U^{\ast} + (L_2^U -
  L_{\tmop{ap}}^{U}) \delta U^{\ast}, \delta U^{\ast} \rangle =-\langle (\widehat{U}_1 + G^{U} \widehat{U}_1) \cdot \nabla (G^{U} \delta U) + ( \delta U + G^{U} \delta U ) \cdot
  \nabla (U_1^{\ast} + G^{U} U_1), \delta U^{\ast} \rangle ,\label{tu dif eqn}
\end{equation}
where we denote $\delta U^{\ast}\assign U_1^{\ast} - \widehat{U}_1^{\ast}$, $\delta U\assign U_1- \widehat{U}_1$, and use the cancellation \eqref{eq:Ustar canq}.

Notice that for the continuity estimate, we do not require sharp constants, but want to gain some regularity to the extent possible. We first provide some estimates for $G^{U}f$ by $\|f\|$. Recall the definition \eqref{nl def u} and notice that $\overline\phi^U_j\in H^2$ by Lemma \ref{lem:fe est}. 
We conclude that \begin{equation}
         \label{eq: reg gain}
\|\nabla^iG^{U}f\|_{L^2}\leqslant C\|f\|_{L^2},
     \end{equation}
     where $i=0,1$ and $C$ is a constant depending on $\overline\phi^U_j,\overline\psi^U_j$.

We estimate the difference of the nonlinear term as follows:
\begin{align}
  & | \langle (\widehat{U}_1 + G^{U} \widehat{U}_1) \cdot \nabla (G^{U} \delta U) + ( \delta U + G^{U} \delta U ) \cdot
  \nabla (U_1^{\ast} + G^{U} U_1), \delta U^{\ast} \rangle |  \label{eq:TU nonlin est1}\nonumber
  \\&\leqslant C(\|\widehat{U}_1\|_{H^1}+\|{U}_1^{\ast}  \|_{H^1}+\| {U}_1\|_{H^1})\|\delta U \|_{L^4}\|\delta U^{\ast} \|_{L^4},
\end{align}where $C$ is some constant only depending on $\overline\phi^U_j,\overline\psi^U_j,R$.

\begin{proof}[Proof of \eqref{eq:TU nonlin est1}]
     Combining the boundedness of $G^U$ in \eqref{eq: reg gain} with Holder's inequality with multiple functions and Sobolev embedding, we bound the LHS by 
     $$\begin{aligned}
        &\leqslant(\|\widehat{U}_1 + G^{U} \widehat{U}_1\|_{L^4}\|\nabla (G^{U} \delta U)\|_{L^2}+\|\delta U + G^{U} \delta U\|_{L^4}\|\nabla (U_1^{\ast} + G^{U} U_1)\|_{L^2})\|\delta U^{\ast} \|_{L^4}\\&\lesssim(\|\widehat{U}_1 + G^{U} \widehat{U}_1\|_{H^1}\| \delta U\|_{L^2}+(\|\delta U\|_{L^4} + \|G^{U} \delta U\|_{H^1})(\|{U}_1^{\ast}  \|_{H^1}+\| {U}_1\|_{H^1}))\|\delta U^{\ast} \|_{L^4}\\&\lesssim(\|\widehat{U}_1\|_{H^1}+\|{U}_1^{\ast}  \|_{H^1}+\| {U}_1\|_{H^1})\|\delta U \|_{L^4}\|\delta U^{\ast} \|_{L^4},
     \end{aligned}$$
 where in the last inequality, we use the fact that $\delta U$ is compactly supported.
\end{proof}

Combined with \eqref{tu dif eqn}, the coercivity \eqref{eq:TU lin est}, Sobolev embedding of $H^1$ into $L^4$ with $\|f\|_{L^4} \lesssim \|f\|_{H^1}$, and the previous step on the a priori bound of $H^1$ norm, we conclude that 
\[
\|\delta U^{\ast}\|_{L^4} \lesssim \|\delta U\|_{L^4}.
\]
Thus $T^U_R$ is a continuous mapping in $L^4$, in the space $\mathcal B_R^U$.

Notice that $\mathcal B_R^U$ is bounded in $\mathcal{H}^1_R$, and by the Rellich–Kondrachov embedding theorem we know that it is a compact subset of $L^4$.
It is closed in $L^4$: an $L^4$ limit of a bounded sequence in this set is also its weak $H^1$ subsequential limit, which preserves the weighted norm bound. As a Hilbert norm ball, $\mathcal B_R^U$ is convex. We apply the Schauder fixed-point theorem to the topology induced by $L^4$. Since $T^U_R$ is continuous, maps the compact subset $\mathcal B_R^U$ to itself, there exists a fixed point $U_1^R$ such that by \eqref{eq:Ustar def}, for any $g\in \mathcal{H}^1_R$, we have 
\begin{equation}  \langle L_1^U U_1^R + (L_2^U - L_{\tmop{ap}}^{U}) U_1^R , g \rangle =  \langle-(U_1^R +
   G^{U} U_1^R) \cdot \nabla   (U_1^R +
   G^{U} U_1^R) - \Epsilon^U,g \rangle.\label{eq:Ufix def}
\end{equation}

\paragraph{Refined bounds at the fixed point.}
At a fixed point, integration by parts and \eqref{eq:solenoidal-strain-bound} give
\[
\begin{aligned}
 |\langle(U_1+U_2)\cdot\nabla U_2,U_1\rangle|
 &=|\langle U_1\otimes U_1,S_{U_2}\rangle
       -\langle U_2\otimes U_2,S_{U_1}\rangle|\\
 &\le\frac1{\sqrt3}\bigl(
       \|\nabla U_2\|_{L^2}\|U_1\|_{L^4}^2
       +\|\nabla U_1\|_{L^2}\|U_2\|_{L^4}^2\bigr).
\end{aligned}
\]
For $i=1,2$, Lemma~\ref{lem:sobolev const} gives
\[
 \|U_i\|_{L^4}^2
 \le\|U_i\|_{L^2}^{1/2}\|U_i\|_{L^6}^{3/2}
 \le K_4\|U_i\|_{L^2}^{1/2}
                  \|\nabla U_i\|_{L^2}^{3/2}.
\]
Substituting these estimates and then applying \eqref{eq:phi gram U}, we obtain
\begin{equation}
|\langle(U_1+U_2)\cdot\nabla U_2,U_1\rangle|
 \le\frac{K_4}{\sqrt3}\Theta_{\nabla\phi}^U
       \|U_1\|_{L^2}^{3/2}\|\nabla U_1\|_{L^2}^{3/2} +\frac{K_4}{\sqrt3}\sqrt{\Theta_\phi^U}
       (\Theta_{\nabla\phi}^U)^{3/2}\|U_1\|_{L^2}^2\|\nabla U_1\|_{L^2}.
\label{eq:U fixed strain}
\end{equation}

The energy ball gives $\eta_U\|U_1\|_2^2+\|\nabla U_1\|_2^2\le R_U^2$, hence
\begin{equation*}
 \|U_1\|_2\le\frac{R_U}{\sqrt{\eta_U}},\qquad
 \sqrt{\|U_1\|_2\|\nabla U_1\|_2}\le\frac{R_U}{\sqrt2\,\eta_U^{1/4}}.
\end{equation*}
Retain $a^U,b^U$ and choose $c^U$ so that
\begin{equation}\label{eq:U-refined-coefficients}
a^U\le\tfrac14-\eta_1-M_2^U,\qquad b^U\le1-M_1^U,\qquad c^U\ge\frac{K_4R_U}{\sqrt{3\eta_U}}
 \left(\sqrt{\Theta_\phi^U}(\Theta_{\nabla\phi}^U)^{3/2}
 +\frac{\eta_U^{1/4}}{\sqrt2}\Theta_{\nabla\phi}^U\right).
\end{equation}
Using these bounds in \eqref{eq:U fixed strain} gives
\[
 |\langle(U_1+U_2)\cdot\nabla U_2,U_1\rangle|\le c^U\|U_1\|_2\|\nabla U_1\|_2.
\]

Testing \eqref{eq:Ufix def} with $U_1$ and using \eqref{eq:TU lin est} gives

\begin{equation}
 a^U\|U_1\|_2^2+b^U\|\nabla U_1\|_2^2-c^U\|U_1\|_2\|\nabla U_1\|_2\le\varepsilon^U\|U_1\|_2,
\label{eq:U fixed energy}
\end{equation}
If $a^U,b^U>0$ and $4a^Ub^U>(c^U)^2$, completing the square gives, for $p,q\ge0$,
\begin{equation}\label{eq:U-joint-bound}
 p\|U_1\|_2+q\|\nabla U_1\|_2\le\mathcal J(a^U,b^U,c^U,\varepsilon^U;p,q),
\end{equation}

where
\[
 \mathcal J(a,b,c,\varepsilon;p,q):=\frac{2\varepsilon}{4ab-c^2}
 \left(bp+\frac{cq}{2}+\sqrt{b(bp^2+cpq+aq^2)}\right).
\]
Since $U=(I+G^U)U_1$, \eqref{eq:phi gram U} gives
\begin{equation}\label{eq:U-full-L2-bounds}
 \|U\|_2\le(1+\Theta_\phi^U)\|U_1\|_2,\qquad
 \|\nabla U\|_2\le\Theta_{\nabla\phi}^U\|U_1\|_2+\|\nabla U_1\|_2.
\end{equation}

\paragraph{Taking weak limits.}
Notice that the divergence-free condition is preserved in the weak limit, and $U_1^R$ has uniform bounds in $H^1$ by definition of $\mathcal B_R^U$. We can thus take weak limits for a subsequence of $U_1^{R}$ in $\mathcal{H}^1$ for $R\to\infty$. Denote $U_1^{\infty}$ as the weak limit.

Now fix any test function $\varphi\in C^\infty_c$ that is divergence free, for sufficiently large $R$, $B_{R}(0)$ contains the support of $\varphi$. We compute by taking limits of \eqref{eq:Ufix def} to get $$\langle L_1^U U_1^\infty + (L_2^U - L_{\tmop{ap}}^{U}) U_1^\infty , \varphi \rangle =  \langle-(U_1^\infty +
   G^{U} U_1^\infty) \cdot \nabla   (U_1^\infty +
   G^{U} U_1^\infty) - \Epsilon^U,\varphi \rangle,$$
   where we use an integration by parts for the diffusion term on the left-hand side.
  For the convergence of the right-hand side, we use the fact that in $L^2$, $\nabla^i G^{U}U_1^R\to\nabla^i G^{U}U_1^\infty$ by definition \eqref{nl def u}, $i=0,1$, and $(U_1^R +
   G^{U} U_1^R)\varphi\to(U_1^\infty +
   G^{U} U_1^\infty)\varphi$ due to the compact support of $\varphi$ upon taking a subsequence by Rellich–Kondrachov again, and $\langle f_R,g_R\rangle\to\langle f_\infty,g_\infty\rangle$ for $f_R\to f_\infty$ and $g_R \rightharpoonup g_\infty$.
By symmetry, it suffices to consider test functions in the prescribed
symmetry class. Since the test function was arbitrary within this class,
\eqref{eq: fix u 1} holds for $U_1^\infty$, completing the proof of
Proposition~\ref{prop:U}.

\subsection{Proof of Proposition \ref{prop:u lambda}}\label{Sec:prop 4}
We work on Proposition \ref{prop:u lambda} following an argument similar  to the previous subsection by assuming that Proposition \ref{prop:U} holds, and we only highlight the key changes.

\paragraph{Finite-rank approximation.}
Recall $L^U$ from \eqref{eq:profile-operators},
$P_{\overline v}$ from \eqref{eq:projection-bar-v}, and
$\lambda^{\mathrm{con}}$ from \eqref{eq:eigenvalue-correction-decomposition}.

Define
\[
 L_0f=P_{\overline v}\Pi(L^Uf-\overline\lambda f),\qquad D_Uf=P_{\overline v}\Pi\big((U\cdot\nabla)f+(f\cdot\nabla)U
                              -\lambda^{\mathrm{con}}f\big).
\]

As in \eqref{eq:phi eq}, for $j=1,\ldots,m$ we construct
$\overline\phi_j^v\in\mathcal H^1\cap H^2$ satisfying
\begin{equation*}
 L_0\overline\phi_j^v
 =P_{\overline v}\Pi Q\overline\psi_j^v+r_j^v,\qquad
 P_{\overline v}\Pi\overline\phi_j^v=\overline\phi_j^v,
\end{equation*}
where $r_j^v$ is the residual for the background operator $L_0$.
\begingroup
We construct the finite-rank approximation operator $L_{\mathrm{ap}}^v$
of $L_2^v$ as:
\begin{equation}
 L_{\mathrm{ap}}^v f
 \assign-\sum_{j=1}^{m}(\overline\lambda_j^v)^{-1}
 \langle f,Q\overline\psi_j^v\rangle_\Omega
 P_{\overline v}\Pi(Q\overline\psi_j^v+r_j^v).
 \label{eq:Lap def2}
\end{equation}
\endgroup
Define
\begingroup
\begin{equation}
 G^vv_1=\sum_{j=1}^{m}(\overline\lambda_j^v)^{-1}
 \langle v_1,Q\overline\psi_j^v\rangle_\Omega\overline\phi_j^v,
       \qquad v_2=G^vv_1.\label{v2 def}
\end{equation}
\endgroup
It follows by applying $L_0$ to this finite sum that
\begin{equation}
 L_0G^v=-L_{\mathrm{ap}}^v.
 \label{eq:odd-background-inverse}
\end{equation}

\paragraph{Definition of the fixed point map.}
Recall the space $\mathcal H^1$ defined in \eqref{Def-H0H1-space}.
For $R>0$, let $\mathcal H_R^1$ denote the closure in $\mathcal H^1$
of its smooth elements compactly supported in $B_R(0)$.
Recall also $\lambda^{\text{lin}}$ from
\eqref{eq:eigenvalue-correction-decomposition} and
$\widehat{\Epsilon}^v$ from \eqref{v-def-eqn}.

Define the fixed-point map $T_R^v$ by requiring its output
$v_1^\ast\assign T_R^v(v_1)$ to satisfy
\begin{equation}\label{eq:odd-fpm-definition}
 (L_0-L_{\mathrm{ap}}^v)v_1^\ast
     +D_Uv_1^\ast+D_UG^v v_1^\ast -\lambda^{\text{lin}}(v_1+G^v v_1)
                   (v_1^\ast+G^v v_1^\ast)=-\widehat{\Epsilon}^v.
\end{equation}
Instead of evaluating the entire nonlinear
term at the input, we keep the scalar coefficient
$\lambda^{\mathrm{lin}}(v_1+G^v v_1)$ evaluated at the input, while
replacing the vector factor $v_1+G^v v_1$ by
$v_1^\ast+G^v v_1^\ast$.
For fixed $R$, the bilinear form is bounded on $\mathcal H_R^1$,
so coercivity gives existence and uniqueness by the Lax--Milgram theorem.
At a fixed point, \eqref{eq:odd-background-inverse} yields
$L^vv=-\widehat{\Epsilon}^v-N^v(v)$ for $v=(I+G^v)v_1$.

\paragraph{The linear eigenvalue functional.}
We estimate $\lambda^{\text{lin}}((I+G^v)v_1)$.
Since $(I+G^v)v_1\perp\overline v$, integration by parts in
\eqref{eq:eigenvalue-correction-decomposition} gives
\begin{equation}\label{eq:eigenvalue-functional-ibp}
 \lambda^{\text{lin}}((I+G^v)v_1)=\langle v_1,g_0+g_U\rangle,
\end{equation}
where
\begin{equation*}
\begin{aligned}
 h_0&:=\tfrac12x\cdot\nabla\overline v-\Delta\overline v
       -(\overline U\cdot\nabla)\overline v
       +(\nabla\overline U)^T\overline v,\\
 g_0&:=(I+(G^v)^T)h_0,\qquad
 g_U:=(I+(G^v)^T)\bigl[-(U\cdot\nabla)\overline v+(\nabla U)^T\overline v\bigr].
\end{aligned}
\end{equation*}
By \eqref{v2 def},
\begin{equation*}
g_0=h_0+\sum_{j=1}^{m}(\overline\lambda_j^v)^{-1}
       \langle h_0,\overline\phi_j^v\rangle
       Q\overline\psi_j^v\mathbf1_\Omega.
\end{equation*}
The fields $Q\overline\psi_j^v$ are extended by zero from $\Omega$.

For $\eta_v>0$, we seek a numerical solution of
\[
 \eta_v w-\Delta w+\nabla p+\alpha\overline v=g_0,\qquad w\in\mathcal H^1,
\]
and choose approximations $\overline w_{\eta_v}\in\mathcal H^1\cap H^2$,
$\overline p_{\eta_v}$ and $\overline\alpha_{\eta_v}$. Define
\begin{equation}\label{eq:dual-riesz-residual}
 \Epsilon_{\eta_v}:=g_0+\Delta\overline w_{\eta_v}
       -\nabla\overline p_{\eta_v}-\overline\alpha_{\eta_v}\overline v.
\end{equation}
For $v_1\in\mathcal H^1$, integration by parts, $\nabla\cdot v_1=0$ and
$\langle v_1,\overline v\rangle=0$ give
\[
 \langle g_0,v_1\rangle
 =\langle\nabla\overline w_{\eta_v},\nabla v_1\rangle
       +\langle\Epsilon_{\eta_v},v_1\rangle.
\]
Weighted Cauchy--Schwarz then yields
\begin{equation}\label{eq:dual-riesz-upper-bound}
 |\langle g_0,v_1\rangle|\le M_5^v\|v_1\|_{H^1,\eta_v},\qquad
 M_5^v:=\left(\|\nabla\overline w_{\eta_v}\|_2^2
       +\frac{\|\Epsilon_{\eta_v}\|_2^2}{\eta_v}\right)^{1/2}.
\end{equation}

Cauchy--Schwarz gives
\begin{equation}\label{eq: eps5 def}
 \|g_U\|_2\le\varepsilon_5^v:=(1+\Theta_\phi^v)\big(
 \|\nabla\overline v\|_\infty\|U\|_2
       +\|\overline v\|_\infty\|\nabla U\|_2\big).
\end{equation}
Combining \eqref{eq:dual-riesz-upper-bound} and \eqref{eq: eps5 def}, we obtain
\begin{equation}\label{eq:Tv nonlin est}
 |\lambda^{\text{lin}}((I+G^v)v_1)|
 \le \left(M_5^v+\frac{\varepsilon_5^v}{\sqrt{\eta_v}}\right)
       \|v_1\|_{H^1,\eta_v}.
\end{equation}
Since $g_0,g_U\in L^2$, $\lambda^{\text{lin}}(I+G^v)$ extends to a
continuous linear functional on $L^2$.

\paragraph{Coercivity.}
We estimate the four contributions to the bilinear form in
\eqref{eq:odd-fpm-definition}: $L_0-L_{\mathrm{ap}}^v$, $D_U$,
$D_UG^v$, and the term involving $\lambda^{\mathrm{lin}}$.
For $v_1\in\mathcal H_R^1$ and $v_2=G^vv_1$, Lemma~\ref{lem:fe est} gives
\begin{equation*}
 \langle(L_0-L_{\mathrm{ap}}^v)v_1,v_1\rangle
 \ge(\tfrac14-\eta_1-\overline\lambda-M_2^v)\|v_1\|_2^2
                         +(1-M_1^v)\|\nabla v_1\|_2^2 .
\end{equation*}
Since $\nabla\cdot U=0$, $\langle(U\cdot\nabla)v_1,v_1\rangle=0$;
\eqref{eq:solenoidal-strain-bound} and Lemma~\ref{lem:sobolev const} give
\[
 |\langle(v_1\cdot\nabla)U,v_1\rangle|
 \le\frac1{\sqrt3}\|\nabla U\|_2\|v_1\|_4^2
 \le\frac{\|\nabla U\|_2}{2\pi\sqrt3}
       (\|v_1\|_2^2+\|\nabla v_1\|_2^2).
\]
Hence

\begin{equation}
\begin{aligned}
 |\langle(U\cdot\nabla)v_1+(v_1\cdot\nabla)U,v_1\rangle|
 &\le\frac{\|\nabla U\|_2}{2\pi\sqrt3}(\|v_1\|_2^2+\|\nabla v_1\|_2^2),\\
 |\lambda^{\mathrm{con}}|
 &\le\|\overline v\|_\infty \|\nabla U\|_2 .
\end{aligned}
 \label{eq:new coercivity v}
\end{equation}

Integration by parts,
\eqref{eq:solenoidal-strain-bound}, and \eqref{eq:sobolev-product} give
\[
\begin{aligned}
 |\langle(U\cdot\nabla)v_2+(v_2\cdot\nabla)U,v_1\rangle|
 &=2|\langle U\otimes v_2,S_{v_1}\rangle| \le\frac2{\sqrt3}\|\nabla v_1\|_2\|Uv_2\|_2\\
 &\le\frac{2K_4}{\sqrt3}\|\nabla U\|_2\|\nabla v_1\|_2
       \|v_2\|_2^{1/2}\|\nabla v_2\|_2^{1/2}.
\end{aligned}
\]

Using \eqref{eq:phi gram U} yields
\begin{equation*}
 |\langle(U\cdot\nabla)v_2+(v_2\cdot\nabla)U,v_1\rangle|
 \le\frac{2K_4}{\sqrt3}\|\nabla U\|_2
       \sqrt{\Theta_\phi^v\Theta_{\nabla\phi}^v}\,\|v_1\|_2\|\nabla v_1\|_2.
\end{equation*}
For the terms involving $\lambda^{\mathrm{con}}$, we use
\[
 |\lambda^{\mathrm{con}}\langle v_1+v_2,v_1\rangle|
 \le\Theta_{\mathrm s}^v|\lambda^{\mathrm{con}}|\|v_1\|_2^2,
\]
where we use Remark~\ref{rem:finite-rank-norm-certification} to estimate
$\Theta_{\mathrm s}^v$, defined by
\[
 \Theta_{\mathrm s}^v\assign
 1+\left\|\frac{G^v+(G^v)^*}{2}\right\|.
\]
Finally, for a fixed input $\widehat v_1$, write
$\lambda^{\mathrm{lin}}:=\lambda^{\mathrm{lin}}((I+G^v)\widehat v_1)$; then
\[
 |\lambda^{\mathrm{lin}}\langle v_1+v_2,v_1\rangle|
 \le\Theta_{\mathrm s}^v|\lambda^{\mathrm{lin}}|\|v_1\|_2^2.
\]

Combining these estimates gives
\begin{equation}
\begin{aligned}
 &\langle (L_0-L_{\mathrm{ap}}^v)v_1+D_Uv_1+D_Uv_2
                  -\lambda^{\mathrm{lin}}(v_1+v_2),v_1\rangle\\
 &\qquad\ge\bigl(a^v-\Theta_{\mathrm s}^v|\lambda^{\mathrm{lin}}|\bigr)\|v_1\|_2^2
       +b^v\|\nabla v_1\|_2^2-c^v\|v_1\|_2\|\nabla v_1\|_2,
\end{aligned}
 \label{eq:odd-mixed-coercivity}
\end{equation}
where we choose
\begin{equation}\label{eq:odd-coercivity-constants}
 \begin{aligned}
a^v&\le\tfrac14-\eta_1-\overline\lambda-M_2^v
       -\frac{\|\nabla U\|_2}{2\pi\sqrt3}
       -\Theta_{\mathrm s}^v|\lambda^{\mathrm{con}}|,\\
 b^v&\le1-M_1^v-\frac{\|\nabla U\|_2}{2\pi\sqrt3},\qquad c^v\ge\frac2{\sqrt3}K_4\|\nabla U\|_2
               \sqrt{\Theta_\phi^v\Theta_{\nabla\phi}^v}.
\end{aligned}
\end{equation}
The form is coercive when $a^v>\Theta_{\mathrm s}^v|\lambda^{\mathrm{lin}}|$, $b^v>0$, and
$(c^v)^2<4b^v\bigl(a^v-\Theta_{\mathrm s}^v|\lambda^{\mathrm{lin}}|\bigr)$.

\paragraph{Forcing estimate.}
Writing $U=U_1+U_2$, the projected forcing is
\begin{equation*}
 \widehat{\Epsilon}^v=P_{\overline v}\Pi\bigl[
 \Epsilon^v
 +\bigl((U_1\cdot\nabla)\overline v+(\overline v\cdot\nabla)U_1\bigr)
 +\bigl((U_2\cdot\nabla)\overline v+(\overline v\cdot\nabla)U_2\bigr)\bigr].
\end{equation*}

Define
\begin{equation}\label{eq:odd-forcing-Gram}
\begin{aligned}
 F_j^U&\assign(\overline\phi_j^U\cdot\nabla)\overline v
       +(\overline v\cdot\nabla)\overline\phi_j^U,\\
 \Theta_F^U&\assign\sup_{\|f\|_2=1}
 \left\|\sum_j(\overline\lambda_j^U)^{-1}
 \langle f,Q\overline\psi_j^U\rangle_\Omega F_j^U\right\|_2.
\end{aligned}
\end{equation}

For $U_2=G^UU_1$, the finite-rank kernel gives
\[
 \|(U_2\cdot\nabla)\overline v+(\overline v\cdot\nabla)U_2\|_2
 \le\Theta_F^U\|U_1\|_2.
\]
By contractivity of $P_{\overline v}\Pi$ and \eqref{eq:U-joint-bound},
with $a^U,b^U$ from \eqref{eq:U fixed energy} and
$p=\|\nabla\overline v\|_\infty+\Theta_F^U$, $q=\|\overline v\|_\infty$,
\begin{equation}
\begin{aligned}
 \|\widehat{\Epsilon}^v\|_2
 &\le\varepsilon^v+(\|\nabla\overline v\|_\infty+\Theta_F^U)\|U_1\|_2
       +\|\overline v\|_\infty\|\nabla U_1\|_2\\
 &\le\varepsilon^v+\mathcal J(a^U,b^U,c^U,\varepsilon^U;p,q)
 =:\widehat{\varepsilon}^v.
\end{aligned}
 \label{eq: res est v}
\end{equation}

\begin{assumption}[Bounded mapping assumption for $v$]\label{asm:cond ref v}

Fix $\eta_v>0$ and $R_v>0$.
For $a^v,b^v,c^v$ chosen as in \eqref{eq:odd-coercivity-constants},

set

\[
 \widetilde a^v:=a^v-\Theta_{\mathrm s}^v
 \left(M_5^v+\frac{\varepsilon_5^v}{\sqrt{\eta_v}}\right)R_v.
\]

We assume
\begin{equation}\label{eq:odd-scalar-margin}
 b^v>\frac{c^v}{2\sqrt{\eta_v}},\qquad
 (\widetilde a^v-\eta_vb^v)
 \left(b^v-\frac{c^v}{2\sqrt{\eta_v}}\right)
 >\frac{(\widehat\varepsilon^v)^2}{4R_v^2}.
\end{equation}
These conditions imply $\widetilde a^v,b^v>0$ and $(c^v)^2<4\widetilde a^vb^v$.

\end{assumption}

{We will verify Assumption~\ref{asm:cond ref v} in
Subsection~\ref{Sec:num prev}.}

\paragraph{{A priori estimate in $H^1$.}}
Define
\begin{equation}\label{eq:odd-invariant-set}
 \mathcal B_R^v:=\{v_1\in\mathcal H_R^1:\|v_1\|_{H^1,\eta_v}\le R_v\}.
\end{equation}
By \eqref{eq:Tv nonlin est}, every input in $\mathcal B_R^v$ satisfies
\begin{equation}\label{eq:odd-input-scalar-bound}
 |\lambda^{\text{lin}}((I+G^v)v_1)|\le\left(M_5^v+\frac{\varepsilon_5^v}{\sqrt{\eta_v}}\right)R_v.
\end{equation}
Testing \eqref{eq:odd-fpm-definition} with $v_1^\ast$ and using
\eqref{eq:odd-mixed-coercivity}, Assumption~\ref{asm:cond ref v}, and \eqref{eq: res est v} gives
\begin{equation}\label{eq:Tv col}
 \widetilde a^v\|v_1^\ast\|_2^2+b^v\|\nabla v_1^\ast\|_2^2
 -c^v\|v_1^\ast\|_2\|\nabla v_1^\ast\|_2
 \le \widehat{\varepsilon}^v\|v_1^\ast\|_2.
\end{equation}

Using $2\sqrt{\eta_v}\|v_1^\ast\|_2\|\nabla v_1^\ast\|_2
\le\|v_1^\ast\|_{H^1,\eta_v}^2$ in \eqref{eq:Tv col} and completing the square gives
\[
 \left(b^v-\frac{c^v}{2\sqrt{\eta_v}}\right)
 \|v_1^\ast\|_{H^1,\eta_v}^2
 \le\widehat\varepsilon^v\|v_1^\ast\|_2
       -(\widetilde a^v-\eta_vb^v)\|v_1^\ast\|_2^2
 \le\frac{(\widehat\varepsilon^v)^2}{4(\widetilde a^v-\eta_vb^v)}.
\]
By \eqref{eq:odd-scalar-margin}, $\|v_1^\ast\|_{H^1,\eta_v}<R_v$;
thus $T_R^v(\mathcal B_R^v)\subset\mathcal B_R^v$.

\paragraph{{Continuity estimate in $L^4$.}}
{For $v_1^\ast\assign T_R^v(v_1)$,
$\widehat v_1^\ast\assign T_R^v(\widehat v_1)$, we denote
$\delta v^\ast\assign v_1^\ast-\widehat v_1^\ast$ and
$\delta v\assign v_1-\widehat v_1$.}
{By \eqref{eq:odd-fpm-definition}, we have}
\begin{equation}\label{tv dif eqn}
\begin{aligned}
&\langle(L_0-L_{\mathrm{ap}}^v)\delta v^\ast
       +D_U\delta v^\ast+D_UG^v\delta v^\ast,\delta v^\ast\rangle -\lambda^{\text{lin}}((I+G^v)v_1)
                \langle(I+G^v)\delta v^\ast,\delta v^\ast\rangle\\
 &=\lambda^{\text{lin}}((I+G^v)\delta v)
                \langle(I+G^v)\widehat v_1^\ast,\delta v^\ast\rangle.
\end{aligned}
\end{equation}
Using \eqref{eq:eigenvalue-functional-ibp} and the a priori bound of
$\widehat v_1^\ast$, the right-hand side of \eqref{tv dif eqn} is bounded by
\[
 C\|\delta v\|_{L^2}\|\delta v^\ast\|_{L^2}.
\]
By \eqref{tv dif eqn}, coercivity, and the Sobolev embedding of $H^1$ into $L^4$,
we conclude that, for fixed $R$,
\[
 \|\delta v^\ast\|_{L^4}
 \lesssim\|\delta v^\ast\|_{H^1}
 \lesssim\|\delta v\|_{L^2}
 \lesssim_R\|\delta v\|_{L^4}.
\]
{Thus $T_R^v$ is a continuous mapping in $L^4$} on $\mathcal B_R^v$.
The Hilbert norm ball $\mathcal B_R^v$ is convex and is closed in $L^4$ by weak lower semicontinuity of its weighted norm. The uniform
$H^1$ bound makes the image relatively compact in $L^4(B_R)$ by
Rellich--Kondrachov. {By the Schauder fixed-point theorem, there exists
a fixed point $v_1^R$ such that}, for any $g\in\mathcal H_R^1$,
\begin{equation*}
\langle L_1^vv_1^R+(L_2^v-L_{\mathrm{ap}}^v)v_1^R+D_UG^vv_1^R,g\rangle =\langle-N^v(v_1^R+G^vv_1^R)-\widehat{\Epsilon}^v,g\rangle.
\end{equation*}

\paragraph{Taking weak limits.}
Notice that the divergence-free condition and orthogonality condition are preserved in the weak limit.  Denote $v_1^{\infty}\in\mathcal{H}^1$ as the weak limit for a subsequence of $v_1^{R}$ as $R$ tends to infinity. For any test function $\varphi\in C^\infty_c\cap\mathcal{H}^1$, we have 
\begingroup
\begin{equation*}
 \langle L_1^v v_1^\infty+(L_2^v-L_{\mathrm{ap}}^v)v_1^\infty
                       +D_UG^vv_1^\infty,\varphi\rangle
 =\langle-N^v(v_1^\infty+G^vv_1^\infty)
                       -\widehat{\Epsilon}^v,\varphi\rangle.
\end{equation*}
\endgroup
where the only change is to show $\lambda^{\text{lin}} (v_1^R + G^{v} v_1^R)\to\lambda^{\text{lin}} (v_1^\infty + G^{v} v_1^\infty)$, this is true by definition of the linear operator $\lambda^{\text{lin}}$ in \eqref{eq:eigenvalue-functional-ibp}. For the eigenpair, the same symmetry argument applies, and the remaining
component along $\overline v$ vanishes by
\eqref{eq:eigenvalue-correction-decomposition}.
Since the test function $\varphi$ is arbitrary, Proposition~\ref{prop:u lambda} follows.

Weak lower semicontinuity preserves $\|v_1^\infty\|_{H^1,\eta_v}\le R_v$.

The numerical verification in Section~\ref{Sec:num prev} gives
\begin{equation}\label{eq:odd-eigenvalue-sign}
 \widetilde\lambda=\overline\lambda+\lambda<0.
\end{equation}

\subsection{Numerical results}\label{Sec:num prev}
We summarize below the numerical bounds used in the proof.
The numerical methods are described in Sections~\ref{Sec:num}--\ref{sec7}.

\subsubsection{\texorpdfstring{Numerical inputs and their roles}{Numerical inputs and their roles}}

The approximate fields and the matrix $Q$ are the fixed candidates constructed in Sections~\ref{Sec:num} and~\ref{Sec:numfr}; the inequalities below are their certified bounds. The compensation weight $\eta_1=0.1909$ and the spectral energy weight $\eta_2=0.005$ are chosen parameters for which these bounds hold. The coercivity coefficients, energy weights and ball radii used below are then chosen to satisfy the analytic estimates; the constants $M_1^i,M_2^i,M_3^U,M_4^U$ and the resulting correction bounds are derived from the certified inputs.

\paragraph{Residuals and the approximate eigenpair.}

The defects in Assumption~\ref{asm:num} measure the error in the profile and eigenpair equations. The negative approximate eigenvalue provides the instability margin, while the pointwise bounds on $\overline v$ control transport by that field. The residual integrals are certified as in Section~\ref{subsec:analytic-L2}; the pointwise bounds use Sections~\ref{subsec:Linf-envelope} and~\ref{subsec:Q-envelope}.

\begingroup
\begin{eqnarray}
\|\Epsilon^U\|_{L^2}&\le&\varepsilon^U=4.97\times10^{-6},\label{eq:45-input-EU}\\
\|\Epsilon^v\|_{L^2}&\le&\varepsilon^v=4.81\times10^{-6},\label{eq:45-input-Ev}\\
\overline\lambda&\le&-0.1131,\label{eq:45-input-lambda}\\
\|\overline v\|_{L^\infty}&\le&0.0319,\label{eq:45-input-v}\\
\|\nabla\overline v\|_{L^\infty}&\le&0.0432,\label{eq:45-input-grad-v}
\end{eqnarray}
\endgroup

\paragraph{Compensation and background coercivity.}

The matrix inequalities below give the coercive decomposition of Section~\ref{sec:analysis}. Positivity of $Q$ and the compensation inequality are separate requirements: the former permits the $Q$-geometry used in the finite-rank approximation, and the latter leaves $L^2$ damping $1/4-\eta_1$. Section~\ref{subsec:Q-envelope} certifies both inequalities and the maximum eigenvalue bound.

\begingroup
\begin{eqnarray}
Q&\succeq&0,\label{eq:45-input-Q-positive}\\
\frac{\nabla\overline U+\nabla^T\overline U}{2}+\eta_1 I+Q
   &\succeq&0,\qquad\eta_1=0.1909,\label{eq:45-input-Q-condition}\\
\|\lambda_{\max}(Q)\|_{L^\infty}&\le&4.47,\label{eq:45-input-Q-max}
\end{eqnarray}
\endgroup

\paragraph{Spectral tail and retained eigensystem.}

The omitted-mode estimate of Section~\ref{subSec:numfr} controls $K_5^i$ through \eqref{eq:K5-definition}. Inside $V_N^i$, the interval matrix tests of Section~\ref{sec:Q eig} bound $K_6^i$ on the $Q$-orthogonal complement of the retained modes. The errors $\varepsilon_3^i,\varepsilon_4^i$ account for the approximate Gram relations and orthogonality. These bounds enter $M_1^i,M_2^i$ in \eqref{eq:finite-rank-constants}.

\begingroup
\begin{eqnarray}
\sup_{\substack{u\in V,\ u\ne0\\u\perp_{L^2}V_N^i}}
 \frac{\|u\|_{L^2(\Omega)}^2}{\|u\|_{H^1(\Omega),\eta_2}^2}
 &\le&0.1024,\qquad i=U,v,\label{eq:45-input-tail}\\
K_6^U&\le&0.449,\label{eq:45-input-K6}\\
K_6^v&\le&0.445,\label{eq:45-input-K6v}\\
\varepsilon_3^U,\ \varepsilon_4^U,\ \varepsilon_3^v,\ \varepsilon_4^v
 &\le&2.8\times10^{-7},\label{eq:45-input-errors}
\end{eqnarray}
\endgroup

\paragraph{Finite-rank inverse corrections.}

The fields constructed in Section~\ref{sec:inv L2}, the norm estimates of Remark~\ref{rem:finite-rank-norm-certification}, and the quadrature bounds of Section~\ref{subsec:Quadrature} give the following enclosures. The constants $\Theta_r^i$ control the linear-solve residuals; $\Theta_\phi^i$ and $\Theta_{\nabla\phi}^i$ control amplification of a correction and its gradient. The mixed constant $\Theta_F^U$ bounds the contribution of $G^UU_1$ to the eigenpair forcing in \eqref{eq:odd-forcing-Gram}.

\begingroup
\begin{eqnarray}
{\Theta_{r}^U}&\le&4.54\times10^{-4},\label{eq:45-input-rU}\\
{\Theta_{r}^v}&\le&0.027,\label{eq:45-input-rv}\\
\Theta_{\phi}^U&\le&4.29,\label{eq:45-input-phi-U}\\
\Theta_{\nabla\phi}^U&\le&4.8,\label{eq:45-input-gradphi-U}\\
\Theta_\phi^v&\le&52.1,\label{eq:45-input-Gv}\\
\Theta_{\nabla\phi}^v&\le&45.6,\label{eq:45-input-grad-Gv}\\
{\Theta_F^U}&{\le}&0.0638,\label{eq:45-input-FU}
\end{eqnarray}
\endgroup

\paragraph{Inputs to the nonlinear profile estimate.}

The coefficient bound $\Theta_Q^U$ from \eqref{eq:qusi-or1} and the two aggregates below determine $M_3^U$ in \eqref{eq:U-M3-definition}. They are enclosed using Remark~\ref{rem:finite-rank-norm-certification} and the pointwise and quadrature estimates of Sections~\ref{subsec:Linf-envelope} and~\ref{subsec:Quadrature}.

\begingroup
\begin{eqnarray}
\Theta_Q^U&\le&5.25,\label{eq:45-input-GU}\\
\sqrt{\sum_{i,j}\|\overline\phi_i^U\cdot\nabla\overline\phi_j^U\|_{L^2}^2}
 &\le&0.68,\label{eq:45-input-product-U}\\
\sqrt{\sum_j\|\nabla\overline\phi_j^U\|_{L^\infty}^2}&\le&0.37,\label{eq:45-input-grad-phi-U}
\end{eqnarray}
\endgroup

\paragraph{The auxiliary eigenvalue functional.}

The constrained auxiliary solve and residual estimate \eqref{eq:dual-riesz-upper-bound}, with the certified energy and residual integrals of Sections~\ref{subsec:analytic-L2}--\ref{subsec:Quadrature}, give

\begingroup
\begin{eqnarray}
M_5^v&\le&15.\label{eq:45-input-M5v}
\end{eqnarray}
\endgroup

The certificate uses mass weight $0.23187$. This is a dual weighted-energy bound on the eigenvalue functional, not a pointwise bound. Its use at the chosen eigenpair weight is justified below.

\subsubsection{\texorpdfstring{Profile existence and refined error estimates}{Profile existence and refined error estimates}}

\paragraph{Step 1: Retain positive coercivity.}

We first check how much coercivity is lost through the finite-rank approximation. Substituting the spectral tail, $K_6^U$ and Gram errors into \eqref{eq:finite-rank-constants} gives

\begin{equation}\label{eq:45-U-M1}
M_1^U
 =K_5^U+\frac{K_6^U}{(1-\varepsilon_4^U)^2} \le 4.47\times 0.1024
       +\frac{0.449}{(1-2.8\times10^{-7})^2}
 <0.9068.
\end{equation}

The mass contribution of this error, the Gram mismatch and the residual $\Theta_r^U$ then give

\begin{equation}\label{eq:45-U-M2}
 M_2^U\le0.005\times0.9068
 +(2.8\times10^{-7})\times4.47+0.000454<0.005.
\end{equation}

Thus we may choose the conservative lower coefficients $a^U=0.0541$ and $b^U=0.0932$: by \eqref{eq:45-input-Q-condition} and \eqref{eq:U-refined-coefficients},

\begin{equation}\label{eq:45-U-ab}
\begin{aligned}
 \tfrac14-\eta_1-M_2^U&\ge0.25-0.1909-0.005=0.0541=:a^U,\\
 1-M_1^U&\ge1-0.9068=0.0932=:b^U.
\end{aligned}
\end{equation}

Both the $L^2$ damping and gradient dissipation remain positive. These are chosen lower bounds, as permitted in Assumption~\ref{eq:cond ref U}, rather than equalities with the unrounded optimal coefficients.

\paragraph{Step 2: Close the profile energy ball.}

Next we control the nonlinear term in \eqref{eq:TU nonlin est}. The inputs \eqref{eq:45-input-GU}--\eqref{eq:45-input-grad-phi-U} give

\begin{equation}\label{eq:45-U-M3}
\begin{aligned}
 M_3^U
 &=\left(\Theta_Q^U\sum_j\|\nabla\overline\phi_j^U\|_\infty^2\right)^{1/2}
 +\Theta_Q^U
 \left(\sum_{i,j}\|\overline\phi_i^U\cdot\nabla\overline\phi_j^U\|_2^2\right)^{1/2}\\
 &\le\sqrt{5.25}\times0.37
 +5.25\times0.68<4.42.
\end{aligned}
\end{equation}

The gradient aggregate in \eqref{eq:45-input-grad-phi-U} already bounds a square root, so its contribution is $\sqrt{5.25}\,(0.37)$. The resulting discriminant is positive:

\begin{equation}\label{eq:45-U-M4}
 M_4^U=(a^U)^2-4M_3^U\varepsilon^U
 \ge(0.0541)^2-4\times4.42\times(4.97\times10^{-6})>0.00283>0.
\end{equation}

Assumption~\ref{eq:cond ref U} therefore holds. Using the same conservative $a^U,b^U$ in the definitions \eqref{eq:weighted-H1-norm}, we obtain

\begin{equation}\label{eq:45-U-radius}
 \eta_U=\frac{a^U}{b^U}=\frac{0.0541}{0.0932},\qquad
 R_U=\frac{2\sqrt{\eta_U}\,\varepsilon^U}{a^U+\sqrt{M_4^U}} \le\frac{2\sqrt{\frac{0.0541}{0.0932}}\times(4.97\times10^{-6})}
 {0.0541+\sqrt{0.00283}}
 <7.06\times10^{-5}.
\end{equation}

For an input in this weighted energy ball, the profile-map estimate in Section~\ref{Sec:prop 3} is

\[
\|U_1^*\|_{H^1,\eta_U}
\le\frac{\varepsilon^U+(M_3^U/\eta_U)R_U^2}{b^U\sqrt{\eta_U}}.
\]

The profile ball is invariant because
\[
 \varepsilon^U+\frac{M_3^U}{\eta_U}R_U^2
 \le b^U\sqrt{\eta_U}\,R_U.
\]
At the outward-rounded radius $7.06\times10^{-5}$, the right side
minus the left side is greater than $5.2\times10^{-9}$ using
$M_3^U\le4.42$.

The exact radius in \eqref{eq:45-U-radius} is the smaller root of the invariance quadratic; the rounded value is an upper bound and also satisfies the displayed inequality.

\paragraph{Step 3: Obtain the exact profile.}

The continuity and compactness established in Section~\ref{Sec:prop 3} now yield a Schauder fixed point on every truncated space. The energy bound is independent of the truncation radius, so the weak-limit argument there gives a whole-space correction $U_1$. With $U=(I+G^U)U_1$, the field $\widetilde U=\overline U+U$ is the exact profile of Proposition~\ref{prop:U}.

\paragraph{Step 4: Refine the profile-error estimates.}

At a fixed point, the cancellation in \eqref{eq:U fixed strain} gives a sharper estimate than the invariant ball alone. Substitution into \eqref{eq:U-refined-coefficients} permits the following upper coefficient:

\begin{equation}\label{eq:45-U-c}
\begin{aligned}
 &\frac{K_4R_U}{\sqrt{3\eta_U}}
 \left(\sqrt{\Theta_\phi^U}(\Theta_{\nabla\phi}^U)^{3/2}
 +\frac{\eta_U^{1/4}}{\sqrt2}\Theta_{\nabla\phi}^U\right)\\
 &\le\frac{2}{3^{5/4}\pi}\times\frac{7.06\times10^{-5}}{\sqrt{\frac{0.0541}{0.0932}}}
 \left(\sqrt{4.29}\times(4.8)^{3/2}+\frac{\left(\frac{0.0541}{0.0932}\right)^{1/4}}{\sqrt2}\times4.8\right)
 <0.00037=:c^U.
\end{aligned}
\end{equation}

Consequently \eqref{eq:U fixed energy} reads

\[
a^U\|U_1\|_2^2+b^U\|\nabla U_1\|_2^2
-c^U\|U_1\|_2\|\nabla U_1\|_2\le\varepsilon^U\|U_1\|_2.
\]

\begin{equation*}
 4a^Ub^U-(c^U)^2
 \ge4\times0.0541\times0.0932-(0.00037)^2
 >0.0201.
\end{equation*}

The positive discriminant makes this an elliptic constraint on the two norms. Recall the joint bound \eqref{eq:U-joint-bound}: for $p,q\ge0$,

\[
p\|U_1\|_2+q\|\nabla U_1\|_2
\le\mathcal J(a^U,b^U,c^U,\varepsilon^U;p,q),
\]

\[
\mathcal J(a,b,c,\varepsilon;p,q)
=\frac{2\varepsilon}{4ab-c^2}
\left(bp+\frac{cq}{2}+\sqrt{b(bp^2+cpq+aq^2)}\right).
\]

This bound retains the relation between the two norms instead of maximizing them independently. Applying it to the combinations in \eqref{eq:U-full-L2-bounds} gives

\begin{equation}\label{eq:45-v-U-L2}
 \|U\|_2\le(1+\Theta_\phi^U)\|U_1\|_2
 \le5.29\|U_1\|_2<0.000487.
\end{equation}

\begin{equation}\label{eq:45-v-U-gradient}
 \|\nabla U\|_2\le\Theta_{\nabla\phi}^U\|U_1\|_2+\|\nabla U_1\|_2
 \le4.8\|U_1\|_2+\|\nabla U_1\|_2<0.000445.
\end{equation}

These estimates control the perturbation of the eigenpair operator. For its forcing, we will also use $\mathcal J$ directly with the smaller coefficients of that particular expression, avoiding the loss from first bounding the full correction norms separately.

\subsubsection{\texorpdfstring{Eigenpair existence and preservation of instability}{Eigenpair existence and preservation of instability}}

\paragraph{Step 1: Include the profile correction in the forcing.}

The eigenpair problem is posed at the exact profile $\widetilde U$, so its forcing contains more than the raw residual $\Epsilon^v$. By \eqref{eq: res est v},

\[
\|\widehat{\Epsilon}^v\|_2
\le\varepsilon^v+
(\|\nabla\overline v\|_\infty+\Theta_F^U)\|U_1\|_2
+\|\overline v\|_\infty\|\nabla U_1\|_2.
\]

The corresponding coefficients are bounded by

\begin{equation}\label{eq:45-U-pq}
 \|\nabla\overline v\|_\infty+\Theta_F^U
 \le0.0432+0.0638=0.107=:p,
 \qquad \|\overline v\|_\infty\le0.0319=:q.
\end{equation}

Using the joint profile bound with these $p,q$ gives $\mathcal J(a^U,b^U,c^U,\varepsilon^U;p,q)<10^{-5}$ and therefore

\begin{equation}\label{eq:45-U-to-v-forcing}
 \widehat\varepsilon^v\le\varepsilon^v
 +\mathcal J(a^U,b^U,c^U,\varepsilon^U;p,q)
 <4.81\times10^{-6}+10^{-5}<1.49\times10^{-5}.
\end{equation}

\paragraph{Step 2: Bound the eigenvalue feedback on the input ball.}

The scalar coefficient in the semiimplicit map must be controlled for every input, before any eigenpair fixed point has been obtained. We bound the two terms of \eqref{eq:eigenvalue-correction-decomposition}. First, \eqref{eq:new coercivity v} and the exact-profile estimate \eqref{eq:45-v-U-gradient} give

\begin{equation}\label{eq:45-v-lambda-con}
 |\lambda^{\mathrm{con}}|
 \le\|\overline v\|_\infty\|\nabla U\|_2
 \le0.0319\times0.000445<0.0000142.
\end{equation}

For the profile-dependent part of the linear functional, \eqref{eq: eps5 def} and \eqref{eq:U-full-L2-bounds} give coefficients bounded as follows:

\[
\begin{aligned}
 &(1+\Theta_\phi^v)\bigl[\|\nabla\overline v\|_\infty(1+\Theta_\phi^U)
       +\|\overline v\|_\infty\Theta_{\nabla\phi}^U\bigr]
 \le53.1\times(0.0432\times5.29+0.0319\times4.8)<20.3=:p,\\
 &(1+\Theta_\phi^v)\|\overline v\|_\infty
 \le53.1\times0.0319<1.7=:q.
\end{aligned}
\]

Applying the joint bound with these coefficients yields

\begin{equation}\label{eq:45-v-epsilon5}
 \varepsilon_5^v\le p\|U_1\|_2+q\|\nabla U_1\|_2
 \le\mathcal J(a^U,b^U,c^U,\varepsilon^U;p,q)<0.00188.
\end{equation}

We choose the input ball \eqref{eq:odd-invariant-set} with

\begin{equation}\label{eq:45-v-ball-parameters}
 \eta_v=0.232,\qquad R_v=0.00016.
\end{equation}

The certificate for $M_5^v$ obtained at mass weight $0.23187$
also applies at $\eta_v=0.232$, since the weighted dual norm decreases when
the mass weight increases.

Thus \eqref{eq:dual-riesz-upper-bound} and \eqref{eq:odd-input-scalar-bound} give, uniformly for every input $v_1$ in this ball,

\begin{equation}\label{eq:45-v-lambda-lin}
|\lambda^{\mathrm{lin}}((I+G^v)v_1)|
 \le\left(M_5^v+\frac{\varepsilon_5^v}{\sqrt{\eta_v}}\right)R_v
 \le\left(15+\frac{0.00188}{\sqrt{0.232}}\right)\times0.00016<0.00241.
\end{equation}

Combining the constant and linear parts, the scalar assigned to every such input satisfies

\begin{equation}\label{eq:45-v-lambda}
 |\lambda|\le|\lambda^{\mathrm{con}}|+|\lambda^{\mathrm{lin}}|
 \le0.0000142+0.00241
 <0.00243.
\end{equation}

These are input-ball estimates used to retain coercivity; they do not assume existence of the eigenpair. The sign of the exact eigenvalue will be concluded only after existence is established.

\paragraph{Step 3: Retain coercivity for the eigenpair map.}

As in the profile argument, the finite-rank errors first give

\begin{equation}\label{eq:45-v-M1}
M_1^v=K_5^v+\frac{K_6^v}{(1-\varepsilon_4^v)^2} \le4.47\times0.1024
 +\frac{0.445}{(1-2.8\times10^{-7})^2}<0.9028.
\end{equation}

\begin{equation}\label{eq:45-v-M2}
 M_2^v\le0.005\times0.9028
 +(2.8\times10^{-7})\times4.47+0.027<0.0316.
\end{equation}

The scalar feedback enters a quadratic pairing with $I+G^v$. We therefore retain the sharper symmetric-part bound, rather than replace it by the full operator norm. The two spectral-endpoint estimates of Remark~\ref{rem:finite-rank-norm-certification} give

\[
 \Theta_{\mathrm s}^v
 =1+\left\|\frac{G^v+(G^v)^*}{2}\right\|\le33.6.
\]

After allowing for the profile perturbation and the constant eigenvalue term, \eqref{eq:odd-coercivity-constants} permits $a^v=0.14$. Subtracting the uniform linear-feedback bound gives $\widetilde a^v$:

\begin{equation}\label{eq:45-v-a}
\begin{aligned}
 &\tfrac14-\eta_1-\overline\lambda-M_2^v
 -\frac{\|\nabla U\|_2}{2\pi\sqrt3}
 -\Theta_{\mathrm s}^v|\lambda^{\mathrm{con}}|\\
 &\ge0.25-0.1909+0.1131-0.0316
 -0.092\times0.000445-33.6\times0.0000142>0.14=:a^v,\\
 &\widetilde a^v=a^v-\Theta_{\mathrm s}^v
 \left(M_5^v+\frac{\varepsilon_5^v}{\sqrt{\eta_v}}\right)R_v
 \ge0.14-33.6\times0.00241>0.059.
\end{aligned}
\end{equation}

The gradient and mixed coefficients may likewise be chosen as

\begin{equation}\label{eq:45-v-b}
1-M_1^v-\frac{\|\nabla U\|_2}{2\pi\sqrt3} \ge1-0.9028-0.092\times0.000445
 >0.0971=:b^v.
\end{equation}

\begin{equation}\label{eq:45-v-c}
\frac{2K_4}{\sqrt3}\|\nabla U\|_2
      \sqrt{\Theta_\phi^v\Theta_{\nabla\phi}^v} \le\frac4{3^{5/4}\pi}\times0.000445
      \times\sqrt{52.1\times45.6}<0.007=:c^v.
\end{equation}

The estimate \eqref{eq:Tv col} consequently controls the output $v_1^*$ of the map by

\[
\widetilde a^v\|v_1^*\|_2^2+b^v\|\nabla v_1^*\|_2^2
-c^v\|v_1^*\|_2\|\nabla v_1^*\|_2
\le\widehat\varepsilon^v\|v_1^*\|_2.
\]

\paragraph{Step 4: Verify invariance of the eigenpair ball.}

The two factors in \eqref{eq:odd-scalar-margin} are strictly positive:

\begin{equation}\label{eq:45-v-quadratic-margin}
 \widetilde a^v-\eta_vb^v>0.0364,\qquad
 b^v-\frac{c^v}{2\sqrt{\eta_v}}>0.0898.
\end{equation}

Using $2\sqrt{\eta_v}\|v_1^*\|_2\|\nabla v_1^*\|_2\le\|v_1^*\|_{H^1,\eta_v}^2$ and completing the square in $\|v_1^*\|_2$ gives

\[
\|v_1^*\|_{H^1,\eta_v}
\le\frac{\widehat\varepsilon^v}
{2\sqrt{(\widetilde a^v-\eta_vb^v)
\left(b^v-\frac{c^v}{2\sqrt{\eta_v}}\right)}}.
\]

It therefore suffices to compare the forcing with the capacity of the chosen ball. The bounds above give

\begin{equation}\label{eq:45-v-radius-budget}
 2R_v\sqrt{(\widetilde a^v-\eta_vb^v)
 \left(b^v-\frac{c^v}{2\sqrt{\eta_v}}\right)}
 >1.82\times10^{-5}>1.49\times10^{-5}>\widehat\varepsilon^v.
\end{equation}

This verifies the bounded mapping conditions in Assumption~\ref{asm:cond ref v}; in particular, the output satisfies

\[
\|v_1^*\|_{H^1,\eta_v}
<\frac{1.49\times10^{-5}}{2\sqrt{0.0364\times0.0898}}
<0.000131<0.00016=R_v.
\]

The forcing capacity exceeds the forcing upper bound by more than
$3.3\times10^{-6}$, or $22\%$ of that upper bound.

\paragraph{Step 5: Obtain the exact eigenpair.}

All bounds on the map are uniform in the truncation radius. The continuity, compactness and weak-limit argument of Section~\ref{Sec:prop 4} therefore gives an exact correction $v=(I+G^v)v_1$ and scalar $\lambda$ satisfying the eigenpair equation, with the same energy-ball bound. The uniform estimate \eqref{eq:45-v-lambda} now applies to this exact correction.

\paragraph{Step 6: Preserve instability and nontriviality.}

At the fixed point, \eqref{eq:45-input-lambda} and \eqref{eq:45-v-lambda} give

\[
 \widetilde\lambda=\overline\lambda+\lambda
 \le-0.1131+0.00243<-0.1106<0.
\]

By \eqref{Def-H0H1-space} and Assumption~\ref{asm:num},
$\langle\overline v+v,\overline v\rangle=1$.

Thus the eigenfunction is nonzero and the exact eigenvalue is negative, completing the verification of Proposition~\ref{prop:u lambda}.

\section{Numerical Construction of Approximate Solutions}\label{Sec:num}

In this section, we detail the numerical procedures used to solve the
equation \eqref{eq:NS eq eig}. In prior research {\cite{guillod2023numerical}}, boundary conditions of the
form: 
\[ A_\alpha (\theta) = \alpha ( 0, 0, e^{- 4 \cot^2 \theta} ), \]
were used, where $\alpha$ is a parameter. As $\alpha$ increased, the
eigenvalue associated with the linear operator $-\mathcal{L} (\overline{U})$
crossed the imaginary axis at approximately $\alpha \approx 292$. However, the finite-rank approximation in Section \ref{subsec:finiterank} for solutions with such large values of boundary conditions will be inefficient, rendering the numerical verification challenging. To overcome this, we introduce an alternative
numerical strategy aimed at obtaining a smaller solution associated with a
negative eigenvalue.
Instead of tuning $\alpha$ and tracking eigenvalue crossings, we optimize the field $\overline U$ directly: we compute $\nabla \lambda(\overline U)$ and run a projected gradient descent to obtain a small solution with a negative eigenvalue. Then, we interpolate the approximate solution to
enforce exact divergence-free conditions.

This section is structured as follows. In Section \ref{sec:gd}, we derive the
gradient descent method for the eigenvalue $\lambda (\widetilde{U})$. Section
\ref{sec:sym} simplifies the problem by leveraging symmetries (axisymmetry in spherical coordinates). In Section
\ref{sec:fem}, we discretize the equations using the finite element method.
Section \ref{sec:eig} details the computation of the eigenvalue. The first four subsections produce a finite-element candidate; we then refine it to a high-precision candidate using a spectral basis, which is divergence-free and smooth away from the origin. In Section
\ref{sec:ext}, we extend the obtained solutions to the entire space. Section
\ref{sec:div-free} presents the construction of a divergence-free basis. In
Section \ref{sec:res cg}, we interpolate our numerical solutions onto this
basis and reduce residuals. Finally, Section \ref{sec:sum} summarizes our
numerical approach and provides pseudo-code for the implemented algorithm.

\subsection{Gradient descent for eigenvalue}\label{sec:gd}

We write $\mathcal L(\widetilde U):=\mathcal L_{\widetilde U}$, with $\mathcal L_{\widetilde U}$ defined in \eqref{eq:LU def}.

To compute the desired eigenvalue, we adopt a gradient descent approach. The
iterative update from $\bar{U}_n$ to $\bar{U}_{n + 1}$ is defined as follows.
Recall that $\lambda (\tilde{U})$, as a functional of $\tilde{U}$, is the
eigenvalue of $-\mathcal{L} (\tilde{U})$ with the smallest real part. First, we
calculate the Fr{\'e}chet derivative of $\lambda (\tilde{U})$. Throughout, we
assume that $\lambda (\tilde{U})$ is real, algebraically simple, and isolated, which
guarantees that it is also a simple and isolated eigenvalue for $(-\mathcal{L}
(\tilde{U}))^{\ast}$. Let $v^l (\tilde{U}), v^r (\tilde{U})$ denote the left
and right eigenvectors associated with the eigenvalue $\lambda (\tilde{U})$,
respectively satisfying:
\[ -\mathcal{L} (\tilde{U}) v^r (\tilde{U}) = \lambda (\tilde{U}) v^r
   (\tilde{U}), \quad (-\mathcal{L} (\tilde{U}))^{\ast} v^l (\tilde{U}) = \lambda
   (\tilde{U}) v^l (\tilde{U}) . \]
Denoting $v^r_n = v^r (\bar{U}_n)$ and $v^l_n = v^l (\bar{U}_n)$ respectively,
the gradient of $\lambda (\bar{U}_n)$ is explicitly given by:
\begin{equation}
  \frac{\delta \lambda (\bar{U}_n)}{\delta \tilde{U}} = \frac{v^l_n \cdot
  \nabla^T v^r_n - v^r_n \cdot \nabla v^l_n}{\langle v^l_n, v^r_n \rangle} .
  \label{eq:gradient lambda}
\end{equation}
A formal derivation of the above equation is provided in Appendix
\ref{sec:nabla lambda}. This is also known as the non-self-adjoint
Hellmann--Feynman formula:
\[ \partial \lambda = -\langle v^l, (\partial \mathcal{L}) v^r \rangle /
   \langle v^l, v^r \rangle, \quad \partial \mathcal{L} (\delta U) \assign -\Pi\bigl(
   (\delta U)  \hspace{-0.17em} \cdummy \hspace{-0.17em} \nabla
   \hspace{0.17em} (\cdummy) \hspace{0.17em} + (\cdummy)  \hspace{-0.17em}
   \cdummy \hspace{-0.17em} \nabla (\delta U)\bigr) . \]
Without additional constraints, the increment $g_n$ would ideally be given by
\[ g_n = - k_n \frac{\delta \lambda (\bar{U}_n)}{\delta \tilde{U}}, \]
where $k_n$ denotes step size. In practice, $k_n$ can be chosen by different
strategies; we use the Barzilai--Borwein method {\cite{barzilai1988two}} for
efficiency.

However, since $\tilde{U}$ is constrained to lie on the solution manifold
defined by equation \eqref{eq:NS-eq-self-similar0}, the increment $g_n$ must
belong to the tangent space of this manifold at the current iterate
$\bar{U}_n$:
\[ T_{\bar{U}_n} \assign \{ \hspace{0.17em} g : \mathrm{div} g = 0,
   \mathcal{L}(\bar{U}_n) \hspace{0.17em} g = 0 \hspace{0.17em} \} . \]
Moreover, the boundary condition $A_n (\omega)$ with $\omega = \frac{x}{|x|}$
is updated together with $\bar{U}_n$. In addition to the tangent-space
constraint, we impose conditions on $A (\omega)$ that governs the far-field
behavior of $\tilde{U}$. The key motivation is to ensure that the most
negative eigenvalue of the symmetric gradient
\[ Q_s (x) = \frac{1}{2} (\nabla\overline U (x) + \nabla^T\overline U (x)), \]
decays as rapidly as possible in the far field. Using the ansatz $\bar{U}_n
\approx \frac{A_n (\omega)}{| x |}$, we obtain
\[ Q_{s, n} \approx \frac{Q^A_{s, n} (\omega)}{| x |^2}, \]
so the radial decay rate is fixed at $| x |^{- 2}$, while the prefactor $Q_{s,
n}^A (\omega)$ depends on $A_n (\omega)$ and controls the size of the most
negative eigenvalue. To control this prefactor and to ensure numerical
stability, we introduce a positive constant $\gamma_U > 0$ and impose two
boundary constraints on $A_{n + 1}$:
\begin{enumerate}
  \item Regularity bound (to sufficient smoothness of $A (\omega)$):
  \begin{equation}
    \| \nabla^2 A (\omega)\|_{L^\infty} \leqslant \gamma_U, \label{eq:bd an}
  \end{equation}
  \item Spectral bound (to suppress the most negative eigenvalue in the far
  field):
  \begin{equation}
    \lambda_{\max} (- Q_{s, n}^A (\omega)) \leqslant \gamma_U . \label{eq:lbd
    an}
  \end{equation}
\end{enumerate}
Combining these requirements, we determine the increment by solving the convex
problem:
\[ \min_{g_n} \| g_n + k_n \frac{\delta \lambda (\bar{U}_n)}{\delta
   \tilde{U}} \|_{L^2}, \text{ \ s.t. $g_n \in T_{\bar{U}_n}$, $A_{n + 1}$
   satisfies \eqref{eq:bd an} and \eqref{eq:lbd an}.} \]
This problem can be solved efficiently in MATLAB. With the increment $g_n$
determined by solving the constrained optimization problem, we obtain the
tentative update
\[ \bar{U}^{\ast}_{n + 1} = \bar{U}_n + g_n . \]
Since $\bar{U}^{\ast}_{n + 1}$ is not yet a solution to equation
\eqref{eq:NS-eq-self-similar0}, we project it onto the solution manifold while
fixing the boundary condition $A_{n + 1} (\omega)$, yielding the updated
solution $\bar{U}_{n + 1}$. In practice, we discretize the system
\eqref{eq:NS-eq-self-similar0} using the finite element method, resulting in a
nonlinear equation for the numerical solution $\bar{U}$. We then fix the
boundary condition $A_{n + 1}$ and apply Newton's method with initial guess
$\bar{U}^{\ast}_{n + 1}$ to solve the discretized version of equation
\eqref{eq:NS-eq-self-similar0}, thereby obtaining $\bar{U}_{n + 1}$. A few
Newton steps suffice; we stop once the discrete residual falls below a
prescribed tolerance.

We iterate this procedure until convergence. If the eigenvalue satisfies
$\lambda < - 0.1$, the iteration terminates. Otherwise, we incrementally
increase $\gamma_U$ and repeat the above steps.

\subsection{Axisymmetry, parity and boundary conditions}\label{sec:sym}
We consider the axisymmetric case throughout the computations and denote the cylindrical coordinates as $(\rho,z,\phi)$. We also impose symmetry conditions in $z$: even for the profile problem $(U,P)$ and odd for the eigenvalue problem $(v,q)$. In the Euclidean coordinates, a vector field $(u_x,u_y,u_z)$ is \emph{even in $z$} if 
\[
u_x(x,y,-z)=u_x(x,y,z),\quad  
u_y(x,y,-z)=u_y(x,y,z),\quad  
u_z(x,y,-z)=-u_z(x,y,z).
\]
In other words, the horizontal components are even functions of $z$, while the vertical component is odd. In spherical coordinates $(r,\theta,\phi)$, the reflection $z\mapsto -z$ corresponds to $\theta \mapsto \pi-\theta$. The condition of being even in $z$ then reads  
\[
u_r(r,\pi-\theta,\phi) = u_r(r,\theta,\phi),\quad  
u_\phi(r,\pi-\theta,\phi) = u_\phi(r,\theta,\phi),\quad  
u_\theta(r,\pi-\theta,\phi) = -\,u_\theta(r,\theta,\phi).
\]
The definition of \emph{odd in $z$} is the opposite.

Throughout the analysis in Section~\ref{sec:analysis}, we work in the subspaces of axisymmetric functions with the prescribed even or odd symmetry. All the estimates and fixed-point arguments remain valid in these subspaces without modification.

We note that if $\widetilde{\lambda}$ is a simple eigenvalue, since $\widetilde{U}$ is even in $z$, $\widetilde{v}$ must be either even or odd
.
Numerical evidence indicates that eigenvalues associated with odd
eigenvectors decrease more rapidly. Therefore, we assume $\widetilde{v}$ is odd
with respect to $z$.

In practice, for the construction of accurate approximate solutions in this section, we adopt spherical coordinates $(r, \theta, \phi)$.  Due to the axisymmetry and parity, the computational
domain reduces to $r \in \mathbb{R}_+$ and $\theta \in \left[ 0, \frac{\pi}{2}
\right]$. Since the spherical coordinates $(r,\theta)$ correspond to a polar representation  of the cylindrical coordinates $(\rho,z)$, boundary conditions in the spherical coordinates are thus given by:
\begin{itemizedot}
  \item At $\theta = 0$, which corresponds to $\rho=0$, we have $e_r =e_z $ and $e_{\theta}=e_\rho$. To ensure $\Delta U,\Delta v\in L^2$ and smoothness of the pressure, we impose a Dirichlet zero condition on the $\rho$ and $\phi$ components of the velocity field and a Neumann zero condition on the $z$ component and the pressure.  This implies $\widetilde{U}_{\theta} (r, 0) = \widetilde{U}_{\phi} (r, 0)
  = 0$ and  $\partial_\theta
  \widetilde{U}_r (r, 0) =\partial_\theta\widetilde{P}(r, 0)= 0$. Likewise, $\widetilde{v}_{\theta} (r, 0) = \widetilde{v}_{\phi} (r, 0)
  = 0$ and  $\partial_\theta
  \widetilde{v}_r (r, 0) =\partial_\theta\widetilde{q}(r, 0)= 0$.
  
  \item At $\theta = \frac{\pi}{2}$, which corresponds to $z=0$, we have $e_r =e_\rho $ and $e_{\theta}=-e_z$, implying  $\partial_\theta \widetilde{U}_r \left( r, \frac{\pi}{2} \right) =
  \widetilde{U}_{\theta} \left( r, \frac{\pi}{2} \right) = \partial_\theta
  \widetilde{U}_{\phi} \left( r, \frac{\pi}{2} \right) = \partial_\theta
  \widetilde{P} \left( r, \frac{\pi}{2} \right)= 0$ by even symmetry. Conversely, by odd symmetry, we have $ \widetilde{v}_r \left( r, \frac{\pi}{2} \right) =
 \partial_\theta \widetilde{v}_{\theta} \left( r, \frac{\pi}{2} \right) = 
  \widetilde{v}_{\phi} \left( r, \frac{\pi}{2} \right) = 
  \widetilde{q} \left( r, \frac{\pi}{2} \right)= 0$.
  
  \item At $r = 0$,  which corresponds to the origin in the cylindrical coordinates, the even symmetry in $z$ and the Dirichlet zero condition in $\rho$ directly
  lead to $\widetilde{U} (0, \theta) = (0, 0, 0) .$ Similarly, we impose $\partial_r \widetilde{P}(0, \theta)=0$ and  $\widetilde{q}(0, \theta)=0$. Finally, since at  the origin in cylindrical coordinates, $\widetilde{v}$ must align with the direction
  $e_z$, we know that $\widetilde{v}$ takes the following form for a constant $C$:\begin{equation}
  \label{v-boundary-condition}
 \widetilde{v} =  C
  (\cos \theta, -\sin \theta, 0).
  \end{equation} 
  Another reason why we need to impose this boundary condition at $r=0$ is because we have to cancel the two  terms that contain the singular factor $\frac{1}{r^2}$ in the Laplacian operator $\Delta$ in the spherical coordinates, see Appendix \ref{sec sphdel}.
If we do not enforce this boundary condition and allow $\widetilde{v}$ to have modes $\sin(m\theta)$ or $\cos(m\theta)$ with $m \geqslant 2$, then the equation for $\widetilde{v}$ will generate some singular terms that do not cancel each other, rendering $\widetilde{v}\notin H^2$.
  
  \item At the far-field boundary, the condition is given by the function $A (\theta)$: $\lim_{r
  \to + \infty} r \widetilde{U} (r, \theta) = A (\theta) .$ We require $\widetilde{v}$ to decay faster than $\widetilde{U}$ by imposing  $\lim_{r
  \to + \infty} r \widetilde{v} (r, \theta) = 0.$ In
  {\cite{guillod2023numerical}}, the authors estimate the decay rate as
  approximately $O (r^{- 4})$.
\end{itemizedot}
Here, $A (\theta)$ satisfies the consistency condition derived from the
incompressibility constraint \eqref{eq:A constraint}:
\begin{equation}
  A_r (\theta) + \frac{1}{\sin \theta} \partial_{\theta}  (A_{\theta} (\theta)
  \sin \theta) = 0. \label{eq:A constraint sph}
\end{equation}

\subsection{Spherical coordinates and finite element method}\label{sec:fem}

In this section, we detail our setting for finite element approximation. The
incompressible Navier--Stokes equations in the spherical coordinates read:
\begin{equation}
  \left\{\begin{array}{l}
    - \frac{1}{2}  \widetilde{U} - \frac{1}{2} r \partial_r  \widetilde{U} + \widetilde{U}
    \cdot \nabla_{\tmop{sph}}  \widetilde{U} + \nabla_{\tmop{sph}} \widetilde{P} -
    \Delta_{\tmop{sph}}  \widetilde{U} = 0,\\
    \tmop{div}_{\tmop{sph}}  \widetilde{U} = 0,
  \end{array}\right. \label{eq:NS ss sph}
\end{equation}
where $\nabla_{\text{sph}}$, $\text{div}_{\text{sph}}$, and
$\Delta_{\text{sph}}$ are the gradient, divergence, and Laplacian operators in
the spherical coordinates, respectively (see Appendix \ref{sec:ops} for explicit
expressions). The weak form of \eqref{eq:NS ss sph} is:
\begin{equation*}
  \langle \widetilde{U} \cdummy \nabla \widetilde{U}, V \rangle - \frac{1}{2} \langle
  x \cdummy \nabla \widetilde{U}, V \rangle - \frac{1}{2} \langle \widetilde{U}, V
  \rangle - \langle \widetilde{P}, \tmop{div} V \rangle - \langle Q, \tmop{div}
  \widetilde{U} \rangle + \langle \nabla \widetilde{U}, \nabla V \rangle = 0,
  \label{eq:weak NS}
\end{equation*}
for all test functions $V \in H_0^1 (\mathbb{R}^3)^3, Q \in L^2
(\mathbb{R}^3)$. In the spherical coordinates, the weak form
reads:
\begin{equation*}
  \begin{array}{l}
    \langle \widetilde{U} \cdummy \nabla_{\tmop{sph}} \widetilde{U}, V
    \rangle_{\tmop{sph}} - \frac{1}{2} \langle r \partial_r \widetilde{U}, V
    \rangle_{\tmop{sph}} - \frac{1}{2} \langle \widetilde{U}, V
    \rangle_{\tmop{sph}} - \langle \widetilde{P}, \tmop{div}_{\tmop{sph}} V
    \rangle_{\tmop{sph}}\\
    - \langle Q, \tmop{div}_{\tmop{sph}}  \widetilde{U} \rangle_{\tmop{sph}} +
    \langle \nabla_{\tmop{sph}} \widetilde{U}, \nabla_{\tmop{sph}} V
    \rangle_{\tmop{sph}} = 0,
  \end{array}
\end{equation*}
for all test functions $V \in H^1_{0,\tmop{sph}} ( \mathbb{R}_+ \times
[ 0, \frac{\pi}{2} ] )^3, Q \in L^2_{\tmop{sph}} (
\mathbb{R}_+ \times[ 0, \frac{\pi}{2} ] )$.
 The inner product is
\[ \langle f, g \rangle_{\tmop{sph}} = \int_{\left[ 0, \frac{\pi}{2}
   \right]}\int_{\left[ 0, \infty\right]} f^T (r, \theta) \nosymbol g (r, \theta) r^2 \sin \theta \mathd r \mathd \theta . \]

In practice, we truncate the
domain to a finite domain $B
(0, R)$ with $R = 20$. The far-field boundary condition $\lim \nospace_{r
\to + \infty} r \widetilde{U} (r, \theta) = A (\theta)$ is approximated by
the Dirichlet boundary condition:
\[ \widetilde{U} (r = R, \theta) = \frac{1}{R} A (\theta), \]
where $A (\theta)$ must satisfy the consistency constraint \eqref{eq:A
constraint sph} derived from incompressibility. We impose a Neumann boundary condition for the pressure: $\partial_r\widetilde{P} (R, \theta) = 0$.

We discretize the domain using finite elements, following closely the method
detailed in {\cite{guillod2023numerical}}. Specifically, the computational
domain is divided into $N_r \times N_{\theta}$ squares, with $N_r = 128$ and
$N_{\theta} = 64$. Each square is subdivided into two triangles. As the
parameter $\gamma_U$ increases, a larger computational domain is required to
ensure that truncation errors remain negligible. To efficiently handle this
domain expansion without modifying the finite element mesh, we adopt a spatial
rescaling strategy following the approach described in
{\cite{guillod2023numerical}}. Specifically, we fix the computational domain
radius at $R = 20$ and introduce a spatial rescaling given by:
\[ \widetilde{U}_{\kappa} (x) = \widetilde{U} (\kappa x). \]
This rescaling is equivalent to expanding the computational domain from $B (0,
R)$ to $B (0, \kappa R)$, while allowing the reuse of the same finite element mesh
for varying values of $\gamma_U$. The scaling factor $\kappa$ is chosen to be:
\[ \kappa = \sqrt{1 + \frac{\gamma_U}{4}} . \]

Thus, the rescaled equations are given by:
\begin{equation}
  \begin{array}{l}
    \kappa^{-1} \langle \widetilde{U} \cdummy \nabla_{\tmop{sph}} \widetilde{U}, V
    \rangle_{\tmop{sph}} - \frac{1}{2} \langle r \partial_r \widetilde{U}, V
    \rangle_{\tmop{sph}} - \frac{1}{2} \langle \widetilde{U}, V
    \rangle_{\tmop{sph}} - \kappa^{-1} \langle \widetilde{P}, \tmop{div}_{\tmop{sph}} V
    \rangle_{\tmop{sph}}\\
    - \kappa^{-1} \langle Q, \tmop{div}_{\tmop{sph}}  \widetilde{U}
    \rangle_{\tmop{sph}} + \kappa^{-2} \langle \nabla_{\tmop{sph}} \widetilde{U},
    \nabla_{\tmop{sph}} V \rangle_{\tmop{sph}} = 0,
  \end{array} \label{eq:weak NS sph}
\end{equation}
with boundary conditions:
\[ \widetilde{U} (r = R, \theta) = \frac{1}{\kappa R} A (\theta) . \]
Here, we omit the subscript $\kappa$ for simplicity.
 Finally, we discretize
the weak form \eqref{eq:weak NS sph} using Lagrange quadratic polynomials (P2
elements) for $\widetilde{U}$ and linear polynomials (P1 elements) for
$\widetilde{P}$. The discretization involves assembling not only bilinear forms
but trilinear forms, such as the convection term $\langle \widetilde{U} \cdot
\nabla_{\text{sph}} \widetilde{U}, V \rangle_{\text{sph}}$, which require handling
third-order tensors. We use scikit-fem because it provides a convenient implementation for explicitly
assembling and accessing the third-order tensors used in our eigenvalue optimization.

\subsection{Eigenvalue solver}\label{sec:eig}

Suppose we have obtained a solution $\widetilde{U}$. Our next step is to compute
the eigenvalues of the linearized operator $-\mathcal{L} (\widetilde{U})$. The eigenvalue problem in the spherical coordinates reads:
\[ \left\{\begin{array}{l}
     - \frac{1}{2}  \widetilde{v} - \frac{1}{2} r \partial_r  \widetilde{v} +
     \widetilde{U} \cdot \nabla_{\tmop{sph}}  \widetilde{v} + \widetilde{v} \cdot
     \nabla_{\tmop{sph}}  \widetilde{U} + \nabla_{\tmop{sph}} \widetilde{q} -
     \Delta_{\tmop{sph}}  \widetilde{v} = \widetilde{\lambda}  \widetilde{v},\\
     \tmop{div}_{\tmop{sph}}  \widetilde{v} = 0.
   \end{array}\right. \]

We start
by truncating the domain to a ball $B (0, R)$. 
The same spatial rescaling by $\kappa$ discussed earlier is then applied. Consequently,
the weak form of the eigenvalue problem is expressed as:
\begin{equation}
  \begin{array}{l}
    \kappa^{-1} \langle \widetilde{U} \cdummy \nabla_{\tmop{sph}} v, u
    \rangle_{\tmop{sph}} + \kappa^{-1} \langle v \cdummy \nabla_{\tmop{sph}}
    \widetilde{U}, u \rangle_{\tmop{sph}} - \frac{1}{2} \langle r \partial_r v, u
    \rangle_{\tmop{sph}} - \frac{1}{2} \langle v, u \rangle_{\tmop{sph}} -
    \kappa^{-1} \langle q, \tmop{div}_{\tmop{sph}} u \rangle_{\tmop{sph}}\\
    - \kappa^{-1} \langle p, \tmop{div}_{\tmop{sph}} v \rangle_{\tmop{sph}} +
    \kappa^{-2} \langle \nabla_{\tmop{sph}} v, \nabla_{\tmop{sph}} u
    \rangle_{\tmop{sph}} = \lambda \langle v, u \rangle_{\tmop{sph}},
  \end{array} \label{eq:weak NS eig sph}
\end{equation}
for all test functions $u \in H^1_{0,\tmop{sph}} (\mathbb{R}_+ \times [0, \frac\pi2])^3, p
\in L^2_{\tmop{sph}} (\mathbb{R}_+ \times [0, \frac\pi2])$. We impose a homogeneous Dirichlet boundary condition at the truncation
  radius $r = R$: $\widetilde{v} (R, \theta) = (0, 0, 0)$ and a Neumann boundary condition for the pressure: $\partial_r\widetilde{q} (R, \theta)= 0$. 

We discretize this weak form using the finite element method described in the
previous section, which yields a generalized matrix eigenvalue problem of the
form:
\[ A x = \lambda B x, \]
where $A$ corresponds to the discretization of the left-hand side of equation
\eqref{eq:weak NS eig sph}, and $B$ represents the discretization of the inner
product term on the right-hand side. The resulting eigenvalue problem is
solved numerically using MATLAB.

\begin{remark}
   In \cite{guillod2023numerical}, the authors noted that directly imposing
  Dirichlet boundary conditions at a finite radius $R$ might slightly distort
  the computed solutions. For this reason, the numerical settings described in
  this subsection and the previous section are employed only within the
  gradient descent iterations for obtaining coarse solutions $\overline{U}$
  and $\overline{v}$, rather than for rigorous numerical analysis. To minimize
  potential distortions and obtain accurate numerical solutions, we adopt an
  interpolation-based extension technique, as detailed in the subsequent
  section.
\end{remark}

\subsection{Extension to whole space}\label{sec:ext}

After obtaining a numerical solution $(\overline{U}, \overline{P})$ on the
bounded domain, we need to extend it to the entire space. To achieve this, we
introduce the coordinate transformation:
\[ r = \tan \beta . \]
Under this transformation, we define new variables to capture the decay at the far field:
\[ \widehat{U} (\beta, \theta) = \frac{\widetilde{U} (\tan\beta, \theta)}{\cos \beta},
   \quad \widehat{P} (\beta, \theta) = \frac{\widetilde{P} (\tan\beta, \theta)}{\cos
   \beta}.\]
In these new variables, equation \eqref{eq:NS ss sph} takes the following
form:
\begin{equation*}
  \left\{\begin{array}{l}
    - \frac{1}{2} \cos^2 \beta (\cos \beta  \widehat{U} + \sin
    \beta \partial_{\beta}\widehat{U}) + \cos \beta\widehat{U} \cdot \widetilde{\nabla}_{\tmop{sph}}  \widehat{U} +
    \widetilde{\nabla}_{\tmop{sph}} \widehat{P} - \widetilde{\Delta}_{\tmop{sph}}  \widehat{U}
    = 0,\\
    \widetilde{\tmop{div}}_{\tmop{sph}}  \widehat{U} = 0,
  \end{array}\right. \label{eq:NS ss sph cv}
\end{equation*}
where the definitions of $\widetilde{\nabla}_{\tmop{sph}}$,
$\widetilde{\Delta}_{\tmop{sph}}$ and $\widetilde{\tmop{div}}_{\tmop{sph}}$ are
listed in Appendix \ref{sec:ops}. The hatted operators in that appendix are the tilded operators used here; for example, $\widetilde\nabla_{\mathrm{sph}}f=\nabla_{\mathrm{sph}}(\cos\beta\,f)$ after the change of variables. Taking the physical velocity and scalar test functions to be $\cos\beta\,V$ and $\cos\beta\,Q$, respectively, the corresponding weak form is
\begin{equation}\label{eq:weak NS sph cv}
\begin{aligned}
 &\kappa^{-1}\langle\cos\beta\widehat U\cdot\widetilde\nabla_{\mathrm{sph}}\widehat U,\cos\beta\,V\rangle_{\mathrm{sphcv}}
 -\tfrac12\langle\cos^2\beta(\cos\beta\widehat U+\sin\beta\partial_\beta\widehat U),\cos\beta\,V\rangle_{\mathrm{sphcv}}\\
 &-\kappa^{-1}\langle\cos\beta\,\widehat P,\widetilde{\mathrm{div}}_{\mathrm{sph}}V\rangle_{\mathrm{sphcv}}
 -\kappa^{-1}\langle\cos\beta\,Q,\widetilde{\mathrm{div}}_{\mathrm{sph}}\widehat U\rangle_{\mathrm{sphcv}}\\
 &+\kappa^{-2}\langle\widetilde\nabla_{\mathrm{sph}}\widehat U,\widetilde\nabla_{\mathrm{sph}}V\rangle_{\mathrm{sphcv}}=0.
\end{aligned}
\end{equation}

for all test functions $V \in H^1_{0, \tmop{sphcv}} ([0, \frac\pi2]^2)^3, Q \in L^2_{\tmop{sphcv}} ([0, \frac\pi2]^2)$, with the
inner product
\[ \langle f, g \rangle_{\tmop{sphcv}} = \int_{\left[ 0, \frac{\pi}{2}
   \right]^2} f^T (\beta, \theta) \nosymbol g (\beta, \theta) \frac{\tan^2
   \beta}{\cos^2 \beta} \sin \theta \mathd \beta \mathd \theta . \]
The far-field boundary condition transforms into:
\[ \widehat{U} \left( \frac{\pi}{2}, \theta \right) = \frac{A (\theta)}{\kappa} .
\]
When numerically solving the weak form \eqref{eq:weak NS sph cv}, it is
beneficial to select a larger scaling factor $\kappa$:
\[ \kappa = 2 \sqrt{1 + \gamma_U / 4}. \]
We employ the same discretization method as before to solve this weak form
numerically. Given a solution $\overline{U}$ defined on $\mathbb{R}_+ \times
\left[ 0, \frac{\pi}{2} \right]$, we first interpolate it onto the finite
element basis corresponding to the new mesh. Here $r=\tan\beta$ denotes
the rescaled radius, with physical radius $\kappa r$. Beyond the original
mesh, we initialize the profile by $A(\theta)/(\kappa r)$ and the
eigenvector by zero. Although the interpolated
function generally does not exactly satisfy equation \eqref{eq:weak NS sph
cv}, it provides a sufficiently accurate initial guess. Consequently, by
applying Newton's iteration with fixed boundary conditions, we can efficiently
solve the transformed system within a few steps, obtaining a good
approximation over the entire spatial domain.

Next, we linearize the operator at the solution profile $\overline{U}$
obtained from equation \eqref{eq:weak NS sph cv} and compute the corresponding
eigenvectors. We note that the computed solution profile $\overline{U}$ and
eigenvector $\overline{v}$ are close to being divergence-free, though not exactly
so. In the next subsections, we introduce a procedure to enforce the
divergence-free condition exactly.

\begin{remark}
  Numerical instabilities occasionally lead to spurious eigenvalues with
  highly oscillatory eigenvectors, which do not correspond to  the actual
  eigenmodes of the operator. However, despite these numerical artifacts, the
  method still enables us to approximate the genuine eigenvectors accurately.
\end{remark}

\subsection{Construction of divergence-free basis}\label{sec:div-free}

To exactly enforce the divergence-free condition, we construct a
divergence-free basis. We first expand the vector $\widetilde{U}$ using axisymmetric vector spherical harmonics $X_l, Y_l, Z_l$, defined later in \eqref{xyz}:
\[ \widetilde{U} (r, \theta) = \sum_{l = 0}^{+ \infty} (\widetilde{U}_l^Y (r) Y_l
   (\theta) + \widetilde{U}_l^X (r) X_l (\theta) + \widetilde{U}_l^Z (r) Z_l (\theta))
   . \]
    Here, for a scalar function $p$, we can expand $p$ with the scalar spherical harmonics:
\[ p (r, \omega) = \sum_{l = 0}^{+ \infty} \sum_{m = - l}^l p_{l m} (r) S_{l
   m} (\omega), \quad r \geqslant 0, \quad \omega \in \mathbb{S}^2 , \]
where $S_{l m} (\omega)$ is the spherical harmonics. For a vector field $u$,
we can also expand $u$ with the vector spherical harmonics:
\[ u (r, \omega) = \sum_{l = 0}^{+ \infty} \sum_{m = - l}^l (u_{l m}^Y (r)
   Y_{l m} (\omega) + u_{l m}^X (r) X_{l m} (\omega) + u_{l m}^Z (r) Z_{l m}
   (\omega)), \]
where
\[ Y_{l m} (\omega) = S_{l m} (\omega) \vec{e}_r, \quad X_{l m} (\omega) = \nabla_{\mathbb{S}^2}
   S_{l m} (\omega), \quad Z_{l m} (\omega) = \vec{e}_r \times \nabla_{\mathbb{S}^2} S_{l m}
   (\omega) . \]
   
Since we only consider axisymmetric functions, we choose $m = 0$ and $\omega$
reduces to $\theta$:
\begin{equation}
    \label{xyz}
 Y_l (\theta) = S_l (\theta) \vec{e}_r, \quad X_l (\theta) = S_l' (\theta)
   \vec{e}_{\theta}, \quad Z_l (\theta) = \vec{e}_{\phi} S_l' (\theta), \quad
   S_l (\theta) = P_l (\cos \theta), \end{equation}
where $P_l$ denotes the Legendre polynomial. Here we drop the subscript $m =
0$ for simplicity. 

We have the reduced expansions
$$
  p (r, \theta) = \sum_{l = 0}^{+ \infty} p_l (r) S_l (\theta),\quad
  u (r, \theta) = \sum_{l = 0}^{+ \infty} (u_l^Y (r) Y_l (\theta) + u_l^X
  (r) X_l (\theta) + u_l^Z (r) Z_l (\theta)) .
$$
We note that $X_l, Y_m, Z_n$ are mutually orthogonal for all
indices $l, m, n$. Their $L^2$ inner products are:
$$
  \int_0^{\pi} Y_l (\theta) Y_m (\theta) \sin \theta \mathd \theta =
  \frac{2 \delta_{l m}}{2 l + 1}, \quad
  \int_0^{\pi} X_l (\theta) X_m (\theta) \sin \theta \mathd \theta =
  \int_0^{\pi} Z_l (\theta) Z_m (\theta) \sin \theta \mathd \theta =\frac{2 l (l
  + 1)\delta_{l m}}{2 l + 1}  .
$$
The gradients of $p$ and $u$ have the following forms, by Appendix \ref{sec sphdel} and the eigenvector properties of the spherical harmonics and the Legendre polynomials:
\begin{eqnarray}
  \nabla_{\tmop{sph}} p & = & \sum_{l = 0}^{+ \infty} \left( p_l' (r) Y_l
  (\theta) + \frac{1}{r} p_l (r) X_l (\theta) \right),  \label{eq:grad p}\\
  \Delta_{\tmop{sph}} p & = & \sum_{l = 0}^{+ \infty} \left( \Delta_r p_{
  l} (r) - \frac{\mu_l}{r^2} p_{l} (r) \right) S_l (\theta),  \label{eq:lap
  p}\\
  \tmop{div}_{\tmop{sph}} u & = & \sum_{l = 0}^{+ \infty} \left(
  \frac{\mathd}{\mathd r} u_l^Y + \frac{2}{r} u_l^Y - \frac{\mu_l}{r} u_l^X
  \right) S_l,  \label{eq:div u}\\
  \Delta_{\tmop{sph}} u & = & \sum_{l = 0}^{+ \infty} (\Delta_r u_l^Y Y_l +
  \Delta_r u_l^X X_l + \Delta_r u_l^Z Z_l) \nonumber\\
  & + & \frac{1}{r^2} \sum_{l = 0}^{+ \infty} ((2 \mu_l u_l^X - (\mu_l + 2)
  u_l^Y) Y_l + (2 u_l^Y - \mu_l u_l^X) X_l - \mu_l u_l^Z Z_l),  \label{eq:lap
  u}
\end{eqnarray}
where we define
$$
  \mu_l = l (l + 1),\quad
  \Delta_r = \frac{1}{r^2} \frac{\mathd}{\mathd r} r^2
  \frac{\mathd}{\mathd r} .
$$
 We note that the superscripts $X,Y,Z$ refer to the spherical harmonics and have nothing to do with Cartesian coordinates.
We note that when $l$ is even, the vector field $(Y_l, X_l, 0)$ is even with
respect to $z$ while $(0, 0, Z_l)$ is odd. When $l$ is odd, they have the
opposite parities. Therefore, for the even $\widetilde{U}$, we choose even $l$ for
$Y_l, X_l$ and odd $l$ for $Z_l$. For the odd $\widetilde{v}$, it is the other way
around. Therefore, we modify our expansion to:
\begin{align}
  \widetilde{U} (r, \theta) & = \sum_{l = 0}^{+ \infty} (\widetilde{U}_l^Y (r) Y_{2
  l} (\theta) + \widetilde{U}_l^X (r) X_{2 l} (\theta)) + \sum_{l = 1}^{+ \infty}
  \widetilde{U}_l^Z (r) Z_{2 l - 1} (\theta),  \label{eq:Utld exp}\\
  \widetilde{v} (r, \theta) & = \sum_{l = 1}^{+ \infty} (\widetilde{v}_l^Y (r) Y_{2
  l - 1} (\theta) + \widetilde{v}_l^X (r) X_{2 l - 1} (\theta) + \widetilde{v}_l^Z (r)
  Z_{2 l} (\theta)) . \nonumber
\end{align}  We compute
the divergence of $\widetilde{U}$ by \eqref{eq:div u} as
\[ \tmop{div} \widetilde{U} = \sum_{l = 0}^{+ \infty} \left( \frac{\mathd}{\mathd
   r}  \widetilde{U}_l^Y + \frac{2}{r}  \widetilde{U}_l^Y - \frac{\mu_{2 l}}{r} 
   \widetilde{U}_l^X \right) S_{2l}, \quad \mu_{2 l} = 2 l (2 l + 1) . \]
Therefore, the divergence-free condition is
\[ r \frac{\mathd}{\mathd r}  \widetilde{U}_l^Y + 2 \widetilde{U}_l^Y - \mu_{2 l} 
   \widetilde{U}_l^X = 0. \]
Under the previously introduced coordinate transformation, the divergence-free
condition becomes:
\begin{equation}
  \frac{1}{2} \sin 2 \beta \frac{\mathd}{\mathd \beta}  \widehat{U}_l^Y +
  \frac{1}{2}  (3 + \cos 2 \beta)  \widehat{U}_l^Y - \mu_{2 l}  \widehat{U}_l^X = 0,\quad \forall l\geqslant0.
  \label{eq:div sph cv 1}
\end{equation}
For each mode $l$, since both $\widehat{U}_l^Y$ terms are multiplied with $\sin 2 \beta$ or $\cos 2 \beta$, we need one more Fourier mode in $\beta$ for $\widehat{U}_l^X$ to make sure \eqref{eq:div sph cv 1} holds. In practice, we truncate the series in \eqref{eq:Utld exp} up to $L$ terms. Therefore, the divergence
free condition splits into $L$ equations relating the $\beta$ coefficient
functions $\widehat{U}_l^Y$ and $\widehat{U}_l^X$.

In order to expand the coefficient functions into trigonometric series, we first examine the boundary behavior. Specifically, we recall the boundary conditions in Section 
\ref{sec:sym} as \mbox{$\widehat{U} (\beta = 0, \theta) = (0, 0, 0)$} at the origin. For the approximate profile $\overline U$, the far field behavior \eqref{boundary u} in Assumption \ref{asm:num} implies that
\[
    \widehat{U} = \frac{\overline{U}}{\cos \beta} = A(\theta) + O((\frac{\pi}{2}-\beta)^2),
\]
where we have used the change of variable $r=\tan \beta = O((\frac{\pi}{2}-\beta)^{-1})$. These conditions lead naturally to the following expansion of $\widehat{U}_l^Y
(\beta)$ and $\widehat{U}_l^X (\beta)$ as:
\begin{equation}
  \widehat{U}_l^Y (\beta) = \sum_{m = 1}^M \Beta^{U, Y}_{m, l} \sin ((2 m - 1)
  \beta), \quad \widehat{U}_l^X (\beta) = \sum_{m = 1}^{M + 1} \Beta^{U, X}_{m, l}
  \sin ((2 m - 1) \beta), \label{eq:U1 ansatz 1}
\end{equation}
since 
\[
 \sin ((2 m - 1) \beta) = (-1)^{m-1} + O((\frac{\pi}{2}-\beta)^2).
\]
It should be noted, however, that this expansion lacks smoothness at the origin $\beta=0$. Substituting these into the divergence-free constraint \eqref{eq:div sph cv 1}
yields the algebraic relations:
\begin{equation}
  \Beta^{U, X}_{m, l} = - \frac{1}{\mu_{2 l}}  \left( \frac{m}{2} \Beta^{U,
  Y}_{m + 1, l} - \frac{3}{2} \Beta^{U, Y}_{m, l} - \frac{m - 1}{2} \Beta^{U,
  Y}_{m - 1, l} \right), \quad \forall m \geqslant 1, l \in \mathbb{Z}^+
  \label{eq:div sph cv 2}
\end{equation}
with boundary conditions: $\Beta^{U, Y}_{0, l} = \Beta^{U, Y}_{M + 1, l} =
\Beta^{U, Y}_{M + 2, l} = 0$. Now we have solved the coefficients  $\Beta^{U, X}$ from $\Beta^{U, Y}$ to enforce the divergence-free condition.

When $l = 0$, we have $\Beta^{U, Y}_{m, l} = 0$, and hence $\widehat{U}_l^Y
(\beta)$ becomes trivial. Moreover, since $X_0 (\theta) = 0$, the first term
$\widetilde{U}_0^Y (r) Y_0 (\theta) + \widetilde{U}_0^X (r) X_0 (\theta)$ vanishes.
Consequently, in the first series of \eqref{eq:Utld exp}, the summation
effectively starts from 1.

For the eigenvector $\overline{v}$, the boundary conditions are given by: $\overline{v}
(\beta = 0, \theta) = C (\cos \theta, - \sin \theta, 0)$ and $\overline{v} \left(
\beta = \frac{\pi}{2}, \theta \right) = (0, 0, 0)$. Consequently, only the
functions $\widehat{v}_1^X$ and $\widehat{v}_1^Y$ remain nonzero at $\beta = 0$. To
satisfy these boundary conditions, we adopt the following expansions:
\begin{equation*}
  \widehat{v}_1^Y (\beta) = \sum_{m = 1}^M \Beta^{v, Y}_{m, 1} \cos ((2 m - 1)
  \beta), \quad \widehat{v}_1^X (\beta) = \sum_{m = 1}^{M + 1} \Beta^{v, X}_{m, 1}
  \cos ((2 m - 1) \beta), \quad \text{for } l = 1, \label{eq:v ansatz 2.1}
\end{equation*}
and
\begin{equation*}
  \widehat{v}_l^Y (\beta) = \sum_{m = 1}^M \Beta^{v, Y}_{m, l} \sin (2 m \beta),
  \quad \widehat{v}_l^X (\beta) = \sum_{m = 1}^{M + 1} \Beta^{v, X}_{m, l} \sin (2
  m \beta) . \quad \text{for $l \geqslant 2$.} \label{eq:v1 ansatz 2}
\end{equation*}
The coefficients must satisfy the algebraic relations:
\begin{equation}
\begin{aligned}
  \Beta^{v, X}_{m, 1} & = - \frac{1}{\mu_{2 l - 1}}  \left( \frac{m}{2}
  \Beta^{v, Y}_{m + 1, 1} - \frac{3}{2} \Beta^{v, Y}_{m, 1} - \frac{m - 1}{2}
  \Beta^{v, Y}_{m - 1, 1} \right),\\
  \Beta^{v, X}_{m, l} & = - \frac{1}{\mu_{2 l - 1}}  \left( \frac{2 m + 1}{4}
  \Beta^{v, Y}_{m + 1, l} - \frac{3}{2} \Beta^{v, Y}_{m, l} - \frac{2 m -
  1}{4} \Beta^{v, Y}_{m - 1, l} \right), \quad \forall l \in \mathbb{Z}^+,
\end{aligned}\label{eq:v ansatz 2}\end{equation}
with boundary conditions: $\Beta^{v, Y}_{0, l} = \Beta^{v, Y}_{M + 1, l} =
\Beta^{v, Y}_{M + 2, l} = 0$. These algebraic constraints ensure strict divergence-freeness of $\overline{U}$ and
$\overline{v}$. From \eqref{eq:div sph cv 2} and \eqref{eq:v ansatz 2}, we can construct a divergence-free basis from any trigonometric basis of $U_r$ and $U_\phi$. For example, we take the coefficients 
$\Beta^{U, Y}, \Beta^{U, Z}$
as independent variables, and define 
$\Beta^{U, X}$
 by the divergence-free constraint \eqref{eq:div sph cv 2}. This defines a linear reconstruction map from the space of free coefficients into the truncated vector-field space. Its image is precisely the truncated exactly divergence-free subspace. Therefore, by applying this reconstruction map to the standard coordinate basis in the free coefficient space, we obtain a basis of the truncated exactly divergence-free subspace.
 In the next subsection, we demonstrate how to interpolate our
solutions onto this basis.

\subsection{Divergence-free interpolation}\label{sec:res cg}

In the preceding subsections, we obtained numerical approximations for the
profile $\overline{U}$ and the eigenvector $\overline{v}$ using the finite element
method. To ensure strict adherence to the divergence-free condition, we now
interpolate these solutions onto the carefully structured divergence-free
basis developed in the previous subsection.

We remark that $\overline{U}_{\theta}$, $\overline{U}_{r}$, $\overline{U}_{\phi}$ corresponds to the $X$, $Y$ and $Z$ family coefficients in the harmonic expansions respectively, and we enforce the divergence-free condition in a mode-by-mode fashion, to be precise, via the algebraic relations \eqref{eq:div sph cv 2} and \eqref{eq:v ansatz 2}.
We use $\overline{U}_{r}$ as an example. We select a computational mesh
defined by $\beta_j = \frac{j \pi}{2 M}, \theta_k = \frac{k \pi}{2 L}$, and
evaluate $\overline{U}_{r}$ at each grid point $(\beta_j, \theta_k)$. We apply
discrete Fourier transforms in $\theta$ to obtain expansions:
\[ \overline{U}_{r} (\beta_j, \theta) = \sum_{l = 0}^L \mathcal{F}_{\theta}
   \overline{U}_{r}^l (\beta_j) \cos (2 l \theta), \quad j = 1, \ldots, M. \]
where $\mathcal{F}_{\theta}$ denotes the discrete Fourier transform with
respect to $\theta$. We then convert the resulting Fourier expansion in $\theta$ into the eigenfunction basis $Y_{2 l}$, which amounts to a finite change of basis in the angular variable
$\theta$. Therefore, we get the value of $\widehat{U}_l^Y$ at
each $\beta_j$. Applying another discrete Fourier transform in $\beta$, we
derive the coefficients $\Beta^{U, Y}_{m, l}$. Similarly we obtain $\Beta^{v, Y}_{m, l}$.
We explicitly enforce the divergence-free conditions by setting coefficients
$\Beta^{U, X}_{m, l}$ and $\Beta^{v, X}_{m, l}$ according to the derived
algebraic constraints \eqref{eq:div sph cv 2} and \eqref{eq:v ansatz 2}. In practice, we observe only small deviations from
the original values.

Typically, the initial interpolation yields residuals insufficiently small for
rigorous numerical analysis. To address this, we further refine our solutions.
Defining the residual:
\[ R_{\tmop{sph}}^U  (\overline{U}, \overline{P}) = - \frac{1}{2}  \overline{U} - \frac{1}{2}
   r \partial_r  \overline{U} + \overline{U} \cdot \nabla_{\tmop{sph}}  \overline{U} +
   \nabla_{\tmop{sph}}  \overline{P} - \Delta_{\tmop{sph}}  \overline{U} . \]
We want to find a correction $\delta \overline{U}, \delta \overline{P}$ to minimize
$\|R_{\tmop{sph}}^U (\overline{U} + \delta \overline{U}, \overline{P} + \delta \overline{P})\|$.
We approximate $\delta \overline{U}$ within the structured divergence-free basis as
follows:
\begin{equation}
  \delta \overline{U}_r (\beta, \theta) = \cos\beta\sum_{l = 1}^L \sum_{m = 1}^M \delta
  \Beta^{U, Y}_{ml} \sin ((2 m - 1) \beta) S_{2 l} (\theta) . \label{eq:delta
  U ansatz}
\end{equation}
For the other entries, we adopt similar ansatz expansions. Motivated by
Newton's method, we linearize $R_{\tmop{sph}}^U$ around $(\overline{U}, \overline{P})$,
yielding the following equation:
\begin{equation}
  - \frac{1}{2} \delta \overline{U} - \frac{1}{2} r \partial_r \delta \overline{U} +
  \overline{U} \cdot \nabla_{\tmop{sph}} \delta \overline{U} + \delta \overline{U} \cdot
  \nabla_{\tmop{sph}}  \overline{U} + \nabla_{\tmop{sph}} \delta \overline{P} -
  \Delta_{\tmop{sph}} \delta \overline{U} = - R_{\tmop{sph}}^U  (\overline{U}, \overline{P}),
  \label{eq:delta U eq}
\end{equation}
which cannot be solved precisely by the approximated $\delta \overline{U}$. Hence,
we solve the resulting optimization problem:
\begin{equation}
  \min_{\delta \overline{U}} \left| - \frac{1}{2} \delta \overline{U} - \frac{1}{2} r
  \partial_r \delta \overline{U} + \overline{U} \cdot \nabla_{\tmop{sph}} \delta \overline{U}
  + \delta \overline{U} \cdot \nabla_{\tmop{sph}}  \overline{U} + \nabla_{\tmop{sph}}
  \delta \overline{P} - \Delta_{\tmop{sph}} \delta \overline{U} + R_{\tmop{sph}}^U
  (\overline{U}, \overline{P}) \right| . \label{eq:min pb}
\end{equation}
The above optimization problem is essentially a least squares problem for the
coefficients of $\delta U$ and $\delta P$ and can be efficiently solved using
the conjugate gradient method without the need to construct the matrix
corresponding to the bilinear form. The updated profile $(\overline{U} + \delta
\overline{U}, \overline{P} + \delta \overline{P})$ exhibits a smaller residual. We iterate
this refinement process until the residual is sufficiently small for precise
numerical analysis. This iterative refinement method is similarly applied to
refine the eigenvector $\overline{v}$. Notice that by Assumption \ref{asm:num}, the
residue is orthogonal to $\overline{v}$ and $\lambda$ is uniquely determined given
$\overline{v}$, and we thus update $\overline{v}$ and $\overline{q}$ before updating
$\overline{\lambda}$.

\subsection{Summary}\label{sec:sum}

For clarity, we distinguish four stages of the numerical approximation.
We denote by $(\overline{U}_{n},\overline{P}_{n})$ the finite-element solution
computed on the bounded computational domain using $n$-steps of Newton iterations. After applying the
coordinate transformation described above, we obtain the whole-space
candidate $(\overline{U}^{\mathrm{tr}},\overline{P}^{\mathrm{tr}})$. Interpolating this field
into the truncated exactly divergence-free spectral representation
yields $(\overline{U}^{\mathrm{df}},\overline{P}^{\mathrm{df}})$. Finally, we use $(\overline{U}^{(k)},\overline{P}^{(k)})$ for the iteratively refined approximations produced by
the residual correction procedure at iteration $k$. The converged output of this residual correction procedure is denoted by $(\overline{U}^{\mathrm{cv}},\overline{P}^{\mathrm{cv}})$ and $(\overline{\lambda},\overline{v}^{\mathrm{cv}},\overline{q}^{\mathrm{cv}})$.

We summarize below the four-stage numerical pipeline, from the finite-element approximation on the bounded computational domain to the final corrected profile and associated eigenpair. We initialize the boundary condition by
$A_0(\theta)=(0,0,0.1\sin\theta)$, which satisfies
the prescribed symmetry and divergence-free constraints.
We choose $\gamma_\lambda=0.1$ as the target eigenvalue
threshold and gradually increase the boundary constraint
$\gamma_U$ during optimization, starting from a small
positive value. 

\begin{algorithm}[H]
  \SetAlgoLined
  \KwIn{Initial bound $\gamma_{\lambda}$ and boundary
    condition $A_0$}
  \KwOut{Profile $\overline{U}^{\text{cv}}$ and eigenpair $(\overline{\lambda}, \overline{v}^{\text{cv}})$}
    Construct a mesh on the computational domain $[0, R] \times \left[ 0,
      \frac{\pi}{2} \right]$;
      
    Assemble the weak form and obtain the corresponding matrices and
    tensors;
    
    Set initial bound $\gamma_U, \varepsilon$;
    
    Initialize iteration index $n = 0$ and eigenvalue $\lambda_0 = 2$;
    
    \While{$\lambda_n > - \gamma_{\lambda}$}{
        Update scaling factor $\kappa = \sqrt{1 + \gamma_U / 4}$ and adjust boundary condition $A_n$;
        
        \While{$n=0$ or $| \lambda_n - \lambda_{n - 1} | > \varepsilon$}{
            Solve equation \eqref{eq:weak NS sph} under boundary condition $A_n$ using Newton's method;

            Compute the smallest eigenvalue $\lambda_n$ of the linearized operator $-\mathcal{L} (\overline{U}_n)$;

            Evaluate the gradient $\frac{\delta \lambda}{\delta U}$ using \eqref{eq:gradient lambda};

            Solve the projected gradient problem to determine descent direction;

            Update solution and boundary condition: $\overline{U}_n, A_n \to \overline{U}_{n+1}^{\ast}, A_{n+1}$; \tcp{Note $\overline{U}_{n+1}^{\ast}$ is not yet on the solution manifold and will be corrected in the next iteration of the loop.} 

            Increment the iteration index $n \leftarrow n + 1$;
        }
        Increase $\gamma_U$;
    }
    
    Construct a new mesh on the extended domain $\left[ 0, \frac{\pi}{2} \right]^2$;

    Transform the computed solution $\overline{U}_n$ onto the new mesh;
      
    Solve equation \eqref{eq:weak NS sph cv} using Newton's method: $\overline{U}_n \to \overline{U}^{\tmop{tr}}$;
    
    Compute eigenpair $(\overline{\lambda}, \overline{v}^{\tmop{tr}})$;
    
    Interpolate the solutions onto a divergence-free basis: $\overline{U}^{\tmop{tr}}, \overline{v}^{\tmop{tr}} \to \overline{U}^{\tmop{df}}, \overline{v}^{\tmop{df}}$;
    
    Initialize iteration index $k = 0$, with initialization $\overline{U}^{(0)}, \overline{v}^{(0)}=\overline{U}^{\tmop{df}}, \overline{v}^{\tmop{df}}$;
    
    Set the bound for residual $\varepsilon_{\tmop{res}}$;
    
    \While{$\| R_{\tmop{sph}}^U (\overline{U}^{(k)},
      \overline{P}^{(k)}) \| > \varepsilon_{\tmop{res}}$ or $\|
      R_{\tmop{sph}}^v (\overline{v}^{(k)},
      \overline{q}^{(k)}) \| > \varepsilon_{\tmop{res}}$}{
        Solve \eqref{eq:min pb} to obtain $(\delta
      \overline{U}^{(k)}, \delta \overline{P}^{(k)}$ and an analogous problem for $(\delta \overline{v}^{(k)}, \delta
      \overline{q}^{(k)})$;

      Update $\overline{U}^{(k+1)}, \overline{P}^{(k+1)}, \overline{v}^{(k+1)}, \overline{q}^{(k+1)} \leftarrow \overline{U}^{(k)} + \delta \overline{U}^{(k)}, \overline{P}^{(k)} + \delta \overline{P}^{(k)}, \overline{v}^{(k)} + \delta \overline{v}^{(k)}, \overline{q}^{(k)} + \delta \overline{q}^{(k)}$;

      Update $\overline{\lambda}^{(k+1)}$ based on orthogonality by Assumption \ref{asm:num};

      Increment the iteration index $k \leftarrow k + 1$;
    }
  \Return{Profile $\overline{U}^{\text{cv}}=\overline{U}^{(k)}$ and eigenpair $(\overline{\lambda}, \overline{v}^{\text{cv}})=(\overline{\lambda}^{(k)}, \overline{v}^{(k)})$}\;
  \caption{Gradient descent for eigenvalue}
\end{algorithm}

We also present our numerical results in cylindrical coordinates. 
Figure \ref{fig:U figs} shows the three components of $\overline{U}$ together with the streamlines, while Figure \ref{fig:v figs} displays the corresponding components of $\overline{v}$. Compared with the results of Guillod and Sverak  \cite{guillod2023numerical}, our profile $\overline{U}$ exhibits relatively larger components in the $\rho$- and $z$-directions and resembles a hyperbolic flow. The azimuthal component of the eigenvector $\overline{v}$ is comparatively small, indicating that the bifurcation primarily occurs in the $\rho$-$z$ plane, where it disrupts the hyperbolic structure. Figure \ref{fig:Uv figs} illustrates $\overline{U} + \sigma 
\overline{v}$ for $\sigma = 10, 20, 30, 50$, which illustrates the effect of increasing the amplitude of the unstable eigenmode. As $\sigma$ increases, the parity symmetry with respect to $z$ is broken, and small swirls emerge near the $\rho$-axis. In figure \ref{fig:omega figs}, we display the vorticity $\omega_\phi$ of $\overline{U}, \overline{v}$ and $\overline{U} + 10 \overline{v}$. The results show that the vorticity is concentrated along the $\rho$-axis, where $z=0$.

\begin{figure}[htbp]
  \centering
  \begin{subfigure}{0.24\textwidth}
    \includegraphics[width=\linewidth]{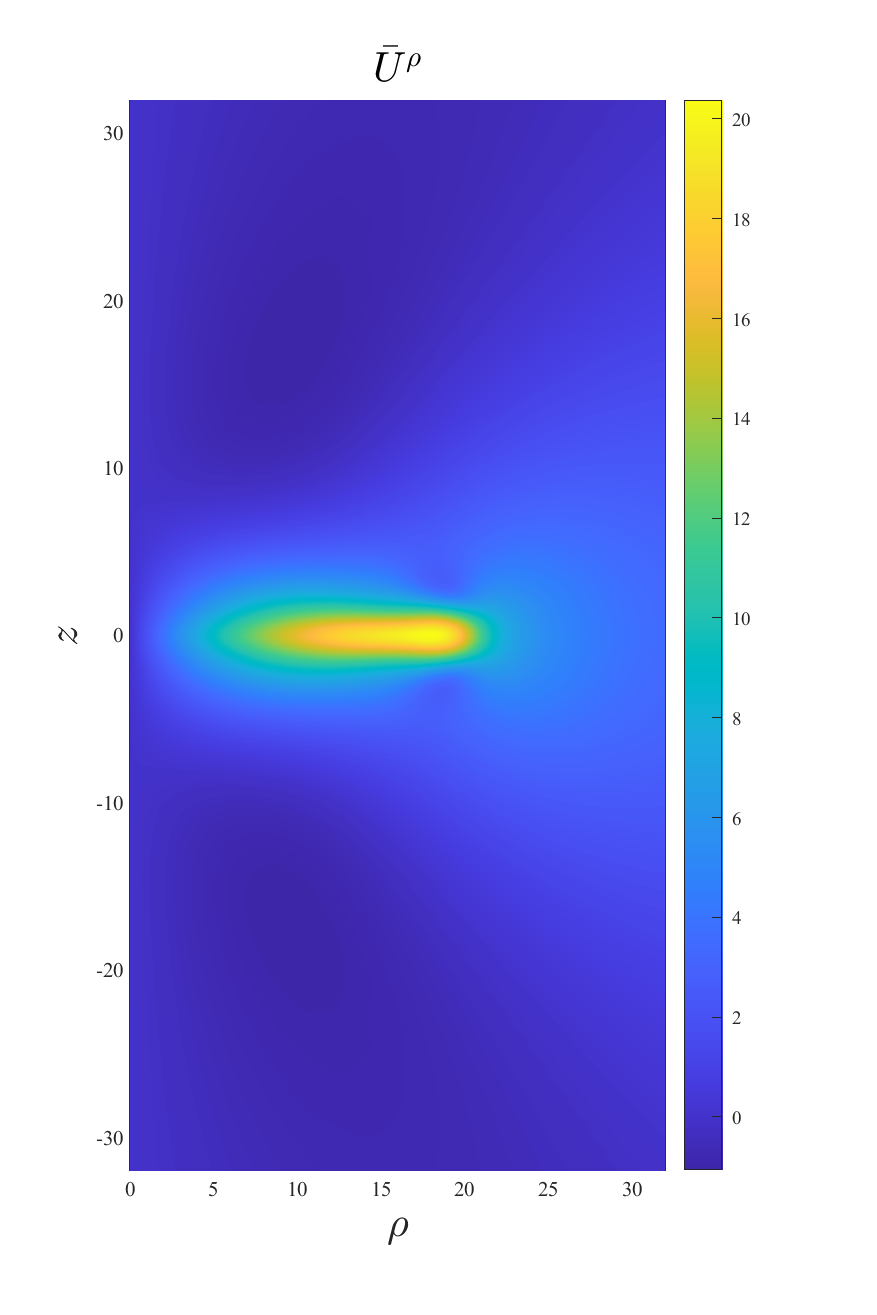}
    \caption{$\overline{U}^\rho$}
  \end{subfigure}
  \begin{subfigure}{0.24\textwidth}
    \includegraphics[width=\linewidth]{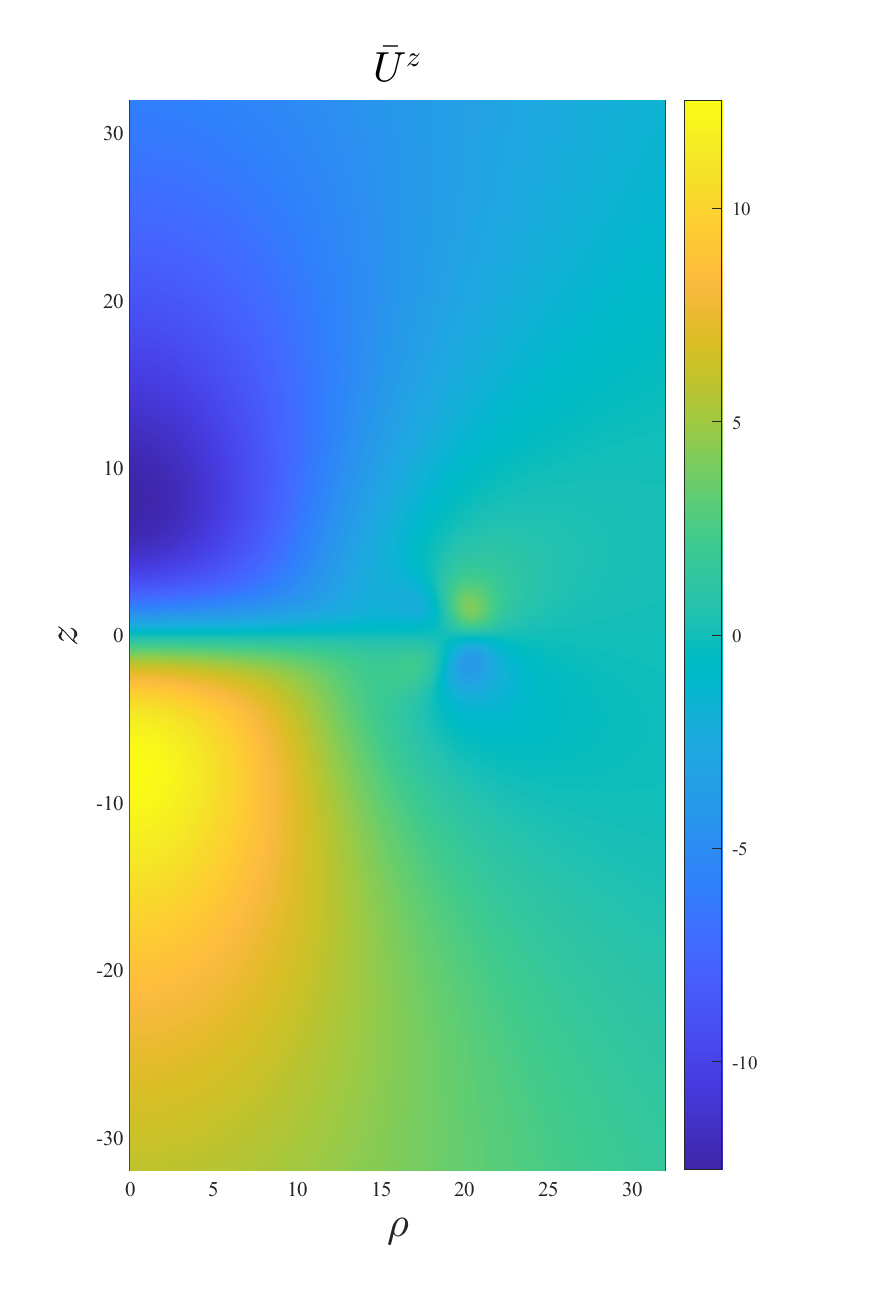}
    \caption{$\overline{U}^z$}
  \end{subfigure}
  \begin{subfigure}{0.24\textwidth}
    \includegraphics[width=\linewidth]{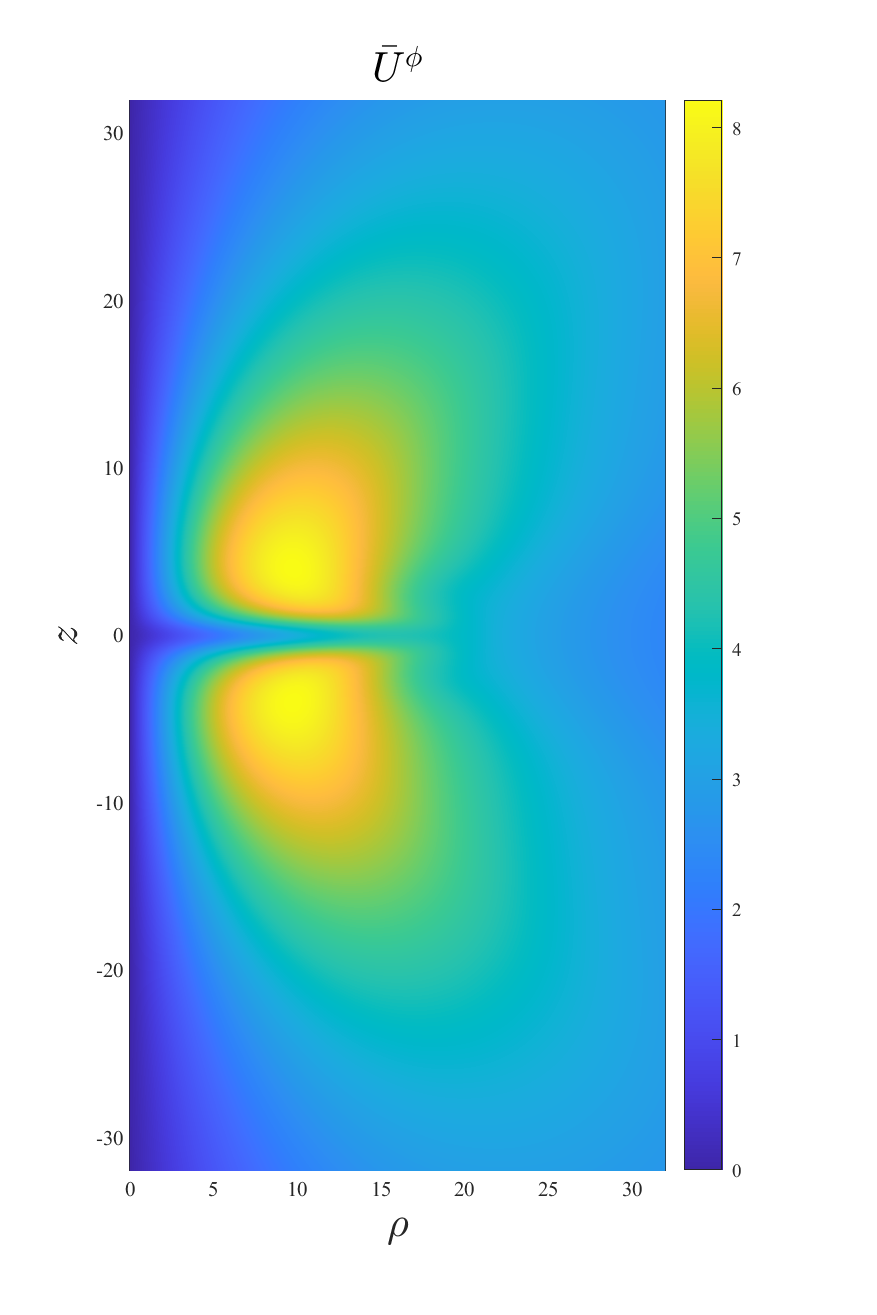}
    \caption{$\overline{U}^\phi$}
  \end{subfigure}
  \begin{subfigure}{0.24\textwidth}
    \includegraphics[width=\linewidth]{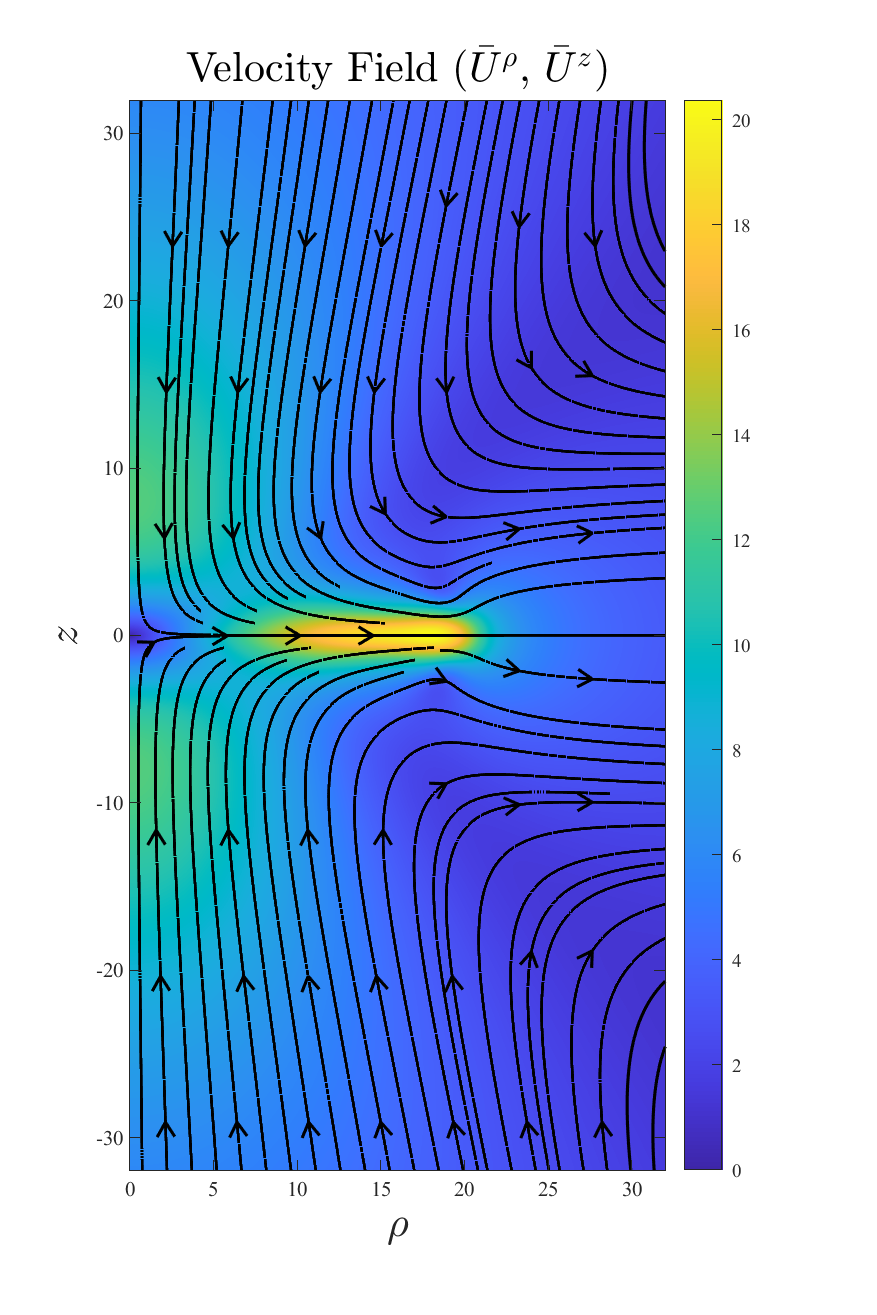}
    \caption{$\overline{U}$ on the $\rho-z$ plane}
  \end{subfigure}
  \caption{Figures of $\overline{U}$}
  \label{fig:U figs}
\end{figure}

\begin{figure}[htbp]
  \centering
  \begin{subfigure}{0.24\textwidth}
    \includegraphics[width=\linewidth]{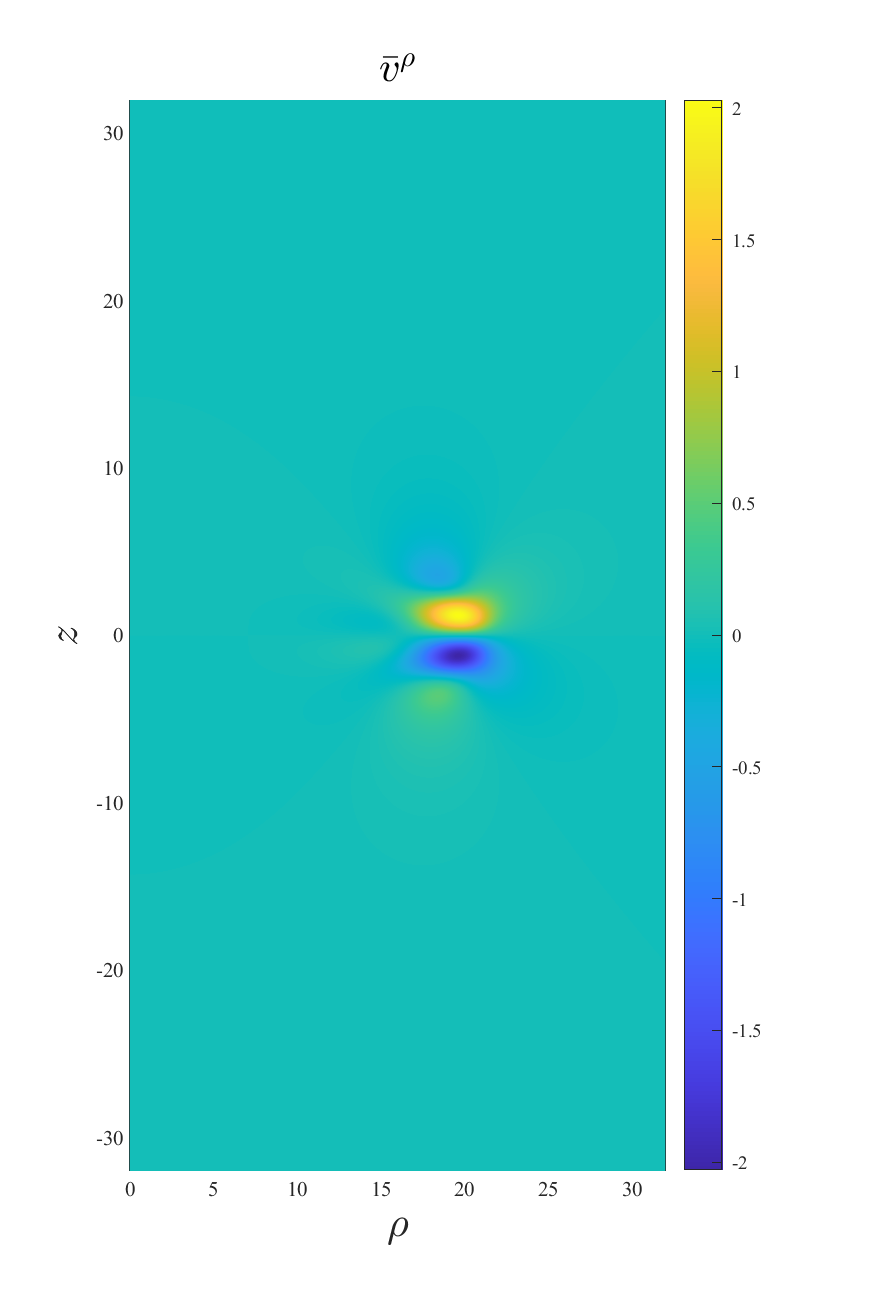}
    \caption{$\overline{v}^\rho$}
  \end{subfigure}
  \begin{subfigure}{0.24\textwidth}
    \includegraphics[width=\linewidth]{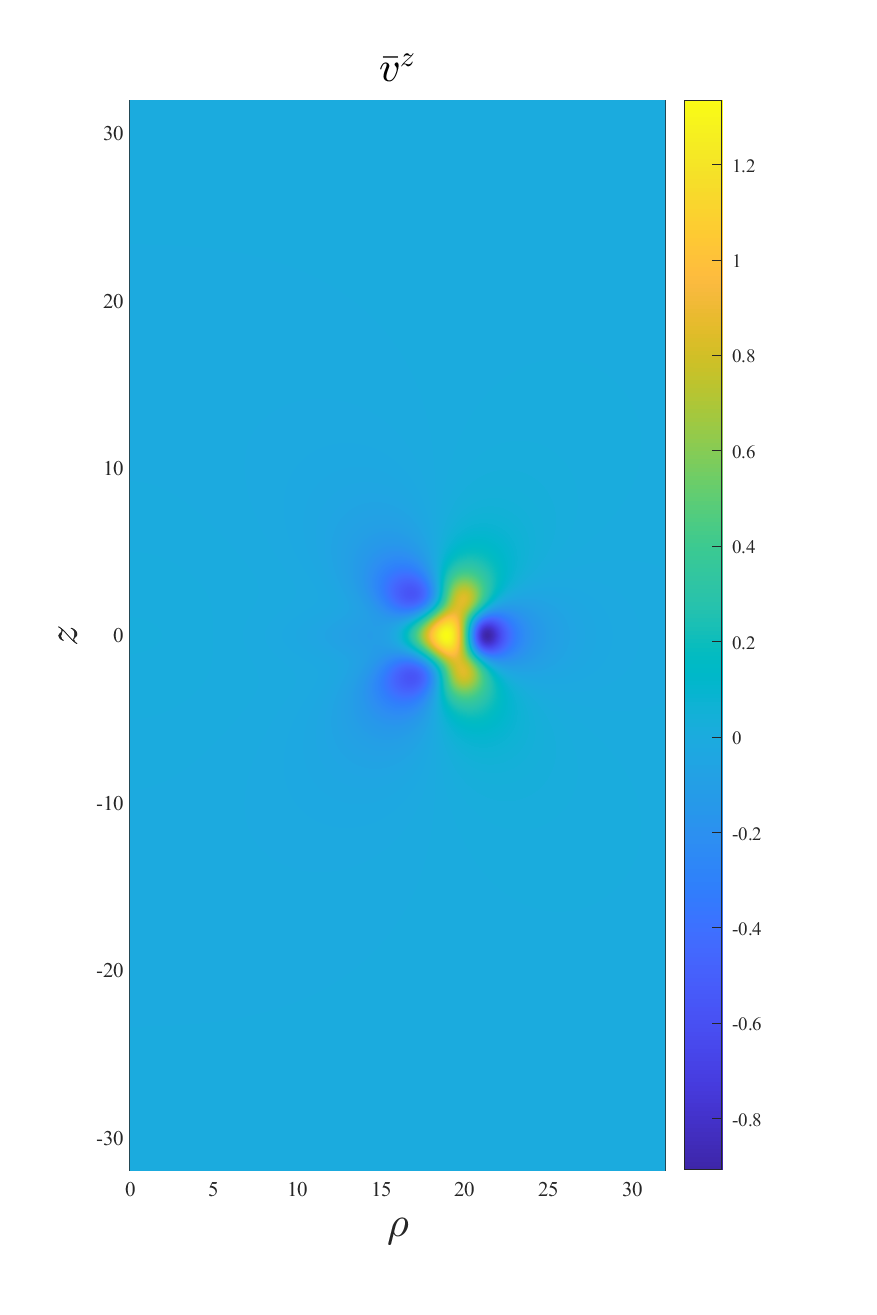}
    \caption{$\overline{v}^z$}
  \end{subfigure}
  \begin{subfigure}{0.24\textwidth}
    \includegraphics[width=\linewidth]{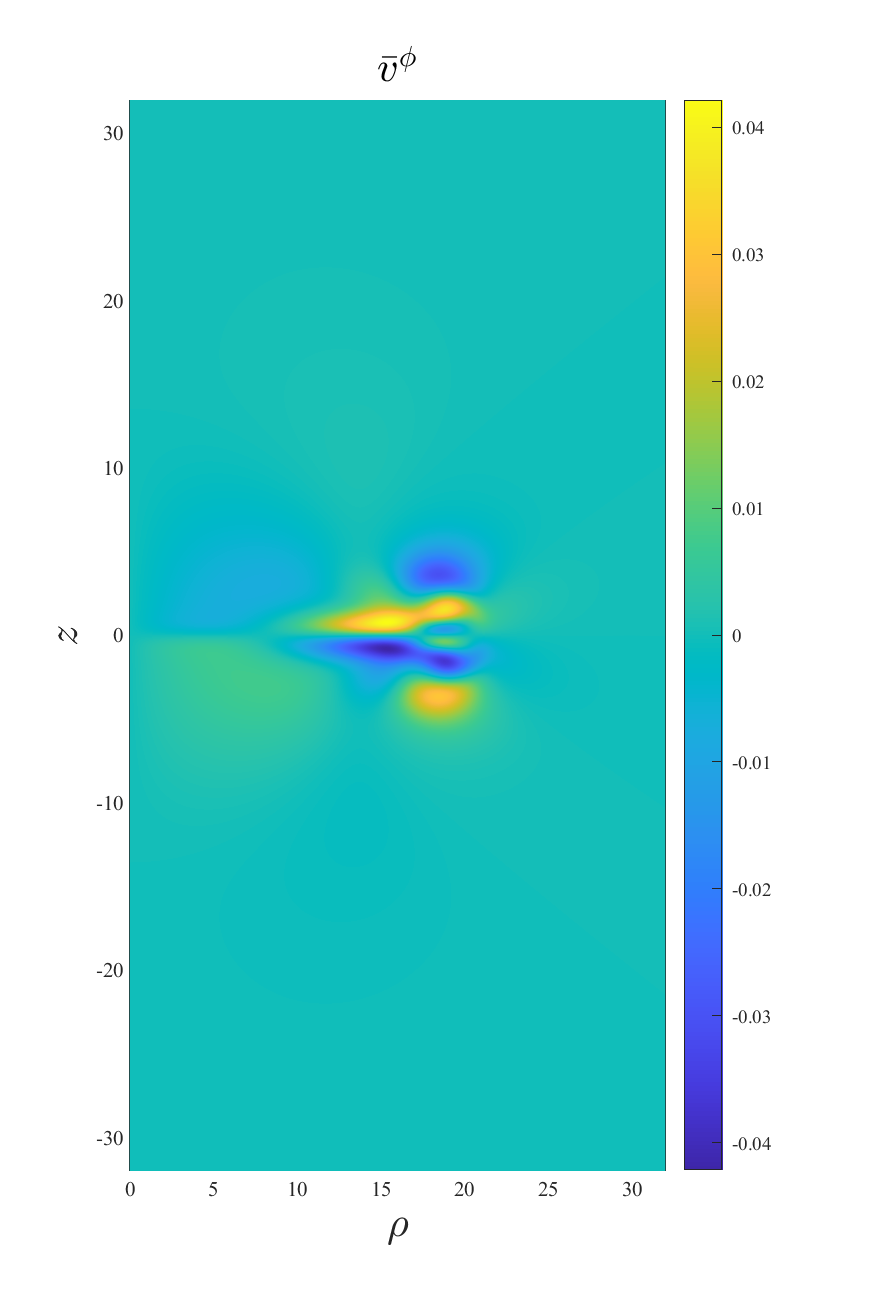}
    \caption{$\overline{v}^\phi$}
  \end{subfigure}
  \begin{subfigure}{0.24\textwidth}
    \includegraphics[width=\linewidth]{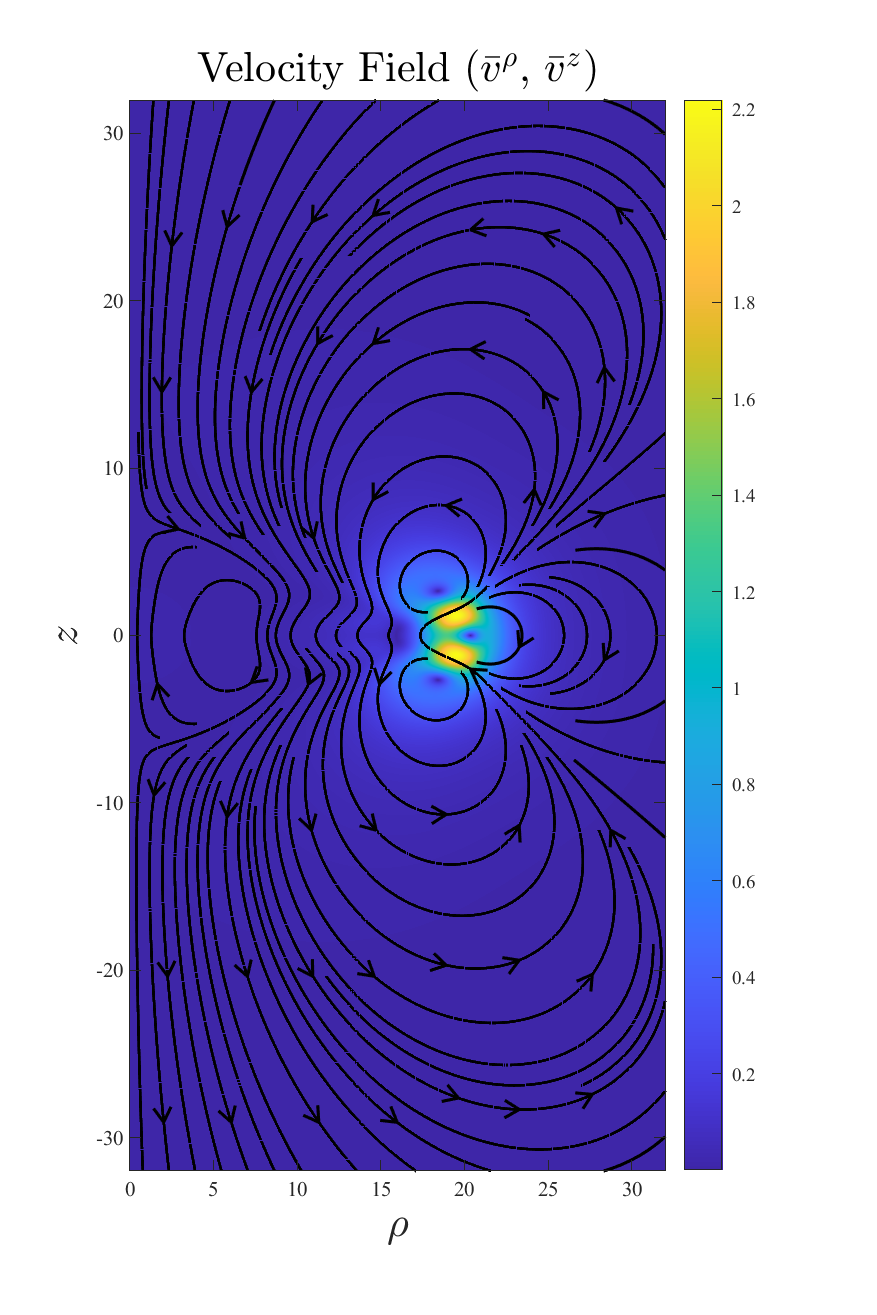}
    \caption{$\overline{v}$ on the $\rho-z$ plane}
  \end{subfigure}
  \caption{Figures of $\overline{v}$}
  \label{fig:v figs}
\end{figure}

\begin{figure}[htbp]
  \centering
  \begin{subfigure}{0.24\textwidth}
    \includegraphics[width=\linewidth]{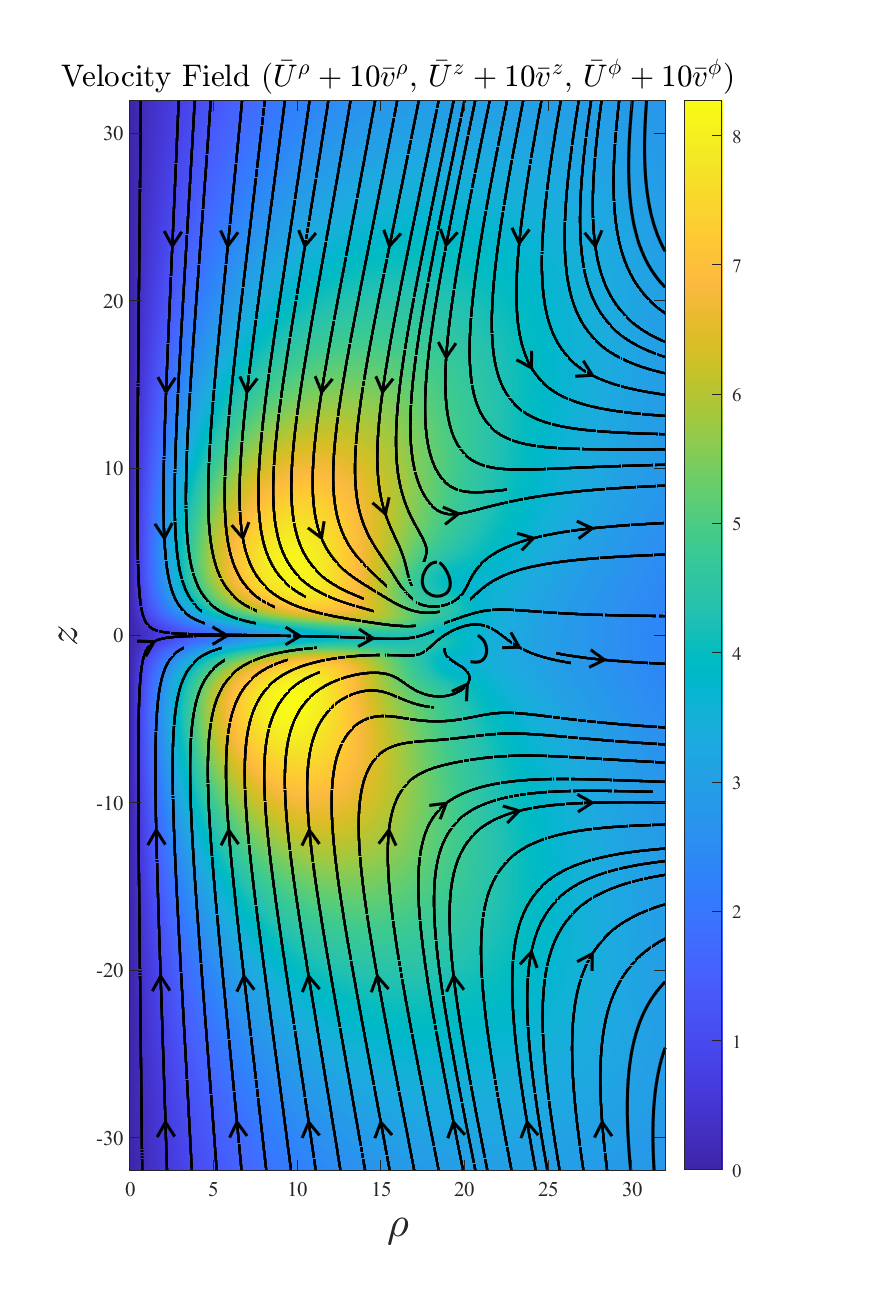}
    \caption{$\overline{U} + 10 \overline{v}$}
  \end{subfigure}
  \begin{subfigure}{0.24\textwidth}
    \includegraphics[width=\linewidth]{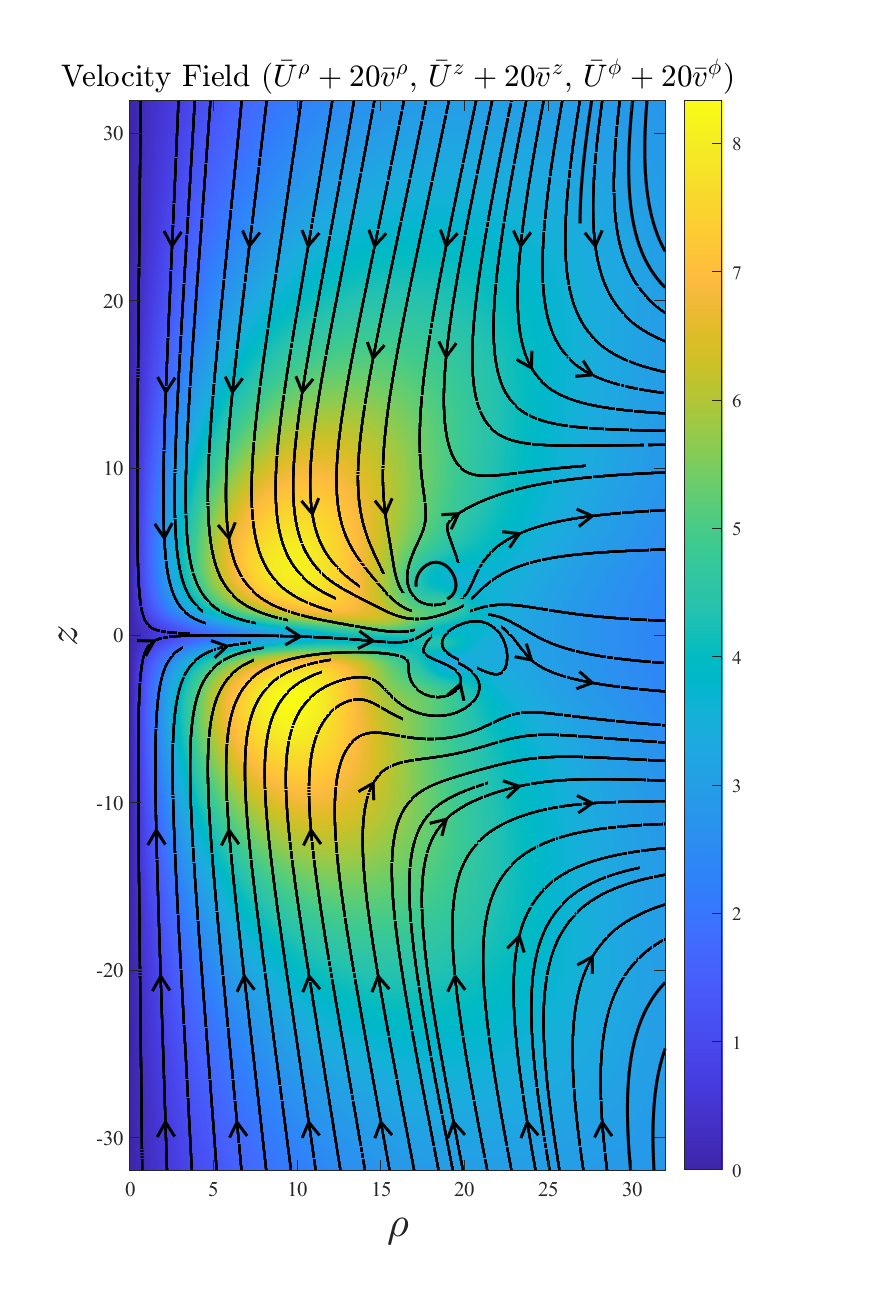}
    \caption{$\overline{U} + 20 \overline{v}$}
  \end{subfigure}
  \begin{subfigure}{0.24\textwidth}
    \includegraphics[width=\linewidth]{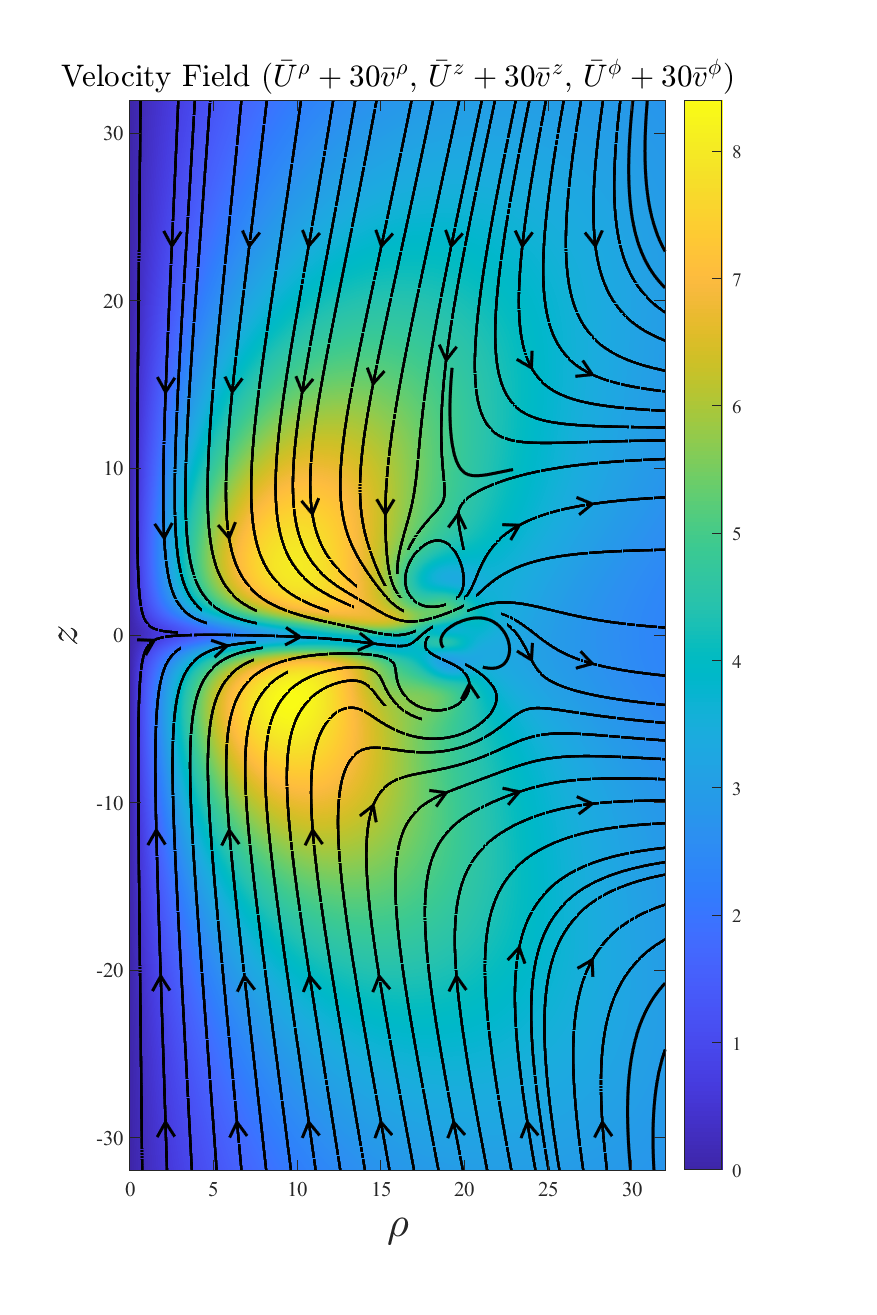}
    \caption{$\overline{U} + 30 \overline{v}$}
  \end{subfigure}
  \begin{subfigure}{0.24\textwidth}
    \includegraphics[width=\linewidth]{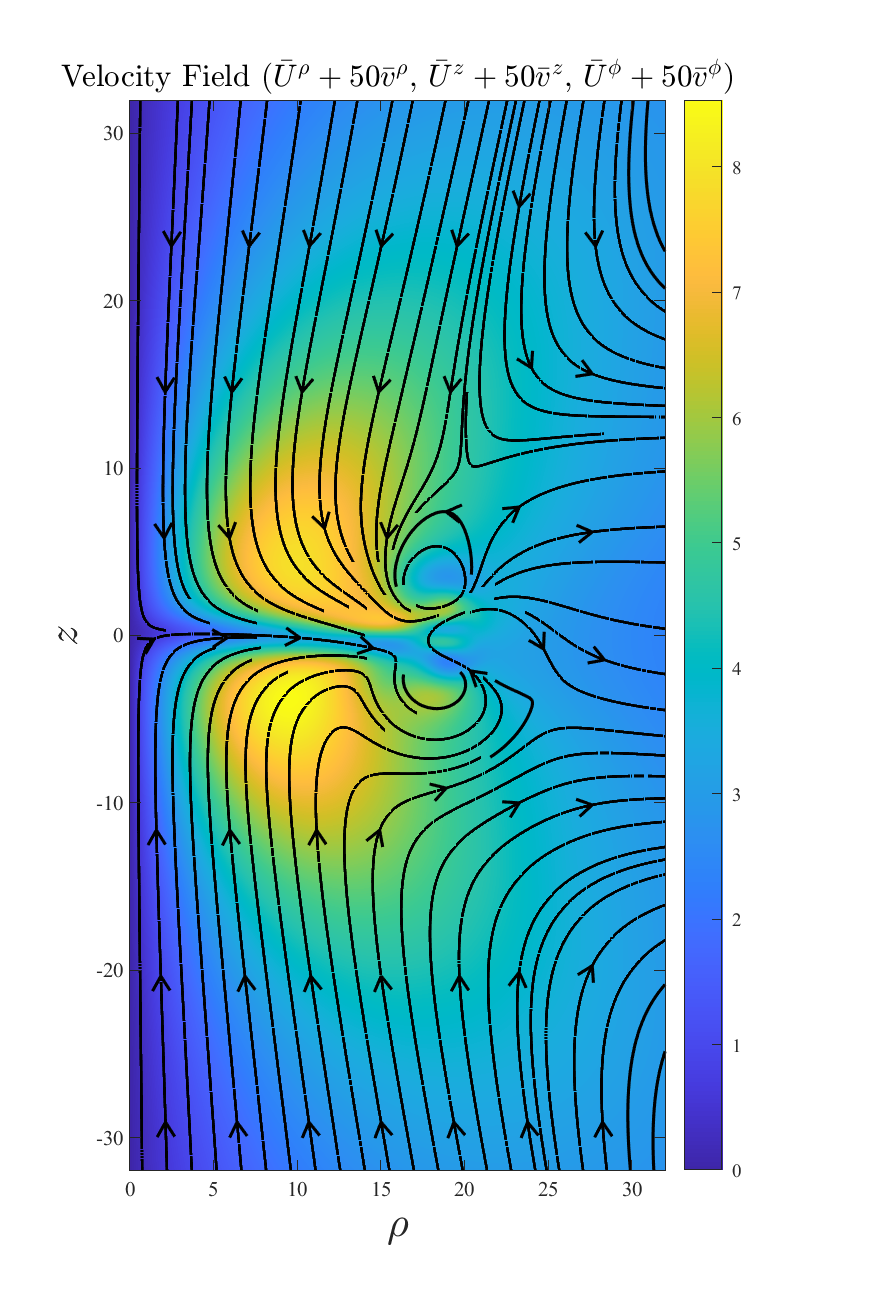}
    \caption{$\overline{U} + 50 \overline{v}$}
  \end{subfigure}
  \caption{Figures of $\overline{U} + \sigma \overline{v}$ with different $\sigma$}
  \label{fig:Uv figs}
\end{figure}

\begin{figure}[htbp]
  \centering
  \begin{subfigure}{0.24\textwidth}
    \includegraphics[width=\linewidth]{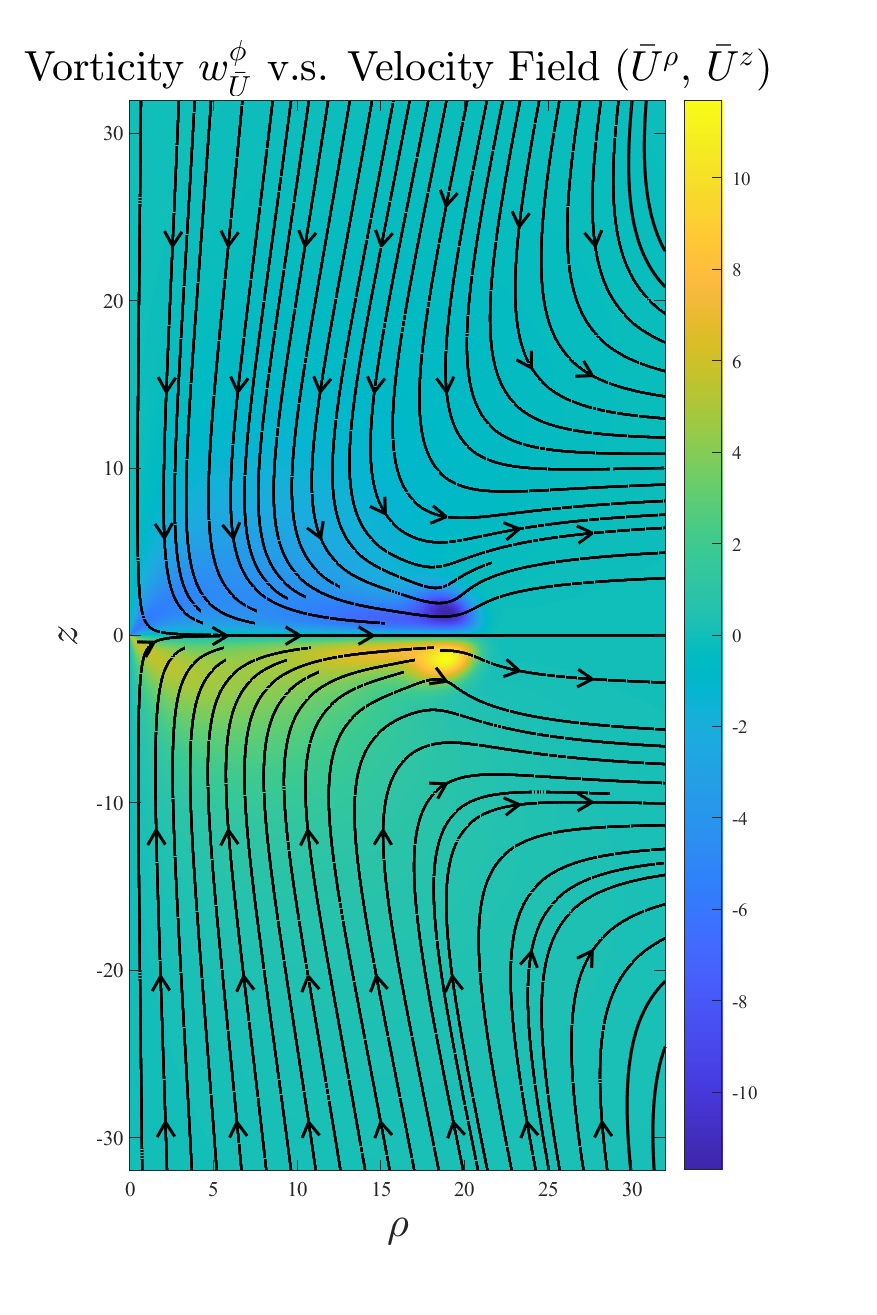}
    \caption{$\omega_\phi^U$}
  \end{subfigure}
  \begin{subfigure}{0.24\textwidth}
    \includegraphics[width=\linewidth]{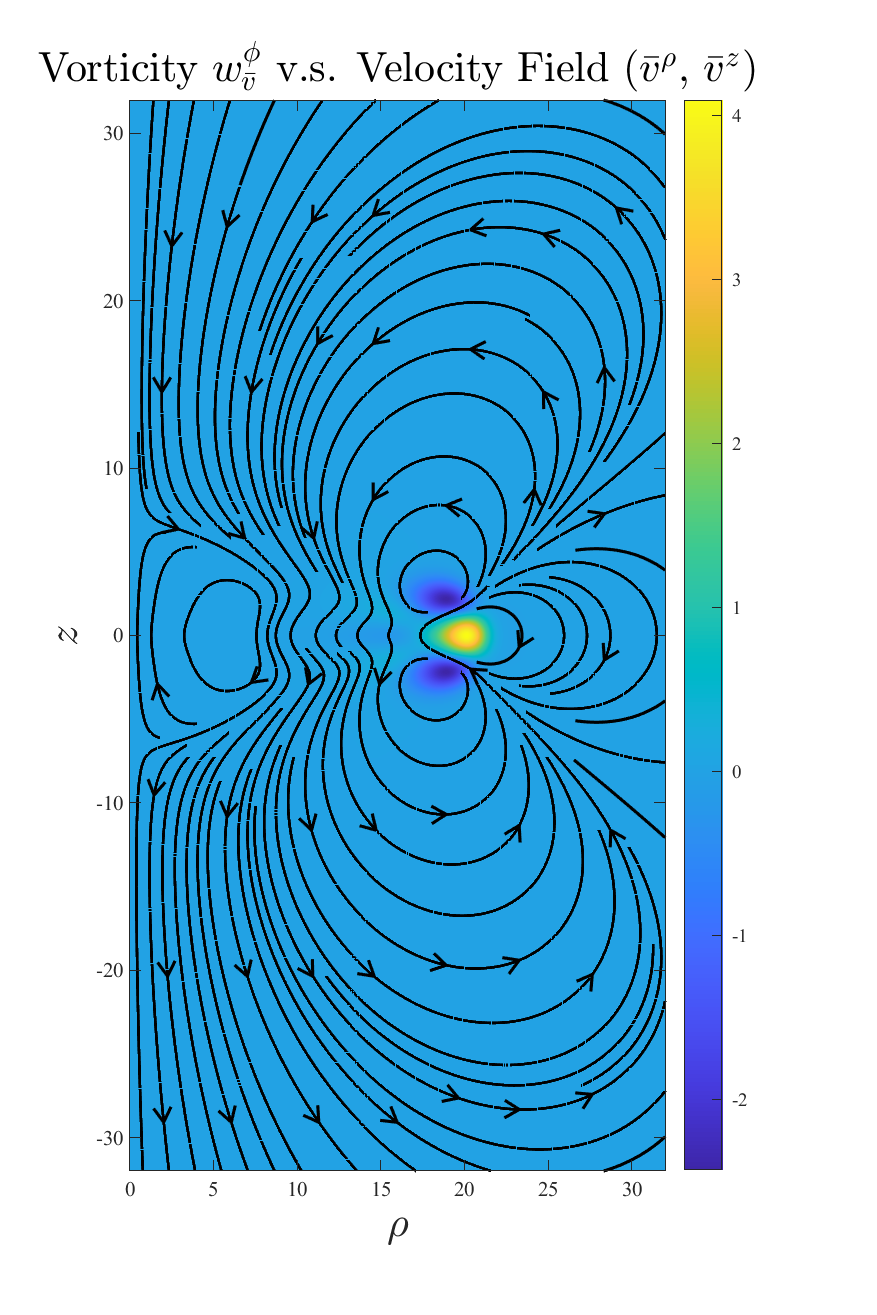}
    \caption{$\omega_\phi^v$}
  \end{subfigure}
  \begin{subfigure}{0.24\textwidth}
    \includegraphics[width=\linewidth]{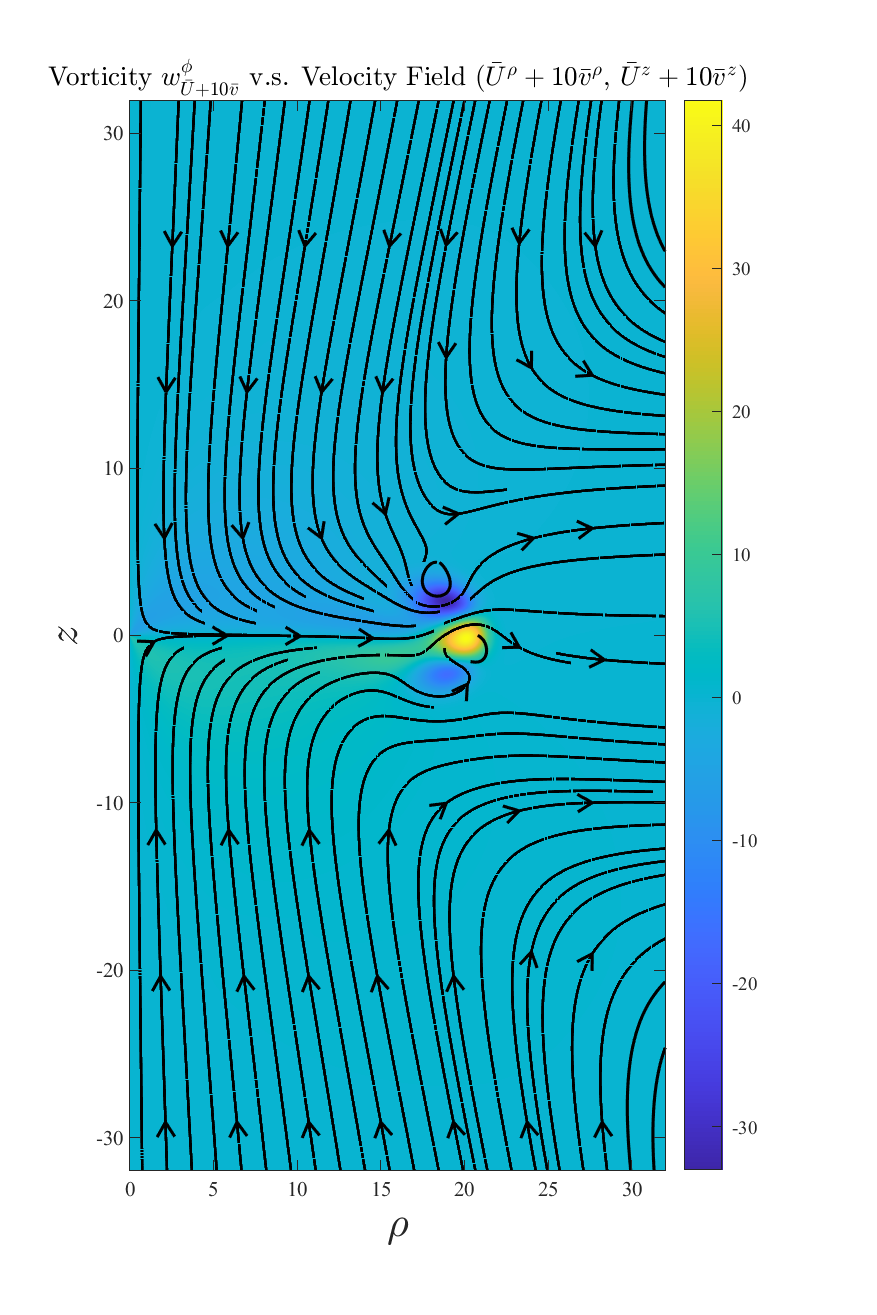}
    \caption{$\omega_\phi^U + 10 \omega_\phi^v$}
  \end{subfigure}
  \caption{Vorticity of $\overline{U}, \overline{v}$ and $\overline{U} + 10 \overline{v}$}
  \label{fig:omega figs}
\end{figure}

\section{Finite Rank Approximation}\label{Sec:numfr}

In this section, we construct the finite rank approximation for $L_2 = - \Pi
(U \cdot Q)$. Recall the decomposition of $\nabla \overline{U}$:
\[ Q_s = \frac{\nabla \overline{U} + \nabla^T \overline{U}}{2}, \quad Q_s +
   \eta_1 I = Q_{+} - Q_{-}, \quad Q_{+} \succeq 0, \quad Q_{-} \succeq 0. \]
Here, $Q_{-}=Q$ is the compactly supported positive semidefinite matrix constructed in Section~\ref{sec:def Q}, and $Q_{+}:=Q_s+\eta_1I+Q_{-}$. We use the computational domain $\Omega = [0, R_0] \times \left[ 0,
\frac{\pi}{2} \right]$ with $R_0 = 32$.

Taking the exact spectral negative part of $Q_s+\eta_1I$ may produce a nonsmooth matrix, which complicates the future error analysis of numerical
integrals. Fortunately, we only need $Q$ to satisfy
\begin{equation}
  Q_s + \eta_1 I + Q \succeq 0. \label{eq:Q cond}
\end{equation}
Therefore, we design a matrix function $Q(x)$ supported
in $\Omega$ that satisfies this condition throughout $\mathbb R^3$.
The global verification is given in Section \ref{subsec:Q-envelope}. The algorithm is described in Section \ref{sec:def Q}. Next,
we construct an appropriate basis to approximate functions defined on $\Omega$
in Section \ref{subSec:numfr} and solve an eigenvalue problem to build
$L_{\tmop{ap}}$ in Section \ref{sec:Q eig}. Finally, in Section \ref{sec:inv
L2}, we solve
\[ L \phi_i = Q \psi_i, \]
to define the operator $G$.

\subsection{Construction of $Q$}\label{sec:def Q}

We now describe the construction of the symmetric matrix function $Q$,
which is required to satisfy
\[
Q \succeq 0, \qquad Q_s+\eta_1 I+Q \succeq 0.
\]

We represent $Q$ by a finite trigonometric expansion on the ball of
radius $R_0=32$ and extend it by zero outside this ball. Since $Q$ is
symmetric, we only need to construct its six independent entries,
\[
Q_{11},\quad Q_{12},\quad Q_{13},\quad
Q_{22},\quad Q_{23},\quad Q_{33}.
\]

The choice of trigonometric basis is guided by the boundary conditions
and symmetry of the corresponding components. In the radial variable
$\rho$, we use
\[
\cos\left(\frac{(2m-1)\pi\rho}{2R_0}\right),
\qquad m=1,\ldots,M,
\]
which satisfies the Neumann condition at $\rho=0$ and the Dirichlet
condition at $\rho=R_0$. For the angular variable, we use the following
basis functions:
\[
\begin{array}{c|c}
\text{Component} & \text{Angular basis}\\ \hline
Q_{11},Q_{22},Q_{33} & \cos(2l\theta)\\
Q_{12} & \sin(2l\theta)\\
Q_{13} & \sin((2l-1)\theta)\\
Q_{23} & \cos((2l-1)\theta)
\end{array}
\]
Here $l=0,\ldots,L-1$ for the even cosine basis, while
$l=1,\ldots,L$ for the remaining bases.

These choices automatically enforce the desired Dirichlet or Neumann
conditions on $\theta=0$ and $\theta=\pi/2$. For example,
$Q_{12}=Q_{13}=0$ at $\theta=0$, while $Q_{12}=Q_{23}=0$ at
$\theta=\pi/2$. The remaining components satisfy the corresponding
angular Neumann conditions. We note that these componentwise conditions
do not, by themselves, enforce the additional smooth-axis relations
$Q_{23}(\rho,0)=0$ and $Q_{22}(\rho,0)=Q_{33}(\rho,0)$.

We fix $M=300$ and $L=150$. For example, the expansion of $Q_{11}$
takes the form
\[
Q_{11}(\rho,\theta)
=
\sum_{m=1}^{M}\sum_{l=0}^{L-1}
q_{11,ml}
\cos\left(\frac{(2m-1)\pi\rho}{2R_0}\right)
\cos(2l\theta).
\]
The other five entries are represented similarly using their
corresponding angular bases.

To construct $Q$, we choose a slightly smaller parameter $\eta_0=0.19$
and initialize $Q^{(0)}=0$. At each iteration, we compute the positive
semidefinite negative part
\[
D_i=\bigl(Q_s+\eta_0 I+Q^{(i)}\bigr)_-,
\]
and approximate it in the prescribed trigonometric basis using
Fej\'er smoothing in both radial and angular variables. We use
Fej\'er smoothing because its nonnegative convolution kernel helps
preserve matrix positivity, unlike direct Fourier truncation, which
may introduce negative eigenvalues. We then add this approximation
to $Q^{(i)}$ and repeat until
\[
\|\lambda_{\max}(D_i)\|_{L^\infty}
<
\eta_1-\eta_0.
\]
In our computation, five iterations suffice.

Finally, to improve the positivity of $Q$, we add a small diagonal
correction using the first radial cosine mode:
\[
Q
=
Q^{(i)}
+
2.2\times10^{-4}
\cos\left(\frac{\pi\rho}{2R_0}\right)I.
\]
This correction is positive in the interior and vanishes at
$\rho=R_0$, preserving the prescribed Dirichlet boundary condition.

The construction above produces a numerical candidate, while the
required matrix inequalities are verified independently in
Section~\ref{Sec:num prev}. In particular, both $Q\succeq0$ and
\eqref{eq:Q cond} are included among the numerical premises, and
the global lower eigenvalue bound in
Section~\ref{subsec:Q-envelope} certifies the latter condition with
$\eta_1=0.1909$.

\subsection{Construction of spectral basis}\label{subSec:numfr}

Since the $Q$-norm $\| \cdummy \|_Q$ is controlled by the standard $L^2$-norm,
we begin by considering the following eigenvalue problem in the Rayleigh quotient
form:
\begin{equation}
  \lambda_n^s = \sup_{\dim (V_n) = n} \inf_{u \neq 0, u \in V, u \bot V_n}
  \frac{\| \nabla u\|^2_{\Omega}}{\|u\|^2_{\Omega}}, \label{rayleigh}
\end{equation}
where $V$ is the space of admissible functions, namely divergence-free
axisymmetric functions. We obtain a sequence of $\lambda_n^s$ in ascending
order with associated eigenfunctions $\psi_n^s$. These eigenfunctions will
constitute our spectral basis.

The solutions to this eigenvalue problem admit explicit representations in
terms of Bessel functions and spherical harmonics, as established in the
following lemma.

\begin{lemma}
  \label{lem:lap eig}Consider the eigenvalue problem \eqref{rayleigh} with
  $\Omega$ being the unit ball in 3D. The first eigenpair is
  \[ \lambda_0^s = 0, \quad \psi_0^s = (\cos \theta, - \sin \theta, 0) . \]
  The positive eigenvalues can be written as
  \[ \lambda_n^s = (\alpha_{l, m}^i)^2, \quad l \geqslant 1, \quad m \geqslant
     1, \quad i = Y \infixor Z. \]
  {\tmstrong{Case 1: $i = Z$}}
  
  Here, $\alpha_{l, m}^Z > 0$ denotes the m-th zero of the equation for $r$:
  \begin{equation}
    rJ_{l + \frac{1}{2}}' (r) - \frac{1}{2} J_{l + \frac{1}{2}} (r) = 0,
    \label{eq:alphaZ}
  \end{equation}
  where $J_{\nu}$ is the Bessel function of the first kind. The corresponding
  eigenfunction is
  \[ \psi_n^s = \psi_{ml}^Z = \frac{J_{l + \frac{1}{2}}  (\alpha^Z_{ml}
     r)}{\sqrt{r}} Z_l (\theta) . \]
  {\tmstrong{Case 2: $i = Y$}}
  
  Here, $\alpha_{l, m}^Y > 0$ denotes the m-th zero of the equation for $r$:
  \begin{equation}
    (r^2 - l (l - 1) (l + 2)) r \tilde{j}_l' (r) + (r^2 - l (l - 1))  (r^2 -
    (l - 1) (l + 2))  \tilde{j}_l (r) = 0, \label{eq:alphaY}
  \end{equation}
  where
  \[ \tilde{j}_l (r) \assign \frac{J_{l + \frac{1}{2}} (r)}{r^{\frac{3}{2}}} .
  \]
  The corresponding eigenfunction is
  \[
\begin{aligned}
 \psi_n^s=\psi_{ml}^Y={}&\bigl((\alpha_{ml}^Y)^{3/2}B_{ml}\tilde j_l(\alpha_{ml}^Yr)+lA_{ml}r^{l-1}\bigr)Y_l(\theta)\\
 &+\left[\frac{(\alpha_{ml}^Y)^{3/2}B_{ml}}{\mu_l}\bigl(r\alpha_{ml}^Y\tilde j_l'(\alpha_{ml}^Yr)+2\tilde j_l(\alpha_{ml}^Yr)\bigr)+A_{ml}r^{l-1}\right]X_l(\theta),
\end{aligned}
\]

  with
  \[ A_{ml}=-\bigl((\alpha_{ml}^Y)^2-(l-1)(l+2)\bigr)\tilde j_l(\alpha_{ml}^Y),\qquad
 B_{ml}=\frac{(\alpha_{ml}^Y)^2-l(l-1)(l+2)}{(\alpha_{ml}^Y)^{3/2}}.
  \]
\end{lemma}

The proof of the lemma is given in Appendix \ref{eig cons}.

\begin{remark}[Scaling and Parity]
  In our setting, $\Omega = B_0 (R_0)$ with $R_0 = 32$. The eigenvalues and
  eigenfunctions can be obtained by rescaling with respect to $R_0$.
  Specifically, the eigenvalues are $\frac{\lambda^s_n}{R_0^2}$ and the
  eigenvectors are $\psi_n^s \left( \frac{r}{R_0}, \theta \right)$. Recall
  that, since $U$ and $v$ also satisfy different boundary conditions, in the
  vector spherical harmonics expansion as in Section \ref{sec:div-free}, only
  even $l$ or only odd $l$ appear. Hence, we retain only eigenpairs of certain
  parities in the construction. We denote by $\lambda_k^{s, U}$ the $k$-th
  smallest eigenvalue among all $\frac{\lambda^s_n}{R_0^2}$ with even
  eigenfunctions $\psi_k^{s, U} = \psi_n^s \left( \frac{r}{R_0}, \theta
  \right)$, and similarly define $\lambda_k^{s, v}$ and $\psi_k^{s, v}$.
\end{remark}

For the projection below, we normalize each $\psi_k^{s,i}$, $i=U,v$, in $L^2(\Omega)$, retaining the same notation.

Let $V_N^i$ denote the space spanned by the first $N$ eigenvectors $\psi_n^{s,
i}$ for $i = U, v$ respectively. We define the projection $I^i_N$ as
\begin{equation}
    I^i_N \psi = \sum_{k = 1}^N \langle \psi, \psi_k^{s, i} \rangle_{\Omega} \psi_k^{s,
   i} . \label{eq:I def}
\end{equation}
It is both the $L^2$ and $H^1$ projection onto $V_N^i$.
The ordering is within the symmetry sector of $\psi$; the zero mode is included first in the odd sector. The eigenfunction expansion gives
\[
 \|\psi-I_N^i\psi\|_{L^2(\Omega)}
 \le\frac{\|\psi\|_{H^1(\Omega),\eta_2}}
 {\sqrt{\eta_2+\lambda_{N+1}^{s,i}}},\qquad
 \|\psi-I_N^i\psi\|_{L^2(\Omega)}\le\|\psi\|_{L^2(\Omega)}.
\]
Recall the constant $K_5^i$ from \eqref{k5}, we can set
\[ K_5^i = \frac{\| \lambda_{\max} (Q)\|_{\infty}}{\eta_2 + \lambda_{N
   + 1}^{s, i}} . \]

Numerically, $\|\lambda_{\max}(Q)\|_\infty\le4.47$.
We retain all unscaled Bessel parameters $\alpha<100$.
With $R_0=32$, every omitted physical Laplace eigenvalue is at least
$(100/32)^2$, and therefore
\[
 K_5^i\le\frac{4.47}{0.005+(100/32)^2}.
\]
The resulting dimensions are $N={1257}$ for $U$ and $N={1260}$ for $v$.
The following lemma is used to certify that all roots below the
unscaled cutoff $\alpha=100$ have been identified.

\begin{lemma}
  \label{lem:zeros}Let $0 = \beta_{l, 0} < \beta_{l, 1} < \beta_{l, 2} <
  \cdots$ denote the zeros of $J_{l + \frac{1}{2}} (k)$ and $\gamma = \sqrt{l
  (l - 1)  (l + 2)}$.
  \begin{itemizedot}
    \item If $l = 1$, define $\tilde{\beta}_{1, j} = \beta_{1, j}$;
    
    \item If $l \geqslant 2$ and $\gamma \neq \beta_{l, j}$ for any $j > 0$,
    denote $\tilde{\beta}_{l, j}$ as the sequence obtained by sorting
    $\{\beta_{l, j}, \gamma\}$ in an ascending order.
  \end{itemizedot}
  Let $s_l$ be the unique positive root of:
  \[ Q_l (s) = s^3 - 4 l^2 s^2 + l (l - 1)  (6 l^2 + 11 l + 7) s - 3 l (l -
     1)^2  (l + 1)^2  (l + 2), \]
  and $t_l = \sqrt{l (l + 1)}$. Assume that $\sqrt{s_l}$ is not a root of
  \eqref{eq:alphaY} and that $t_l$ is not a root of \eqref{eq:alphaZ}.
  
  Then the following statements hold:
  \begin{enumeratenumeric}
    \item The equation \eqref{eq:alphaY} has exactly one root on every
    interval $(\tilde{\beta}_{l, j}, \tilde{\beta}_{l, j + 1})$ for $j
    \geqslant 0$.
    
    \item The equation \eqref{eq:alphaZ} has exactly one root on every
    interval $(\beta_{l, j}, \beta_{l, j + 1})$ for $j \geqslant 0$.
    
    \item The first zeros $\alpha_{l, 1}^Y$ and $\alpha_{l, 1}^Z$ satisfy the
    lower bounds:
    \[ \alpha_{l, 1}^Y \geqslant \sqrt{s_l}, \quad \alpha_{l, 1}^Z \geqslant
       t_l . \]
  \end{enumeratenumeric}
\end{lemma}

The proof of the lemma is given in Appendix \ref{eig cons}. In practice, we
use the lower bounds of $\alpha_{l, 1}^Y$ and $\alpha_{l, 1}^Z$ to determine
the largest index $l$ required, denoted by $l_{\max}$. For every $1 \leqslant
l \leqslant l_{\max}$, we then search all the intervals $(\beta_{l, j},
\beta_{l, j + 1})$ and $(\widetilde{\beta}_{l, j}, \widetilde{\beta}_{l, j + 1})$ for
$\alpha_{l, 1}^Z$ and $\alpha_{l, 1}^Y$, respectively, until the left endpoint
of the interval exceeds the prescribed threshold.

\subsection{Eigenvalue problem}\label{sec:Q eig}

We subsequently solve the finite-dimensional eigenvalue problem numerically:
\[ \langle \psi, Q \psi_n^{Q, i} \rangle_{\Omega} = \lambda^{Q, i}_n  \langle
   \psi, \psi_n^{Q, i} \rangle_{H^1(\Omega),\eta_2}, \quad \forall \psi \in V_N^i , \]
on the space $V_N^i$. Recall that $\psi_k^{s, i}$ are orthogonal in $L^2(\Omega)$ and
$H^1(\Omega)$. We normalize them in $H^1(\Omega),\eta_2$ to obtain an orthonormal basis as
$\widetilde{\psi}_k^{s, i}$ and compute the matrix:
\begin{equation}
  (A_Q^i)_{k, j} = \langle \widetilde{\psi}_k^{s, i}, Q \widetilde{\psi}_j^{s, i}
  \rangle_{\Omega}, \label{def:aq}
\end{equation}
using numerical integration, and solve the resulting eigenvalue problem:
\[ A_Q^i \psi^{Q, i}_n = \lambda^{Q, i}_n \psi^{Q, i}_n . \]
Here, by an abuse of notation, we use $\psi^{Q, i}_n$ to denote both the
eigenvector and its coordinates in the basis $\widetilde{\psi}_k^{s, i}$.

We solve the matrix eigenvalue problem numerically to determine all
eigenvalues $\lambda^{Q, i}_n \geqslant 0.45$. Numerical computations yield:
\[ \lambda^U_{19} \approx 0.449, \quad \lambda^v_{26} \approx 0.444. \]
These numerics determine the rank used in the finite-rank operator. The
computed eigenvectors provide the foundation for a finite-rank approximation
of $L_2$, which corresponds to the second step in Section
\ref{subsec:finiterank}.

\subsection{Inversion on the subspace}\label{sec:inv L2}

Finally, we invert the operator $L$ on the subspace spanned by the finitely many of eigenfunctions, as in the last step in
Section \ref{subsec:finiterank}. We numerically solve equation \eqref{eq:phi
eq} by employing the finite element method described in Section \ref{sec:ext},
using the $(\beta, \theta)$ coordinates. For $\phi_n^v$, we use the background operator $L_0$ defined in Section~\ref{Sec:prop 4}. The correction $D_U$ is estimated separately. Additionally, we must
impose the orthogonality constraint:
\[ \langle \phi_n^v, \overline{v} \rangle = 0, \]
which is numerically enforced by solving the following modified linear system:
\[ \left( \begin{array}{cc}
     A & y\\
     y^T & 0
   \end{array} \right) \left( \begin{array}{c}
     x\\
     \mu
   \end{array} \right) = \left( \begin{array}{c}
     b\\
     0
   \end{array} \right), \]
where the matrix $A$ stands for the discretization of the background operator $L_0$, and
the vector $x, y, b$ stands for $\phi^v_n, \overline{v}, Q \overline{\psi}_n^v$
respectively. $\mu$ is a scalar Lagrange multiplier ensuring the orthogonality
constraint. 

The inverse velocities use the opposite sign to the stored fields; their pressures and constraint multipliers retain their stored signs.
The numerical solutions $\Phi$ are then interpolated onto the explicitly
constructed divergence-free basis presented in Section \ref{sec:div-free}. To maintain accuracy, we subsequently refine our solutions using iterative
correction techniques detailed in Section \ref{sec:res cg}, until sufficiently
small residuals are obtained. Specifically, we iteratively minimize the
residual:
\[ \left| - \frac{1}{2} \phi_n^U - \frac{1}{2} r \partial_r \phi_n^U + \overline{U}
   \cdot \nabla_{\tmop{sph}} \phi_n^U + \phi_n^U \cdot \nabla_{\tmop{sph}} 
   \overline{U} + \nabla_{\tmop{sph}} P_n - \Delta_{\tmop{sph}} \phi_n^U - Q
   \psi_n^U \right|, \]
for $U$ and
\[ \left| - \left( \frac{1}{2} + \overline{\lambda} \right) \phi_n^v - \frac{1}{2}
   r \partial_r \phi_n^v + \overline{U} \cdot \nabla_{\tmop{sph}} \phi_n^v +
   \phi_n^v \cdot \nabla_{\tmop{sph}}  \overline{U} + \nabla_{\tmop{sph}} q_n -
   \Delta_{\tmop{sph}} \phi_n^v - Q \psi_n^v +\mu\overline v \right|, \]
where $\mu$ should be again understood as the scalar Lagrange multiplier
ensuring the orthogonality constraint. Finally, we obtain exactly
divergence-free functions $\overline{\phi}^U_n$, $\overline{\phi}^v_n$, and we have to
modify $\overline{\phi}^v_n$ by a projection to ensure orthogonality to $\overline{v}$.
\begin{remark}
    Since we require $\phi^U_n$ belongs to $L^2$, the boundary condition at $\beta=\frac{\pi}{2}$ becomes
\[
\phi^U_n(\frac{\pi}{2}, \theta) = 0.
\]
When solving for $\phi^U_n$ with a finite element method, we must impose the homogeneous Dirichlet boundary condition, which differs from the condition for $\overline{U}$. Moreover, in the expansion \eqref{eq:U1 ansatz 1}, the basis functions $\sin((2 m-1)\beta)$ must be replaced by $\sin(2 m \beta)$. Consequently, the coefficients $\Beta_{m,l}$-s are required to satisfy a modified condition.
\end{remark}

\section{On Rigorous Numerical Verification}\label{sec7}

In this section, we outline the framework we use for rigorous computer-assisted validation; in particular, the rigorous justification for the constants in Section \ref{Sec:num prev}. 
We release the code and reproducibility scripts in a public GitHub repository. 
Throughout, we work in the spherical coordinates and employ the change of variables
\[
r=\tan\beta,\qquad (\beta,\theta)\in[0,\tfrac{\pi}{2})\times[0,\tfrac{\pi}{2}],
\]
so that all fields are expanded on trigonometric basis functions in $(\beta,\theta)$.
Here $r$ is the rescaled radius. In the supplied verification data,
the physical radius is $\rho=\kappa r=\kappa\tan\beta$, with $\kappa=22$.
On the finite ball we use $s=\rho/R_0$, $R_0=32$. For axisymmetric fields in the same prescribed
parity sector, set $w(\beta)=\sin\beta/\cos^2\beta$. Physical inner products are
$4\pi\kappa^3\int_0^{\pi/2}\int_0^{\pi/2}
(wF)\cdot(wH)\sin\theta\,d\theta\,d\beta$.
Equivalently, on the ball they are
$4\pi R_0^3\int_0^1\int_0^{\pi/2}F\cdot H\,s^2\sin\theta\,d\theta\,ds$.
All derivatives and norms in Section \ref{Sec:num prev} are physical ones.
The items below detail how we bound the residuals, $L^\infty$-norms, eigenvalue envelopes, and quadrature errors used in the verification of Section~\ref{Sec:num prev}. The finite-dimensional matrix inequalities are verified separately using interval arithmetic, including certified $LDL^T$ positivity checks.

The high-level idea and organization of this section are as follows: we first discuss estimates of numerical integration and, in particular, $L^2$-residues in Section \ref{subsec:analytic-L2} and \ref{subsec:Quadrature}. We leverage the trigonometric representation in the profiles and use exact analytic integration if possible in Section \ref{subsec:analytic-L2}, and discuss how we perform numerical quadrature with certified error bounds using $L^\infty$ bounds on higher-order derivatives in Section \ref{subsec:Quadrature}. Then we detail in the remaining part of this section on how to build $L^\infty$ estimates. We discuss the classical linear interpolation bound in 2D using higher-order derivatives in Section \ref{subsec:Linf-envelope} and discuss in particular the use of the Bernstein inequality for trigonometric functions on bounds for these higher-order derivatives. Concerning the construction and the finite-rank approximation of $Q$ and its associated spectral bound estimate, we need to use a coordinate transform on top of the linear interpolation bound in Section \ref{subsec:Q-envelope}. Finally, for Bessel functions associated with the eigenpairs constructed from the finite-rank approximation, we derive $L^\infty$ estimates based on the known spherical Bessel integral representation and  plane-wave expansions via Legendre polynomials in Section \ref{sec:bes ln}.

\subsection{Analytic $L^2$ residuals via a trigonometric representation}\label{subsec:analytic-L2}
We expand the approximate self-similar profile $\overline U$ and the eigenfunction $\overline v$ as finite trigonometric sums
\[
\overline U(\beta,\theta)
= \sum_{p\in\mathcal P_\beta}\ \sum_{q\in\mathcal P_\theta}\Big(
\mathbf{A}^{ss}_{pq}\,\sin(p\beta)\sin(q\theta)
+\mathbf{A}^{sc}_{pq}\,\sin(p\beta)\cos(q\theta)
+\mathbf{A}^{cs}_{pq}\,\cos(p\beta)\sin(q\theta)
+\mathbf{A}^{cc}_{pq}\,\cos(p\beta)\cos(q\theta)
\Big),
\]
where each coefficient \(\mathbf{A}^{\star}_{pq}\in\mathbb R^3\) collects the three components of \(\overline U\), and \(\mathcal P_\beta,\mathcal P_\theta\subset\mathbb N\) are finite index sets.
Similarly, we have an expansion for $\overline v$. Because sums, products, and derivatives of sines/cosines remain finite trigonometric sums, every term in the residuals $E^U$ and $E^v$ stays in the same class. The only division introduced arises from the Jacobian in the change of variables
$r=\tan\beta$, which produces factors of the form either $\sin(n\beta)/\sin(\beta)$ or $\cos((2m+1)\beta)/\cos(\beta)$. Note that we never have $\cos((2m)\beta)/\cos(\beta)$ in our equations: such terms will be unbounded for $\beta \approx \pi/2$.
For $\sin(n\beta)/\sin(\beta)$, we know that this ratio admits an exact trigonometric expansion, so no true division remains.
Indeed, recalling that
\[
\frac{\sin(n\beta)}{\sin(\beta)} = U_{n-1}(\cos\beta),
\]
where $U_{n-1}$ denotes the Chebyshev polynomial of the second kind, we obtain
the explicit cosine series
\[
\frac{\sin(n\beta)}{\sin(\beta)} =
\begin{cases}
\displaystyle 1 + 2\sum_{k=1}^{m} \cos(2k\beta), & n=2m+1, \\[6pt]
\displaystyle 2\sum_{k=1}^{m} \cos\big((2k-1)\beta\big), & n=2m.
\end{cases}
\]
For $\frac{\cos\big((2m+1)\beta\big)}{\cos\beta}$,
we have for any integer $m\geqslant 0$ and $\beta\not\equiv \tfrac{\pi}{2}\ (\mathrm{mod}\ \pi)$,
\[
\frac{\cos\big((2m+1)\beta\big)}{\cos\beta}
= (-1)^{m}\!\left( 1 + 2\sum_{k=1}^{m} (-1)^{k}\cos(2k\,\beta)\right).
\]
In particular, the ratio is a finite cosine polynomial in $\beta$.
Hence, every such Jacobian factor can be represented exactly as a finite trigonometric sum, and we can evaluate certain integrals analytically. 
 We remark that
while in principle, inner products for
$
\|E^U\|_{L^2}, \|E^v\|_{L^2}$
could be evaluated {analytically}, we choose to use numerical integration for the sake of simplicity; see the next subsection.

Finally we note that we need to keep track of the max
frequency. Only products of two trigonometric polynomial increase the max
frequency. Therefore, the term with the largest frequency is the transport term
$\overline{U} \cdummy \nabla \overline{U}$ with max frequency smaller than $2
M + 10$, where we include a constant wavenumber to account for Jacobians.  Here, $M$ stands for the largest frequency of $\overline{U}$.
The benefit is that we can easily use the Bernstein inequality to estimate the
$L^{\infty}$ bound of any derivative of a trigonometric polynomial, which
brings huge convenience for the error estimate of numerical integrals in the
following sections.


\subsection{Quadrature for numerical integrals and its error}\label{subsec:Quadrature}

Now we deal with numerical integrals. In contrast to Section \ref{subsec:analytic-L2},  due to a multiplication of $Q$, the integrals related to the finite rank approximation cannot be evaluated analytically, and we use numerical integrations. For integrands involving $Q$, quadrature is performed separately on its support, so that no quadrature cell crosses the support boundary. For a function $f$, we use the composite Newton–Cotes 7-point rule in each direction, denoted by $\mathrm{Quad} (f)$. A 1D version can be written with an error estimate as  
\[
\int_a^b f(x)= \frac{h}{140}\Big(41f(x_0)+216f(x_1)+27f(x_2)+272f(x_3)+27f(x_4)+216f(x_5)+41f(x_6)\Big) -\frac{9}{1400}\,h^{9}\,f^{(8)}(\xi),
\]
where $\xi\in(a,b), h=\frac{b-a}{6},x_i=a+ih\ \ (i=0,\dots,6)$.
Let $h_{\beta}$ and $h_{\theta}$
denote the step size in $\beta$ and $\theta$ directions, respectively. We know that the error
is controlled by
\[ \left| \int_0^{\frac{\pi}{2}} \int_0^{\frac{\pi}{2}} f (\beta, \theta)
   \mathd \beta \mathd \theta - \mathrm{Quad} (f) \right| \leqslant
   \frac{3\pi^2}{11200} (\| \partial_{\beta}^8 f \|_{L^{\infty}} h_{\beta}^8 + \|
   \partial_{\theta}^8 f \|_{L^{\infty}} h_{\theta}^8) . \]
   
In many cases, we need to calculate the $L^2$ inner product of two functions
$f$ and $g$, given by
\[ \left| \int_0^{\frac{\pi}{2}} \int_0^{\frac{\pi}{2}} f (\beta, \theta) g
   (\beta, \theta) \sin \theta \mathd \beta \mathd \theta - \mathrm{Quad} (f g
   \sin \theta) \right| \leqslant \frac{3\pi^2}{11200} (F_{\beta} (f, g)
   h_{\beta}^8 + F_{\theta} (f, g, \sin \theta) h_{\theta}^8), \]
where the bound of $\| \partial^8_t (f_1 f_2 \cdots f_n) \|, t = \beta,
\theta$ can be derived via Leibniz's rule as
\[ F_t (f_1, f_2, \ldots, f_n) \leqslant \sum_{i_1 + \cdots + i_n = 8}
   \left(\begin{array}{c}
     8\\
     i_1, \ldots, i_n
   \end{array}\right) \prod_{j = 1}^n \| \partial_t^{i_j} f_j \|_{L^{\infty}}
   . \]
   In the subsequent subsections, we will detail how to obtain tight $L^\infty$ estimates on the functions and their derivatives.

\subsection{Certified $L^{\infty}$ bounds on trigonometric
polynomials}\label{subsec:Linf-envelope}

Let $f$ be any scalar component of $\bar{\phi}_j$, $\bar{v}$, or $\bar{U}$.
Partition the $\beta$ (and similarly for $\theta$) interval into $N_{\beta}$
subintervals $[\beta_i, \beta_{i + 1}]$ of uniform mesh size $h_{\beta}
\assign \beta_{i + 1} - \beta_i$, and denote $f_{i, j} = f (\beta_i,
\theta_j)$. In each square $I_{i, j} \assign [\beta_i, \beta_{i + 1}] \times
[\theta_j, \theta_{j + 1}]$, using the classical error bound for linear
interpolation with the Lagrange (Cauchy) form of the remainder, we obtain the following estimate:
\begin{equation}
  |f (\beta, \theta) | \leqslant \max \{|f_{i, j}|, |f_{i + 1, j}|, |f_{i, j + 1}|,
  |f_{i + 1, j + 1}|\} + \frac{h_{\beta}^2 \|f_{\beta \beta} \|_{L^{\infty}
  (I_{i, j})} + h_{\theta}^2 \|f_{\theta \theta} \|_{L^{\infty} (I_{i,
  j})}}{8}, \label{eq:f Iij est}
\end{equation}
for all $(\beta, \theta) \in I_{i, j}$, where $f_{i, j} = f (\beta_i,
\theta_j)$.

For a 1D trigonometric polynomial with the highest wavenumber $M$ we have the
classical bound via Bernstein's inequality that $\|f'' \|_{\infty} \leqslant M^2 \|f\|_{\infty}$. Taking the supremum and using Bernstein's inequality, we arrive at
\begin{equation}
  \|f\|_{\infty} \leqslant (\max_{0 \leqslant i \leqslant N_{\beta}, 0
  \leqslant j \leqslant N_{\theta}} |f_{i, j}|) / \left( 1 - \tfrac{(M_{\beta}
  h_{\beta})^2 + (M_{\theta} h_{\theta})^2}{8} \right) .
  \label{eq:linf-envelope}
\end{equation}
We apply Bernstein's inequality on a full period or a symmetry-reduced domain with the same absolute supremum. Taking $M_\beta h_\beta,M_\theta h_\theta\ll1$ ensures a positive denominator and makes the amplification factor close to $1$.

To evaluate the upper bound of derivatives of $f$, we can use similar
techniques via finite differences. We take $f_{\theta}$ as an example. We first evaluate $f$ on the
mesh grids. Then we know that for every $i, j$, there exists a $\xi_{i, j} \in
[\theta_j, \theta_{j + 1}]$ such that
\[ f_{\theta} (\beta_i, \xi_{i, j}) = \frac{f_{i, j + 1} - f_{i,
   j}}{h_{\theta}} . \]
Although we do not know the exact value of $\xi_{i, j}$, since $| \xi_{i, j +
1} - \xi_{i, j} | < 2 h_{\theta}$, we can still use the same argument to
estimate $\| f_{\theta} \|_{L^{\infty}}$. In fact, as in the previous subsection, we need to obtain sharp estimates of as high as 8th-order derivatives of functions. We can use Bernstein's inequality to get a crude estimate of the 12th-order derivatives before applying \eqref{eq:f Iij est} to get, recursively, a tighter estimate of 10th-order and 8th-order derivatives.

Recall that in our run, we use a grid in $(\beta,\theta)$ for the numerical construction of the approximate solutions. So with a fixed spectral cutoff $M$, we pick $N$ as the \emph{certification} grid $(N_\beta,N_\theta)$ used to obtain rigorous $L^\infty$ bounds for trigonometric polynomials via the interpolation/Bernstein enclosure, large enough that
\[
\frac{(M_\beta h_\beta)^2+(M_\theta h_\theta)^2}{8}<1.
\]
In our computation, we take $M_\beta=1400$ and $M_\theta=700$ for $\overline U$, and $M_\beta=600$ and $M_\theta=300$ for all other quantities. The certification grid sizes $(N_\beta,N_\theta)$ are chosen accordingly. All pointwise evaluations $f_{i,j}$ are computed with interval arithmetic to enclose round-off.

\subsection{Largest eigenvalue estimate: ${L^{\infty}}$-estimate via coordinate transform}\label{subsec:Q-envelope}

We also need a uniform eigenvalue bound for the symmetric matrix $Q$.
Its stored trigonometric representation uses $s=\rho/R_0\in[0,1]$,
and $Q$ is set to zero for $\rho>R_0=32$. On each grid node
$(s_i,\theta_j)$ we enclose the eigenvalues by interval arithmetic.
Let $D_{ss}$ and $D_{\theta\theta}$ denote the supremum norms of the
second derivatives of $Q$, measured in the Frobenius norm:
\[
 D_{ss}\assign\sup_{s,\theta}\|\partial_{ss}Q(s,\theta)\|_{\mathrm F},
 \qquad
 D_{\theta\theta}\assign\sup_{s,\theta}\|\partial_{\theta\theta}Q(s,\theta)\|_{\mathrm F},
\]
where the suprema are taken over $[0,1]\times[0,\pi/2]$.
Bilinear interpolation then gives
\[
 \sup_{\rho\le R_0}\lambda_{\max}(Q)
 \le\max_{i,j}\lambda_{\max}(Q_{i,j})
 +\frac{1}{8}
   \left(h_s^2D_{ss}+h_\theta^2D_{\theta\theta}\right)
 \le4.47 .
\]
To verify the compensation condition $Q_s+Q+\eta_1I\succeq0$,
we bound the smallest eigenvalue of $Q_s+Q$ from below.

The interpolation cells are split at the support boundary
$\rho=32$, so that no cell crosses it. We use one-sided
derivative bounds inside the support and $Q=0$ outside;
thus no $C^1$ regularity of the zero extension is needed.

We apply the same additive matrix interpolation estimate to $Q_s+Q$.
On the interpolation cells, we assume an upper bound $D_{\beta\beta}$ such that
$\|\partial_{\beta\beta}Q\|_{\mathrm F}\le D_{\beta\beta}$.
Together with the corresponding second-derivative bounds for $Q_s$
and the $\theta$-derivative bounds, this controls the interpolation error
in $(\beta,\theta)$ and yields the global bound
\[
 \lambda_{\min}(Q_s+Q)\ge-0.1909\qquad\text{throughout }\mathbb R^3.
\]
Hence $\eta_1=0.1909$ verifies \eqref{eq:Q cond} throughout the space.
To verify $Q\succeq0$, we factor
\[
Q(s,\theta)=\cos\left(\frac{\pi s}{2}\right)\widehat Q(s,\theta),
\qquad s=\rho/R_0\in[0,1].
\]
Since the cosine factor is nonnegative, it suffices to certify $\widehat Q\succeq0$ using the same matrix interpolation estimate.

\subsection{$L^{\infty}$ bounds for Bessel functions}\label{sec:bes ln}
Now we deal with Bessel functions arising from the finite-rank approximation, whose $L^\infty$ bounds are trickier than trigonometric functions. We recall the construction of our finite-rank approximation in detail in Lemma \ref{lem:lap eig}, where the  basis 
$\psi_{m l}^Y$ and $\psi_{m l}^Z$ are of the form:
\begin{eqnarray*}
  \psi_{m l}^Y & = & (U_{m l}^Y (r) Y_l (\theta), U_{m l}^X (r) X_l (\theta),
  0),\\
  U_{m l}^Y & = & B_{m l} \frac{J_{l + \frac{1}{2}} (\alpha^Y_{m l}
  r)}{r^{\frac{3}{2}}} + l A_{m l} r^{l - 1},\\
  U_{m l}^X & = & \frac{B_{m l}}{\mu_l \sqrt{r}} \left( \alpha^Y_{m l} J_{l -
  \frac{1}{2}} (\alpha^Y_{m l} r) - \frac{l J_{l + \frac{1}{2}} (\alpha^Y_{m
  l} r)}{r} \right) + A_{m l} r^{l - 1},\\
  A_{m l} & = & - ((\alpha^Y_{m l})^2 - (l - 1) (l + 2)) \frac{J_{l +
  \frac{1}{2}} (\alpha^Y_{m l} )}{(\alpha^Y_{m l})^{\frac{3}{2}}},\\
  B_{m l} & = & \frac{(\alpha^Y_{m l})^2 - l (l - 1) (l + 2)}{(\alpha^Y_{m
  l})^{\frac{3}{2}}},\\
  \psi_{m l}^Z & = & (0, 0, U_{m l}^Z (r) Z_l (\theta)),\\
  U_{m l}^Z & = & \frac{J_{l + \frac{1}{2}} (\alpha^Z_{m l} r)}{\sqrt{r}} .
\end{eqnarray*}
The quantities $\alpha_{ml}^{Y}$, $\alpha_{ml}^{Z}$, $A_{ml}$, and $B_{ml}$
are themselves computed with rigorous interval
enclosures. In particular, these constants are not treated as floating
point inputs; instead, all subsequent bounds for the basis functions
$\psi_{ml}^{Y}$, $\psi_{ml}^{Z}$ and their derivatives are propagated
from these validated enclosures.

In a similar spirit to the two previous subsections, we would like to give $L^\infty$ estimates on $\frac{J_{l + \frac{1}{2}} (r)}{r^{\frac{1}{2}}}$ and its derivatives before using a finer linear interpolation estimate like \eqref{eq:f Iij est}. From that, we can derive $L^\infty$ estimates on $\frac{\mathd^k}{\mathd r^k} U_{m l}^i
(r)$ and consequently $\psi_{m l}^Y$ and $\psi_{m l}^Z$.

We work with $\frac{J_{l + \frac{1}{2}} (r)}{r^{\frac{1}{2}}}$ since it relates to the spherical Bessel functions, which admit an integral expression via plane-wave expansions; see \eqref{int rep} and we can further differentiate under the integrand. To be specific, from the plane-wave expansion with $t=\cos\theta\in[-1,1]$, we have \[e^{izt}=\sum_{n\geq0}(2n+1)i^nj_n(z)P_n(t),\]
where $P_n$ are Legendre polynomials, $j_n(z)=\sqrt{\frac{\pi}{2z}}J_{n + \frac{1}{2}} (z)$ are spherical Bessel functions. Using the orthogonality of the Legendre polynomials, we compute \[\int_{- 1}^1 e^{i z t} P_l (t) \mathd t=2i^lj_l(z).\]
As a consequence, we derive
\begin{equation}
    \label{int rep}
 \frac{J_{l + \frac{1}{2}} (r)}{r^{\frac{1}{2}}} = \frac{i^{- l}}{\sqrt{2
   \pi}} \int_{- 1}^1 e^{i r t} P_l (t) \mathd t, \end{equation}
and we know that from Cauchy-Schwarz 
\[
  \left\| \frac{\mathd^k}{\mathd r^k} \left( \frac{J_{l + \frac{1}{2}}
  (r)}{r^{\frac{1}{2}}} \right) \right\|_{L^{\infty}}  \leqslant 
  \frac{1}{\sqrt{2 \pi}} \int_{- 1}^1 |t|^k | P_l (t) | \mathd t \leqslant  \frac{\sqrt{2}}{\sqrt{\pi (2 k + 1) (2 l + 1)}} ,
\]
where we have used the fact that $|e^{i r t}|=1$ and the normalization of the Legendre polynomials.

Therefore, we have
\[ \left\| \frac{\mathd^k}{\mathd r^k} U_{m l}^Z (r) \right\|_{L^{\infty} ([0,
   1])} \leqslant \frac{\sqrt{2} (\alpha^Z_{m l})^{k+1/2}}{\sqrt{\pi (2 k + 1) (2 l
   + 1)}} . \]
We use a recursion formula for the Bessel functions to relate $U_{m l}^Y$ and $U_{m l}^X$ again to $\frac{J_{l + \frac{1}{2}} (r)}{r^{\frac{1}{2}}}$:
\[ (2 l + 1) \text{ } \frac{J_{l + \frac{1}{2}} (r)}{r^{\frac{3}{2}}} = \text{
   } \frac{J_{l - \frac{1}{2}} (r)}{r^{\frac{1}{2}}} + \text{ } \frac{J_{l +
   \frac{3}{2}} (r)}{r^{\frac{1}{2}}}, \]
and we get
\begin{eqnarray*}
  U_{m l}^Y & = & \frac{\alpha^Y_{m l} B_{m l}}{(2 l + 1) \sqrt{r}} \left(
  J_{l - \frac{1}{2}} (\alpha^Y_{m l} r) + J_{l + \frac{3}{2}} (\alpha^Y_{m l}
  r) \right) + l A_{m l} r^{l - 1},\\
  U_{m l}^X & = & \frac{B_{m l} \alpha^Y_{m l}}{\mu_l \sqrt{r}} \left( \frac{l
  + 1}{2 l + 1} J_{l - \frac{1}{2}} (\alpha^Y_{m l} r) - \frac{l}{2 l + 1}
  J_{l + \frac{3}{2}} (\alpha^Y_{m l} r) \right) + A_{m l} r^{l - 1}.
\end{eqnarray*}
From the above formula, with all following norms taken on $[0,1]$, we obtain
\begin{eqnarray*}
  \|\partial_r^kU_{ml}^Y\|_\infty & \leqslant & \frac{\sqrt{2} (\alpha^Y_{m l})^{\frac{3}{2} + k} |B_{m
  l}|}{\sqrt{\pi (2 k + 1)} (2 l + 1)} \left( \frac{1}{\sqrt{(2 l - 1)}} +
  \frac{1}{\sqrt{(2 l + 3)}} \right) + l|A_{ml}|\|\partial_r^k r^{l-1}\|_\infty,\\
  \|\partial_r^kU_{ml}^X\|_\infty & \leqslant & \frac{\sqrt{2} (\alpha^Y_{m l})^{\frac{3}{2} + k} |B_{m
  l}|}{\sqrt{\pi (2 k + 1)} (2 l + 1) \mu_l} \left( \frac{l + 1}{\sqrt{(2 l -
  1)}} + \frac{l}{\sqrt{(2 l + 3)}} \right) + |A_{ml}|\|\partial_r^k r^{l-1}\|_\infty .
\end{eqnarray*}

Suppose we have the bound
\[ \left\| \frac{\mathd^k}{\mathd r^k} U_{m l}^i (r) \right\|_{L^{\infty}}
   \leqslant M^i_{k m l}, \quad i = X, Y, Z. \]
By Bernstein's inequality of spherical harmonics defined in Section \ref{sec:div-free}, we also have
\[ \left\| \frac{\mathd^k}{\mathd \theta^k} S_l (\theta) \right\|_{L^{\infty}}
   \leqslant l^k . \]
Therefore, we obtain
\[ \| \partial_r^j \partial_{\theta}^k \psi_{m l}^Y \|_{L^{\infty}} \leqslant
   l^k M^Y_{j m l} + l^{k + 1} M^X_{j m l}, \quad \| \partial_r^j
   \partial_{\theta}^k \psi_{m l}^Z \|_{L^{\infty}} \leqslant l^{k + 1} M^Z_{j
   m l} . \]
Since
\[ \psi^{Q, i} = \sum_{m, l} \kappa^{i, Y}_{m l} \psi_{m l}^Y (r, \theta) +
   \sum_{m, l} \kappa^{i, Z}_{m l} \psi_{m l}^Z (r, \theta), \quad i = U, v,
\]
we have
\[ \| \partial_r^j \partial_{\theta}^k \psi^{Q, i} \|_{L^{\infty}} \leqslant
   \sum_{m, l} |\kappa^{i, Y}_{m l}| (l^k M^Y_{j m l} + l^{k + 1} M^X_{j m l}) +
   \sum_{m, l} |\kappa^{i, Z}_{m l}| l^{k + 1} M^Z_{j m l} . \]

In a similar fashion, when we want to give estimates on the Bessel functions $J_{l+\frac12}(r)$ directly instead of via spherical harmonics, for example, when we want to give a point value estimate of $J$, we can proceed using the integral representation \eqref{int rep} similarly as follows.
For an interval $[a, b]$ with $0<a<b$, to compute $J_{l+\frac{1}{2}}([a,b])$, we first use the package Arblib to compute the end point values $J_{l+\frac{1}{2}}(a)$ and $J_{l+\frac{1}{2}}(b)$ rigorously, and then perform continuum enclosure between endpoints via a second-derivative remainder bound. Let $I_J$ be the affine interpolant of the two endpoint values. Then
\[
\left|J_{l+1/2}(x)-I_J(x)\right|\le\frac{(b-a)^2}{8}\sup_{z\in[a,b]}|J''_{l+1/2}(z)|,\qquad x\in[a,b].
\]
Thus the hull of the two interval endpoint values, enlarged by this remainder in both directions, encloses the entire range.
To get a rigorous enclosure of $J_{l+\frac{1}{2}}([a,b])$, we only need to bound $|| J''_{l+\frac{1}{2}}(z)  ||_{L^\infty([a,b])}$, which follows from the following bound, obtained from \eqref{int rep} and Cauchy-Schwarz
\[
    \left|J''_{l+\frac{1}{2}}(z)\right| \leqslant \sqrt{\frac{2}{(2l+1)\pi}}\left( \frac{1}{4 z^{3/2}} + \frac{1}{\sqrt{3} z^{1/2}}  + \frac{z^{1/2}}{\sqrt{5}} \right).
\]

\vspace{0.2in}

\textbf{Acknowledgements:} The research was in part supported by the NSF grants DMS-2205590 and DMS-2512878. We would like to thank  Jiajie Chen, Joel Dahne, Jia Hao, Xuefeng Liu, Xiang Qin, and Vladimir Sverak for a number of stimulating discussions. 

\appendix
\section{\texorpdfstring{Sobolev and Strain Estimates}{Sobolev and Strain Estimates}}\label{appen:sobolev}

\begin{lemma}[Sobolev estimates]\label{lem:sobolev const}
For $f,g\in H^1(\mathbb R^3;\mathbb R^m)$, we have
\begin{equation}\label{eq:sobolev-product}
 \|fg\|_{L^2}
 \le K_4\|\nabla f\|_{L^2}
       \|g\|_{L^2}^{1/2}\|\nabla g\|_{L^2}^{1/2},
 \qquad K_4:=\frac{2}{3^{3/4}\pi}.
\end{equation}
In particular, taking $g=f$ gives
\[
 \|f\|_{L^4}^2
 \le K_4\|f\|_{L^2}^{1/2}\|\nabla f\|_{L^2}^{3/2}
 \le\frac{\|f\|_{L^2}^2+\|\nabla f\|_{L^2}^2}{2\pi}.
\]
\end{lemma}

\begin{proof}
The sharp scalar Sobolev inequality \cite{talenti1976best}, applied to
$|u|$, and $|\nabla|u||\le|\nabla u|$ give
\[
 \|u\|_{L^6}\le K_s\|\nabla u\|_{L^2},\qquad K_s=\frac{2^{2/3}}{\sqrt3\,\pi^{2/3}},
 \qquad\text{for }u\in H^1(\mathbb R^3;\mathbb R^m).
\]
H\"older's inequality and interpolation therefore yield
\[
\|fg\|_{L^2}
 \le\|f\|_{L^6}\|g\|_{L^3}
 \le\|f\|_{L^6}\|g\|_{L^2}^{1/2}\|g\|_{L^6}^{1/2} \le K_4\|\nabla f\|_{L^2}
       \|g\|_{L^2}^{1/2}\|\nabla g\|_{L^2}^{1/2}.
\]
Setting $g=f$ proves the first $L^4$ bound. The second follows from
\[
 a^{1/2}b^{3/2}\le\frac{3^{3/4}}4(a^2+b^2)
 \quad(a,b\ge0),
 \qquad \frac{3^{3/4}}4K_4=\frac1{2\pi}.
\]
\end{proof}

\begin{lemma}\label{lem:strain-estimate}
Let $f\in H^1_\sigma(\mathbb R^3)$ and $S_f=(\nabla f+\nabla^Tf)/2$.
For vector fields $g,h$ with $gh:=g\otimes h\in L^2(\mathbb R^3)$,
\begin{equation}\label{eq:solenoidal-strain-bound}
 |\langle g\otimes h,S_f\rangle|
 \le\frac1{\sqrt3}\|\nabla f\|_2\|gh\|_2.
\end{equation}
\end{lemma}

\begin{proof}[Proof of Lemma~\ref{lem:strain-estimate}]
Matrix norms and contractions below are Frobenius.
Since $\nabla\cdot f=0$, we have $\operatorname{tr}S_f=0$.
For smooth, compactly supported divergence-free $f$, integration by parts gives
\[
 \sum_{i,j}\int\partial_i f_j\,\partial_j f_i\,dx
 =-\sum_j\int f_j\,\partial_j(\nabla\cdot f)\,dx=0.
\]
Expanding the square therefore yields
\[
\|S_f\|_2^2
 =\frac14\sum_{i,j}\int(\partial_i f_j+\partial_j f_i)^2\,dx =\frac12\|\nabla f\|_2^2
   +\frac12\sum_{i,j}\int\partial_i f_j\,\partial_j f_i\,dx
 =\frac12\|\nabla f\|_2^2.
\]
Approximation by smooth, compactly supported divergence-free fields
extends this identity to $f\in H^1_\sigma(\mathbb R^3)$.

Because $S_f$ is symmetric and trace-free, pointwise we have
\[
 (g\otimes h):S_f
 =\left(\frac{g\otimes h+h\otimes g}{2}
              -\frac{g\cdot h}{3}I\right):S_f.
\]
For $a,b\in\mathbb R^3$, expanding the Frobenius norm gives
\[
\left|\frac{a\otimes b+b\otimes a}{2}-\frac{a\cdot b}{3}I\right|^2
 =\left|\frac{a\otimes b+b\otimes a}{2}\right|^2
                       -\frac{(a\cdot b)^2}{3} =\frac12|a|^2|b|^2+\frac16(a\cdot b)^2 \le\frac23|a|^2|b|^2,
\]
where the last step uses $(a\cdot b)^2\le|a|^2|b|^2$.
Pointwise Cauchy--Schwarz consequently gives
\begin{equation*}
 |(g\otimes h):S_f|\le\sqrt{\frac23}\,|g||h||S_f|.
\end{equation*}
Integrating, applying Cauchy--Schwarz, and using the norm identity above yield
\[
|\langle g\otimes h,S_f\rangle|
 \le\sqrt{\frac23}\int|g||h||S_f|\,dx \le\sqrt{\frac23}\|gh\|_2\|S_f\|_2 =\frac1{\sqrt3}\|\nabla f\|_2\|gh\|_2.
\]
\end{proof}

\section{Estimates of Finite Rank Approximation}
\label{appen:lem1}
We prove Lemmas \ref{lem:quasi ort} and \ref{lem:quasi ort1} on estimates of quasi-orthonormal projection, and then prove Lemma \ref{lem:fe est} on estimates of finite-rank approximation. We recall that $\langle\cdot,\cdot\rangle_{H^1(\Omega),\eta_2}$ is the energy inner product induced by the norm  $\|u\|_{H^1(\Omega),\eta_2}^2=\eta_2\|u\|_{L^2(\Omega)}^2+\|\nabla u\|_{L^2(\Omega)}^2$, the diagonal matrix for approximate eigenvalues is $\overline\Lambda=\mathrm{diag}(\overline\lambda_1,\dots,\overline\lambda_m)$, and the Gram matrix is $G_Q=(\langle \overline\psi_i,Q\overline\psi_j\rangle_\Omega)_{1\leqslant i,j\leqslant m}$.

\begin{proof}[Proof of Lemma \ref{lem:quasi ort}]
We assume the  representation \[f^{Q,\perp}_{\tmop{lg}}=\sum_{j = 1}^m d_j\overline\psi_j=\overline{\Psi}D\,,\]
where $D\assign(d_1,\cdots,d_m)^T$, and $\overline{\Psi}\assign(\overline\psi_1,\cdots,\overline\psi_m)$. 
By definition, we have 
\begin{equation}
    \label{eq111}\|Q^{\frac{1}{2}}f^{Q,\perp}_{\tmop{lg}}\|^2_{L^2}=|(G_Q)^{\frac{1}{2}}D|^2,\quad \sum_{j = 1}^m \overline\lambda_j^{-1} \langle f, Q\overline\psi_j \rangle_{\Omega} ^2 =|\overline\Lambda^{\frac{-1}{2}}G_QD|^2.
\end{equation}
     Hence we have
\[|\|Q^{\frac{1}{2}}f^{Q,\perp}_{\tmop{lg}}\|^2_{L^2}-\sum_{j = 1}^m \overline\lambda_j^{-1} \langle f, Q\overline\psi_j \rangle_{\Omega} ^2|\leqslant\|\mathbb{1}_{m}-(G_Q)^{\frac{1}{2}}\overline\Lambda^{-1}(G_Q)^{\frac{1}{2}}\|\|Q^{\frac{1}{2}}f^{Q,\perp}_{\tmop{lg}}\|^2_{L^2}.\]This is the ``Gram mismatch'' constant $\varepsilon_3$ defined in \eqref{def-eps3} in Section \ref{sec:analysis}.
We thus conclude the proof of \eqref{fn 1}
 using the definition of projections.\end{proof}
 
\begin{proof}[Proof of Lemma \ref{lem:quasi ort1}]

Write $f=\sum_jd_j\overline\psi_j$. The residual identity and
$Q$-orthogonality give
$\langle f,u\rangle_{H^1(\Omega),\eta_2}=-\sum_jd_j\overline\lambda_j^{-1}\langle r_j^e,u\rangle_{H^1(\Omega),\eta_2}$.
By Cauchy--Schwarz,
\[
\begin{aligned}
 |\langle f,u\rangle_{H^1(\Omega),\eta_2}|
 &\le |d|\big(\sum_j\overline\lambda_j^{-2}\|r_j^e\|_{H^1(\Omega),\eta_2}^2\big)^{1/2}\|u\|_{H^1(\Omega),\eta_2},\\
 \|f\|_{H^1(\Omega),\eta_2}^2
 &=d^T{G_E}d\ge(1-\|I-{G_E}\|_{\mathrm F})|d|^2.
\end{aligned}
\]
Combining these inequalities proves the claim.

\end{proof}

\begin{proof}[Proof of Lemma \ref{lem:fe est}]
We decompose $V_N$ into $V_{\tmop{lg}}=\tmop{span}\{\overline\psi_j\}_{j=1}^m$ and its orthogonal complement in the $Q$-inner product $V_{\tmop{sm}}$. 
For any $f$, due to the orthogonality of $V_{\tmop{sm}}$ and $V_{\tmop{lg}}$, we can write an orthogonal decomposition in the $Q$-inner product as $$f=f^Q_{\tmop{or}}+f^Q=f^Q_{\tmop{or}}+f^Q_{\tmop{sm}}+f^Q_{\tmop{lg}},$$
  where $f^Q_{\tmop{or}}\perp V_N$, $f^Q_{\tmop{sm}}\in V_{\tmop{sm}}$, $f^Q_{\tmop{lg}}\in V_{\tmop{lg}}$ hence they are mutually orthogonal. In particular $f^Q_{\tmop{lg}}$ coincides with the orthogonal projection $f^{Q,\perp}_{\tmop{lg}}$.

\paragraph{Step 1: Quasi-orthogonal projection in $H^1$.}

We invoke another orthogonal decomposition in the $H^1$-inner product as 
$$f=f^{H}_{\tmop{or}}+f^{H}=f^{H}_{\tmop{or}}+f^{H}_{\tmop{sm}}+f^{H}_{\tmop{lg}},$$
Here the superscript $H$ refers to the $H^1$ inner product, and $f^{H}_{\tmop{or}}\perp V_N$, $f^{H}_{\tmop{sm}}\in V_{\tmop{sm}}$, $f^{H}_{\tmop{lg}}\in V_{\tmop{lg}}$.   Notice that $f-f^{H}_{\tmop{or}}$ is the $H^1$-projection of $f$ onto the space $V_N$, which coincides with $I_Nf$. Thus $f^{H}_{\tmop{or}}=f-I_Nf$, and 
 from \eqref{k5} we have \begin{equation}
      \label{k51}\| Q^{\frac{1}{2}} f^{H}_{\tmop{or}}
    \|_{L^2}^2\leqslant K_5 \|  f^{H}_{\tmop{or}} \|_{H^1(\Omega),\eta_2}^2.\end{equation}
    
    $V_{\tmop{sm}}$ and $V_{\tmop{lg}}$ are not orthogonal in the $H^1$-inner product, so $f^{H}_{\tmop{lg}}$ does not coincide with $f^{H,\perp}_{\tmop{lg}}$, the orthogonal projection of $f^{H}$ in the $H^1$-inner product in $V_{\tmop{lg}}$, but we show that they are close. In fact\begin{equation*}
\begin{aligned}
 0&=\langle f^{H}-f^{H,\perp}_{\tmop{lg}}, f^{H}_{\tmop{lg}}-f^{H,\perp}_{\tmop{lg}} \rangle_{H^1(\Omega),\eta_2}\\
 &=\|f^{H}_{\tmop{lg}}-f^{H,\perp}_{\tmop{lg}} \|^2_{H^1(\Omega),\eta_2}+\langle f^{H}_{\tmop{sm}}, f^{H}_{\tmop{lg}}-f^{H,\perp}_{\tmop{lg}} \rangle_{H^1(\Omega),\eta_2}.
\end{aligned}
\end{equation*}As a consequence of Lemma \ref{lem:quasi ort1}, we have \begin{equation*}
    \|f^{H}_{\tmop{lg}}-f^{H,\perp}_{\tmop{lg}} \|_{H^1(\Omega),\eta_2}\leqslant\varepsilon_4\| f^{H}_{\tmop{sm}}\|_{H^1(\Omega),\eta_2}\leqslant\varepsilon_4\| f^{H}_{\tmop{lg}}-f^{H,\perp}_{\tmop{lg}} \|_{H^1(\Omega),\eta_2}+\varepsilon_4\| f^{H}-f^{H,\perp}_{\tmop{lg}}\|_{H^1(\Omega),\eta_2},
\end{equation*}
and we solve  that 
\begin{equation}
\begin{aligned}
 \|f^{H}_{\tmop{lg}}-f^{H,\perp}_{\tmop{lg}} \|_{H^1(\Omega),\eta_2}&\leqslant\frac{\varepsilon_4}{1-\varepsilon_4}\| f^{H}-f^{H,\perp}_{\tmop{lg}}\|_{H^1(\Omega),\eta_2},\\
 \|f^{H}_{\tmop{sm}}\|_{H^1(\Omega),\eta_2}&\leqslant\frac{1}{1-\varepsilon_4}\| f^{H}-f^{H,\perp}_{\tmop{lg}}\|_{H^1(\Omega),\eta_2}\leqslant\frac{1}{1-\varepsilon_4}\| f^{H}\|_{H^1(\Omega),\eta_2}.\label{h_proj_1}
\end{aligned}
\end{equation}
 \paragraph{Step 2: Conclusion of the estimate.}

   Now
for $\mathcal{P}f=f$, the difference  consists of three parts thanks to the orthogonal decomposition,
  $$\begin{aligned}
    &\langle -(L_2 - L_{\tmop{ap}}) f, f \rangle 
    \backassign  I_1 + I_2+I_3,\quad
    I_1  \assign  \|Q^{\frac{1}{2}}(f-f^Q_{\tmop{lg}})\|_{L^2}^2,\\
    &I_2  \assign  \|Q^{\frac{1}{2}}f^Q_{\tmop{lg}}\|_{L^2}^2- \sum_{j = 1}^m \overline\lambda_j^{-1}\langle f, Q\overline\psi_j \rangle_{
  \Omega}^2,\quad
    I_3\assign- \sum_{j = 1}^m\overline{\lambda}_j^{-1} \langle f, Q \overline\psi_j
    \rangle_{\Omega} \langle f, r_j \rangle_{\mathbb R^3} .
  \end{aligned}$$
  By the $Q$-orthogonality, we have the estimate using \eqref{k51}, \eqref{constant small}, and  \eqref{h_proj_1} that\begin{equation}\begin{aligned}
  \label{fn eq1}&|I_1|\leqslant\| Q^{\frac{1}{2}} (f-f^Q_{\tmop{lg}}+f^Q_{\tmop{lg}}-f^{H}_{\tmop{lg}})
    \|^2_{L^2}=\| Q^{\frac{1}{2}} (f^{H}_{\tmop{or}}+f^{H}_{\tmop{sm}})
    \|_{L^2}^2\\&\leqslant (1 + \mu) \|Q^{\frac{1}{2}}f^{H}_{\tmop{or}}\|_{L^2}^2 + ( 1 + \frac{1}{\mu} ) \|Q^{\frac{1}{2}}f^{H}_{\tmop{sm}}\|_{L^2}^2\\&\leqslant(1 + \mu)K_5\|f^{H}_{\tmop{or}}\|_{H^1(\Omega),\eta_2}^2+ ( 1 + \frac{1}{\mu} ) \frac{K_6}{(1-\varepsilon_4)^2}\| f^{H}\|_{H^1(\Omega),\eta_2}^2\\&=(K_5+\frac{K_6}{(1-\varepsilon_4)^2})\| f\|_{H^1(\Omega),\eta_2}^2,    
  \end{aligned}
  \end{equation}
   where $\mu=\frac{K_6}{K_5(1-\varepsilon_4)^2}$ is chosen to balance the coefficients of the two terms using the $H^1$-orthogonality, leveraging the weighted AM-GM inequality. 

For $I_3$, the definition of $\Theta_r$ in Lemma~\ref{lem:fe est} gives
\begin{equation}
 |I_3|
 \le\Theta_r\|f\|_{L^2(\mathbb R^3)}^2.
 \label{est i3}
\end{equation}
The residual pairing in this definition is over $\mathbb R^3$, including
the part outside $\Omega$ where $Q$ vanishes.

Combining the estimates \eqref{fn eq1} and \eqref{est i3} with \eqref{fn 1} on $I_2$, we have   \[\begin{aligned}
    &|\langle (L_2 - L_{\tmop{ap}}) f, f \rangle |\leqslant (K_5+\frac{K_6}{(1-\varepsilon_4)^2})\| f\|_{H^1(\Omega),\eta_2}^2+(\varepsilon_3\| \lambda_{\max} (Q) \|_{L^\infty}+{\Theta_{r}})\|f\|_{L^2(\mathbb R^3)}^2.
\end{aligned}\]
Here $\varepsilon_3$ is the ''Gram mismatch'' constant from Lemma \ref{lem:quasi ort}, and $K_5,\varepsilon_4,K_6,{\Theta_{r}}$ are the constants from Section \ref{sec:analysis}.
We thus conclude the proof of the lemma.
 \end{proof}

\section{Derivation of Gradient}\label{sec:nabla lambda}

We derive the gradient of $\lambda (\widetilde{U})$ by differentiating the
eigenvalue equation with respect to $\widetilde{U}$:
\[ -\frac{\delta \mathcal{L} (\widetilde{U})}{\delta \widetilde{U}} g v^r (\widetilde{U})
   -\mathcal{L} (\widetilde{U})  \frac{\delta v^r (\widetilde{U})}{\delta \widetilde{U}} g
   = \frac{\delta \lambda (\widetilde{U})}{\delta \widetilde{U}} g v^r (\widetilde{U}) +
   \lambda (\widetilde{U})  \frac{\delta v^r (\widetilde{U})}{\delta \widetilde{U}} g, \]
where $g$ represents the descent direction. Taking the inner product with $v^l
(\widetilde{U})$, we have:
\[ -\left\langle v^l, \frac{\delta \mathcal{L} (\widetilde{U})}{\delta \widetilde{U}} g
   v^r \right\rangle + \left\langle v^l, -\mathcal{L} (\widetilde{U}) \frac{\delta
   v^r (\widetilde{U})}{\delta \widetilde{U}} g \right\rangle = \langle v^l, v^r
   \rangle \frac{\delta \lambda (\widetilde{U})}{\delta \widetilde{U}} g + \lambda
   (\widetilde{U})  \left\langle v^l, \frac{\delta v^r (\widetilde{U})}{\delta
   \widetilde{U}} g \right\rangle . \]
Since
\[ \left\langle v^l, -\mathcal{L} (\widetilde{U})  \frac{\delta v^r
   (\widetilde{U})}{\delta \widetilde{U}} g \right\rangle = \left\langle (-\mathcal{L}
   (\widetilde{U}))^{\ast} v^l, \frac{\delta v^r (\widetilde{U})}{\delta \widetilde{U}} g
   \right\rangle = \lambda (\widetilde{U})  \left\langle v^l, \frac{\delta v^r
   (\widetilde{U})}{\delta \widetilde{U}} g \right\rangle, \]
we simplify the expression by assuming $ \langle v^l, v^r \rangle \neq 0$:
\[ \frac{\delta \lambda (\widetilde{U})}{\delta \widetilde{U}} g = -\frac{\left\langle
   v^l, \frac{\delta \mathcal{L} (\widetilde{U})}{\delta \widetilde{U}} g v^r
   \right\rangle}{\langle v^l, v^r \rangle} . \]
   
Additionally, we note that
\[ \frac{\delta \mathcal{L} (\widetilde{U})}{\delta \widetilde{U}} g v^r = -\Pi (g
   \cdot \nabla v^r + v^r \cdot \nabla g) . \]
It follows via an integration by parts that
\[ \frac{\delta \lambda (\widetilde{U})}{\delta \widetilde{U}} g = \frac{\langle g
   \cdot \nabla v^r + v^r \cdot \nabla g, v^l \rangle}{\langle v^l, v^r
   \rangle} = \frac{\langle g \cdot \nabla v^r, v^l \rangle - \langle v^r
   \cdot \nabla v^l, g \rangle}{\langle v^l, v^r \rangle} . \]
   We thus conclude the proof of \eqref{eq:gradient lambda}.
\section{Expression of Operators in Spherical
Coordinates}\label{sec sphdel}
For axisymmetric scalar $f$, our velocity vector $\widetilde{U}$, and pressure $\widetilde{P}$, we know we have the following operators in the spherical coordinates:
\label{sec:ops}
\begin{eqnarray*}
\Delta_{\tmop{sph}} f & = & \frac{1}{r^2} \partial_r  (r^2 
  \partial_r f) + \frac{1}{r^2 \sin \theta} \partial_{\theta} 
  (\sin \theta\partial_{\theta}f ) ,\\
  \nabla_{\tmop{sph}}  \widetilde{U} & = & \left(\begin{array}{ccc}
    \partial_r  \widetilde{U}_r & \partial_r  \widetilde{U}_{\theta} & \partial_r 
    \widetilde{U}_{\phi}\\
    \frac{1}{r} \partial_{\theta}  \widetilde{U}_r - \frac{\widetilde{U}_{\theta}}{r}
    & \frac{1}{r} \partial_{\theta}  \widetilde{U}_{\theta} +
    \frac{\widetilde{U}_r}{r} & \frac{1}{r} \partial_{\theta}  \widetilde{U}_{\phi}\\
    - \frac{\widetilde{U}_{\phi}}{r} & - \frac{\widetilde{U}_{\phi} \cot \theta}{r} &
    \frac{\widetilde{U}_{\theta} \cot \theta}{r} + \frac{\widetilde{U}_r}{r}
  \end{array}\right),\\
  \Delta_{\tmop{sph}} \widetilde{U} & = & \left( \Delta \widetilde{U}_r - \frac{2
  \widetilde{U}_r}{r^2} - \frac{2 \partial_{\theta}  \widetilde{U}_{\theta}}{r^2} -
  \frac{2 \widetilde{U}_{\theta} \cot \theta}{r^2}, \Delta \widetilde{U}_{\theta} +
  \frac{2 \partial_{\theta}  \widetilde{U}_r}{r^2} - \frac{\widetilde{U}_{\theta}}{r^2
  \sin^2 \theta}, \Delta \widetilde{U}_{\phi} - \frac{\widetilde{U}_{\phi}}{r^2 \sin^2
  \theta} \right),\\
  \nabla_{\tmop{sph}} \widetilde{P} & = & \left( \partial_r \widetilde{P}, \frac{1}{r}
  \partial_{\theta} \widetilde{P}, 0 \right),\\
  \tmop{div}_{\tmop{sph}}  \widetilde{U} & = & \frac{1}{r^2} \partial_r  (r^2 
  \widetilde{U}_r) + \frac{1}{r \sin \theta} \partial_{\theta} 
  (\widetilde{U}_{\theta} \sin \theta) .
\end{eqnarray*}
And in variables of Section \ref{sec:ext}, we compute: 
\begin{eqnarray*}
  \widehat{\partial}_{\beta} f & \assign & \sin \beta (\cos \beta
  \partial_{\beta} f - \sin \beta f),\\
  \widehat{\nabla}_{\tmop{sph}} \widehat{U}  & \assign & \frac{\cos^2 \beta}{\sin \beta}
  \left(\begin{array}{ccc}
    \widehat{\partial}_{\beta} \widehat{U} _r & \widehat{\partial}_{\beta} \widehat{U} _{\theta} &
    \widehat{\partial}_{\beta} \widehat{U} _{\phi}\\
    \partial_{\theta} \widehat{U} _r - \widehat{U} _{\theta} & \partial_{\theta} \widehat{U} _{\theta} + \widehat{U} _r &
    \partial_{\theta} \widehat{U} _{\phi}\\
    - \widehat{U} _{\phi} & - \widehat{U} _{\phi} \cot \theta & \widehat{U} _{\theta} \cot \theta + \widehat{U} _r
  \end{array}\right),\\
  \widehat{\Delta}_{\tmop{sph}} \widehat{U} & \assign & \cot^2 \beta \cos \beta \left(
  \cot \theta \partial_{\theta} \widehat{U} + \frac{1}{4}  (2 \sin 4 \beta
  \partial_{\beta} \widehat{U} + \sin^2 2 \beta (\partial_{\beta \beta} \widehat{U} - 3 \widehat{U})) +
  \partial_{\theta \theta}  \widehat{U} \right)\\
  & + & \cot^2 \beta \cos \beta \left( - 2 \widehat{U}_r - 2 (\cot \theta \widehat{U}_{\theta} +
  \partial_{\theta} \widehat{U}_{\theta}), 2 \partial_{\theta} \widehat{U}_r -
  \frac{\widehat{U}_{\theta}}{\sin^2 \theta}, - \frac{\widehat{U}_{\phi}}{\sin^2 \theta}
  \right),\\
  \widehat{\nabla}_{\tmop{sph}} \widehat{P} & \assign & \frac{\cos^2 \beta}{\sin \beta}
  (\widehat{\partial}_{\beta} \widehat{P}, \partial_{\theta} \widehat{P}, 0),\\
  \widehat{\tmop{div}}_{\tmop{sph}} \widehat{U} & \assign & \frac{\cos^2 \beta}{\sin
  \beta} \left( \frac{1}{2} \sin 2 \beta \partial_{\beta} \widehat{U}_r + \frac{1}{2} (3
  + \cos 2 \beta) \widehat{U}_r + \frac{1}{\sin \theta} \partial_{\theta}  (\widehat{U}_{\theta}
  \sin \theta) \right) .
\end{eqnarray*}

\section{Eigenfunction in the Ball}\label{eig cons}

In this section, we aim to prove Lemma \ref{lem:lap eig} and Lemma
\ref{lem:zeros}.

\begin{proof}[Proof of Lemma \ref{lem:lap eig}]
  \paragraph{Step 1: Variational formulation.}
  
  The problem \eqref{rayleigh} can be equivalently written as
  \[ \lambda_n^s = \sup_{\dim (V_n) = n} \inf_{\|u\|^2_{\Omega} = 1, u \in
     V, u \bot V_n}  \| \nabla u\|^2_{\Omega} . \]
  We apply a variational approach to reformulate it as an eigenvalue problem.
  Consider the Lagrangian functional with multiplier $p$ and $\lambda$:
  \[ \mathfrak{L} (u, p, \lambda) = \frac{1}{2} \| \nabla u\|^2_{\Omega} -
     \frac{\lambda}{2} (\|u\|^2_{\Omega} - 1) - \langle p, \tmop{div} u
     \rangle_{\Omega} . \]
  Taking variations gives
  \[ \delta \mathfrak{L}= \langle \nabla u, \nabla \delta u \rangle_{\Omega} -
     \lambda \langle u, \delta u \rangle_{\Omega} - \langle p, \tmop{div}
     \delta u \rangle_{\Omega} = \langle - \Delta u + \nabla p - \lambda u,
     \delta u \rangle_{\Omega} + \langle \partial_{\vec{n}} u^b - p^b 
     \vec{n}, \delta u \rangle_{\partial \Omega}, \]
  where we integrate by parts in the second step. The corresponding
  eigenvalue problem reads:
  \begin{equation}
    - \Delta_{\tmop{sph}} u^b + \nabla_{\tmop{sph}} p^b = \lambda^b u^b, \quad
    \tmop{div}_{\tmop{sph}} u^b = 0, \label{eq:sph eig U}
  \end{equation}
  subject to the homogeneous Neumann boundary condition on $\partial \Omega$
  with the unit vector $\vec{n}$:
  \[ \partial_{\vec{n}} u^b - p^b  \vec{n} = 0. \]
  Here $u$ is constrained to the divergence-free subspace of axisymmetric
  vector fields.
  
  \paragraph{Step 2: Expansion in Spherical Harmonics.}
  
  We expand $u^b$ and $p^b$ in the vector spherical harmonics as follows:
  \[ u^b = \sum_{l = 0}^{+ \infty} (u_l^Y (r) Y_l (\theta) + u_l^X (r) X_l
     (\theta) + u_l^Z (r) Z_l (\theta)), \quad p^b = \sum_{l = 0}^{+ \infty}
     p_l (r) S_l (\theta) . \]
  From \eqref{eq:sph eig U}, taking the divergence yields $\Delta_{\tmop{sph}}
  p = 0$. Thanks to \eqref{eq:lap p}, we obtain for each $l$,
  \[ \Delta_r p_l (r) - \frac{\mu_l}{r^2} p_l (r) = 0. \]
  The regular solution is $p_l (r) = C_l r^l$, while the singular branch $r^{-
  l - 1}$ is discarded. Substituting into \eqref{eq:grad p}\eqref{eq:lap u},
  the eigenvalue problem becomes
  \begin{equation}
    \begin{aligned}
      & - \Delta_r u_l^Y - \frac{1}{r^2}  (2 \mu_l u_l^X - (\mu_l + 2) u_l^Y)
      + lC_l r^{l - 1} = \lambda u_l^Y,\\
      & - \Delta_r u_l^X - \frac{1}{r^2}  (2 u_l^Y - \mu_l u_l^X) + C_l r^{l
      - 1} = \lambda u_l^X, \quad - \Delta_r u_l^Z + \frac{\mu_l}{r^2} u_l^Z =
      \lambda u_l^Z . \label{eq:eig UY}
    \end{aligned}
  \end{equation}
  Using \eqref{eq:div u}, the divergence-free condition reduces to
  \begin{equation}
    \frac{\mathd}{\mathd r} u_l^Y + \frac{2}{r} u_l^Y - \frac{\mu_l}{r} u_l^X
    = 0. \label{eq:div free UX}
  \end{equation}
  And the boundary condition is
  \[ \frac{\mathd}{\mathd r} u_l^X (1) = \frac{\mathd}{\mathd r} u_l^Y (1) -
     p_l (1) = \frac{\mathd}{\mathd r} u_l^Z (1) = 0. \]
  Hence the system decouples into an eigenvalue problem for $u_l^Z$ and an
  eigenvalue system for $u_l^Y$ and $u_l^X$.
  
  \paragraph{Step 3: Solve $u_l^Y$ and $u_l^X$.}
  
  We first solve for $u_l^Y$ and $u_l^X$. When $l = 0$, the divergence-free
  condition \eqref{eq:div free UX} forces $u^Y_l = 0$. Thus, we only consider
  $l \geqslant 1$. Plugging \eqref{eq:div free UX} into \eqref{eq:eig UY}, we
  obtain:
  \begin{equation}
    - \frac{\mathd^2}{\mathd r^2} u_l^Y - \frac{4}{r}  \frac{\mathd}{\mathd r}
    u_l^Y + \frac{1}{r^2}  (l (l + 1) - 2) u_l^Y + lC_l r^{l - 1} = \lambda
    u_l^Y . \label{eq:eq UY}
  \end{equation}
  A particular solution for this equation is
  \[ u_l^Y = lA_l r^{l - 1}, \quad C_l = \lambda A_l . \]
  The homogeneous equation is related to the Bessel ODE. Its solution is:
  \[ u_l^Y = B_l  \tilde{j}_l  (kr), \quad \lambda = k^2, \quad \tilde{j}_l
     (r) \assign \frac{J_{l + \frac{1}{2}} (r)}{r^{\frac{3}{2}}}, \]
  where $J_{l + \frac{1}{2}}$ is the first kind of Bessel function. Hence the
  general solution reads
  \[ u_l^Y = B_l  \tilde{j}_l  (kr) + lA_l r^{l - 1} . \]
  It follows that
  \[ u_l^X = \frac{B_l}{\mu_l}  \left( r \frac{\mathd}{\mathd r} + 2 \right) 
     \tilde{j}_l  (kr) + A_l r^{l - 1} = \frac{B_l}{\mu_l}  (rk \tilde{j}_l'
     (kr) + 2 \tilde{j}_l (kr)) + A_l r^{l - 1} . \]
  It is not hard to check that the pair indeed solves \eqref{eq:eig UY}.
  
  Finally, we use the boundary condition to decide $k$. The boundary
  condition for $u_l^Y$ is:
  \begin{equation}
    \frac{\mathd}{\mathd r} u_l^Y (1) - p_l (1) = B_l k \tilde{j}_l' (k) -
    (k^2 - l (l - 1)) A_l = 0. \label{eq:bc UY}
  \end{equation}
  For $u_l^X$, we simplify the condition $\frac{\mathd}{\mathd r} u_l^X (1) =
  0$ into
  \[ \begin{aligned}
       \frac{\mathd}{\mathd r} u_l^X (1) & = \frac{B_l}{\mu_l}  \left( r
       \frac{\mathd^2}{\mathd r^2} + 3 \frac{\mathd}{\mathd r} \right) 
       \tilde{j}_l  (kr) |_{r = 1} + (l - 1) A_l\\
       & = - \frac{B_l}{\mu_l}  ((k^2 - (l - 1) (l + 2)) \tilde{j}_l (k) + k
       \tilde{j}_l' (k)) + (l - 1) A_l,
     \end{aligned} \]
  where we use that $\tilde{j}_l  (kr)$ is a homogeneous solution to
  \eqref{eq:eq UY}. By \eqref{eq:bc UY}, this becomes
  \[ - B_l  (k^2 - (l - 1) (l + 2))  \tilde{j}_l (k) - (k^2 - l (l - 1) (l +
     2)) A_l = 0. \]
  To ensure nontrivial solutions with $A_l$ and $B_l$, we require that
  \eqref{eq:alphaY} holds:
  \begin{equation}
    (k^2 - l (l - 1) (l + 2)) k \tilde{j}_l' (k) + (k^2 - l (l - 1))  (k^2 -
    (l - 1) (l + 2))  \tilde{j}_l (k) = 0.
  \end{equation}
  This condition determines all eigenvalues $\lambda = k^2$ and the
  corresponding coefficients $A_l$ and $B_l$ are uniquely determined up to
  normalization, when $k>0$:
  \[ A_l = - (k^2 - (l - 1) (l + 2))  \tilde{j}_l (k), \quad B_l = k^2 - l (l
     - 1)  (l + 2) . \]
  
  \paragraph{Step 4: $k = 0$ case.}
  
  Testing the variational eigenproblem with $u$ at $\lambda=0$ gives $\|\nabla u\|_2^2=0$. Hence $u$ is a Cartesian constant, and axisymmetry leaves only multiples of $e_z$. Thus the unique axisymmetric zero mode is
  \[ \psi_0^s=(\cos\theta,-\sin\theta,0). \]
  It has odd vector reflection symmetry, so it is retained in the $v$ sector and is absent from the even $U$ sector.
  
  \paragraph{Step 5: Solve $u_l^Z$.}
  
  We now turn to the component $u_l^Z$. The corresponding eigenpairs are given
  explicitly by
  \[ u_l^Z = \frac{J_{l + \frac{1}{2}}  (kr)}{\sqrt{r}}, \quad \lambda = k^2 .
  \]
  The boundary condition yields
  \begin{equation}
    \frac{\mathd}{\mathd r} u_l^Z (1) = kJ_{l + \frac{1}{2}}' (k) -
    \frac{1}{2} J_{l + \frac{1}{2}} (k) = 0,
  \end{equation}
  which characterizes all admissible values of $k$, and therefore determines
  the corresponding eigenvalues $\lambda = k^2$. In particular, $k = 0$ is
  excluded since $u_l^Z$ vanishes, so $\lambda = 0$ is not an eigenvalue in
  this case.
  
  \paragraph{Summary}
  
  In summary, let $\alpha^Y_{ml}$ be the $m$-th nonzero root of equation
  \eqref{eq:alphaY} and $\alpha^Z_{ml}$ be the $m$-th root of equation
  \eqref{eq:alphaZ}. Then the eigenpairs are:
  The intermediate coefficient $B_l$ above equals $(\alpha_{ml}^Y)^{3/2}B_{ml}$ in the following normalized notation.
  \[ \lambda_{ml}^Y=(\alpha_{ml}^Y)^2,\quad
  A_{ml}=-\bigl((\alpha_{ml}^Y)^2-(l-1)(l+2)\bigr)\tilde j_l(\alpha_{ml}^Y),\qquad
 B_{ml}=\frac{(\alpha_{ml}^Y)^2-l(l-1)(l+2)}{(\alpha_{ml}^Y)^{3/2}}.
  \]
  \[\begin{aligned}
 \psi_{ml}^{Y}={}&\bigl((\alpha_{ml}^Y)^{3/2}B_{ml}\tilde j_l(\alpha_{ml}^Yr)+lA_{ml}r^{l-1}\bigr)Y_l(\theta)\\
 &+\left[\frac{(\alpha_{ml}^Y)^{3/2}B_{ml}}{\mu_l}\bigl(r\alpha_{ml}^Y\tilde j_l'(\alpha_{ml}^Yr)+2\tilde j_l(\alpha_{ml}^Yr)\bigr)+A_{ml}r^{l-1}\right]X_l(\theta).
\end{aligned}\]
  and
  \[ \lambda_{ml}^Z = (\alpha^Z_{ml})^2, \quad \psi_{ml}^Z = \frac{J_{l +
     \frac{1}{2}}  (\alpha^Z_{ml} r)}{\sqrt{r}} Z_l (\theta) . \]
  Plus a zero eigenmode $(\lambda_0, \psi_0) = (0, (\cos \theta, - \sin
  \theta, 0))$.
\end{proof}

\begin{remark}
  In practice, it is often inconvenient to evaluate the derivative $J_{l +
  \frac{1}{2}}' (r)$ directly. Instead, we make use of the recurrence identity
  \[ J_{l + \frac{1}{2}}' (r) = J_{l - \frac{1}{2}} (r) - \frac{l +
     \frac{1}{2}}{r} J_{l + \frac{1}{2}} (r) . \]
  Substituting this relation, equation \eqref{eq:alphaY} can be rewritten as:
  \[ k (k^2 - l (l - 1) (l + 2)) J_{l - \frac{1}{2}} (k) + (k^4 - l (2 l + 1)
     k^2 + l (l - 1) (l + 2) (2 l + 1)) J_{l + \frac{1}{2}} (k) = 0, \]
  while equation \eqref{eq:alphaZ} becomes:
  \[ kJ_{l - \frac{1}{2}} (k) - (l + 1) J_{l + \frac{1}{2}} (k) = 0. \]
  Similarly, the expression for $u_l^X$ takes the form:
  \[ u_l^X = \frac{B_l}{\mu_l k^{3/2}\sqrt r}
  \left(kJ_{l-1/2}(kr)-\frac{lJ_{l+1/2}(kr)}r\right)+A_lr^{l-1}. \]

\end{remark}

\begin{proof}[Proof of Lemma \ref{lem:zeros}]
  First we consider the equation \eqref{eq:alphaY}. Since $J_{l + \frac{1}{2}}
  (\gamma) \neq 0$ for $l \geqslant 2$, it can be written in the form:
  \[ \frac{kJ_{l + \frac{1}{2}}' (k)}{J_{l + \frac{1}{2}} (k)} - \frac{3}{2} +
     \frac{(k^2 - l (l - 1))  (k^2 - (l - 1) (l + 2))}{(k^2 - l (l - 1) (l +
     2))} = 0. \]
  If $l = 1$, the equation becomes:
  \[ \frac{kJ_{l + \frac{1}{2}}' (k)}{J_{l + \frac{1}{2}} (k)} - \frac{3}{2} +
     k^2 = 0. \]
  Define:
  \[ R_l (k) = \frac{kJ_{l + \frac{1}{2}}' (k)}{J_{l + \frac{1}{2}} (k)},
     \quad T_l (k) = \frac{(k^2 - l (l - 1))  (k^2 - (l - 1) (l + 2))}{(k^2 -
     l (l - 1) (l + 2))}, \quad F_l (k) = R_l (k) - \frac{3}{2} + T_l (k) . \]
  We observe that $F_l (k)$ is continuous except at $\tilde{\beta}_{l, j}$ ($j
  \geqslant 1$). If $\tilde{\beta}_{l, j}$ is a zero of $J_{l + \frac{1}{2}}
  (x)$, then
  \[ \lim_{x \rightarrow \tilde{\beta}_{l, j} +} R_l (x) = + \infty, \quad
     \lim_{x \rightarrow \tilde{\beta}_{l, j} -} R_l (x) = - \infty, \]
  and $T_l (\tilde{\beta}_{l, j})$ is finite. If $\tilde{\beta}_{l, j}$ is
  $\gamma$, we have
  \[ \lim_{x \rightarrow \tilde{\beta}_{l, j} +} T_l (x) = + \infty, \quad
     \lim_{x \rightarrow \tilde{\beta}_{l, j} -} T_l (x) = - \infty, \]
  and $R_l (\tilde{\beta}_{l, j})$ is finite. Therefore, for every
  $\tilde{\beta}_{l, j}$, in both cases we obtain 
  \[ \lim_{x \rightarrow \tilde{\beta}_{l, j} +} F_l (x) = + \infty, \quad
     \lim_{x \rightarrow \tilde{\beta}_{l, j} -} F_l (x) = - \infty . \]
  At $\tilde{\beta}_{l, 0} = 0$, the power series
  expansion
  \[ J_{l + \frac{1}{2}} (x) = \frac{(x / 2)^{l + \frac{1}{2}}}{\Gamma (l +
     \frac{3}{2})}  (1 - \frac{x^2}{4 (l + \frac{3}{2})} + O (x^4)) \]
  shows that
  \[ F_l (k) = \frac{3 (1 + l)^2}{l (2 + l)  (3 + 2 l)} k^2 + o (k^2) . \]
  Therefore, we obtain $F_l (k) > 0$ for sufficiently small $k$. We conclude that $F_l
  (k)$ has at least one root on every interval $(\tilde{\beta}_{l, j},
  \tilde{\beta}_{l, j + 1})$ with $j \geqslant 0$. The case $l = 1$ follows by
  a similar argument.
  
  Next we derive a lower bound for the first root, and prove that assuming two
  zeros in the same interval leads to a contradiction. Since $F_l 
  (\tilde{\beta}_{l, j} +) = \infty$ and $F_l  (\tilde{\beta}_{l, j + 1} -) =
  - \infty$, we denote $k_2 > k_1$ as the first two roots in
  $(\tilde{\beta}_{l, j}, \tilde{\beta}_{l, j + 1})$, which implies that $F_l'
  (k_1) \leqslant 0$.
  
  Straightforward calculation shows that
  \[ R_l' (k) = - \frac{R_l (k)^2 + k^2 - \left( l + \frac{1}{2}
     \right)^2}{k}, \quad T'_l (k) = 2 k \left( 1 - \frac{l (l - 1)^3 (l + 1)
     (l + 2)}{(k^2 - l (l - 1) (l + 2))^2} \right) . \]
  When $F_l (k) = 0$, we have $R_l (k) = \frac{3}{2} - T_l (k)$. Thus $F_l'$
  simplifies into:
  \[ F_l' (k) = - \frac{\left( \frac{3}{2} - T_l (k) \right)^2 + k^2 - \left(
     l + \frac{1}{2} \right)^2}{k} + T'_l (k) = - \frac{kQ_l (k^2)}{(k^2 - l
     (l - 1) (l + 2))^2}, \]
  where
  \[ Q_l (s) = s^3 - 4 l^2 s^2 + l (l - 1)  (6 l^2 + 11 l + 7) s - 3 l (l -
     1)^2  (l + 1)^2  (l + 2) . \]
  It is easy to verify that for $l \geqslant 2$, we have by quadratic
  equations that
  \[ Q_l' (s) = 3 s^2 - 8 l^2 s + l (l - 1)  (6 l^2 + 11 l + 7) > 0. \]
  Therefore, $Q_l (s)$ has a unique root $s_l > 0$ such that $Q_l (s) < 0$ for
  $s < s_l$, and $Q_l (s)$ is positive and strictly increasing for $s > s_l$.
  
  For $l = 1$, we have $Q_1 (s) = s^3 - 4 s^2$, which satisfies the same
  property.
  
  Since $F_l' (k_1) \leqslant 0$, it follows that
  \[ k_1 \geqslant \sqrt{s_l}, \]
  and that $0 \leqslant Q_l (k_1^2) < Q_l (k_2^2)$. Hence, we obtain $F_l'
  (k_2) < 0$. Thus $F_l$ cannot cross the axis at $k_1$, and it has to hold
  that $F_l (k_1) = F_l' (k_1) = 0$. We verify that $F_l (\sqrt{s_l}) \neq 0$.
  Hence, we reach a contradiction and conclude the proof of the lemma.
  
  By a similar but simpler argument, the equation \eqref{eq:alphaZ} has
  exactly one root in every interval $(\beta_{l, j}, \beta_{l, j + 1})$ for
  all $l \geqslant 1, j \geqslant 0$. Moreover the first root satisfies the
  lower bound
  \[ \alpha_{l, 1}^Z \geqslant \sqrt{l (l + 1)}, \]
  provided that $\sqrt{l (l + 1)}$ is not a root.
\end{proof}

\bibliographystyle{abbrv} 
\bibliography{ref}

\end{document}